\pdfoutput=1
\documentclass[a4paper,twoside,12pt]{report}

\usepackage[utf8]{inputenc} 
\usepackage[T1]{fontenc} 
\usepackage{textcomp} 
\usepackage{hyperref} 
\usepackage{tikz} 

\usepackage{verbatim} 
\usepackage{calc} 
\usepackage{fmtcount} 
\usepackage{remreset} 
\usepackage{relsize} 
\usepackage{fancyhdr} 
\usepackage{paralist} 

\usepackage{amsmath}
\usepackage{amssymb}
\usepackage{amsthm}

\usepackage{pifont} 
\usepackage{lettrine} 

\hbadness=\tolerance 
\usepackage{fouriernc} 

\swapnumbers 

\theoremstyle{definition}

\newtheorem{Def}{Definition}[section] 
\newtheorem{Not}[Def]{Notation}
\newtheorem{Hyp}[Def]{Assumption}

\theoremstyle{plain}

\newtheorem{Pro}[Def]{Proposition}
\newtheorem{Clm}[Def]{Claim}
\newtheorem{Lem}[Def]{Lemma}
\newtheorem{Thm}[Def]{Theorem}
\newtheorem{Cor}[Def]{Corollary}

\theoremstyle{remark}

\newtheorem{Rmk}[Def]{Remark}
\newtheorem{Xpl}[Def]{Example}

\makeatletter\@removefromreset{equation}{chapter}\makeatother 
\makeatletter\@removefromreset{figure}{section}\makeatother 

\newcommand\mbb{\mathbb}
\newcommand\mbf{\mathbf}
\newcommand\mcal{\mathcal}
\newcommand\mfrak{\mathfrak}
\newcommand\mrm{\mathrm}
\newcommand\msf{\mathsf}

\newcommand\NN{\mbb{N}}

\newcommand\RR{\mbb{R}}
\newcommand\ZZ{\mbb{Z}}

\newcommand\PP{\Pr}
\newcommand\EE{\mathop{\smash{\mbf{E}}}\nolimits}
\DeclareMathOperator\Var{Var}

\DeclareMathOperator\Law{\mathit{Law}}

\newcommand\ab{\allowbreak}

\renewcommand\epsilon{\varepsilon}
\renewcommand\phi{\varphi}
\renewcommand\geq{\geqslant}
\renewcommand\leq{\leqslant}
\renewcommand\setminus{\smallsetminus}
\renewcommand\ell{l}

\renewcommand\Pr{\mathop{\smash{\mbf{P}}}\nolimits}
\renewcommand\ln{\log}
\renewcommand\Re{\mathop{\mfrak{Re}}}
\renewcommand\Im{\mathop{\mfrak{Im}}}
\renewcommand\arcsin{\mathop{\mathrm{Arcsin}}}
\renewcommand\arccos{\mathop{\mathrm{Arccos}}}

\renewcommand\limsup{\varlimsup}
\renewcommand\liminf{\varliminf}

\renewcommand\colon{\:\mathpunct{:}\enskip}
\renewcommand\[{\begin{equation}}
\renewcommand\]{\end{equation}}
\renewenvironment{itemize}{\vspace{-.5\parskip}\begin{compactitem}}{\end{compactitem}\vspace{-.5\parskip}}
\renewenvironment{enumerate}{\vspace{-.5\parskip}\begin{compactenum}}{\end{compactenum}\vspace{-.5\parskip}}

\makeatletter

\let\wedge@TeX=\wedge
\renewcommand\wedge{\mathbin{\text{\relsize{-1}$\wedge@TeX$}}}

\let\vee@TeX=\vee
\renewcommand\vee{\mathbin{\text{\relsize{-1}$\vee@TeX$}}}

\let\bigtriangleup@TeX=\bigtriangleup
\renewcommand\bigtriangleup{\mathbin{\text{\relsize{-1}$\bigtriangleup@TeX$}}}

\usepackage{accents} 

\def\CM@fonts#1{\begingroup\fontfamily{cmr}\selectfont#1\endgroup}

\def\acute@low{\CM@fonts{\relsize{+4}\raisebox{-1.2ex}{\textasciiacute}}}
\def\bar@low  {\CM@fonts{\relsize{+2}\raisebox{-1.4ex}{\textasciimacron}}}
\def\breve@low{\CM@fonts{\relsize{+3}\raisebox{-1.2ex}{\textasciibreve}}}
\def\check@low{\CM@fonts{\relsize{+3}\raisebox{-1.2ex}{\textasciicaron}}}
\def\ddot@low {\CM@fonts{\relsize{+4}\raisebox{-0.0ex}{.\hspace{-.1em}.}}}
\def\dot@low  {\CM@fonts{\relsize{+4}\raisebox{-0.0ex}{.}}}
\def\grave@low{\CM@fonts{\relsize{+4}\raisebox{-1.2ex}{\textasciigrave}}}
\def\hat@low  {\CM@fonts{\relsize{+3}\raisebox{-1.0ex}{\textasciicircum}}}
\def\tilde@low{\CM@fonts{\relsize{+3}\raisebox{+0.1ex}{\texttildelow}}}
\def\vec@low  {\CM@fonts{\relsize{+0}\raisebox{-0.0ex}{$\rightarrow$}}}

\renewcommand\acute[1]{\accentset{\smash{\text{\acute@low}}\rule{0pt}{.3ex}}{#1}}
\renewcommand\bar  [1]{\accentset{\smash{\text{\bar@low  }}\rule{0pt}{.3ex}}{#1}}
\renewcommand\breve[1]{\accentset{\smash{\text{\breve@low}}\rule{0pt}{.3ex}}{#1}}
\renewcommand\check[1]{\accentset{\smash{\text{\check@low}}\rule{0pt}{.3ex}}{#1}}
\renewcommand\ddot [1]{\accentset{\smash{\text{\ddot@low }}\rule{0pt}{.3ex}}{#1}}
\renewcommand\dot  [1]{\accentset{\smash{\text{\dot@low  }}\rule{0pt}{.3ex}}{#1}}
\renewcommand\grave[1]{\accentset{\smash{\text{\grave@low}}\rule{0pt}{.3ex}}{#1}}
\renewcommand\hat  [1]{\accentset{\smash{\text{\hat@low  }}\rule{0pt}{.3ex}}{#1}}
\renewcommand\tilde[1]{\accentset{\smash{\text{\tilde@low}}\rule{0pt}{.3ex}}{#1}}
\renewcommand\vec  [1]{\accentset{\smash{\text{\vec@low  }}\rule{0pt}{.3ex}}{#1}}

\let\enskip@TeX=\enskip
\renewcommand\enskip{\penalty-0\enskip@TeX\relax}

\let\quad@TeX=\quad
\renewcommand\quad{\penalty-0\quad@TeX\relax}

\renewcommand\big[1]{\text{\relsize{+1}\raisebox{-.09ex}{$#1$}}}
\renewcommand\Big[1]{\text{\relsize{+2}\raisebox{-.17ex}{$#1$}}}
\let\bigg=\bigg 

\newcount\@DT@loopN
\newcount\@DT@modctr
\newcount\@DT@X
\newcommand*{\@fnsymbol@six}[1]{%
\@DT@loopN=#1\relax
\advance\@DT@loopN by -1\relax
\divide\@DT@loopN by 6\relax
\@DT@modctr=\@DT@loopN
\multiply\@DT@modctr by 6\relax
\@DT@X=#1\relax
\advance\@DT@X by -1\relax
\advance\@DT@X by -\@DT@modctr
\advance\@DT@loopN by 1\relax
\advance\@DT@X by 1\relax
\@fnsymbol\@DT@X
}
\providecommand*{\fnsymbol@six}[1]{%
\expandafter\protect\expandafter\@fnsymbol@six{%
\expandafter\the\csname c@#1\endcsname}}
\renewcommand\fnsymbol\fnsymbol@six

\let\figurename@TeX=\figurename%
\renewcommand\figurename{\small\figurename@TeX}

\makeatother

\newcommand\coloneqq{\mathrel{\text{\rlap{\raisebox{-.25ex}{$\cdot$}}\raisebox{+.25ex}{$\cdot$}$=$}}} 
\newcommand\eqqcolon{\mathrel{\text{$=$\rlap{\raisebox{-.25ex}{$\cdot$}}\raisebox{+.25ex}{$\cdot$}}}} 
\newcommand\dx[1]{d\mspace{-2.mu}\mathord{#1}} 
\newcommand\1[1]{\mbf{1}_{\text{\relsize{+1}$#1$}}} 
\newcommand\pushfwd{\mathbin{\text{\relsize{-1}$\#$}}} 
\newcommand\Bcdot{\mathord{\Big\cdot}} 

\def\clap#1{\hbox to 0pt{\hss#1\hss}}
\def\mathclap{\mathpalette\mathclapinternal}
\def\mathclapinternal#1#2{%
\clap{$\mathsurround=0pt#1{#2}$}}

\newenvironment{ienumerate}{\vspace{-.5\parskip}\begin{compactenum}[(i)]}{\end{compactenum}\vspace{-.5\parskip}} 
\newenvironment{NOTA}{\par\noindent\rm\ding{42}\quad\it}{\par} 

\author{R\'emi Peyre}
\title{Tensorizing maximal correlations}
\date{December 30, 2010}

\usepackage{tikz}
\usetikzlibrary{calc,shapes.misc}
\newcommand{\minisquare}[1]{\draw #1 +(-.7pt,-.7pt) rectangle +(+.7pt,+.7pt);}
\newcommand{\rightangle}[3]{\draw (#2)
	let	\p1 = ($(#1)-(#2)$), \p3 = ($(#3)-(#2)$), \n1 = {veclen(\p1)}, \n3 = {veclen(\p3)}
	in	+($5/\n1*(\p1)$) -- +($5/\n1*(\p1) + 5/\n3*(\p3)$) -- +($5/\n3*(\p3)$);}

\newcommand\dist{\mathit{dist}}
\newcommand\C[1]{#1^{\msf{c}}}
\newcommand\vechat[1]{\vec{\hat{#1}}} 
\newcommand\longto{\to} 
\newcommand{\ldb}{\smash{\bar{L}}^2}
\DeclareMathOperator\ecty{Sd}
\DeclareMathOperator\Cov{Cov}
\DeclareMathOperator\Corr{Corr}
\newcommand\footrel[2]{\underset{#1}{#2}}
\newcommand\T[1]{#1^{\msf{\!T}}}
\newcommand\onecirc{\raisebox{-.25ex}{\large\ding{172}}} 
\newcommand\twocirc{\raisebox{-.25ex}{\large\ding{173}}} 
\newcommand\threecirc{\raisebox{-.25ex}{\large\ding{174}}} 
\newcommand{\barLip}{\overline{\mathit{Lip}}}
\newcommand\boldepsilon{\boldsymbol\epsilon} 
\newcommand{\notR}{\mcal{R}\llap{\raisebox{-.22ex}{$\Big/$}$\mspace{3.mu}$}}
\DeclareMathOperator\valency{val}
\DeclareMathOperator\II{II}
\DeclareMathOperator\volume{vol}
\newcommand\teps{\tilde\epsilon} 
\newcommand\offset{\mathit{offset}}
\DeclareMathOperator*{\argmin}{arg\,min}

\begin{document}


\maketitle

\begin{abstract}
The \emph{maximal \emph{(or \emph{Hilbertian})} correlation coefficient} between two random variables $X$ and~$Y$, denoted by ${\{X:Y\}}$, is the supremum of the $\big|\Corr\big(f(X),g(Y)\big)\big|$ for real measurable functions~$f,g$, where ``$\Corr$'' denotes Pearson's correlation coefficient.
It is a classical result that for independent pairs of variables $(X_i,Y_i)_{i\in I}$, ${\{\vec{X}_I : \vec{Y}_I\}}$ is the supremum of the~${\{X_i:Y_i\}}$. The main goal of this monograph is to prove similar tensorization results when one only has partial independence between the~$(X_i,Y_i)$; more generally, for random variables $(X_i)_{i\in I}, \ab (Y_j)_{j\in J}$, we will look for bounds on~${\{\vec{X}_I:\vec{Y}_J\}}$ from bounds on the~${\{X_i:Y_j\}}, \enskip i\in I, \ab j\in J$.

My tensorization theorems will imply new decorrelation results for models of statistical physics exhibiting asymptotic independence, like the subcritical Ising model. I shall prove that for such models, two distant bunches of spins are decorrelated (in the Hilbertian sense) \emph{uniformly in their sizes and shapes}: if $I$ and~$J$ are two sets of spins such that $\dist(i,j) \geq d$ for all~$i\in I, \ab j\in J$, then one gets a nontrivial bound for~${\{\vec{X}_I:\vec{Y}_J\}}$ only depending on~$d$.

Still for models like the subcritical Ising one, I shall also prove how Hilbertian decorrelations may be used to get the spatial CLT or the (strict) positiveness of the spectral gap for the Glauber dynamics, \emph{via} tensorization techniques again.

Besides all that, I shall finally prove a new criterion to bound the maximal correlation~${\{\mcal{F}:\mcal{G}\}}$ between two $\sigma$-algebras~$\mcal{F}$ and~$\mcal{G}$ form a uniform bound on the $\big|{\Pr[A\cap B]}-{\Pr[A]\Pr[B]}\big| \mathbin{\big/} \sqrt{\Pr[A]\Pr[B]}$ for all~$A\in\mcal{F}, \ab B\in\mcal{G}$. Such criteria were already known, but mine strictly improves those and can moreover be proved to be optimal.
\end{abstract}

\chapter*{Introduction}%
\addcontentsline{itc}{chapter}{Introduction}

\section*{Overview of the monograph}

This monograph is devoted to the study
of Hilbertian correlations (also called ``maximal correlations'' or ``$\rho$-mixing coefficients''),
in particular to showing how this concept can be `tensorized'
to yield new results on systems of statistical mechanics exhibiting asymptotic independence.
I have divided it into six chapters:
\begin{itemize}
\item The first chapter, numbered ``\ref{parMotivation}'',
aims at motivating the study of Hilbertian correlations and their tensorization.
In this chapter, I will recall some classical results on the subcritical Ising model,
which is a classical model showing asymptotic independence between pairs of spins.
When one gets interested in very large `bunches' of spins,
it is known that asymptotic independence cannot be captured by $\beta$-mixing any more,
but that, in certain cases at least, it still holds in terms of $\rho$-mixing.
The techniques used so far to establish $\rho$-mixing for bunches of spins
are strongly limited by technical assumptions looking somehow artificial,
which will motivate studying $\rho$-mixing `for itself' and trying to tensorize it.
\item In Chapter~\ref{parBasicHilbertianDecorrelations},
I shall recall the definition of the Hilbertian correlation coefficient;
I shall also recall some classical facts about this concept and give some examples.
This chapter can be seen as a `crash course' on $\rho$-mixing for the non-specialist reader:
almost nothing in it is new.
\item In Chapter~\ref{parEventSufficientConditions}, I shall give some new criteria
to bound the Hilbertian correlation between two $\sigma$-algebras,
which criteria assume bounds on the~$(\Pr[A\cap B] - \Pr[A]\Pr[B])$
for events~$A$ and~$B$ belonging to these respective $\sigma$-algebras.
My ``strong event sufficient condition'', which improves previous results by several authors,
shall even be shown to be optimal.
\item Chapter~\ref{parTensorization} is the core of this work:
in it I will handle tensorization of Hilbertian decorrelations.
This chapter begins with a refined version of the concept of correlation,
called ``subjective correlation'', which is necessary to write the subsequent tensorization results.
Then I shall state and prove my three main tensorization theorems:
Theorem~\ref{thm5750} (`$N$~against~$1$' theorem)
bounds the correlation between
a `simple' and a `vector' variable;
Theorem~\ref{thm6413} (`$N$~against~$M$' theorem)
deals with correlation between two vector variables,
and Theorem~\ref{thm0119} (`$\ZZ$~against~$\ZZ$' theorem)
refines the previous one in the case where certain symmetries are present.
Then, I will discuss some refinement and optimality statements
about these theorems; in \S~\ref{par3lines},
I will also present a geometric corollary of tensorization results
which underlines quite well the Hilbertian aspect of maximal correlations.
\item In Chapter~\ref{parOtherApplicationsOfTensorizationTechniques},
I will continue to use the tensorization techniques of Chapter~\ref{parTensorization},
but that time instead of proving tensorization results \emph{stricto sensu}
I will turn to different types of results, namely the spatial central limit theorem
and the presence of spectral gap for the Glauber dynamics.
\item Finally, Chapter~\ref{parApplications} will present some concrete applications
of the results of this monograph. For instance, I shall prove new results about decorrelation
between distant bunches of spins in Ising's model (see Theorem~\ref{t3730});
I will also give results of the same type for quite general models of statistical mechanics (see e.g.\ Theorems~\ref{thmTchou} and~\ref{thmLama}),
also proving spatial CLT and spectral gap for the Glauber dynamics for these models.
I will also show how tensorization of Hilbertian correlations can be used
to get `hypocoercivity' results~[Theorem~\ref{thmParis}].
\end{itemize}

\section*{Conventions and notation}

Notation will not always be perfectly rigorous:
to make reading easier, it may occur sometimes that formalism is slightly loose,
or that some writing conventions or assumptions are implicit.
However this shall only be done in situations
where adding the missing information by the reader is (hopefully) obvious.

Here is some notation used throughout this text:

\paragraph{Miscellaneous}

\begin{itemize}
\item The symbol $\NN$ denotes the set of nonnegative integers, including $0$.
The set of positive integers $\NN\setminus\{0\}$ is denoted by~$\NN^*$.
\item For~$a,b$ real numbers, $a\wedge b$ denotes $\min\{a,b\}$,
resp.~$a\vee b$ denotes $\max\{a,b\}$; $a_+$ denotes the positive part of~$a$,
i.e.\ $a\vee0$.
\item For~$A$ a set, $\C{A}$ denotes the complement set of~$A$
(the set of reference shall always be clear);
$\1{A}$ denotes the indicator function of~$A$,
that is, the function being~$1$ on~$A$ and~$0$ on~$\C{A}$.
\item For~$A,B$ sets, $A\bigtriangleup B$ denotes the symmetric difference of~$A$ and~$B$,
i.e.\ ${(A\setminus B)}\uplus(B\setminus A)$,
where ``$\uplus$'' means the same as ``$\cup$'', but with underlining that the union is disjoint.
\item The identity matrix in dimension $n$ will be denoted by~$\mathbf{I}_n$.
The transpose of a matrix $A$ will be denoted by~$\T{A}$.
\item If $\Theta$ is a set endowed with a metric $\dist$, then for
$I,J\subset\Theta$, $\dist(I,J)$ denotes the distance between~$I$ and~$J$, that is,
$\dist(I,J) \coloneqq \inf\{\dist(i,j) \colon {i\in I}, \ab {j\in J}\}$.
\item As is customary in physical literature, $\propto$ means ``proportional to''.
\item Whenever $I$ is a set and $X$ a symbol,
$\vec{X}_I$ will be a shorthand for ``$(X_i)_{i\in I}$''.
\end{itemize}

\paragraph{Probability}

\begin{itemize}
\item We will always work on an implicit probability space $(\Omega,\mcal{B})$
equipped with a probability measure~$\Pr$. Sub-$\sigma$-algebras of~$\mcal{B}$
will be merely called ``$\sigma$-algebras'';
I will also often write ``variable'' for ``random variable''.
Unless explicitly specified, variables on~$\Omega$ can be valued in any set.
\item If $f$ is a real random variable, the expectation of~$f$ is denoted by~$\EE[f]$;
its variance is denoted by~$\Var(f)$; its standard deviation is denoted
by $\ecty(f) \coloneqq \sqrt{\Var(f)}$; if $g$ is another real variable,
the covariance between~$f$ and~$g$ is denoted by~$\Cov(f,g) \coloneqq \EE[fg] - \EE[f]\EE[g]$.
All that notation extends to the case where $f$ and~$g$ are valued
in some vector space $\RR^N$, except that in that case it refers to vectors or matrices.
\item If $B$ is an event with $\Pr[B]>0$, then $\Pr[A|B]$, $\EE[f|B]$, $\Var(f|B)$, \dots\ 
stand resp.\ for the probability of~$A$, the expectation of~$f$, the variance of~$f$, \dots\ 
under the conditional law~$\dx\Pr[\Bcdot|B] \coloneqq \1{B}\,\dx\Pr[\Bcdot] \div\Pr[B]$.
Similarly, if $\mcal{F}$ is a $\sigma$-algebra,
$\Pr[A|\mcal{F}]$, $\EE[f|\mcal{F}]$, \dots\ stand for the conditional probability of~$A$,
the conditional expectation of~$f$, \dots\ w.r.t.~$\mcal{F}$.
\item Concerning conditional expectations, I will actually use two different conventions:
for~$\mcal{G}$ a $\sigma$-algebra, $\EE[f|\mcal{G}]$ can also be denoted by~$f^{\mcal{G}}$.
Both conventions can be used inside the same formula.%
\footnote{The use of the first or the second convention will depend on
the way we prefer to see the conditional expectation of~$f$ w.r.t.~$\mcal{G}$:
if it is rather seen as the expectation of~$f$ knowing the information of~$\mcal{G}$,
notation $\EE[f|\mcal{G}]$ will be chosen, while if it is more seen like
the $\mcal{G}$-measurable function best approximating $f$, we will use the notation $f^{\mcal{G}}$.}%
\footnote{One must not confuse $\Var(f|\mcal{G})$,
which is the variance of~$f$ under the law $\Pr[\Bcdot|\mcal{G}]$,
with~$\Var(f^{\mcal{G}})$ which is the (unconditioned) variance of the random variable $f^{\mcal{G}}$.
One has the well-known identity $\Var(f) = \Var(f^{\mcal{G}}) + \EE[\Var(f|\mcal{G})]$,
which I shall refer to as \emph{associativity of variance}.}
\item If $X$ is a variable on~$\Omega$, the $\sigma$-algebra generated by~$X$
(that is, the smallest $\sigma$-algebra w.r.t.\ which $X$ is measurable)
is denoted by~$\sigma(X)$. If $\mcal{F}$ and~$\mcal{G}$ are $\sigma$-algebras,
the $\sigma$-algebra generated by~$\mcal{F}$ and~$\mcal{G}$
(that is, the smallest $\sigma$-algebra containing both~$\mcal{F}$ and~$\mcal{G}$)
is denoted by~$\mcal{F} \vee \mcal{G}$,
and for an arbitrary number of $\sigma$-algebras this notation extends into the $\infty$-ary operator~$\bigvee$.
\item An event $A\in\mcal{B}$ is said to have trivial probability, or to be trivial,
if $\Pr[A]\in\{0,1\}$. A $\sigma$-algebra is said to be trivial if all its events are trivial.
The $\sigma$-algebra $\{\emptyset,\Omega\}$, which is trivial under any law $\Pr$,
will be denoted by~$\mcal{O}$ and refered to as ``the'' trivial sigma-algebra.
\item The Lebesgue measure on~$\RR^n$ will be denoted by~$\dx{x}$,
``$x$'' being the name of the integration variable.
For a Borel set $A\subset\RR^n$, $\int_{x\in A} \dx{x}$ will sometimes be denoted by~$|A|$.
\item For~$C$ a positive-semidefinite matrix (possibly of dimension $1$,
in which case it is identified with~$\sigma^2 \in \RR_+$), $\mcal{N}(C)$ denotes
the law of the centered Gaussian vector with covariance matrix~$C$.
I will write $\mcal{N}(C)+m$
to denote the non-centered Gaussian vector with variance $C$ and mean $m$.
\end{itemize}

\paragraph{Functional analysis}

\begin{itemize}
\item Unless otherwise specified, all the functional spaces considered in this monograph
shall be real.
\item For~$I$ an open interval of~$\RR$ and~$k\in\NN\cup\{\infty\}$,
$\mcal{C}^k_0(I)$ denotes the subset of functions of~$\mcal{C}^k(I)$ with compact support.
\item If $\mu$ is a nonnegative measure on some measurable space $(\Omega,\mcal{B})$,
$L^2(\mu)$ denotes the set of measurable functions $f$ (up to~$\mu$-a.e.\ equality) such that
$\int_{\Omega} f(\omega)^2 \,\dx\mu(\omega) < \infty$.
If $I$ is a countable set, $L^2(I)$ denotes the set of functions $f : I\to\RR$ such that
$\sum_{i\in I} f(i)^2 < \infty$.
If $\mcal{F}$ is a $\sigma$-algebra, $L^2(\mcal{F})$ denotes the space of
$\mcal{F}$-measurable functions (up to a.s.\ equality)
which are square-integrable w.r.t.~$\Pr$.
All these spaces are equipped with their natural Hilbertian product
$\langle\Bcdot,\Bcdot\rangle$ and the associated norm $\|\Bcdot\|$.
\item For~$\mu$ a finite measure, in~$L^2(\mu)$ the constant functions make a line
which can be identified with~$\RR$;
then, $\ldb(\mu)$ will denote the quotient $L^2(\mu)/\RR$,
equipped with its natural Hilbert structure.
In other words, if $\bar{f}\in \ldb(\mu)$ is the projection of~$f\in L^2(\mu)$,
$\|\bar{f}\|_{\ldb} \coloneqq \inf\{\|f-a\|_{L^2}\colon {a\in\RR}\} =
\big( \|f\|_{L^2}^2 - \langle f,1/\|1\|\rangle_{L^2}^2 \big)^{1/2}$.
$\ldb(\mu)$ can also be seen as the subspace of centered functions of~$L^2(\mu)$,
i.e.\ as $\{ f\in L^2(\mu)\colon \langle f,1\rangle = 0 \}$;
throughout the monograph we will implicitly switch between both interpretations.
\item If $L \colon H_1\to H_2$ is a linear operator between two Hilbert spaces,
then $L^* \colon H_2\to H_1$ denotes the adjoint operator of~$L$, characterized by the relationship
$\langle L^*y, x \rangle_{H_1} = \langle y, Lx \rangle_{H_2}$.
\item If $L \colon E\to F$ is a linear operator between two Banach spaces
(not necessarily Hilbert) with respective norms~$\|\Bcdot\|_E$ and~$\|\Bcdot\|_F$,
the operator norm of~$f$, denoted by~$\VERT f\VERT$, is defined as
$\sup\{\|Lx\|_F\colon \|x\|_E=1\}$.
\item If $L \colon E\to E$ is a bounded linear operator on a Banach space,
then $\rho(L)$ denotes the spectral radius of~$f$, that is,
$\rho(L) \coloneqq \lim_{k\longto\infty} \VERT L^k \VERT^{1/k}$
—this limit always exists.
\item A column vector $(a_i)_{i\in I}$ will automatically be identified with
the corresponding element of~$L^2(I)$.
Likewise, a matrix $A = (\!(a_{ij})\!)_{(i,j)\in I\times J}$ will be identified
with the corresponding linear operator from $L^2(J)$ to~$L^2(I)$.
\item In our computations we will often use the Cauchy\,-\,Schwarz inequality
and its variants%
\footnote{For instance, the discrete form $\big|\sum_{i=1}^N a_ib_i\big|
\leq \big(\sum_{i=1}^N a_i^2\big)^{1/2}\big(\sum_{i=1}^N b_i^2\big)^{1/2}$,
the probabilistic form $|\Cov[fg]| \leq \ecty(f)\ecty(g)$, etc..};
when using such an inequality,
we will indicate it by writing ``CS'' under the inequality sign concerned.
Similarly, ``IP'' under an equality sign will mean
that this equality follows from integrating by parts.
\end{itemize}

\section*{Acknowledgements}

This work owes to many people's help.
First of all I must mention V.~Beffara, who launched me on this topic incidentally.
His relevant comments, as well of those of (among others)
Y.~Ollivier, C.~Villani, Y.~Velenik, T.~Bodineau, C.~Shalizi and R.~Bradley,
have also been the origin for several important improvements of this monograph.

Several colleagues provided me with some help on mathematical topics
I was not familiar with. In particular,
Y.~Ollivier suggested to me the use of Lipschitz spaces to prove Lemma~\ref{lem7792};
S.~Martineau pointed out how Gershgorin's Lemma solved a technical point
in the complete proof of Theorem~\ref{thm6657-0};
V.~Calvez had the idea of using Laplace transform to prove Lemma~\ref{l1939a}.
Over the Internet, F.~Martinelli and S.~Shlosman also gave me
precious bibliographic references on the state of the art
about weak and strong mixing in statistical mechanics.

Most of the drawings in this monograph were made thanks to the excellent \LaTeX\ extension \hbox{Ti\textit{k}Z}, combined with computations in C language. The dice of Figure~\ref{fig-3lines} have been kindly drawn for me by A.~Alvarez, using POV-Ray.

\setcounter{tocdepth}{1}
\tableofcontents
\addtocounter{chapter}{-1}
\chapter{Motivation}\label{parMotivation}%
\addcontentsline{itc}{chapter}{\protect\numberline{\thechapter}Motivation}

\section{Some results on Ising's model}\label{parSomeResultsOnIsingsModel}

In this subsection we recall the definition of Ising's model
and give two classical results on it, namely Theorems~\ref{c3793} and~\ref{c6177}.
In \S~\ref{parProblematics}, considerations on these results
will serve as a motivation to the work of this monograph.

\subsection{Ising's model}\label{parIsingsModel}

Ising's celebrated model is a basic model of equilibrium thermodynamics,
which represents a ferromagnetic material:
\begin{Def}
For~$n$ an integer, consider the lattice $\ZZ^n$ endowed
with its usual graph structure (each vertex has $2n$ neighbours),
and denote by~$\dist$ the graph distance.
Define $\Omega = \{\pm1\}^{\ZZ^n}$,
and for~$\vec\omega\in\Omega$, set formally:
\[\label{for9073} H(\vec\omega) = - \frac{1}{2} \sum_{\dist(i,j)=1} \omega_i\omega_j .\]
Then, for~$T\geq0$, the \emph{Ising model on~$\ZZ^n$ at temperature $T$}
is, formally, a probability measure~$\Pr$ on~$\Omega$ such that
$\Pr[\vec\omega] \propto \exp\big(-T^{-1}H(\vec\omega)\big)$.
In rigorous terms, saying that $\Pr$ is an equilibrium measure for Ising's model
means that for all~$i\in\ZZ^n$,
for all~$\vechat\omega_{\C{\{i\}}} \in \{\pm1\}^{\C{\{i\}}}$,
\[\label{f1762} \Pr\big[ \omega_i = \hat\omega_i
\big| \vec\omega_{\C{\{i\}}} = \vechat\omega_{\C{\{i\}}} \big]
\propto \exp\Big( T^{-1} \sum_{\dist(i,j)=1} \hat\omega_i \hat\omega_j \Big) .\]
\end{Def}

Ising's model and the phase transition it exhibits
have been the subject of dozens of works; see~\cite{Grimmett} for an overview.
Here we are interested in the subcritical regime:
\begin{Thm}[Subcritical regime, \cite{Peierls}]
There is a $T_{\mrm{c}}<\infty$ (the `Curie temperature')
such that the solution of~(\ref{f1762}) is unique for $T>T_{\mrm{c}}$ .
\end{Thm}

For $T>T_{\mrm{c}}$ one says that they are in the \emph{subcritical regime}.
An interesting feature of this regime is that for distant~$i$ and~$j$,
the random variables~$\omega_i$ and~$\omega_j$ are `almost independent'.
That phenomenon, called \emph{exponential decay of correlations},
is stated by the following theorem:
\begin{Thm}[Exponential decay of correlations, \cite{sharp_transition}]\label{t3795}
For Ising's model on~$\ZZ^n$ in the subcritical regime,
\begin{ienumerate}
\item\label{i6505} For all~$i\in\ZZ^n$, $\Pr[\omega_i=-1] = \Pr[\omega_i=1] = 1/2$.
\item\label{i1366} There exists~$\psi>0$ and~$C<\infty$
such that for all~$i,j\in\ZZ^n$,
\[ | \EE[\omega_i\omega_j] | \leq C \exp\big(-\psi\dist(i,j)\big) .\]
\end{ienumerate}
\end{Thm}

\subsection{Absence of \texorpdfstring{$\beta$}{beta}-mixing}

Theorem~\ref{t3795} states that two distant spins~$i$ and~$j$ are exponentially decorrelated.
However, it does not inform us about the dependence between `bunches' of spins.
The question is the following: if $I$ and~$J$ are two disjoint, distant subsets of~$\ZZ^n$,
to what extent are~$\vec\omega_I$ and~$\vec\omega_J$ independent?

To answer such a question, the first thing to do is to define a way of measuring independence
between `complicated' variables having an arbitrarily large range
like~$\vec\omega_I$ and~$\vec\omega_J$.
The most common choice is the $\beta$-mixing coefficient:
\begin{Def}\label{def3167}\strut
\begin{enumerate}
\item Recall that for~$\mu$, $\nu$ two probability measures
on the same measurable space $(\Omega,\mcal{F})$,
the \emph{total variation distance} between~$\mu$ and~$\nu$ is
the total mass of both the positive and the negative parts of the signed measure~$\nu-\mu$, that is,
$\dist_{\mathrm{TV}}(\mu,\nu) = \sup_{A\subset\mcal{F}} | \nu(A) - \mu(A) |$.
\item If $X$ and~$Y$ are two random variables (with arbitrary ranges)
defined on the same space, then one defines the \emph{$\beta$-mixing coefficient}
between~$X$ and~$Y$ as
\[ \beta(X,Y) \coloneqq \dist_{\mathrm{TV}} ( \Law_X\otimes\Law_Y\,,\,\Law_{(X,Y)} ) .\]
\end{enumerate}
\end{Def}
Notice that $\beta(X,Y)$ actually only depends on the $\sigma$-algebras%
~$\sigma(X)$ and~$\sigma(Y)$~\cite[Formula~(1.5)]{Bradley-review}.
The following proposition is immediate:
\begin{Pro}\label{pr7428}\strut
\begin{ienumerate}
\item One has always $\beta(X,Y) \in [0,1]$, and
\item $\beta(X,Y) = 0$ if and only if $X$ and~$Y$ are independent ;
\item $\beta(X,Y) = 1$ if and only if $\Law_X\otimes\Law_Y$
and~$\Law_{(X,Y)}$ are mutually disjoint.
\item\label{i7428d} If $X'$ is $X$-measurable and $Y'$ is $Y$-measurable,
then $\beta(X',Y') \leq \beta(X,Y)$.
\end{ienumerate}
\end{Pro}
So, $\beta(X,Y)$ is a way of measuring `how much $X$ and~$Y$ are correlated'.

With that tool at hand, decorrelation between bunches of spins
in statistical physics models has already been thoroughly studied.
Concerning Ising's model, there are two well-known great results:
\begin{Thm}[Weak mixing property, \cite{Martinelli}]\label{t4774}
For Ising's model on~$\ZZ^n$ in the subcritical regime,
there exists~$\psi>0$ and~$C<\infty$ (the same as in Theorem~\ref{t3795})
such that for all disjoint $I,J\subset\ZZ^n$:
\[\label{f7869}
\beta \big( \vec\omega_I , \vec\omega_J \big) \ \leq
\ C \sum_{(i,j)\in I\times J} \exp\big(-\psi \dist(i,j)\big) .\]
\end{Thm}

\begin{Thm}[Complete analyticity, \cite{complete_analyticity}]\label{t4505}
There exists some $T_{\mrm{c}}\leq T_{\mrm{c}}'<\infty$%
\footnote{\label{Tc'=Tc?}It is not known whether $T_{\mrm{c}}'=T_{\mrm{c}}$ today,
but in general situations weak mixing does not always imply complete analyticity.
A classical counterexample is Ising's model with external field~\cite[\S~2]{MO}.}
such that for $T \ab {>T_{\mrm{c}}'}$,
Ising's model is \emph{completely analytical}, i.e.\ 
there exists~$\psi'>0$ and~$C'>0$ such that the following holds:
for all~$K\subset\ZZ^n$, for all `boundary' conditions $\vechat\omega_K \in \{\pm1\}^K$,
denoting $\Pr_{\vechat\omega_K} \coloneqq \Pr[\Bcdot|\vec\omega_K=\vechat\omega_K]$,
Formula~(\ref{f7869}) holds with~$\Law$ replaced by~$\Law_{\vechat\omega_K}$
and~$\psi, C$ replaced resp.\ by~$\psi'$ and~$C'$.
\end{Thm}

Thanks to Theorem~\ref{t4774}, we get an exponential decay of correlation
between two bunches of spins of fixed size when the distance between these bunches increases.
However, we cannot say much about decorrelation between bunches of \emph{variable} size
which are at fixed distance from each other.
For instance, for $n=2$ fix $x>0$ and define
$I_\ell \coloneqq \{(0,y) \colon |y|\leq \ell\}$, resp.\ $J_\ell \coloneqq \{(x,y) \colon {|y|\leq \ell}\}$.
Then Theorem~\ref{t4774} gives us something like:
\[\label{f4961}
\beta \big( \vec\omega_{I_\ell} , \vec\omega_{J_\ell} \big) \lesssim
\ C \ell e^{-\psi x} .\]
But recall that a $\beta$-mixing coefficient is always bounded by~$1$;
so, for $\ell \gtrsim e^{\psi x}/C$, (\ref{f4961}) tell us absolutely nothing about
the decorrelation between~$I_\ell$ and~$J_\ell$.

Though the bound~(\ref{f7869}) is not completely optimal,
the previous point is an \emph{intrinsic} shortcoming of $\beta$-mixing coefficients,
in the sense that it can be proved that bounds like~(\ref{f4961}) \emph{must} become trivial when
$\ell\longto\infty$:
\begin{Thm}\label{c3793}
For all~$T_{\mrm{c}}<T<\infty$, for all~$x>0$,
denoting $I \coloneqq \{0\}\times\ZZ$ and $J \coloneqq \{x\}\times\ZZ$,
one has
\[ \beta \big(\vec\omega_I , \vec\omega_J \big) = 1 .\]
\end{Thm}

\begin{proof}
Denote $i_0 \coloneqq (0,0)$, resp.\ $j_0 \coloneqq (x,0)$.
As we told in Theorem~\ref{t3795}-(\ref{i6505}),
$\EE[\omega_{i_0}], \EE[\omega_{j_0}] = 0$.
Interpretation of Ising's model as a random-cluster model~\cite[\S~1.4]{Grimmett}
shows that $\Pr[\omega_{i_0}=\omega_{j_0}] > 1/2$,
so we define
\[ \gamma \coloneqq \EE[\omega_{i_0}\omega_{j_0}] > 0 .\]

Now, let~$N$ be some large integer, fixed for the time being,
let~$p$ be some large integer and define
$i_1, \ldots, i_N$, resp.~$j_1, \ldots, j_N$,
by~$i_k \coloneqq (0,kp)$, resp.\ $j_k \coloneqq (x,kp)$;
denote by~$P_{N,p}$ the joint law of~%
$(\omega_{i_1},\ldots,\omega_{i_N},\omega_{j_1},\ldots,\omega_{j_N})$.
By translation invariance, for each~$k$,
$(\omega_{i_k},\omega_{j_k})$ has the same law as $(\omega_{i_0},\omega_{j_0})$,
which law we denote by $P_1$.
Then when~$p \longto \infty$, by Theorem~\ref{t4774},
$P_{N,p}$ tends to the law $P_{N,\infty} \coloneqq P_1^{\otimes N}$.%
\footnote{Note that $P_{N,p}$
takes its values in a space of finite dimension,
so there is no ambiguity when speaking of its convergence.}
In other words, $P_{N,\infty}$ is the law such that
all the~$(\omega_{i_k},\omega_{j_k})$ are independent
with $P_{N,\infty}[\omega_{i_k}=\eta \enskip\text{and}\enskip \omega_{j_k}=\theta]
= (1+\gamma\eta\theta)/4$ for all~$k$.
Therefore, the value of~$\beta\big((\omega_{i_k})_k, (\omega_{j_k})_k\big)$
under the law $P_{N,p}$, which by Proposition~\hbox{\ref{pr7428}-(\ref{i7428d})}
is a lower bound for~$\beta \big(\vec\omega_I , \vec\omega_J \big)$,
tends to its value under~$P_{N,\infty}$ when~$p \longto \infty$.
This is summed up by the following formula:
\[\label{fo7691} \beta \big(\vec\omega_I, \vec\omega_J\big) \geq \beta_{P_{N,\infty}}
\big((\omega_{i_k})_{1\leq k\leq N}, (\omega_{j_k})_{1\leq k\leq N}\big) .\]

To end the proof, we will bound the right-hand side of~(\ref{fo7691}) below
by a quantity which tends to~$1$ when~$N\longto\infty$.
Denote by~$\tilde{P}_{N,\infty}$ the product of two the marginals
of~$P_{N,\infty}$ relative resp.\ to the~$(\omega_{i_k})_k$ and the~$(\omega_{j_k})_k$,
so that $\beta_{P_{N,\infty}} \big((\omega_{i_k})_k, (\omega_{j_k})_k\big)
= \dist_{\mathrm{TV}} (P_{N,\infty},\tilde{P}_{N,\infty})$
by the very definition of the $\beta$-mixing coefficient.
Obviously the expression of~$\tilde{P}_{N,\infty}$ is the same
as the expression of~$\tilde{P}_{N,\infty}$, but with~$\gamma$ replaced by~$0$ in the definition of~$P_{N,\infty}$.
Under $P_{N,\infty}$, $(\omega_{i_k}\omega_{j_k})_{1\leq k\leq N}$
is a sequence of i.i.d.\ random variables having a certain law with mean $\gamma > \gamma/2$,
so that by the law of large numbers,
\[ P_{N,\infty} \bigg[N^{-1}\sum_{k=1}^N \omega_{i_k}\omega_{j_k} \leq \frac{\gamma}{2} \bigg]
\stackrel{N\longto\infty}{\longto} 0 .\]
Similarly, since $0 < \gamma/2$,
\[ P_{N,\infty} \bigg[N^{-1} \sum_{k=1}^N \omega_{i_k}\omega_{j_k} \leq \frac{\gamma}{2} \bigg]
\stackrel{N\longto\infty}{\longto} 1 ,\]
so that
\[ \dist_{\mathrm{TV}} \big(P_{N,\infty},\tilde{P}_{N,\infty}\big) \geq
\bigg| \tilde{P}_{N,\infty}
\bigg[ N^{-1} \sum_{k=1}^N \omega_{i_k}\omega_{j_k} \leq \frac{\gamma}{2} \bigg] - P_{N,\infty}
\big[ \text{\it the same} \big] \bigg|
\stackrel{N\longto\infty}{\longto} 1 ,\]
which proves our point.
\end{proof}

\subsection{Presence of \texorpdfstring{$\rho$}{rho}-mixing}

So, Theorem~\ref{c3793} tells us that, for Ising's model on~$\ZZ^2$,
there is a `full' correlation between~$\vec\omega_{I}$ and~$\vec\omega_{J}$
in the sense of $\beta$-mixing.
Yet it is well known too that
Theorem~\ref{t4774} nevertheless implies a Hilbertian form of decorrelation
(called ``$\rho$-mixing'', cf.\ Remark~\ref{rmk8718}) between these variables:
\label{parc6177}\begin{Thm}\label{c6177}
For Ising's model on~$\ZZ^2$ in the subcritical regime,
defining as before $I = \{0\}\times\ZZ$ and $J = \{x\}\times\ZZ$ for some $x>0$,
one has for all~$f\in\ldb(\vec\omega_I)$ and~$g\in\ldb(\vec\omega_J)$:
\[\label{f2881} |\EE[fg]| \leq e^{-\psi x} \ecty(f)\ecty(g) ,\]
where $\psi$ is the same as in Theorem~\ref{t4774}.
\end{Thm}

\begin{proof}
Define the operator
\[ \begin{array}{rrcl} P \colon & \ldb(\vec\omega_I) & \to & \ldb(\vec\omega_J) \\
& f & \mapsto & f^{\sigma(\vec\omega_J)} .\end{array} \]
(Recall that $f^{\sigma(\vec\omega_J)}$ is an alternative notation for~$\EE[f|\vec\omega_J]$,
insisting on the its being a $\sigma(\vec\omega_J)$-measurable function).
Then (\ref{f2881}) is equivalent to proving that
$\VERT P \VERT \leq e^{-\psi x}$ (see \S~\ref{parOperatorInterpretation}).
Now for all~$t \in \{0,\ldots,x\}$, denote $\omega_{(t)} \coloneqq \vec\omega_{\{t\}\times\ZZ}$,
and for all~$t \in \{1,\ldots,x\}$,
\[ \begin{array}{rrcl} \pi_t \colon & \ldb(\omega_{(t-1)}) & \to & \ldb(\omega_{(t)}) \\
& f & \mapsto & f^{\sigma(\omega_{(t)})} .\end{array}\]
Due to the fact that the interactions in Ising's model have only range~$1$,
$\omega_{(0)} \to \omega_{(1)} \to \cdots \to \omega_{(x)}$
is a Markov chain, and therefore
\[\label{f3677} P = \pi_x \circ \cdots \circ \pi_2 \circ \pi_1 .\]
Now, by horizontal translation all the~$\ldb(\omega_{(t)})$ can be identified
with a common Hilbert space $H$.
Then all the~$\pi_t$ are identified with operators on~$\ldb(H)$,
and by the translation invariance of the model all these operators are actually the same.
$P$ is also identified with an operator on~$\ldb(H)$, and (\ref{f3677}) becomes:
\[ P = \pi^x .\]
But $\pi$ is self-adjoint because,
as the model is invariant by translation \emph{and by reflection},
the Markov chain $\omega_{(0)} \to \cdots \to \omega_{(x)}$
is stationary \emph{and reversible}.
In particular $\pi$ is a normal operator, and thus $\VERT P \VERT = \VERT \pi \VERT^x$.
So, proving that $\VERT P \VERT \leq e^{-\psi x}$
is equivalent to proving that $\VERT \pi \VERT \leq e^{-\psi}$,
which will be our new goal.

Take $C<\infty$ like in Theorem~\ref{t4774}.
For~$\ell$ an integer, denote~$I_\ell$ and~$J_\ell$ to be resp.~%
$\{0\}\times\{-\ell,\ldots,\ell\}$ and~$\{x\}\times\{-\ell,\ldots,\ell\}$.
Let~$f$ be a bounded%
\footnote{In fact here it is superfluous to impose that $f$ is bounded since
$\vec\omega_{I_\ell}$ can only take a finite number of values.
I wrote the proof like this just to underline
that the finiteness of the range of the~$\omega_i$ does not play any role in the proof.}
function of~$\ldb(\vec\omega_{I_\ell})$ and denote $M \coloneqq \|f\|_{L^\infty}$.
By translation, $f$ can also be identified with
a function of~$\ldb(\vec\omega_{J_\ell})$,
which is also bounded by~$M$.
Now, since
\[ \EE \big[ f(\vec\omega_{I_\ell})f(\vec\omega_{J_\ell}) \big] =
\Cov \big( f(\vec\omega_{I_\ell}) \,,\, f(\vec\omega_{J_\ell}) \big) \\
= \int f(\vec\omega_{I_\ell})f(\vec\omega_{J_\ell}) \,
\dx{\big(\Law(\vec\omega_{I_\ell\uplus J_\ell})} -
\Law(\vec\omega_{I_\ell})\otimes\Law(\vec\omega_{J_\ell})\big) ,
\]
we can apply~(\ref{f7869}) to~$I_{\ell}$ and~$J_{\ell}$ to obtain:
\[\label{f3079} | \EE[f(\vec\omega_{I_\ell})f(\vec\omega_{J_\ell})] |
\leq M^2\cdot 2C(2\ell+1)^2 e^{-\psi x} .\]

In terms of operators, (\ref{f3079}) means that
\[\label{f3079'} \big|\langle f, Pf \rangle_{\ldb(H)}\big| \leq 2(2\ell+1)^2M^2Ce^{-\psi x} .\]
As the value of~$x$ played no particular role to establish~(\ref{f3079'}),
that formula can be generalized
into
\[ \big|\langle f, \pi^t f \rangle_{\ldb(H)}\big| \leq 2(2\ell+1)^2M^2 C e^{-\psi t} \]
for all~$t\in\NN^*$.
Letting $t$ tend to infinity, we obtain that for all~$\ell$,
for all~$f\in\ldb(\vec\omega_{I_\ell}) \cap L^\infty$,
\[ \limsup_{t\longto\infty} \big( \ln |\langle f, \pi^t f \rangle| \big)^{1/t} \leq e^{-\psi} .\]
But $\bigcup_{\ell\in\NN} \big(\ldb(\vec\omega_{I_\ell})\cap L^\infty\big)$ is a dense subset of
$\ldb(\vec\omega_I)$, so by Lemma~\ref{l3253} set in appendix,
we conclude that $\VERT \pi \VERT_{\ldb(H)} \leq e^{-\psi}$, which is what we wanted.
\end{proof}

\begin{Rmk}
Claim~\ref{c3793} and Theorem~\ref{c6177} adapt straighforwardly,
with similar proofs, to any~$n\geq2$,
replacing~$I$ by~$\{0\}\times\ZZ^{n-1}$ and~$J$ by~$\{x\}\times\ZZ^{n-1}$.
\end{Rmk}

\section{Problematics}\label{parProblematics}

Thanks to Theorem~\ref{c6177}, we see that the Hilbertian concept of $\rho$-mixing
can reveal some independence between infinite bunches of lowly correlated variables
in situations where the $\beta$-mixing coefficient does not show any independence at all.
In the proof we gave, $\rho$-mixing appeared as a corollary
of $\beta$-mixing for finite bunches of spins.
What additional hypotheses did we need to get our corollary?
We used at least the following:
\begin{itemize}
\item To introduce the Markov chain $\omega_{(0)} \to \cdots \to \omega_{(x)}$,
we used that the interactions of our model had finite range.
\item To identify all the spaces $\ldb(\omega_{(t)})$,
we used that $I$ and~$J$ had the same shape and that one could tile up $\ZZ^2$
with a sequence of tiles having that shape (namely, here, tiles of the form $\{t\}\times\ZZ$).
\item To say that all the~$\pi_t$ were the same modulo that identification,
we used the translation invariance of the model.
\item To state that the stationary Markov chain $\omega_{(0)} \to \cdots \to \omega_{(x)}$
was reversible, we used the reflection invariance of the model.
\item To use Lemma~\ref{l3253}, we used the exponential decay of correlations.
\end{itemize}
All these points make the proof of Theorem~\ref{c6177} we gave
in~\S~\ref{parSomeResultsOnIsingsModel} quite difficult to generalize.
What, for instance, if we take~$I$ and~$J$ with arbitrary shapes,
just requiring that $\dist(I,J) \geq x$?
What if we consider statistical physics models with infinite-range interactions? Etc..
The above arguments would not work any more! Yet, we do not have the impression that
the presence of $\rho$-mixing fundamentally relies
on the peculiar symmetries of the case we treated\dots

So, here will be the goal of this work:
\emph{establishing $\rho$-mixing estimates by general methods}.
To achieve this goal, I shall try to concentrate on the properties of $\rho$-mixing
`for itself', rather than on its links with other forms of decorrelation.
I will carry out a thorough study of the $\rho$-mixing coefficient,
in order to get $\rho$-mixing results for `complicated' variables
from decorrelation results \emph{of the same type} for more `basic' variables;
in other words, I will \emph{tensorize} Hilbertian decorrelations.
It turns out that tensorization for such correlation coefficients gives results which are
quite robust as the size of the bunches of variables increases.
Thanks to this method, I shall obtain fairly new decorrelation theorems
for various models of statistical physics.

This monograph is intended to be complete in some sense.
I mean, besides the core of this work—namely, tensorization results—,
I have tried to answer several other questions which appeared natural to me
concerning Hilbertian decorrelation.
This includes studying many examples,
finding sharp criteria for maximal decorrelation,
looking at the optimality issues in the tensorization results
or showing other applications of the tensorization techniques.
Though these topics were initially thought as `sidework',
some of them may be quite interesting for themselves.

\section{Appendix: On the norm of self-adjoint operators}%
\label{parOnTheNormOfSelfAdjointOperators}

In this appendix we prove the following
\begin{Lem}\label{l3253}
Let~$L$ be a self-adjoint operator on a real Hilbert space $H$, and let~$C<\infty$.
Then, to prove that $\VERT L \VERT \leq C$, it suffices to ensure that
\[\big\{ x\in H \colon \limsup_{k\longto\infty} |\langle L^k x, x \rangle|^{1/k} \leq C \big\}\]
is a dense subset of~$H$.
\end{Lem}

\begin{proof}
Reasoning by contraposition, we have to show that,
for~$L$ a self-adjoint operator on~$H$, for all~$C < \VERT L \VERT$,
the set of the~$x \in H$ such that
\[ \label{f8714} \limsup_{k\longto\infty} |\langle L^k x, x \rangle|^{1/k} > C \]
contains a non-empty open subset of~$H$.

Since $L$ is self-adjoint, by the spectral theorem~\cite[Theorem~7.18]{Weidmann},
it is unitarily equivalent to the ``multiplication by identity'' operator $M$
on a space $\bigoplus_{\alpha\in A} L^2(\rho_\alpha)$,
for~$A$ some set and~$\rho_\alpha$ some Radon measures on~$\RR$, that is
[in the following equation, the variable $\lambda$ is free, so that
$f(\lambda)$ is synonymous with~$f$]:
\[ M\Big( \sum_\alpha f_\alpha(\lambda) \Big) = \sum_\alpha \lambda f_\alpha(\lambda) .\]
So we will assume $L$ is of that form.

One has obviously:
\[\label{f7254} \VERT L \VERT = \sup \big\{ \lambda \geq 0 \colon 
(\exists \alpha \in A) \ \big(\rho_\alpha(\C{[-\lambda,\lambda]}) > 0\big) \big\} ;\]
moreover, for all~$f \in H$,
$f = \sum_{\alpha\in A} f_\alpha$ with $f_\alpha \in L^2(\rho_\alpha)$,
\[ \langle L^k f, f \rangle =
\sum_{\alpha\in A} \int_{\RR} \lambda^k |f_\alpha(\lambda)|^2 \,\dx\rho_\alpha(\lambda) ,\]
so that (observing that, for $k$ even, $\lambda^k \geq 0\ \ \forall\lambda$)
\[\label{f7360}
\limsup_{k\longto\infty} \big|\langle L^k f, f \rangle\big|^{1/k}
= \sup \Big\{\lambda \geq 0 \colon (\exists \alpha \in A) \ 
\Big(\int_{\C{[-\lambda,\lambda]}} |f_\alpha(\lambda')|^2 \,\dx\rho_\alpha(\lambda') > 0\Big)\Big\}.\]

Now, for $C < \VERT L \VERT$, the set
\[ U = \Big\{ f \in H \colon (\exists \alpha \in A) \ 
\Big( \int_{\C{[-C,C]}} |f_\alpha(\lambda)|^2 \,\dx\rho_\alpha(\lambda) > 0 \Big) \Big\} \]
is open because $\int_{\C{[-C,C]}} |f_\alpha(\lambda)|^2 \,\dx\rho_\alpha(\lambda)$
is a continuous function of~$f$, and it is non-empty by~(\ref{f7254}).
But (\ref{f8714}) is satisfied for all~$x\in U$ by~(\ref{f7360}),
so $U$ fulfills our quest.
\end{proof}

\chapter{A first approach to Hilbertian correlations}\label{parBasicHilbertianDecorrelations}%
\addcontentsline{itc}{chapter}{\protect\numberline{\thechapter}A first approach to Hilbertian correlations}

\section{Definition and first properties}\label{parDefinition}

\subsection{Equivalent definitions}

\begin{Def}\label{def4393}
Let~$(\Omega,\mcal{B},\Pr)$ be a probability space.
For~$\mcal{F}, \ab \mcal{G}$ two sub-$\sigma$-algebras of~$\mcal{B}$,
the \emph{Hilbertian correlation coefficient} (or merely ``correlation'')
between~$\mcal{F}$ and~$\mcal{G}$ is defined as
\[\label{for3441} \{\mcal{F}:\mcal{G}\}
\coloneqq \sup_{\substack{f\in\ldb(\mcal{F})\setminus\{0\}\\g\in\ldb(\mcal{G})\setminus\{0\}}}
\frac{|\EE[fg]|}{\ecty(f)\ecty(g)} .\]
If the supremum in~(\ref{for3441}) is taken over an empty set,
that is, if $\mcal{F}$ or $\mcal{G}$ is trivial,
we define this supremum to be $0$.
\end{Def}

\begin{Rmk}\label{rmk8718}
$\{\mcal{F}:\mcal{G}\}$ is often called the ``maximal correlation coefficient'' or ``$\rho$-mixing coefficient'' between~$\mcal{F}$ and~$\mcal{G}$, and denoted by~$\rho(\mcal{F},\mcal{G})$ (see~\cite{Bradley-review}).
\end{Rmk}

\begin{Rmk}
In other words, $\{\mcal{F}:\mcal{G}\}$ is the best $k\in\RR_+$ such that the following refined
Cauchy--Schwarz inequality holds in the Hilbert space $\ldb(\mcal{B})$:
\[ \forall f \in \ldb(\mcal{F}) \enskip\forall g \in \ldb(\mcal{G})
\qquad |\langle f,g \rangle| \leq k \|f\| \|g\| .\]

Yet another formulation is that $\{\mcal{F}:\mcal{G}\}$ is the cosine of the angle
between~$\ldb(\mcal{F})$ and~$\ldb(\mcal{G})$, seen as subspaces of~$\ldb(\mcal{B})$
—this angle being defined as the infimum angle
between any two non-zero vectors of these respective subspaces.

If we speak in terms of~$L^2$ spaces rather than $\ldb$ spaces,
$\{\mcal{F}:\mcal{G}\}$ is the best $k\in\RR_+$ such that
for all non-constant square-integrable~$f,g$ resp.\ $\mcal{F}$ and $\mcal{G}$-measurable,
\[\label{for0263} |\Corr(f,g)| \leq k ,\]
where $\Corr(f,g) \coloneqq \Cov(f,g) \div \ecty(f)\ecty(g)$ is the Pearson correlation coefficient between~$f$ and~$g$.
\end{Rmk}

\begin{Def}
We say that $\mcal{F}$ and~$\mcal{G}$ are \emph{$\epsilon$-decorrelated},
resp.\ \emph{$\epsilon$-correlated}, if $\{\mcal{F}:\mcal{G}\}\leq\epsilon$,
resp.\ $\{\mcal{F}:\mcal{G}\} \geq \epsilon$.
\end{Def}

\begin{Def}\label{def3792}
For~$X$ and~$Y$ random variables (with arbitrary ranges), we will denote~${\{X:Y\}}$
for~${\{\sigma(X):\sigma(Y)\}}$.
\end{Def}

\begin{Rmk}
One can rewrite Definition~\ref{def3792} as
\[\label{for3878}
\{ X:Y \} = \sup_{f,g} \frac{\Cov\big( f(X), g(Y) \big)}{\ecty\big(f(X)\big)\ecty\big(g(Y)\big)} ,\]
where it is implied that $f$ and~$g$ have to be measurable, real,
and such that $0 < \ecty(f(X)), \ab \ecty(g(Y)) \ab {< \infty}$.
\end{Rmk}

\begin{NOTA}
More generally, all the questions relative to Hilbertian correlations
may be handled either in terms of $\sigma$-algebras or in terms of random variables.
In the sequel, we will frequently switch implicitly between these two paradigms.
\end{NOTA}

It is natural to enquire what happens if one deals with \emph{complex} $\ldb$ spaces.
In fact it does not change anything:
\begin{Pro}[{\cite[Theorem~1.1]{Withers}}]
Let~$\mcal{F}$ and~$\mcal{G}$ be two $\sigma$-algebras
and let~$f,g$ be two complex centered $L^2$ variables,
measurable w.r.t.\ resp.~$\mcal{F}$ and~$\mcal{G}$.
Then, with~$\ecty(f)$ meaning $\sqrt{\EE[|f-\EE[f]|^2]}$, one has:
\[ |\EE[fg]| \leq \{\mcal{F}:\mcal{G}\} \ecty(f) \ecty(g) .\]
\end{Pro}
\begin{proof}
I recall the proof for the sake of completeness.
Up to multiplying $g$ by a well-chosen unit complex number,
we can assume that $\EE[fg] \ab {\in \RR_+}$.
Then we can apply Definition~\ref{def4393} to the \emph{real} $\ldb$ variables%
~$\Re f$ and~$\Re g$, resp.~$\Im f$ and~$\Im g$, getting:
\begin{multline}
|\EE[fg]| = \Re \EE[fg] = \EE[\Re f \Re g] - \EE[\Im f \Im g] \\
\leq \{\mcal{F}:\mcal{G}\} \big( \ecty(\Re f) \ecty(\Re g) + \ecty(\Im f) \ecty(\Im g) \big) \\
\footrel{\text{CS}}{\leq} \{\mcal{F}:\mcal{G}\} \sqrt{\Var(\Re f)+\Var(\Im f)} \sqrt{\Var(\Re g)+\Var(\Im g)} \\
= \{\mcal{F}:\mcal{G}\} \ecty(f) \ecty(g).
\end{multline}
\end{proof}

Now we turn to a different way of seeing correlation levels.

\begin{Def}
For~$\mcal{F}, \mcal{G}$ two $\sigma$-algebras, we denote by~$\pi_{\mcal{G}\mcal{F}}$
the `projection' operator
\[\label{for0595} \begin{array}{rrcl} \pi_{\mcal{G}\mcal{F}} \colon & \ldb(\mcal{F}) & \to & \ldb(\mcal{G}) \\ & f & \mapsto & f^{\mcal{G}}. \end{array} \]

For~$\mcal{F}, \ldots, \mcal{Z}$ $\sigma$-algebras,
we denote $\pi_{\mcal{Z}\mcal{Y}\mcal{X}\ldots\mcal{G}\mcal{F}} \coloneqq
\pi_{\mcal{Z}\mcal{Y}} \circ \ab \pi_{\mcal{Y}\mcal{X}} \circ \ab \cdots \circ \pi_{\mcal{G}\mcal{F}}$.
\end{Def}

With this vocabulary at hand,
\begin{Pro}\label{pro7764}
For~$\mcal{F}, \mcal{G}$ two $\sigma$-algebras, $\{\mcal{F}:\mcal{G}\} = \VERT \pi_{\mcal{G}\mcal{F}} \VERT$.
\end{Pro}

\begin{proof}
$\pi_{\mcal{G}\mcal{F}}$ is the orthogonal projection from $\ldb(\mcal{F})$
to~$\ldb(\mcal{G})$ in the Hilbert space $\ldb(\mcal{B})$,
so its norm is the cosine of the angle between~$\ldb(\mcal{F})$ and~$\ldb(\mcal{G})$,
i.e.\ ${\{\mcal{F}:\mcal{G}\}}$.
\end{proof}

\begin{Rmk}\label{rmk1265}
One has $\pi_{\mcal{F}\mcal{G}} = \pi_{\mcal{G}\mcal{F}}^*$,
since $\langle \pi_{\mcal{G}\mcal{F}}f , g \rangle = \EE[fg] = \langle f , \pi_{\mcal{F}\mcal{G}}g \rangle$.
Therefore the expression $\VERT \pi_{\mcal{G}\mcal{F}} \VERT$ in Proposition~\ref{pro7764}
can be rewritten into $\sqrt{\VERT \pi_{\mcal{F}\mcal{G}\mcal{F}} \VERT}$,
which is also $\rho(\pi_{\mcal{F}\mcal{G}\mcal{F}})$
since $\pi_{\mcal{F}\mcal{G}\mcal{F}}$ is self-adjoint.
\end{Rmk}

\subsection{Immediate properties}

Having defined Hilbertian correlations, it is now time to study their behaviour.

The following properties are immediate from Definition~\ref{def4393}:
\begin{Pro}
For all $\sigma$-algebras $\mcal{F}$, $\mcal{G}$ and~$\mcal{G}'$,
\begin{ienumerate}
\item $\{\mcal{G}:\mcal{F}\} = \{\mcal{F}:\mcal{G}\}$;
\item $\mcal{G}\subset\mcal{G}' \ \Rightarrow\ \{\mcal{F}:\mcal{G}\} \leq \{\mcal{F}:\mcal{G}'\}$;
\item $\{\mcal{F}:\mcal{G}\} \in [0,1]$;
\item $\{\mcal{F}:\mcal{G}\} = 0$ if and only if $\mcal{F}$ and~$\mcal{G}$ are independent;
\item If $\mcal{F}$ is not trivial, then $\{\mcal{F}:\mcal{F}\} = 1$.
\end{ienumerate}
\end{Pro}

When one is concerned by correlation between variables,
it often occurs that some of these variables are vector-valued.
The following proposition means
that it suffices to know the behaviour of finite-length vectors
to understand the behaviour of all vectors:
\begin{Pro}\label{pro6559}
Let~$I, J$ be possibly infinite sets
and let~$\vec{X}_I, \vec{Y}_J$ be vector-valued variables. Then,
denoting ``$I'\Subset I$'' to mean that $I'$ is a finite subset of~$I$,
\[ \big\{ \vec{X}_I : \vec{Y}_J \big\} =
\sup_{I'\Subset I,J'\Subset J}
\big\{ \vec{X}_{I'} : \vec{Y}_{J'} \big\} .\]
\end{Pro}

\begin{proof}
This is because $\bigcup_{I'\Subset I} \ldb(\vec{X}_{I'})$,
resp.~$\bigcup_{J'\Subset J} \ldb(\vec{X}_{J'})$,
is a dense subset of~$\ldb(\vec{X}_I)$, resp.\ $\ldb(\vec{Y}_J)$.
That property follows by classical approximation arguments
like in the proof of~\cite[Theorem~3.14]{Rudin}.
See~\cite[Theorem~3.16(II-3)]{Bradley-book} for a more detailed proof.
\end{proof}

\subsection{Operator interpretation}\label{parOperatorInterpretation}

\begin{Pro}\label{cor0706}
If $X \to Y \to Z$ is a Markov chain, then $\{X:Z\} \leq {\{X:Y\}} \* {\{Y:Z\}}$.
\end{Pro}

\begin{proof}
The Markov chain property is equivalent to meaning that $\pi_{ZX} = \pi_{ZYX}$,
so the result is a consequence of the submultiplicativity of operator norms.
See also~\cite[\S~VII-4]{Rosenblatt}.
\end{proof}

There is a refined version of Proposition~\ref{cor0706} which is particularly interesting
for reversible chains:
\begin{Pro}\label{pro0281}
If $X \to Y \to Z$ is a Markov chain, then $\{X:Z\} = \sqrt{\rho(\pi_{YZY}\circ\pi_{YXY})}$.
\end{Pro}

\begin{proof}
Because of the Markov chain property, $\pi_{XZ} = \pi_{XY}\circ\pi_{YZ}$
and $\pi_{ZX} = \pi_{ZY}\circ\pi_{YX}$.
Using that for any pair of operators~$\pi \colon H_1 \to H_2$ and~$\tau \colon H_2 \to H_1$
one has $\rho(\pi\circ\tau) = \rho(\tau\circ\pi)$,
we get that $\{X:Z\}^2 = \rho(\pi_{XY}) = \rho(\pi_{XYZYX}) = \rho(\pi_{XY}\circ\pi_{YZYX})
= \rho(\pi_{YZYX}\circ\pi_{XY}) = \rho(\pi_{YZY}\circ\pi_{YXY})$.\linebreak[1]\strut
\end{proof}

\begin{Cor}\label{cor5902}
If $\cdots \to X_{-1} \to X_{0} \to X_{1} \to \cdots$ is a stationary Markov chain
so that $\pi_{X_1X_0}$ and~$\pi_{X_0X_1}$ commute%
\footnote{Reversible chains always satisfy this condition
since then $\pi_{X_1X_0} = \pi_{X_0X_1}$.},
then for all~$k\in\ZZ\setminus\{0\}$, $\{X_0:X_k\} = \{X_0:X_1\}^{|k|}$.
\end{Cor}

\begin{proof}
Since the chain is stationary, all the~$X_n$ have the same law
and thus all the~$\ldb(X_n)$ can be identified;
then the stationarity property is equivalent to saying that
$\pi_{X_{n+1}X_n} = \pi_{X_1X_0}$ for all~$n\in\ZZ$.
Thanks to the commutation hypothesis, one can write for~$k>0$:
\[ \{X_0:X_k\} = \sqrt{\rho(\pi_{X_0X_kX_0})}
= \sqrt{\rho(\pi_{X_0X_1}^k\circ\pi_{X_1X_0}^k)} \\
= \sqrt{\rho(\pi_{X_0X_1X_0}^k)}
= \sqrt{\rho(\pi_{X_0X_1X_0})^k} = \{X_0:X_1\}^k . \]
For the case $k<0$, we use that $\{X_0:X_k\} = \{X_{-k}:X_0\}$.
\end{proof}

\subsection{First criteria for decorrelation}\label{parCriteriaForDecorrelation}

\paragraph{Density sufficient condition}

\begin{Pro}\label{pro0651}
Let~$X$ and~$Y$ be two random variables valued resp.\ in~$E$ and~$F$.
Suppose that $\Law(X,Y)$ has a density $h$ w.r.t.\ the product probability $\Law(X)\otimes\Law(Y)$.
Then:
\[\label{for0255} \{X:Y\} \leq \Big( \int_{E\times F} (h-1)^2 \,\dx{Law_X}\,\dx{Law_Y} \Big)^{1/2} .\]
\end{Pro}

\begin{Rmk}\label{rmk5108}
The integral expression in~(\ref{for0255})
is nothing but $2$ times the bilinearized version of the mutual information
\[\label{eq9402} I(X;Y) \coloneqq \int_{E\times F} h \ln h \,\dx{Law_X}\,\dx{Law_Y} .\]
Yet Example~\ref{x3993} will show that one does not have
${\{X:Y\}} \leq \sqrt{2} \* (X;Y)^{1/2}$ in general.
\end{Rmk}

\begin{proof}
To alleviate notation, denote resp.~$\Pr_X,\ab \Pr_Y,\ab \Pr_{(X,Y)}$ for~$\Law(X),\ab \Law(Y),\ab \Law(X,Y)$.
Let~$f$ and~$g$ be centered $L^2$~functions being resp.\ $X$- and $Y$-measurable.
Observe first that
\[ \int_{E\times F}f(x)g(y) \,\dx{\Pr_X}[x]\dx{\Pr_Y}[y] =
\Big(\int_E f \,\dx{\Pr_X}\Big) \Big(\int_F g \,\dx{\Pr_Y} \Big) = 0 \times 0 = 0 ,\]
so that
\[ \EE[fg] = \int fg \,\dx{\Pr_{(X,Y)}}
= \int hfg \,\dx{\Pr_X}\dx{\Pr_Y} = \int (h-1)fg \,\dx{\Pr_X}\dx{\Pr_Y} \]
and thus
\[\label{cal1083}
|\EE[fg]| \leq \Big( \int (h-1)^2 \,\dx{\Pr_X}\dx{\Pr_Y} \Big)^{1/2}
\Big( \int f^2 g^2 \,\dx{\Pr_X} \dx{\Pr_Y} \Big)^{1/2} \]
by the Cauchy--Schwarz inequality.
But the last factor in the right-hand side of~(\ref{cal1083}) is
\[ \Big( \int f^2(x) g^2(y) \,\dx{\Pr_X}[x] \dx{\Pr_Y}[y] \Big)^{1/2}
= \Big( \int f^2 \dx{\Pr_X} \Big)^{1/2}
\Big( \int g^2 \dx{\Pr_Y} \Big)^{1/2}
= \ecty(f) \ecty(g) ,\]
so that (\ref{for0255}) is proved.

You may also see \cite[Theorem~2.5]{Bryc} for an analogous result.
\end{proof}

\paragraph{Event necessary condition}

\label{parpro0699}\begin{Pro}[event necessary condition]\label{pro0699}
Let~$\mcal{F}$ and~$\mcal{G}$ be two $\sigma$-algebras.
If $\{ \mcal{F} : \mcal{G} \} \leq \epsilon$,
then for all events~$A\in\mcal{F}$ and~$B\in\mcal{G}$
with respective probabilities~$p$ and~$q$,
\[\label{for3346} \big| \Pr[A\cap B] - pq \big| \leq \epsilon \sqrt{p(1-p)q(1-q)} .\]

In particular, if there exists two non-trivial events $A\in\mcal{F}, B\in\mcal{G}$
which are equivalent (in the sense that $\Pr[A\bigtriangleup B] = 0$),
then $\{\mcal{F}:\mcal{G}\} = 1$.\footnote{%
The converse is not true: it can occur that $\{\mcal{F}:\mcal{G}\}=1$
but that no non-trivial events of~$\mcal{F}$ and~$\mcal{G}$ are equivalent.
A counterexample is the following: let~$(X_n)_{n\in\NN}$
be independent $\mathit{Bernoulli}(1/2)$ variables, and define independently
$Y_n = 1-X_n$ with probability~$\epsilon_n$ and $Y_n = X_n$ otherwise,
where $(\epsilon_n)_{n\in\NN}$ is a sequence of numbers such that
$0 < \epsilon_n \leq 1/2$ for all~$n$
and $\epsilon_n\stackrel{n\longto\infty}{\longto}0$.
Then the vectorial variables~$\vec{X}$ and~$\vec{Y}$ obviously satisfy $\{\vec{X}:\vec{Y}\} = 1$,
yet it is not hard to prove that no $\vec{X}$-measurable non-trivial event
is equivalent to a $\vec{Y}$-measurable one.}
\end{Pro}

\begin{proof}
It follows from~(\ref{for0263}) applied to~$\1{A}$ and~$\1{B}$.
\end{proof}

\subsection{Independent tensorization}\label{parIndependentTensorization}

Now we are turning to the basic tensorization theorem,
which will motivate~\S~\ref{parTensorization}:
\begin{Thm}[{\cite[Theorem~6.2]{CsakiFischer}}]\label{pro1252}
Let~$I$ be a set and let~$\vec{X}_I,\vec{Y_I}$ be vectorial variables.
Suppose all the pairs $(X_i,Y_i)$, $i\in I$, are independent, then
\[\label{for6560} \big\{ \vec{X} : \vec{Y} \big\} = \sup_{i\in I} \{ X_i : Y_i \} .\]
\end{Thm}

\begin{proof}
The simplest proof of Theorem~\ref{pro1252}
relies on the operator interpretation of correlations,
see e.g.\ the proof of~\cite[Theorem~1]{Witsenhausen}.
Here however I shall give a proof
based on decomposing functions of several variables into telescopic sums,
for this kind of arguments will be used again
in the proofs of the more general tensorization theorems of~\S~\ref{parTensorization}.

First, observe that the ``$\geq$'' inequality of~(\ref{for6560}) is trivial,
so we only have to prove the ``$\leq$'' inequality.
We denote $\epsilon_i \coloneqq {\{X_i:Y_i\}}$,
and to alleviate notation, $x_i$ will implicitly stand for an element in the range of~$X_i$,
resp.~$y_i$ for an element in the range of~$Y_i$.

By Proposition~\ref{pro6559}, we may assume that $I$ is finite,
say $I = \{1,\ldots,N\}$ for some $N\in\NN$.
Let~$f$ and~$g$ be resp.\ $\vec{X}_I$-measurable and $\vec{Y}_I$-measurable
centered $L^2$ real functions; our goal is to bound above $|\EE[fg]|$.

For~$i\in\{0,\ldots,N\}$, define $\mcal{F}_i \coloneqq \bigvee_{j\leq i} \sigma(X_j,Y_j)$.
I claim that, because of the independence hypothesis,
$f^{\mcal{F}_i}$ only depends on the values of~$X_1,\ldots,X_i$
and not on~$Y_1,\ldots,Y_i$, and similarly
that $g^{\mcal{F}_i}$ only depends on the values of~$Y_1,\ldots,Y_i$:
one can write indeed (in the case of~$f$)
\begin{multline} f^{\mcal{F}_i} (x_1,y_1,\ldots,x_i,y_i)
= \int f(x_1,\ldots,x_i,x_{i+1},\ldots,x_n) \,\dx\Pr[x_{i+1},\ldots,x_n|x_1,y_1,\ldots,x_i,y_i] \\
= \int f(x_1,\ldots,x_i,x_{i+1},\ldots,x_n) \,\dx\Pr[x_{i+1},\ldots,x_n] .
\end{multline}

Now, for~$i\in\{1,\ldots,N\}$, define
\[ f_i(x_1,\ldots,x_i) \coloneqq f^{\mcal{F}_i}(x_1,\ldots,x_i) - \EE[f|x_1,\ldots,x_{i-1}] ,\]
with a similar definition for~$g$.
One has $f = \sum_i f_i$, resp.\ $g = \sum_i g_i$,
and $f_i$ and~$g_i$ are $\mcal{F}_i$-measurable and centered w.r.t.~$\mcal{F}_{i-1}$
(that is, $\EE[f_i|\mcal{F}_{i-1}], \EE[g_i|\mcal{F}_{i-1}] \equiv 0$), so
\[ \Var f = \sum_i \Var f_i ,\]
resp.\ $\Var g = \sum_i \Var g_i$.

We expand:
\[\label{for6921} \EE[fg] = \sum_{(i,j)\in I\times I} \EE[f_ig_j] .\]
In the right-hand side of (\ref{for6921}), if $i\neq j$ then
$\EE[f_ig_j] = 0$ since if, say, $i<j$,
$f_i$ is $\mcal{F}_i$-measurable while $g_j$ is centered w.r.t.~$\mcal{F}_{j-1} \supset \mcal{F}_i$.
So (\ref{for6921}) turns into:
\[ \EE[fg] = \sum_{i\in I} \EE[f_ig_i] .\]

Writing the law of total expectation,
\[ \EE[f_ig_i] = \int \EE[f_ig_i|x_1,y_1,\ldots,x_{i-1},y_{i-1}]
\, \dx\Pr[x_1,y_1,\ldots,x_{i-1},y_{i-1}] .\]
But, as we noticed before,
under~$\Pr[\Bcdot|\ab x_1,y_1,\ldots,x_{i-1},y_{i-1}]$,
$f_i$ only depends on~$X_i$ and~$g_i$ only depends on~$Y_i$.
Moreover, because of the independence property,
the law of~$(X_i,Y_i)$ is the same
under~$\Pr[\Bcdot|\ab x_1,y_1,\ldots,x_{i-1},y_{i-1}]$ as under~$\Pr$,
so under~$\Pr[\Bcdot|\ab x_1,y_1,\ldots,x_{i-1},y_{i-1}]$, $f_i$ and~$g_i$ are centered and $\epsilon_i$-independent.
Thus
\begin{multline}
|\EE[f_ig_i]| \leq
\epsilon_i \int \ecty(f_i|x_1,y_1,\ldots,x_{i-1},y_{i-1}) \ecty(g_i|x_1,y_1,\ldots,x_{i-1},y_{i-1})
\,\dx\Pr[x_1,y_1,\ldots,x_{i-1},y_{i-1}] \\
\footrel{\text{CS}}{\leq} \epsilon_i \sqrt{\int \Var(f_i|x_1,y_1,\ldots,x_{i-1},y_{i-1}) \,\dx\Pr[x_1,y_1,\ldots,x_{i-1},y_{i-1}]}
\sqrt{\textit{the same for~$g$}} \\
\leq \epsilon_i \ecty(f_i)\ecty(g_i) \footnotemark.
\end{multline}
\footnotetext{The last inequality is actually an equality,
because $f_i$ and~$g_i$ are centered w.r.t.~$\mcal{F}_{i-1}$.}
Summing over $i$,
\begin{multline} |\EE[fg]| \leq \sum_{i\in I} \epsilon_i \ecty(f_i)\ecty(g_i)
\leq \sup_{i\in I} \epsilon_i \cdot \sum_{i\in I} \ecty(f_i)\ecty(g_i) \\
\footrel{\text{CS}}{\leq} \sup_{i\in I} \epsilon_i \cdot \sqrt{\sum_{i\in I}\Var(f_i)} \sqrt{\sum_{i\in I}\Var(g_i)}
= \sup_{i\in I} \epsilon_i \cdot \ecty(f) \ecty(g),
\end{multline}
which is the desired bound.
\end{proof}

\section{Examples}\label{parExamples}

\subsection{Finite-ranged variables}

\begin{Pro}\label{pro3798}
Let~$X$ and~$Y$ be random variables with finite ranges
resp.~$\{1,\ldots,N\}$ and~$\{1,\ldots,M\}$,
and denote $p_a \coloneqq \Pr[X=a], \enskip p^b \coloneqq {\Pr[Y=b]}, \enskip p_a^b \coloneqq \Pr[{X=a} \enskip \text{and} \enskip Y=b]$.
Then ${\{X:Y\}} = \VERT \Pi \VERT$, where $\Pi$ is the $N\times M$ matrix with general entry
\[\label{for7497} \Pi_{ab} = \frac{p_a^b-p_ap^b}{\sqrt{p_ap^b}} .\]
\end{Pro}

\begin{Rmk}\label{rmk9549}
In particular, if both $X$ and~$Y$ have range $\{1,2\}$,
using the same notation as before, one has
\[\label{f0699} \{X:Y\} = \frac{|p_a^b-p_ap^b|}{\sqrt{p_1p_2p^1p^2}} ,\]
where the right-hand side of~(\ref{f0699}) does not depend on the choice of
$a,b \in \{1,2\}$.
\end{Rmk}

\begingroup\def\proofname{Proof of Proposition~\ref{pro3798}}\begin{proof}
By Proposition~\ref{pro7764}, ${\{X:Y\}}$ is the norm of the operator
$\pi_{XY} \colon \ldb(Y) \to \ldb(X)$.
Here it will be more convenient to work in $L^2$ spaces than in $\ldb$ spaces,
so we rather compute the norm of
\[ \begin{array}{rrcl} \tilde{\pi} \colon & L^2(Y) & \to & L^2(X) \\
& g & \mapsto & g^X - \EE[g] , \end{array} \]
which is obviously the same as $\VERT\pi_{XY}\VERT$.

A function~$g\in L^2(Y)$ can be identified with a $M$-dimensional vector also denoted by~$g$,
and similarly $\tilde{\pi}g \in L^2(X)$ can be identified with a $N$-dimensional vector.
Denote $P \coloneqq (\!( p_a^b )\!)_{a,b} \ab \in \RR^{N\times M}$,
$I_X \coloneqq (\!( \delta_{aa'} p_a )\!)_{a,a'} \ab \in \RR^{N\times N}$,
$I_Y \coloneqq (\!( \delta_{bb'} p^b )\!)_{b,b'} \ab \in \RR^{M\times M}$,
$1_N \coloneqq 1^{\{1,\ldots,N\}} \ab \in \RR^N$.
Applying Bayes' formula yields that
\[ \tilde{\pi} g \, = \, I_X^{-1} \, P \, g \, - \, 1_N \, \T{(1_N)} \, P \, g .\]
Now, $\|g\|_{L^2(Y)} = \| I_{Y}^{1/2} g \|$, resp.\ 
$\|\tilde{\pi} g\|_{L^2(X)} = \| I_{X}^{1/2} (\tilde{\pi}g) \|$, so:
\[\label{for7396} \{X:Y\} = \sup_{g\neq 0} \frac{
\big\| \, \big( \, I_X^{-1/2} \, P \, - \, I_X^{1/2} \, 1_N \, \T{(1_N)} \, P \, \big) \, g \, \big\|}
{\| \, I_{Y}^{1/2} \, g \, \|} .\]
Performing the change of variables $h = I_{Y}^{1/2} g$, (\ref{for7396}) becomes
$\{ X:Y \} = \sup_{h\neq 0} \ab  \|\Pi h\| \div \|h\| = \VERT \Pi \VERT$,
with
\[ \Pi = I_X^{-1/2} \, P \, I_Y^{-1/2} - \, I_X^{1/2} \, 1_N \, \T{(1_N)} \, P \, I_Y^{-1/2} ,\]
which is Equation~(\ref{for7497}) indeed.
\end{proof}\endgroup

\begin{Rmk}
With the same kind of proof, there is even a similar proposition to calculate $\{X,Y\}$
if \emph{either} $X$ or $Y$ has finite range,
provided you know (in the case it is $X$ which has finite range)
all the~$\Pr[X=x]$ and all the
\[ \int_y \frac{\dx\Pr[Y=y|X=x]\,\dx\Pr[Y=y|X=x']}{\dx\Pr[Y=y]} .\]
\end{Rmk}

\begin{Rmk}
In the case $X$ or~$Y$ has range of cardinality~$2$,
applying Proposition~(\ref{pro3798}) yields that ${\{X:Y\}}^2$ depends smoothly on~$\Law(X,Y)$.
Yet this is not the case in general:
in fact, maximal correlations are nothing more than a particular case of operator norms
(cf.~\S~\ref{parOperatorInterpretation}), and thus they have the same behaviour%
—they are a continuous function of the parameters,
but they can have some $\mcal{C}^1$~singularity.
The following example exhibits such a singularity.
\end{Rmk}

\begin{Xpl}
Suppose both $X$ and~$Y$ have range $\{1,2,3\}$ and
\[\label{for3787} \big(\!\big( \Pr[X=a\enskip\text{and}\enskip Y=b] \big)\!\big)_{a,b}
= \begin{pmatrix} 2/9 & 1/18 & 1/18 \\
1/18 & 2/9+\alpha & 1/18-\alpha \\
1/18 & 1/18-\alpha & 2/9+\alpha \end{pmatrix} \]
for a parameter $\alpha \in [-2/9,1/18]$.
Then the matrix $\Pi$ defined by~(\ref{for7497}) is
\[ \Pi = \begin{pmatrix} 1/3 & -1/6 & -1/6 \\
-1/6 & 1/3+3\alpha & -1/6-3\alpha \\
-1/6 & -1/6-3\alpha & 1/3+3\alpha \end{pmatrix}
= U \begin{pmatrix} 1/2+6\alpha & 0 & 0 \\
0 & 1/2 & 0 \\ 0 & 0 & 0 \end{pmatrix} U^{-1} ,\]
with
\[ U = \begin{pmatrix} 0 & -2/\sqrt{6} & 1/\sqrt{3} \\
1/\sqrt{2} & 1/\sqrt{6} & 1/\sqrt{3} \\
-1/\sqrt{2} & 1/\sqrt{6} & 1/\sqrt{3} \end{pmatrix} \]
being orthogonal.
So by Proposition~\ref{pro3798}, ${\{X:Y\}} = 1/2 + 6\alpha_+$.
\end{Xpl}

\subsection{Gaussian variables}

The following theorem, which I will frequently use in the sequel, computes exactly the Hilbertian correlation between two jointly Gaussian variables:
\begin{Thm}[\cite{Lancaster,KolmogorovRozanov}]\label{pro1857}
Let~$(\vec{X},\vec{Y})$ be an $(N+M)$-dimensional Gaussian vector
whose covariance matrix writes blockwise
\[ \Var\big( \vec{X},\vec{Y} \big) = \begin{pmatrix} \mathbf{I}_N & C \\
\T{C} & \mathbf{I}_M \end{pmatrix}, \]
then $\{\vec{X}:\vec{Y}\} = \VERT C \VERT$.
\end{Thm}

\begin{Rmk}
In other words, Theorem~\ref{pro1857} tells that in the Gaussian case,
the supremum in~(\ref{for3878}) defining~${\{\vec{X}:\vec{Y}\}}$ can be restricted to \emph{linear} functions~$f$ and~$g$.
\end{Rmk}

\begin{Rmk}\label{rmk4919}
By a linear change of variables, Theorem~\ref{pro1857} actually allows us to compute%
~${\{\vec{X}:\vec{Y}\}}$ for \emph{any} Gaussian vector~$(\vec{X},\ab\vec{Y})$.
\end{Rmk}

\begingroup\def\proofname{Proof of Theorem~\ref{pro1857}}
\begin{proof}
I recall a (sketch of) proof for the sake of completeness.
By the properties of Gaussian vectors,
the law of~$Y$ knowing that ${X=x}$ [I dropped the vector arrows]
is the normal law~${\mcal{N}(\mbf{I}_M-\T{C}C)}+\T{C}x$,
and similarly the law of~$X$ knowing that ${Y=y}$
is the normal law~${\mcal{N}(\mbf{I}_N-C\T{C})}+Cy$.
Consequently, the operator~$\pi_{XYX}$ is the generator
of the following random walk on $\RR^N$
(whose equilibrium measure is the standard Gaussian law):
when one is at~$x$, they jump to a point distributed according to the normal law
$\mcal{N}(\mbf{I}_N-C\T{C}C\T{C})+C\T{C}x$.
This walk is a multidimensional AR($1$)-process (see.~\cite[\S~2.6]{AR(1)}),
whose properties are perfectly known;
in particular, the eigenvalue~$f$ of~$\pi_{XYX}$ responsible for its spectral radius
will be a linear function,
so we only have to consider linear~$f$ in the supremum~(\ref{for3878}).
For such~$f$, the optimal~$g$ will also be linear
by the Gaussian nature of the system,
so in the end $\{\vec{X}:\vec{Y}\}$ is equal to~$\VERT C\VERT$.
\end{proof}\endgroup

\subsection{Miscellaneous examples}

\paragraph{Random conditional laws}

\begin{Xpl}\label{xpl5689}
Let~$0<p<n$ be integers. We consider a random variable $(X,Y)$ for which
$Y$ has range~$\mcal{Y} \coloneqq \{1,\ldots,n\}$ and $X$ has range~$\mcal{X} \coloneqq \mathfrak{P}_p(\mcal{Y})$,
$\mathfrak{P}_p(\mcal{Y})$ denoting the set of subsets $y \subset \mcal{Y}$ with cardinality $p$
—so, $\#\mcal{X} = \binom{n}{p}$ and $\#\mcal{Y} = n$ —,
and we take the law of~$(X,Y)$ uniform on the pairs $(x,y)$ such that $y\in x$:
see Figure~\ref{fig5708}.
\begin{figure}
\centering
\begin{tikzpicture}[x=.5cm,y=1.cm,cm={0,-1,-1,0,(0,0)}]
\draw (0,0) -- (10,0) -- (10,5) -- (0,5) -- cycle;
\node[left] at (5,5) {$\mcal{X}$};
\node[above] at (0,2.5) {$\mcal{Y}$};
\draw[step=1,very thin] (0,0) grid (10,5);
\fill (0,4) -- +(1,0) -- +(1,1) -- +(0,1) -- cycle;
\fill (0,3) -- +(1,0) -- +(1,1) -- +(0,1) -- cycle;
\fill (1,4) -- +(1,0) -- +(1,1) -- +(0,1) -- cycle;
\fill (1,2) -- +(1,0) -- +(1,1) -- +(0,1) -- cycle;
\fill (2,4) -- +(1,0) -- +(1,1) -- +(0,1) -- cycle;
\fill (2,1) -- +(1,0) -- +(1,1) -- +(0,1) -- cycle;
\fill (3,4) -- +(1,0) -- +(1,1) -- +(0,1) -- cycle;
\fill (3,0) -- +(1,0) -- +(1,1) -- +(0,1) -- cycle;
\fill (4,3) -- +(1,0) -- +(1,1) -- +(0,1) -- cycle;
\fill (4,2) -- +(1,0) -- +(1,1) -- +(0,1) -- cycle;
\fill (5,3) -- +(1,0) -- +(1,1) -- +(0,1) -- cycle;
\fill (5,1) -- +(1,0) -- +(1,1) -- +(0,1) -- cycle;
\fill (6,3) -- +(1,0) -- +(1,1) -- +(0,1) -- cycle;
\fill (6,0) -- +(1,0) -- +(1,1) -- +(0,1) -- cycle;
\fill (7,2) -- +(1,0) -- +(1,1) -- +(0,1) -- cycle;
\fill (7,1) -- +(1,0) -- +(1,1) -- +(0,1) -- cycle;
\fill (8,2) -- +(1,0) -- +(1,1) -- +(0,1) -- cycle;
\fill (8,0) -- +(1,0) -- +(1,1) -- +(0,1) -- cycle;
\fill (9,1) -- +(1,0) -- +(1,1) -- +(0,1) -- cycle;
\fill (9,0) -- +(1,0) -- +(1,1) -- +(0,1) -- cycle;
\end{tikzpicture}
\caption{Schematic representation of Example~\ref{xpl5689}
for $n=5$ and $p=2$.}\label{fig5708}
\end{figure}
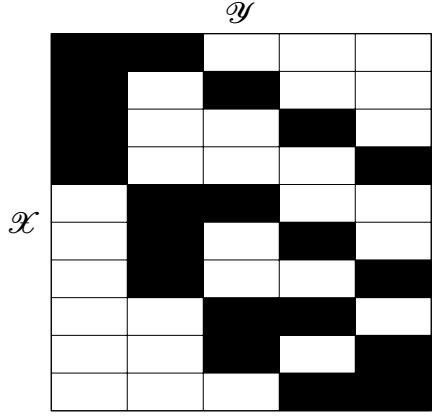

When considered as operators on~$L^2$ spaces, it is obvious that
$\pi_{XY}$ and~$\pi_{YX}$ are characterized by
$\big(\pi_{XY}f\big)(x) = p^{-1}\sum_{y\in x} f(y)$,
resp.\ $\big(\pi_{YX}g\big)(y) = \binom{n-1}{p-1}^{-1} \sum_{y\in x} g(x)$,
so that
\[ \big(\pi_{YXY}f\big)(y)
= \frac{1}{p} f(y) + \sum_{y'\neq y} \frac{p-1}{p(n-1)} f(y') .\]

Thus, on~$\ldb(Y)$, $\pi_{YXY}$ is nothing but the scalar operator $\frac{n-p}{p(n-1)}\mbf{I}$,
and therefore
\[\label{for7767} {\{X:Y\}} = \sqrt{\frac{n-p}{p(n-1)}} \]
by Proposition~\ref{pro7764} and Remark~\ref{rmk1265}.
\end{Xpl}

\paragraph{Weakly coupled particles}

\begin{Pro}\label{p8931}
Let~$V_1$ and~$V_2$ be potentials on~$\RR^n$, $n\geq 1$,
i.e.\ the~$V_i$ are real-valued measurable functions on~$\RR^n$
with $\int_{\RR^n} e^{-V_i(x)} \,\dx{x} < \infty$.
For $i\in\{1,2\}$, denote by~$\PP_i$ the probability measure on~$\RR^n$
proportional to~$e^{-V_i(x)}\dx{x}$, which is to be thought as
the law of the position $X_i$ of a particle $i$ subjected to the potential $V_i$.
Denote $\PP_\otimes \coloneqq \PP_1\otimes\PP_2$, which is the joint law of~$(X_1,X_2)$
in absence of interaction.

Now let~$W$ be an interaction potential on~$(\RR^n)^2$
such that $e^{\big{-[V_1(x_1)+V_2(x_2)+W(x_1,x_2)]}}$ is integrable;
denote by~$\PP$ the probability measure on~$(\RR^n)^2$ proportional to
$e^{-[V_1+V_2+W]} \,\dx{x_1}\dx{x_2}$,
which is the joint law of~$(X_1,X_2)$ in presence of interaction potential $W$.

Then, under the law $\PP$,
\[\label{f8367} \{X_1:X_2\} \leq \frac{\ecty_\otimes(e^{-W})}{\EE_\otimes[e^{-W}]} .\]
\end{Pro}

\begin{proof}
The law $\PP$ has density $h = e^{-W} \div \EE_\otimes[e^{-W}]$ w.r.t.~$\PP_\otimes$,
whence the result by Proposition~\ref{pro0651}.\linebreak[1]\strut
\end{proof}

\begin{Rmk}
Proposition~\ref{p8931} gives a rigorous sense to the intuition that two weakly coupled particles
must have nearly independent positions.
This is valid in a quite general setting,
in particular, $W$ does not have to be bounded.
\end{Rmk}

\paragraph{Non-reversible Markov chain}

\begin{Xpl}
Here is a example showing that the inequality in Proposition~\ref{cor0706}
is strict in general.
Consider the stationary Markov chain on~$\{1,2,3\}$ defined by
\[ P = \big(\!\big( \Pr[X_{k+1}=a | X_k=b] \big)\!\big)_{ab} =
\begin{pmatrix} 0&1/2&1 \\ 1&0&0 \\ 0&1/2&0 \end{pmatrix} ,\]
which has equilibrium measure~$(2/5,2/5,1/5)$.
Diagonalizing $P$ shows that
\[ P^t = \begin{pmatrix} 2/5&2/5&2/5 \\ 2/5&2/5&2/5 \\ 1/5&1/5&1/5 \end{pmatrix}
+ O(2^{-t/2}) ,\]
whence $\{ X_k:X_{k+t} \} = O(2^{-t/2})$ when~$t\longto+\infty$
by Proposition~\ref{pro3798}.
Yet ${\{X_k:X_{k+1}\}} = 1$, since the non-trivial events~${\{X_k = 1\}}$
and~${\{X_{k+1}=2\}}$ are equivalent (cf.\ Proposition~\ref{pro0699}).
\end{Xpl}

\paragraph{Hyperplanes in Ising's model}\label{parHyperplanesInIsingsModel}

As I told in Chapter~\ref{parMotivation},
the initial motivation of this work
was to understand the presence of $\rho$-mixing
in Ising's model (cf.\ \S~\ref{parIsingsModel});
in particular, I intended to re-get a result similar to Theorem~\ref{c6177}
by a more `natural' method.
That shall be achieved indeed in \S~\ref{parBackToIsingsModel}:
\begin{Thm}[Theorem~\ref{t3730}-(\ref{i3743a})]
For Ising's model on~$\ZZ^n$ in the completely analytical regime,
for all disjoint $I,J\subset\ZZ^n$,
\[\label{f46149} \{ \vec\omega_I : \vec\omega_J \} \leq
\exp\big[-\big(\psi'+o(1)\big) \, \dist(I,J)\big] ,\]
where $\psi'$ is the same as in Theorem~\ref{t4505}
and where the ``$o(1)$'' (to be understood ``as $\dist(I,J) \allowbreak {\longto \infty}$'') is uniform in~$I,J$.
\end{Thm}
If we apply that result to the case of parallel `hyperplanes' of~$\ZZ^n$
(I mean, sets of the form $\{t\}\times\ZZ^{n-1}$),
Formula~(\ref{f46149}) looks far less neat than Formula~(\ref{f2881}) in Theorem~\ref{c6177}.

That bound can however be improved by using Proposition~\ref{pro0281}.
Indeed, as we noticed in \S~\ref{parc6177}, the states of two parallel hyperplanes
are elements of some reversible stationary Markov chain.
Therefore, applying Corollary~\ref{cor5902} (in which we let~$k\to\infty$),
we get a result exactly similar to~(\ref{f2881}),
except that we have to replace $\psi$ by~$\psi'$ —recall that it is not known whether $\psi'=\psi$.

\section{Comparing \texorpdfstring{$\rho$}{rho}-mixing to other measures of dependence}

The material of this section is classical; most of it can be found for instance in
\cite[\S\S~3~\&~5]{Bradley-book}.
Here we will say that a sequence of pairs of $\sigma$-algebras $(\mcal{F}^n,\mcal{G}^n)$
is \emph{$\rho$-mixing} to mean that ${\{ \mcal{F}^n : \mcal{G}^n \}} \ab {\stackrel{n\longto\infty}{\longto} 0}$.

\subsection{\texorpdfstring{$\alpha$}{alpha}-mixing}\label{parAlphaMixing}

\begin{Def}
The \emph{$\alpha$-mixing} coefficient of two $\sigma$-algebras~$\mcal{F}$ and~$\mcal{G}$ is
\[ \alpha(\mcal{F},\mcal{G}) \coloneqq
\sup_{\substack{A\in\mcal{A}\\B\in\mcal{B}}} \big| \Pr[A\cap B] - \Pr[A]\Pr[B] \big| .\]
\end{Def}

Proposition~\ref{pro0699} shows that `$\rho$-mixing implies $\alpha$-mixing',
in the sense that one has $\alpha(\mcal{F},\mcal{G}) \leq A\big(\{\mcal{F}:\mcal{G}\}\big)$
for some universal function~$A \colon [0,1] \to [0,1]$
with $A(\rho) \stackrel{\rho\longto0}{\longto} 0$.

\begin{Rmk}
Saying that the correlation of two variables tends to~$0$ means that
their joint law tends in some sense to the product law.
When the variables are ranged in Polish spaces,
a common notion of convergence is weak convergence,
that is, convergence against all bounded continuous function.
\cite[Theorem~2.2]{Billingsley} states that weak convergence is implied
by~$\alpha$-mixing, hence by $\rho$-mixing.
The precise statement is the following:
if $(X^n,Y^n)_{n\in\NN}$ is a sequence of pairs of random variables
such that all the~$X_n$ (resp.~$Y_n$)
have the same law~$\Law(X)$ (resp.~$\Law(Y)$) in some Polish space $E$ (resp.~$F$), then
$\big(\alpha(X^n,Y^n) \stackrel{n\longto\infty}{\longto} 0\big) \ \Rightarrow
\ \big(\Law(X^n,Y^n)\stackrel{n\longto\infty}{\rightharpoonup}\Law(X)\otimes\Law(Y)\big)$.
\end{Rmk}

On the other hand, the following example shows that $\alpha$-mixing does not imply $\rho$-mixing:

\begin{Xpl}\label{x3993}
For~$\epsilon\in (0,1/2]$, define $(X^{\epsilon},Y^{\epsilon})$ in the following way:
\begin{itemize}
\item With probability $\epsilon$, one samples~$X^{\epsilon}$ and~$Y^{\epsilon}$ independently
with common law uniform on~$[0,\epsilon]$;
\item With probability $(1-\epsilon)$, one samples~$X^{\epsilon}$ and~$Y^{\epsilon}$ independently
with common law uniform on~$[\epsilon,1]$.
\end{itemize}
Then for all~$\epsilon>0$ one has $\{X^{\epsilon}:Y^{\epsilon}\}=1$,
since the non-trivial events~${\{X^\epsilon\leq\epsilon\}}$ and~${\{Y^\epsilon\leq\epsilon\}}$
are equivalent (cf.\ Proposition~\ref{pro0699}).
However it is easy to show that $\alpha(X^\epsilon,Y^\epsilon) = \epsilon-\epsilon^2
\stackrel{\epsilon\longto0}{\longto} 0$.
\end{Xpl}

\subsection{\texorpdfstring{$\beta$}{beta}-mixing}

Recall the definition of the $\beta$-mixing coefficient from the previous chapter
[Definition~\ref{def3167}].

\begin{Xpl}\label{xpl9506}
For~$\epsilon \in (0,1)$, consider two random sequences~$(X_i)_{i\in\NN}$ and~$(Y_i)_{i\in\NN}$
defined in the following way: $(X_i)_{i\in\NN}$ is a sequence of i.i.d.\ variables with uniform law on~$\{\pm1\}$,
and for each~$i\in\NN$, independently, one sets
$Y_i = X_i$ with probability $\epsilon$, and with probability $(1-\epsilon)$
one chooses~$Y_i$ uniformly on~$\{\pm1\}$.
Then all the~$(X_i,Y_i)$ are i.i.d.\ 
with $\Pr[X_i=\eta \enskip\text{and}\enskip Y_i=\theta] = (1+\eta\theta\epsilon)/4$
for all~$\eta,\theta\in\{\pm1\}$, thus $\{X_i:Y_i\} = \epsilon$
by Remark~\ref{rmk9549}, whence $\{\vec{X}:\vec{Y}\} = \epsilon$ by Theorem~\ref{pro1252}.
Yet $\Law(\vec{X},\vec{Y})$ and~$\Law(\vec{X})\otimes\Law(\vec{Y})$
are mutually singular for all~$\epsilon>0$.

This shows that $\rho$-mixing does not imply $\beta$-mixing,
and \emph{a fortiori} that there can be no kind of converse to Proposition~\ref{pro0651}.
\end{Xpl}

\subsection{Mutual information}

Recall the definition~(\ref{eq9402}) of mutual information.
\cite[Theorem~5.3(III)]{Bradley-book} states that mutual information controls
the $\beta$-mixing coefficient, so Example~\ref{xpl9506}, which
shows that $\rho$-mixing does not imply $\beta$-mixing in general,
shows that it does not imply mutual information to tend to~$0$ either.

Proposition~\ref{pro0651} suggests that, on the other hand,
maximal correlation could be controlled by mutual information,
but that is not true either: in Example~\ref{x3993} indeed,
$\{X^\epsilon:Y^\epsilon\} = 1$ for all~$\epsilon > 0$, but
\[ I(X^\epsilon;Y^\epsilon) =
\epsilon \ln (\epsilon^{-1}) + (1-\epsilon) \ln \big((1-\epsilon)^{-1}\big)
\stackrel{\epsilon\longto0}{\longto} 0 .\]

Mutual information measures the quantity of information shared by two random variables,
which explains intuitively the following property (\cite[Theorem~2.5.2]{CoverThomas}):
if $X\to Y\to Z$ is a Markov chain, then $I(Y;X,Z) \leq I(X;Y) + I(Y;Z)$.
Does a similar inequality hold for Hilbertian correlation?
In the Gaussian case, the answer is ``yes'' thanks to Theorem~\ref{pro1857}:
one gets that
\[ \{Y:X,Z\}^2 \leq
\frac{(1-\{Y:Z\}^2)\,\{X:Y\}^2 + (1-\{X:Y\}^2)\,\{Y:Z\}^2} {1-\{X:Y\}^2\{Y:Z\}^2}
\leq \{X:Y\}^2 + \{Y:Z\}^2 .\]
But that property does not hold in general, as the following example shows:

\begin{Xpl}\label{xpl1151}
Consider a Markov chain $X\to Y\to Z$,
where~$(Y,X)$ and~$(Y,Z)$ have the same law, which is the joint law described in Example~\ref{xpl5689}%
—the role of ``$Y$'' in that example being played here by~$Y$ in both cases.
Fix $y \in \mcal{Y}$; define event $A$ as ``$Y=y$''
and event $B$ as ``$y\in X\cap Z$''.
Then one computes that $\Pr[A] = n^{-1}$, while
\[ \Pr[B] = \frac{1}{n} + \frac{(p-1)^2}{n(n-1)} .\]
Since $A \subset B$, Proposition~\ref{pro0699}
then yields that
\[\label{for7766} \{Y:X,Z\} \geq \sqrt{\Pr[\C{B}]\Pr[A] \mathbin{\big/} \Pr[\C{A}]\Pr[B]}
= \left(\frac{(n-1)^2 - (p-1)^2} {(n-1)^2 + (n-1)(p-1)^2}\right)^{1/2} .\]

Comparing~(\ref{for7767}) and~(\ref{for7766}),
one sees that taking $n \gg 1$ and $1 \ll p \ll n^{1/2}$
makes~${\{X:Y\}}$ and~${\{Y:Z\}}$ arbitrarily close to~$0$
while $\{Y:X,Z\}$ gets arbitrarily close to~$1$.
\end{Xpl}

\begin{Rmk}
There are similar examples with $(X,Y,Z) = (Y_{-1},Y_0,Y_1)$ for a reversible Markov process $(Y_t)_{t\in\RR}$~\cite{perso}.
\end{Rmk}

\chapter{Event sufficient conditions}\label{parEventSufficientConditions}%
\addcontentsline{itc}{chapter}{\protect\numberline{\thechapter}Event sufficient conditions}

In \S~\ref{parpro0699} we saw that the maximal correlation coefficient $\{\mathcal{F}:\mathcal{G}\}$
controls the difference between~$\Pr[A \cap B]$ and~$\Pr[A]\Pr[B]$
for~$A$ and~$B$ two events resp.\ $\mathcal{F}$- and $\mathcal{G}$-measurable.
A natural question is whether the converse is true,
i.e.\ whether saying that $\Pr[A\cap B]$ is always close in some sense to~$\Pr[A]\Pr[B]$
implies a control on~$\{\mathcal{F}:\mathcal{G}\}$.
We saw in \S~\ref{parAlphaMixing} that $\alpha$-mixing does not fit,
but maybe stronger conditions of the same type would work.

In \S~\ref{parWeakEventSufficientCondition}
I will present a simple such condition (Theorem~\ref{thm0367a}).
This condition demands $|\Pr[A\cap B]-\Pr[A]\Pr[B]|$ to be bounded uniformly
by~$\zeta(\Pr[A])\*\,\theta(\Pr[B])$ for functions $\zeta,\theta \colon [0,1] \to \RR_+$
sufficiently well behaved.
This result, whose proof is rather simple, is apparently new.

Proposition~\ref{pro0699}, however, suggests that the natural condition on events
would be a uniform control on
$\big|\Pr[A\cap B]-\Pr[A]\Pr[B]\big| \mathbin{\big/} \sqrt{\Pr[\C{A}]\Pr[A]}\sqrt{\Pr[\C{B}]\Pr[B]}$,
which is out of the scope of Theorem~\ref{thm0367a}.
Bradley~\cite{Bradley83} proved in 1983 that
that condition was indeed sufficient to get $\rho$-mixing.
His result was improved in the next few years
(see for instance the bound of~\cite{Bulinskii84}),
but the optimal bound was remaining unknown,
though its value was being conjectured.
In \S~\ref{parStrongEventSufficientCondition}, I shall prove this optimal bound.
My method, different from the techniques of~\cite{Bradley83, Bulinskii84},
relies on the analysis of the spectral properties of an operator linked to a law
which I call the ``Chogosov law'', whose study is proceeded to in \S~\ref{parTheChogosovProcess}.

\section{Weak event sufficient condition}\label{parWeakEventSufficientCondition}

To state our next result we need some functional analysis reminders first:
\begin{Def}
On the space $\mathcal{C}^{\infty}_0(0,1)$ of compactly supported fuctions of~$\mathcal{C}^{\infty}(0,1)$,
one defines the scalar product
\[ \langle\phi,\psi\rangle_{H^1_0} = \int_0^1 \phi'(x)\psi'(x) \,\dx{x} .\]
$\mathcal{C}^{\infty}_0(0,1)$ endowed with~$\langle\Bcdot,\Bcdot\rangle_{H^1_0}$ is a pre-Hilbert space;
its completion is denoted by~$H^1_0(0,1)$.
\end{Def}
Recall that elements of~$H^1_0(0,1)$ may be seen as ordinary functions:
\begin{Lem}[Sobolev, {\cite[Theorem~4.12]{Sobolev}}]\label{lem9060}
Any element $f \in H^1_0(0,1)$ can be identified with a unique function
$\bar{f} \in \mathcal{C}^0_0[0,1]$, the space of continuous functions on~$[0,1]$
with $\bar{f}(0), \bar{f}(1) = 0$.
Conversely, a function~$\bar{f} \in \mathcal{C}^0_0[0,1]$ corresponds to
an element of~$H^1_0(0,1)$ if and only if
\[\label{for9029}
\sup_{g\in\mathcal{C}^{\infty}_0(0,1)}
\frac{\big|\int_0^1 f(x)g''(x)\dx{x}\big|}{\sqrt{\int_0^1 g'(x)^2\dx{x}}} \]
is finite, and then there is a unique $f\in H^1_0(0,1)$ associated to~$\bar{f}$,
whose norm is~(\ref{for9029}).
\end{Lem}
In accordance with Lemma~\ref{lem9060}, we will identify functions of~$\mathcal{C}^0_0[0,1]$
with elements of~$H^1_0(0,1)$ whenever it is possible. If $f\in\mathcal{C}^0_0[0,1]$
does not correspond to an element of~$H^1_0(0,1)$, then we will set $\|f\|_{H^1_0} = +\infty$.

Now we can state the
\begin{Thm}[Weak event sufficient condition]\label{thm0367a}
Let~$\mathcal{F}$ and~$\mathcal{G}$ be two $\sigma$-algebras
such that, for all~$A\in\mathcal{F}$ and~$B\in\mathcal{G}$
with respective probabilities~$p$ and~$q$,
\[\label{for0256} \Pr[A\cap B]-pq \leq \zeta(p)\theta(q) \footnotemark\]
\footnotetext{Note that there is no need to put absolute values
in the left-hand side of (\ref{for0256}).}%
for some $\zeta, \theta \in \mathcal{C}^0_0[0,1]$.
Then:
\[\label{for0458} \{\mathcal{F}:\mathcal{G}\} \leq \|\zeta\|_{H^1_0} \|\theta\|_{H^1_0} .\]
\end{Thm}

\begin{proof}
We begin with the following formula for covariance:
\begin{Lem}\label{lem8312}
For~$f$ and~$g$ two real $L^2$ functions,
\[\label{for0257}
\Cov(f,g) = \int_{\RR\times\RR}
\big( \Pr[f\leq x\text{ and }g\leq y] - \Pr[f\leq x] \Pr[g\leq y]\big) \,\dx{x}\dx{y} .\]
\end{Lem}
\begingroup\def\proofname{Proof of Lemma~\ref{lem8312}}
\begin{proof}
Suppose in a first time that $f$ and~$g$ are nonnegative.
A classical Fubini argument (see \cite[Problem~21.6]{Billingsley*})
shows that
\[ \EE[f] = \int_{\RR_+} \Pr[f>x] \,\dx{x} ,\]
with a similar formula for~$g$. By the same method,
\[ \EE[fg] = \int_{\RR_+\times\RR_+} \Pr[f>x \text{ and } g>y] \,\dx{x}\dx{y} ,\]
so that, using the computational formula $\Cov(f,g) = \EE[fg] - \EE[f]\EE[g]$,
\[ \Cov(f,g) = \int_{\RR_+\times\RR_+}
\big( \Pr[f>x\text{ and }g>y] - \Pr[f>x] \Pr[g>y] \big) \,\dx{x}\dx{y}.\]
Observing that the integrand is also
$( \Pr[f\leq x\text{ and }g\leq y] - \Pr[f\leq x] \Pr[g\leq y] )$
and that it is zero for~$(x,y) \notin \RR_+ \times \RR_+$,
we get~(\ref{for0257}) in the nonnegative case.
By translation invariance, the formula remains true for all~$f,g$ bounded below,
and then by approximation for all~$f,g \in L^2$.
\end{proof}\endgroup

Now, let~$f$ and~$g$ be $L^2$ variables resp.\ $\mathcal{F}$- and $\mathcal{G}$-mesurable,
and denote by~$F$ and~$G$ the respective distribution functions of~$f$ and~$g$.
Up to a slight perturbation, $F$ and~$G$
may be supposed to be diffeomorphisms from~$\RR$ onto~$(0,1)$;
denote by~$\alpha$ and~$\beta$ their respective inverse maps.
Then a change of variables in (\ref{for0257}) yields:
\[\label{for0268} \Cov(f,g) = \int_{(0,1)^2}
\Big( \Pr\big[f\leq \alpha(p)\text{ and }g\leq \beta(q)\big] - pq \Big)
\alpha'(p)\beta'(q) \,\dx{p}\dx{q} ,\]
so by assumption (\ref{for0256}):
\[\label{for9181}
\Cov(f,g) \leq \Big(\int_0^1\zeta(p)\alpha'(p) \,\dx{p}\Big) \times
\Big(\int_0^1\theta(q)\beta'(q) \,\dx{q}\Big) .\]
Then our theorem becomes equivalent to the claim stated and proved just below.
\end{proof}
\begin{Clm}\label{clm6860}
If $f$ is a random variable whose repartition function~$F$ is a diffeomorphism
of inverse~$\alpha$, then for~$\zeta \in \mathcal{C}^0_0[0,1]$:
\[ \Big| \int_0^1 \zeta(p)\alpha'(p)\dx{p} \Big| \leq \|\zeta\|_{H^1_0} \ecty(f) .\]
\end{Clm}

\begin{proof}
First note that, replacing $g$ by~$f$ in~(\ref{for0268}),
one has:
\[ \Var(f) = \int_{(0,1)^2} \big[ p(1-q) \wedge q(1-p) \big] \alpha'(p)\alpha'(q) \,\dx{p}\dx{q} .\]
In fact, one can define a scalar product%
\footnote{The positivity of~$\langle\Bcdot,\Bcdot\rangle_V$ follows from the identity
$\langle\phi,\phi\rangle_V = \int_{p<q} \big( \int_p^q \phi(r) \,\dx{r} \big)^2 \,\dx{p}\dx{q}$.}
$\langle\Bcdot,\Bcdot\rangle_V$ on~$\mathcal{C}^0(0,1)$ by setting
\[ \label{for0312}
\langle\phi,\psi\rangle_{V} = \int_{(0,1)^2} \big[ p(1-q) \wedge q(1-p) \big] \phi(p)\psi(q) \,\dx{p}\dx{q} ,\]
so that if $\alpha$ is the inverse distribution fuction
of a variable $f$, $\Var(f) = \|\alpha'\|_V^2$.

So, we are considering three scalar products on some subspaces of~$\mathcal{C}^0(0,1)$:
the ordinary $L^2$ product, which we denote by~$\langle\Bcdot,\Bcdot\rangle_{L^2}$,
the $H^1_0(0,1)$ product $\langle\Bcdot,\Bcdot\rangle_{H^1_0}$
and the variance product $\langle\Bcdot,\Bcdot\rangle_{V}$.
Our goal is to show that for all~$\phi \in H^1_0(0,1), \psi \in \mathcal{C}^0(0,1)$,
\[\label{for0032}
\big| \langle \phi,\psi \rangle_{L^2} \big| \leq \|\phi\|_{H^1_0} \|\psi\|_{V} .\]

By approximation we can suppose that $\phi\in\mathcal{C}^{\infty}_0(0,1)$.
A direct computation shows that
\[ \langle\phi,\psi\rangle_V = \langle L\phi,\psi\rangle_{L^2} ,\]
where the operator $L \colon \mathcal{C}^{\infty}_0(0,1) \to \mathcal{C}^2(0,1)$ is defined by:
\[\label{for0269} \big(L\phi\big)(x) = x\int_0^x (1-y)\phi(y) \,\dx{y} + (1-x)\int_x^1 y\phi(y) \,\dx{y} .\]
But we can make $L$ appear thanks to the following formula:
for~$\phi \in \mathcal{C}^2_0(0,1)$,
\[ \phi = L(-\phi'') ,\]
as one checks by integrating by parts twice.
So,
\[ \big|\langle\phi,\psi\rangle_{L^2}\big| = \big|\langle L(-\phi''),\psi\rangle_{L^2}\big|
= \big|\langle-\phi'',\psi\rangle_V\big| \footrel{\text{CS}}{\leq} \| \phi'' \|_V \| \psi \|_V ,\]
where
\[ \|\phi''\|_V^2 = \langle\phi'',\phi''\rangle_V = \langle L\phi'',\phi''\rangle_{L^2} = -\langle\phi,\phi''\rangle_{L^2}
\footrel{\text{IP}}{=} \|\phi'\|_{L^2}^2 = \|\phi\|_{H^1_0}^2 ,\]
whence~(\ref{for0032}).
\end{proof}

\section{Strong event sufficient condition}\label{parStrongEventSufficientCondition}

\subsection{The strong event sufficient condition}\label{parTheStrongEventSufficientCondition}

A natural choice for functions~$\zeta$ and~$\theta$ in Theorem~\ref{thm0367a}
would be $\zeta(p) = \theta(p) = \epsilon^{1/2} \* \sqrt{p(1-p)}$,
since that would give a converse to Formula~(\ref{for3346}) of Proposition~\ref{pro0699}.
Unfortunately $\| \sqrt{p(1-p)} \|_{H^1_0} = +\infty$,
so Theorem~\ref{thm0367a} does not work in this case.
There is however a specific result then:

\begin{Thm}[Strong event sufficient condition]\label{thm0367b}
Let~$\mathcal{F}$ and~$\mathcal{G}$ be two $\sigma$-algebras
such that, for all~$A$ and~$B$ resp.\ in~$\mathcal{F}$ and~$\mathcal{G}$
with respective probabilities~$p$ and~$q$,
\[\label{for0256b} \Pr[A\cap B]-pq \leq \epsilon \sqrt{p(1-p)q(1-q)} \]
for some $\epsilon\in[0,1]$.
Then
\[\label{for7753} \{\mathcal{F}:\mathcal{G}\} \leq \Lambda(\epsilon) ,\]
where $\Lambda \colon [0,1] \to \RR_+$ is defined by
\[ \Lambda(\epsilon) \coloneqq \begin{cases} \epsilon (1+|\ln\epsilon|) & \text{if $\epsilon>0$,} \\
0 & \text{if $\epsilon=0$.} \end{cases}\]
\end{Thm}

\begin{Rmk}
The function~$\Lambda$ is increasing on~$[0,1]$ and satisfies
$\Lambda(0)=0$, $\Lambda(1)=1$, and $\Lambda(\epsilon)>\epsilon$
for all~$\epsilon\in(0,1)$.
Moreover it is continuous, in particular
$\Lambda(\epsilon) \searrow 0$ as $\epsilon \searrow 0$
(see Figure~\ref{fig-Lambda}).
\end{Rmk}
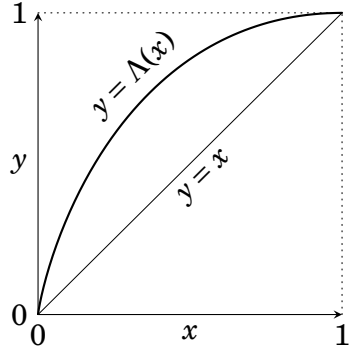
\begin{figure}
\centering
\begin{tikzpicture}[x=4cm, y=4cm, line cap=round, >=stealth]

\draw[->,thin] (0,0) node[anchor=north]{$0$} -- node[anchor=north]{$x$} (1,0) node[anchor=north]{$1$} ;
\draw[->,thin] (0,0) node[anchor=east ]{$0$} -- node[anchor=east ]{$y$} (0,1) node[anchor=east ]{$1$} ;
\draw[thin,dotted] (0,1) -- (1,1) ;
\draw[thin,dotted] (1,0) -- (1,1) ;
\draw[thick] plot[smooth] file{Lambda.txt} ;
\path (.36788,.73576) +(-1pt,-1pt) -- node[above,sloped]{$y=\Lambda(x)$} +(1pt,1pt) ;
\draw[very thin] (0,0) -- node[midway,below,sloped]{$y=x$} (1,1) ;

\end{tikzpicture}
\caption{The function~$\Lambda$.}\label{fig-Lambda}
\end{figure}

\begin{Rmk}
I called Theorems~\ref{thm0367a} and~\ref{thm0367b}
resp.\ ``weak'' and ``strong'' event sufficient conditions;
yet that vocabulary is a bit misleading,
since the strong condition does not imply the weak one \emph{stricto sensu}:
with the hypotheses of Theorem~\ref{thm0367a} indeed,
Theorem~\ref{thm0367b} only implies that
\[\label{fo1611} \{\mathcal{F}:\mathcal{G}\} \leq \Lambda\big(\|\zeta\|_{H^1_0}\|\theta\|_{H^1_0}\big) .\]
But the right-hand side of~(\ref{fo1611}) tends to~$0$
as soon the right-hand side of~(\ref{for0458}) does,
so it is relevant to say that Theorem~\ref{thm0367b} is `qualitatively stronger'
than Theorem~\ref{thm0367a}.
\end{Rmk}

\begin{Rmk}
With the same informal vocabulary,
Theorem~\ref{thm0367b} is a `qualitative converse' of Proposition~\ref{pro0699}:
maximal decorrelation is `qualitatively equivalent' to decorrelation of events
as defined by Formula~(\ref{for3346}).
\end{Rmk}

\begin{Rmk}
One can prove that the bound $\Lambda(\epsilon)$
in~(\ref{for7753}) is the best possible:
see \S~\ref{parOptimalityOfTheStrongEventSufficientCondition}.
\end{Rmk}

\begin{proof}
The core principle of the proof is the same as for Theorem~\ref{thm0367a},
except that we first perform a tricky refinement of the hypothesis:
observing that, for~$A$ and~$B$ with respective probabilities~$p$ and~$q$,
one trivially has $\Pr[A\cap B] \leq p\wedge q$,
the bound~(\ref{for0256b}) can be strengthened into:
\[\label{for0256c} \Pr[A\cap B] \leq \big( pq + \epsilon\sqrt{p(1-p)q(1-q)} \big) \wedge p\wedge q .\]
The right-hand side of~(\ref{for0256c}) will be denoted by~$Z_{\epsilon}(p,q)$.

Now, like in the proof of Theorem~\ref{thm0367a},
if (\ref{for0256b}) is satisfied, for $f$ and~$g$ two $L^2$ real variables
resp.\ $\mathcal{F}$- and $\mathcal{G}$-measurable, having respective distribution functions~$F$ and~$G$
with respective inverses maps~$\alpha$ and~$\beta$:
\[\label{for0268b} \Cov(f,g) \leq \int_{(0,1)^2} \big(Z_{\epsilon}(p,q)-pq\big) \alpha'(p)\beta'(q)
\, \dx{p}\dx{q} .\]
Call~$\langle\alpha',\beta'\rangle_{Z_\epsilon}$ the right-hand side of~(\ref{for0268b}).

To bound $\langle\alpha',\beta'\rangle_{Z_\epsilon}$,
this time we are remaining on a random variable paradigm:
\begin{Def}\label{def5228}
The \emph{Chogosov law}%
\footnote{So called in honour of my dear friend M.~K.~Chogosov.}
is the (unique) probability measure $\mu$ on~$(0,1)^2$ such that
\[\label{for0229}
\forall (p,q) \in (0,1)^2 \qquad \mu\big[ (0,p) \times (0,q) \big] = Z_\epsilon(p,q) .\]
\end{Def}
It shall be proved in \S~\ref{parTheChogosovProcess} that the Chogosov law actually exists.
$\mu$ is invariant under switching $x_1$ and~$x_2$ as the function $Z(p,q)$ is,
and its marginals are uniform on~$(0,1)$ as $Z(p,1) \equiv p$.

The Chogosov law is linked to~$\langle\Bcdot,\Bcdot\rangle_{Z_{\epsilon}}$ by the operator defined next:
\begin{Def}
For $p\in(0,1)$, denote by $\mu_p$ the conditional law of~$x_2$ knowing that $x_1=p$ under~$\mu$:
$(\mu_p)_{p\in(0,1)}$ is the family of probability laws on~$(0,1)$ such that for all measurable $X\subset(0,1)^2$,
\[ \mu[X] = \int_0^1 \mu_p [\{q\colon (p,q)\in X\}] \,\dx{p} .\footnotemark\]
\footnotetext{\emph{Stricto sensu} the family of measures $(\mu_p)_{p\in(0,1)}$
is only defined up to a.s.\ equality; however, this family has a unique version such that
$p \mapsto \mu_p$ is continuous for the weak convergence of measures,
which is the one that we will consider in the sequel.}
\end{Def}
\begin{Def}
Let~$\mcal{L}$ be the operator on bounded functions on~$(0,1)$ defined by
\[ (\mcal{L}\beta)(p) \coloneqq \int_0^1 \beta(q) \,\dx{\mu_p}[q] ;\]
in other words, $\mcal{L}$ is the generator of the stationary Markow chain
$\ldots \to r_0 \to r_1 \to \ldots$ with uniform equilibrium measure on~$(0,1)$
such that the $(r_i,r_{i+1})$ have law~$\mu$.
\end{Def}
\noindent Then the very definition of~$\langle\Bcdot,\Bcdot\rangle_{Z_{\epsilon}}$ yields:
\[ \langle\alpha',\beta'\rangle_{Z_{\epsilon}} = \Cov(\alpha,\mathcal{L}\beta) ,\]
where by writing ``$\Cov(\alpha,\mathcal{L}\beta)$'' I consider
functions~$\alpha$ and~$\mathcal{L}\beta$ as real random variables
on the probability space $(0,1)$ endowed with the uniform measure.

By the Cauchy--Schwarz inequality,
it is then enough to prove that $\Var(\mathcal{L}\beta) \leq \Lambda(\epsilon)\Var(\beta)$,
i.e.\ that the operator norm of~$\mathcal{L}$ on~$\ldb(0,1)$ is bounded above by~$\Lambda(\epsilon)$.
That work is achieved by Lemma~\ref{lem7792} in next subsection.
\end{proof}

\subsection{The Chogosov law}\label{parTheChogosovProcess}

This subsection deals with the ``Chogosov law'',
which we introduced in the proof of Theorem~\ref{thm0367b}.

\begin{NOTA}
Throughout this subsection we suppose $\epsilon\in(0,1)$ fixed
and we write~$\Lambda$ for~$\Lambda(\epsilon)$, resp.~$Z$ for~$Z_{\epsilon}$.
The drawings will be made for $\epsilon=1/2$.
\end{NOTA}

Recall Definition~\ref{def5228} of the Chogosov law.
We first have to check that the Chogosov law actually exists:
\begin{Clm}\label{clm4375}
There exists a (unique) probability measure~$\mu$ on~$(0,1)^2$ such that
\[\label{for7083}
\forall p,q \in [0,1] \quad \mu\big[\big\{(x_1,x_2)\in(0,1)^2 \colon 
x_1\leq p \text{ and } x_2\leq q\big\}\big] = Z(p,q) ,\]
where we recall that
\[ Z(p,q) \coloneqq \big( pq + \epsilon\sqrt{p(1-p)q(1-q)} \big) \wedge p\wedge q .\]
\end{Clm}

\begin{proof}
(\ref{for7083}) means that the density of~$\mu$ on~$(0,1)^2$ is equal to the distribution
$\partial^2_{x_1x_2}Z$; the non-trivial point consists in proving that
that distribution is nonnegative.

\begin{NOTA}
From now on in this subsection elements of~$(0,1)^2$ will be automatically denoted by~$(p,q)$.
Moreover, we will denote $\bar{p} \coloneqq 1-p$ and $\tilde{p} \coloneqq p-1/2$,
resp.\ $\bar{q} \coloneqq 1-q$ and $\tilde{q} \coloneqq q-1/2$.
\end{NOTA}

The analytic formula defining $Z(p,q)$ depends on the zone of~$(0,1)^2$
in which $(p,q)$ lies (see Figure~\ref{fig-mu}):
\begin{itemize}
\item If $p\bar{q} \div q\bar{p} \leq \epsilon^2$, then $Z(p,q) = p$ and we will say that we are in zone~\onecirc;
\item If $\epsilon^2 \leq p\bar{q} \div q\bar{p} \leq \epsilon^{-2}$, then $Z(p,q) = pq + \sqrt{p\bar{p}q\bar{q}}$
and we will say that we are in zone~\twocirc;
\item If $\epsilon^{-2} \leq p\bar{q} \div q\bar{p}$, then $Z(p,q) = q$ and we will say that we are in zone~\threecirc.
\end{itemize}
\begin{figure}
\centering
%
%
\begin{tikzpicture}[x=5.25cm, y=5.25cm, line cap=round, >=stealth]


\draw[->,thin] (0,0) node[anchor=north]{$0$} -- node[midway,anchor=north]{$p$} (1,0) node[anchor=north]{$1$};
\draw[->,thin] (0,0) node[anchor=east ]{$0$} -- node[midway,anchor=east ]{$q$} (0,1) node[anchor=east ]{$1$};
\draw[thin,dotted] (1,0) -- (1,1);
\draw[thin,dotted] (0,1) -- (1,1);


\draw[thick] plot[smooth] file {fuseauD.txt};
\draw[thick] plot[smooth] file {fuseauU.txt};


\node at (.167,.833) {\onecirc};
\node at (.500,.500) {\twocirc};
\node at (.833,.167) {\threecirc};


\coordinate (ligneD) at (0.25000, 0.07692) {};
\coordinate (ligneU) at (0.75000, 0.92308) {};
\path (ligneD) +(0,.125) node[inner sep=1pt] (nomD) {$\mfrak{D}$};
\path (ligneU) +(0,-.125) node[inner sep=1pt] (nomU) {$\mfrak{U}$};
\draw[<-] (ligneD) -- (nomD) ;
\draw[<-] (ligneU) -- (nomU) ;
\end{tikzpicture}
\hfill
%
%
\begin{tikzpicture}[x=5.25cm, y=5.25cm, line cap=round, >=stealth]


\draw[->,thin] (0,0) node[anchor=north]{$0$} -- node[midway,anchor=north]{$p$} (1,0) node[anchor=north]{$1$};
\draw[->,thin] (0,0) node[anchor=east ]{$0$} -- node[midway,anchor=east ]{$q$} (0,1) node[anchor=east ]{$1$};
\draw[thin,dotted] (0,1) -- (1,1);
\draw[thin,dotted] (1,0) -- (1,1);


\begin{scope}[ultra thin]
\input{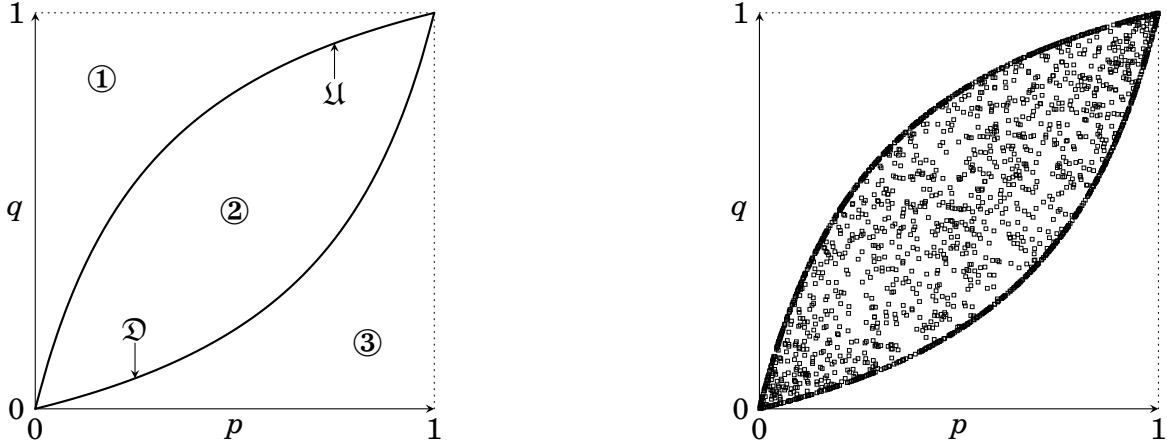}
\end{scope}
\end{tikzpicture}
\caption{The Chogosov law~$\mu$. On the left
are drawn the different zones relative to the support of the measure;
on the right is a cloud of 2,048 independent points with law $\mu$.}%
\label{fig-mu}
\end{figure}

So the expression of~$\partial_pZ$ depends on the zone where one lies:
in~\onecirc\ it is ``$1$'', in~\twocirc\ it is
``$q-\epsilon\tilde{p}\sqrt{q\bar{q} \div p\bar{p}}$'',
and in~\threecirc\ it is ``$0$''. Anyway it is defined and finite evererywhere,
just having jumps at the borders between the zones, which borders we will denote
respectively $\mfrak{U}$ for the border between~\onecirc\ and~\twocirc,
and $\mfrak{D}$ for the border between~\twocirc\ and~\threecirc\ (see Figure~\ref{fig-mu}).
To prove that the distribution~$\partial^2_{pq}Z$ is nonnegative, we have to show
that $\partial_pZ$ is increasing in~$q$ at~$p$ fixed.
Let us check it:
\begin{itemize}
\item In~\onecirc\ and~\threecirc, $\partial_pZ$ is differentiable with
$\partial_q(\partial_pZ) = 0 \geq 0$;
\item In~\twocirc, $\partial_pZ$ is differentiable with
$\partial_q(\partial_pZ) = 1 + \epsilon \tilde{p}\tilde{q} \div \sqrt{p\bar{p}q\bar{q}}$.
Denoting by~$\rho(p,q)$ that expression, let us prove that
$\rho(p,q)$ is nonnegative (and even positive) in~\twocirc:
either $\tilde{p}$ and~$\tilde{q}$ have the same sign and then $\rho(p,q)$ is trivially $\geq 1$,
or $\tilde{p}$ and~$\tilde{q}$ have opposite signs.
In the latter case, say for instance that $(\tilde{p}\geq0\enskip\text{and}\enskip \tilde{q}\ \leq 0)$.
Then $p\geq1/2$ and $q\leq1/2$, so $|\tilde{p}|=p-1/2<p$
and $|\tilde{q}|=1/2-q<\bar{q}$, which implies that
\[ \epsilon \frac{|\tilde{p}\tilde{q}|}{\sqrt{p\bar{p}q\bar{q}}}
< \epsilon \sqrt{\frac{p\bar{q}}{q\bar{p}}} \stackrel{\twocirc}{\leq} \epsilon\sqrt{\epsilon^{-2}} = 1 ,\]
so that $\rho(p,q) > 0$.
\item On~$\mfrak{D}$, $\partial_pZ$ makes a jump. Denote by~$q^{\mfrak{D}}_p$ the unique
$q$ such that $(p,q)\in\mfrak{D}$. When~$q$ tends to~$q^{\mfrak{D}}_p$ by lower values,
$(p,q)$ is in \threecirc, so $\partial_pZ(p,q^{\mfrak{D}}_p-) = 0$,
while when~$q$ tends to~$q^{\mfrak{D}}_p$ by upper values,
$(p,q)$ is in \twocirc, so $\partial_pZ(p,q^{\mfrak{D}}_p+) =
q^{\mfrak{D}}_p-\epsilon\tilde{p}\sqrt{q^{\mfrak{D}}_p\bar{q}^{\mfrak{D}}_p \div p\bar{p}}$.
But on~$\mfrak{D}$, $q\bar{p} = \epsilon^2p\bar{q}$, so
\[ q^{\mfrak{D}}_p- \epsilon\tilde{p}\sqrt{\frac{q^{\mfrak{D}}_p\bar{q}^{\mfrak{D}}_p}{p\bar{p}}}
= q^{\mfrak{D}}_p - \epsilon\tilde{p}\sqrt{\frac{(q^{\mfrak{D}}_p)^2}{\epsilon^2p^2}}
= q^{\mfrak{D}}_p \bigg( 1 - \frac{\tilde{p}}{p} \bigg) = \frac{q^{\mfrak{D}}_p}{2p} > 0 ,\]
so that the jump of~$\partial_pZ(p,\Bcdot)$ at~$q^{\mfrak{D}}_p$
occurs in the increasing sense.
\item Similarly we find that on~$\mfrak{U}$, with obvious notation,
$\partial_pZ(p,q^{\mfrak{U}}_p+) - \partial_pZ(p,q^{\mfrak{U}}_p-) = \bar{q}^{\mfrak{U}}_p \div 2\bar{p} > 0$.
\end{itemize}
So we have proved that $\partial_pZ(p,q)$ is increasing in~$q$, which is what we wanted.
\end{proof}

\begin{Rmk}\label{rmk5379}
The measure~$\mu$ has a rather complicated structure:
it is supported by zone~\twocirc; it has density
$1+\epsilon\tilde{p}\tilde{q} \div \sqrt{p\bar{p}q\bar{q}}$ w.r.t.\ the Lebesgue measure
in the interior of that zone,
and on its boundaries it has a \emph{linear} density giving a mass $(q\div 2p)\,\dx{p}$
to the infinitesimal part of~$\mfrak{D}$ of abscissa~$p$,
resp.\ a mass $(\bar{q}\div 2\bar{p})\,\dx{p}$ to the infinitesimal part of~$\mfrak{U}$ of abscissa~$p$.
See Figure~\ref{fig-mu}.
\end{Rmk}

Now that its existence is ensured, we notice a crucial property of the operator~$\mcal{L}$:
\begin{Pro}\label{pro1288}
$\mcal{L}$ is self-adjoint on~$L^2(0,1)$.
\end{Pro}
\begin{Rmk}
As $\mcal{L}1 \equiv 1$, we can also consider $\mcal{L}$
as an operator on the quotient space $\ldb(0,1)$,
on which it shall also be self-adjoint.
\end{Rmk}
\begingroup\def\proofname{Proof of Proposition~\ref{pro1288}}\begin{proof}
Indeed $\langle \alpha , \mcal{L}\beta \rangle = \EE_{\mu}[\alpha(p)\beta(q)]$,
which is invariant under switching $\alpha$ and~$\beta$
as $\mu$ is invariant under switching $p$ and~$q$.
\end{proof}\endgroup

Now we can turn to the main result of this subsection:
\begin{Lem}\label{lem7792}
The operator norm of~$\mathcal{L}$ on~$\ldb(0,1)$ is bounded above by~$\Lambda$.
\end{Lem}

\begin{proof}
Let~$\eta \in (0,1/2)$, devised to tend to~$0$, and define the distance $d_{\eta}$ on~$(0,1)$ by:
\[ \forall p_1 < p_2 \qquad d_{\eta}(p_1,p_2) \coloneqq \int_{p_1}^{p_2} (p\bar{p})^{-3/2+\eta} \,\dx{p} .\]
For continuous $f \colon (0,1)\to\RR$, define
\[ \|f\|_{\mathit{Lip}(\eta)} \coloneqq \sup_{p_1\neq p_2} \frac{|f(p_2)-f(p_1)|}{d_{\eta}(p_1,p_2)} ,\]
and denote by~$\mathit{Lip}(\eta)$ the set of functions $f$ with $\|f\|_{\mathit{Lip}(\eta)} < \infty$.
$\mathit{Lip}(\eta)$ is obviously complete for~$\|\Bcdot\|_{\mathit{Lip}(\eta)}$,
yet that semi-norm is not definite since it is zero for any constant function.
We thus define $\barLip(\eta)$ as $\mathit{Lip}(\eta)/\RR$,
which is actually a Banach space.
I claim that
\begin{Clm}\label{clm1291}
$\barLip(\eta)$ is continuously imbedded in~$\ldb(0,1)$,
i.e.\ there exists some $C<\infty$ (depending on~$\eta$) such that
for all~$f \in \mathit{Lip}(\eta)$, $\ecty(f) \leq C \|f\|_{\mathit{Lip}(\eta)}$.
\end{Clm}

\begingroup\def\proofname{Proof of Claim~\ref{clm1291}}
\begin{proof}
Fix some arbitrary $p_0\in(0,1)$.
For~$f \in \mathit{Lip}(\eta)$, denoting $y_0 \coloneqq f(p_0)$,
one has, for all~$p \in (0,1)$,
\[ | f(p)-y_0 | \leq \|f\|_{\mathit{Lip}(\eta)} \Big| \int_{p_0}^{p} (q\bar{q})^{-3/2+\eta} \dx{q} \Big| ,\]
whence:
\[\label{for1492} \ecty(f) \leq \sqrt{\int_0^1 \big|f(p)-y_0\big|^2 \dx{p}} \leq
\|f\|_{\mathit{Lip}(\eta)}
\sqrt{\int_0^1 \Big( \int_{p_0}^{p} (q\bar{q})^{-3/2+\eta} \dx{q} \Big)^{\!2} \dx{p}}
.\]
Since $\eta>0$, the integral in the right-hand side of~(\ref{for1492}) is finite,
which proves the claim.
\end{proof}\endgroup

Now, the cruxpoint is the following claim, whose proof is postponed:
\begin{Clm}\label{clm3453}
\begin{ienumerate}
\item\label{itm6574a} There exists a constant $\Lambda_{\eta}<\infty$ such that
for all~$f \in \barLip(\eta)$, $\|\mathcal{L}f\|_{\mathit{Lip}(\eta)} \leq \Lambda_{\eta} \|f\|_{\mathit{Lip}(\eta)}$.
\item\label{itm6574b} It is possible to choose $\Lambda_{\eta}$ so that
$\liminf_{\eta\searrow0} \Lambda_{\eta} \leq \Lambda$.
\end{ienumerate}
\end{Clm}

\noindent Using Claims~\ref{clm1291} and~\ref{clm3453},
one has for all~$f \in \barLip(\eta)$, for~$n\in\NN$,
\[\label{for3674}
\ecty(\mathcal{L}^n f) \leq C \| \mathcal{L}^n f \|_{\mathit{Lip}(\eta)} \leq
C \Lambda_{\eta}^n \|f\|_{\mathit{Lip}(\eta)} \stackrel{n\longto\infty}{=} O(\Lambda_{\eta}^n) .\]
By Lemma~\ref{l3253}, since $\mathcal{L}$ is self-adjoint on~$\ldb(0,1)$
and $\barLip(\eta)$ is a dense subset of~$\ldb(0,1)$,
(\ref{for3674}) implies that $\VERT \mathcal{L} \VERT_{\ldb(0,1)} \leq \Lambda_{\eta}$.
Making $\eta \searrow 0$, $\VERT \mathcal{L} \VERT_{\ldb(0,1)} \leq \Lambda$, \textsc{qed}.
\end{proof}

\begingroup\def\proofname{Proof of Claim~\ref{clm3453}}
\begin{proof}
The proof relies on monotone rearrangement of measures (cf.~\cite[p.~75]{Villani-topics}).
For~$p \in (0,1)$, $\omega \in [0,1]$, define
\[ Q(p,\omega) \coloneqq \inf\big\{ q\in(0,1) \colon \partial_p Z(p,q) \geq \omega \big\} \]
(see Figure~\ref{fig-Q}),
so that $Q(p,\omega)$ is nondecreasing in~$\omega$
and that, for $\omega$ with uniform law on~$[0,1]$,
the law of~$Q(p,\omega)$ is the conditioned version $\mu_p$ of the Chogosov law,
therefore giving:
\[ \big(\mathcal{L}f\big)(p) = \EE\big[f\big(Q(p,\omega)\big)\big] \footnotemark .\]
\footnotetext{Beware that here the expectation is not taken w.r.t.~$p$
but w.r.t.~$\omega$.}%
Then one has the following `coupling formula':
\[\label{for4666}
\big(\mathcal{L}f\big)(p_2) - \big(\mathcal{L}f\big)(p_1) = \EE\big[f\big(Q(p_2,\omega)\big) - f\big(Q(p_1,\omega)\big)\big] .\]
From~(\ref{for4666}) we deduce that
\[ \big| \big(\mathcal{L}f\big)(p_2) - \big(\mathcal{L}f\big)(p_1) \big|
\leq \|f\|_{\mathit{Lip}(\eta)} \EE\big[d_{\eta}\big(Q(p_1,\omega),Q(p_2,\omega)\big)\big] .\]
So, if we can prove that for all~$p_1<p_2$,
\[\label{for5207}
\EE\big[d_{\eta}\big(Q(p_1,\omega),Q(p_2,\omega)\big)\big] \leq \Lambda_{\eta} d_{\eta}(p_1,p_2) ,\]
then we are done.
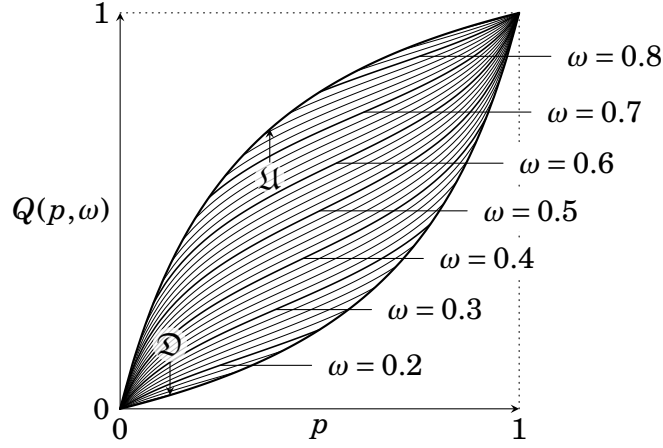
\begin{figure}
\centering
\begin{tikzpicture}[x=5.25cm,y=5.25cm,>=stealth,line cap=round]
\draw[->,anchor=north] (0,0) node{$0$} -- node{$p$}      (1,0) node{$1$};
\draw[->,anchor=east ] (0,0) node{$0$} -- node{$Q(p,\omega)$} (0,1) node{$1$};
\draw[dotted] (0,1) -- (1,1);
\draw[dotted] (1,0) -- (1,1); 

\draw[thick] plot[smooth] file {fuseauD.txt};
\draw[thick] plot[smooth] file {fuseauU.txt};

\begin{scope}[very thin]
\draw plot[smooth] file {equiQ14.txt};
\draw plot[smooth] file {equiQ16.txt};
\draw plot[smooth] file {equiQ18.txt};
\draw plot[smooth] file {equiQ20.txt} [semithick];
\draw plot[smooth] file {equiQ22.txt};
\draw plot[smooth] file {equiQ24.txt};
\draw plot[smooth] file {equiQ26.txt};
\draw plot[smooth] file {equiQ28.txt};
\draw plot[smooth] file {equiQ30.txt} [semithick];
\draw plot[smooth] file {equiQ32.txt};
\draw plot[smooth] file {equiQ34.txt};
\draw plot[smooth] file {equiQ36.txt};
\draw plot[smooth] file {equiQ38.txt};
\draw plot[smooth] file {equiQ40.txt} [semithick];
\draw plot[smooth] file {equiQ42.txt};
\draw plot[smooth] file {equiQ44.txt};
\draw plot[smooth] file {equiQ46.txt};
\draw plot[smooth] file {equiQ48.txt};
\draw plot[smooth] file {equiQ50.txt} [semithick];
\draw plot[smooth] file {equiQ52.txt};
\draw plot[smooth] file {equiQ54.txt};
\draw plot[smooth] file {equiQ56.txt};
\draw plot[smooth] file {equiQ58.txt};
\draw plot[smooth] file {equiQ60.txt} [semithick];
\draw plot[smooth] file {equiQ62.txt};
\draw plot[smooth] file {equiQ64.txt};
\draw plot[smooth] file {equiQ66.txt};
\draw plot[smooth] file {equiQ68.txt};
\draw plot[smooth] file {equiQ70.txt} [semithick];
\draw plot[smooth] file {equiQ72.txt};
\draw plot[smooth] file {equiQ74.txt};
\draw plot[smooth] file {equiQ76.txt};
\draw plot[smooth] file {equiQ78.txt};
\draw plot[smooth] file {equiQ80.txt} [semithick];
\draw plot[smooth] file {equiQ82.txt};
\draw plot[smooth] file {equiQ84.txt};
\draw plot[smooth] file {equiQ86.txt};
\end{scope}

\begin{scope}[mynode/.style={anchor=west,shape=rounded rectangle,fill=white}]
\draw let \p0 = (.2500,.1098) in
(\p0) -- (\p0 -| 0.48,0) node[mynode]{$\omega=0.2$};
\draw let \p0 = (.3889,.2506) in
(\p0) -- (\p0 -| 0.63,0) node[mynode]{$\omega=0.3$};
\draw let \p0 = (.4583,.3797) in
(\p0) -- (\p0 -| 0.76,0) node[mynode]{$\omega=0.4$};
\draw let \p0 = (.5000,.5000) in
(\p0) -- (\p0 -| 0.87,0) node[mynode]{$\omega=0.5$};
\draw let \p0 = (.5417,.6203) in
(\p0) -- (\p0 -| 0.96,0) node[mynode]{$\omega=0.6$};
\draw let \p0 = (.6111,.7494) in
(\p0) -- (\p0 -| 1.03,0) node[mynode]{$\omega=0.7$};
\draw let \p0 = (.7500,.8902) in
(\p0) -- (\p0 -| 1.08,0) node[mynode]{$\omega=0.8$};
\end{scope}

\begin{scope}[mynode/.style={shape=circle,fill=white,fill opacity=.75,text opacity=1,inner sep=0pt}]
\draw[<-] (.1250,.0345) -- +(0,+.0833) node[anchor=south,mynode] {$\mfrak{D}$};
\draw[<-] (.3750,.7059) -- +(0,-.0833) node[anchor=north,mynode] {$\mfrak{U}$};
\end{scope}

\end{tikzpicture}
\caption{The function~$Q(p,\omega)$. This drawing plots the functions $Q(\protect\Bcdot,\omega)$
for values of~$\omega$ running from $0$ to~$1$ with step $0.02$.
Note that all these functions are defined on the whole $(0,1)$:
in fact the graph of~$Q$ `merges' with~$\mfrak{D}$ beyond a certain point for $\omega<1/2$,
resp.\ it merges with~$\mfrak{U}$ below a certain point for $\omega>1/2$.
For $\omega<\epsilon^2/2$, resp.\ $\omega>1-\epsilon^2/2$ (which corresponds here to~$\omega<0.125$, resp.\ $\omega>0.875$),
the whole graph of~$Q(p,\protect\Bcdot)$ is actually equal to the curve $\mfrak{D}$, resp.~$\mfrak{U}$.}\label{fig-Q}
\end{figure}

Now I claim (it will be checked later) that $Q$ is absolutely continuous w.r.t.~$p$,
i.e.\ that there exists an integrable function~$Q' \colon (0,1)\times[0,1] \to \RR$
such that for all~$\omega,p_1,p_2$ one has $Q(p_2,\omega) = Q(p_1,\omega) + \int_{p_1}^{p_2} Q'(p,\omega) \,\dx{p}$.
Introducing that function, (\ref{for5207}) becomes:
\[\label{for5278} \EE\Big[ \Big| \int_{p_1}^{p_2} \big(Q(p,\omega)\bar{Q}(p,\omega)\big)^{-3/2+\eta} Q'(p,\omega) \,\dx{p} \Big| \Big]
\leq \Lambda_{\eta} \EE\Big[ \int_{p_1}^{p_2} (p\bar{p})^{-3/2+\eta} \,\dx{p} \Big] ,\]
so by Fubini's theorem (which is legal here since, as we will see later, $Q'$ is bounded),
proving~(\ref{for5278}) for all~$p_1<p_2$ is tantamount to proving that, for all~$p\in(0,1)$,
\[\label{for5278b} \EE \Big[ \big(Q(p,\omega)\bar{Q}(p,\omega)\big)^{-3/2+\eta} \big|Q'(p,\omega)\big| \Big]
\leq \Lambda_{\eta} (p\bar{p})^{-3/2+\eta} .\]

So we have to compute $Q'(p,\omega)$.
Using the structure of the law $\mu$ (cf.\ Remark~\ref{rmk5379}),
we find the following (see Figure~\ref{fig-Q}):
\begin{itemize}
\item First if $\omega < q^{\mfrak{D}}_p \div 2p$, then $Q(p,\omega) = q^{\mfrak{D}}_p$,
whence $Q'(p,\omega) = \dx{q^{\mfrak{D}}_p}/\dx{p}$.
Differentiating the equality $q\bar{p}=\epsilon^2p\bar{q}$ defining $\mfrak{D}$,
one finds that $\dx{q^{\mfrak{D}}_p}/\dx{p} = (q^{\mfrak{D}}_p + \epsilon^2\bar{q}^{\mfrak{D}}_p) \div (\bar{p}+\epsilon^2p)$,
which simplifies into $q^{\mfrak{D}}_p\bar{q}^{\mfrak{D}}_p \div p\bar{p}$ using once again
that $q\bar{p}=\epsilon^2p\bar{q}$.
\item Similarly if $\omega > 1 - \bar{q}^{\mfrak{U}}_p \div 2\bar{p}$,
one has $Q'(p,\omega) = q^{\mfrak{U}}_p\bar{q}^{\mfrak{U}}_p \div p\bar{p}$.
\item If $q^{\mfrak{D}}_p \div 2p < \omega < 1 - \bar{q}^{\mfrak{U}}_p \div 2\bar{p}$, then
$\partial_p Z(p,q) = q-\epsilon\tilde{p}\sqrt{q\bar{q}/p\bar{p}}$,
thus differentiating the equality $\partial_p Z \big(p,Q(p,\omega)\big) = \omega$,
we get:
\[ Q'(p,\omega) = \frac{\epsilon\sqrt{Q(p,\omega)\bar{Q}(p,\omega)}}
{4\sqrt{p\bar{p}}^3 \bigg( 1 + \epsilon \frac{\tilde{p}\tilde{Q}(p,\omega)}{\sqrt{p\bar{p}Q(p,\omega)\bar{Q}(p,\omega)}} \bigg)} .\]
\item Finally in the critical cases $\omega = q^{\mfrak{D}}_p \div 2p, 1 - \bar{q}^{\mfrak{U}}_p \div 2\bar{p}$,
there is no canonical value for~$Q'(\omega)$ since at these points
$Q(\Bcdot,\omega)$ is not $\mathcal{C}^1$, but that does not matter.
\end{itemize}

\begin{Rmk}
Note that one always has $Q'(p,\omega)>0$, i.e.\ $Q(\Bcdot,\omega)$ is increasing.
In other words, for $p_1 < p_2$, $\mu_{p_1}$ is stochastically smaller than~$\mu_{p_2}$.
\end{Rmk}

We have computed $Q'(p,\omega)$, so now we can tackle~(\ref{for5278b}):
we have to bound 
\[\label{for7401} \int_0^1 \bigg(\frac{p\bar{p}}{Q(p,\omega)\bar{Q}(p,\omega)}\bigg)^{\!3/2-\eta}
Q'(p,\omega) \dx{\omega} ,\]
uniformly in~$p$.
We begin with noticing that
\begin{Clm}\label{clm7320}
For all~$p\in(0,1)$, all~$q\in[q^{\mfrak{D}}_p,q^{\mfrak{U}}_p]$,
one has $q\bar{q}/p\bar{p} \leq \epsilon^{-2}$.
\end{Clm}

\begingroup\def\proofname{Proof of Claim~\ref{clm7320}}
\begin{proof}
The condition $q\in[q^{\mfrak{D}}_p,q^{\mfrak{U}}_p]$ means that
$\epsilon^2 \leq p\bar{q}/q\bar{p} \leq \epsilon^{-2}$.
Then we distinguish two cases:
\begin{itemize}
\item If $p\leq q$, then $q\bar{q}/p\bar{p} = (\bar{q}/\bar{p})^2 q\bar{p}/p\bar{q} \leq q\bar{p}/p\bar{q} = (p\bar{q}/q\bar{p})^{-1} \leq \epsilon^{-2}$;
\item If $p\geq q$, then $q\bar{q}/p\bar{p} = (q/p)^2 p\bar{q}/q\bar{p} \leq p\bar{q}/q\bar{p} \leq \epsilon^{-2}$.
\end{itemize}
\end{proof}\endgroup

\noindent Thanks to Claim~\ref{clm7320}, we bound~(\ref{for7401}) by
\[\label{for7472} \epsilon^{-2\eta} \int_0^1 \bigg(\frac{p\bar{p}}{Q(p,\omega)\bar{Q}(p,\omega)}\bigg)^{\!3/2} Q'(p,\omega) \dx{\omega} ,\]
which we shorthand into $\epsilon^{-2\eta}\lambda(p)$.
Splitting the integral in~(\ref{for7472}) according to the value of~$\omega$
(resp.\ for~$\omega\in(0,q^{\mfrak{D}}_p/2p)$, $(q^{\mfrak{D}}_p/2p,1-\bar{q}^{\mfrak{U}}_p/2\bar{p})$
and~$(1-\bar{q}^{\mfrak{U}}_p/2\bar{p},1)$), one finds:
\begin{eqnarray}
\label{for6433a}
\lambda(p) &=& \frac{q^{\mfrak{D}}_p}{2p} \bigg(\frac{p\bar{p}}{q^{\mfrak{D}}_p\bar{q}^{\mfrak{D}}_p}\bigg)^{\!3/2}
\frac{q^{\mfrak{D}}_p \bar{q}^{\mfrak{D}}_p}{p\bar{p}} \\
\label{for7783}
&+& \int_{q^{\mfrak{D}}_p/2p}^{1-\bar{q}^{\mfrak{U}}_p/2\bar{p}} \bigg(\frac{p\bar{p}}{Q(p,\omega)\bar{Q}(p,\omega)}\bigg)^{\!3/2}
\frac{\epsilon\sqrt{Q(p,\omega)\bar{Q}(p,\omega)}}
{4\sqrt{p\bar{p}}^3 \bigg( 1 + \epsilon \frac{\tilde{p}\tilde{Q}(p,\omega)}{\sqrt{p\bar{p}Q(p,\omega)\bar{Q}(p,\omega)}} \bigg)} \dx{\omega} \\
\label{for6433b}
&+& \frac{\bar{q}^{\mfrak{U}}_p}{2\bar{p}} \bigg(\frac{p\bar{p}}{q^{\mfrak{U}}_p\bar{q}^{\mfrak{U}}_p}\bigg)^{\!3/2}
\frac{q^{\mfrak{U}}_p \bar{q}^{\mfrak{U}}_p}{p\bar{p}}.
\end{eqnarray}
(\ref{for6433a}) simplifies into $\frac{1}{2}\sqrt{q^{\mfrak{D}}_p\bar{p}/p\bar{q}^{\mfrak{D}}_p} = \frac{1}{2}\sqrt{\epsilon^2} = \epsilon/2$;
similarly $(\ref{for6433b}) = \frac{1}{2}\sqrt{p\bar{q}^{\mfrak{U}}_p/q^{\mfrak{U}}_p\bar{p}} = \epsilon/2$.
Concerning term~(\ref{for7783}),
we make the change of variables $q = Q(p,\omega)$, for which
$\dx{\omega} = ( 1 + \epsilon \tilde{p}\tilde{q}/\sqrt{p\bar{p}q\bar{q}} ) \,\dx{q}$
because of the expression of the density of~$\mu$ in zone \twocirc\ (cf.\ Remark~\ref{rmk5379}).
One gets:
\[(\ref{for7783}) = \frac{\epsilon}{4} \int_{q^{\mfrak{D}}_p}^{q^{\mfrak{U}}_p} \frac{1}{q\bar{q}} \dx{q}
= \frac{\epsilon}{4} \bigg[ \ln \frac{q}{\bar{q}} \bigg]_{q^{\mfrak{D}}_p}^{q^{\mfrak{U}}_p}
= \frac{\epsilon}{4} \bigg(\ln \frac{\bar{p}}{\epsilon^2p} - \ln \frac{\epsilon^2\bar{p}}{p} \bigg)
= \frac{\epsilon}{4} \ln\frac{1}{\epsilon^4} = \epsilon|\ln\epsilon|.\]
So in the end we have $\lambda(p) = \epsilon/2+\epsilon/2+\epsilon|\ln\epsilon| = \Lambda$
for all~$p$, thus $\Lambda_{\eta} \leq \epsilon^{-2\eta}\Lambda$ (hence (\ref{itm6574a})),
which tends to~$\Lambda$ as $\eta\searrow0$ (hence (\ref{itm6574b})).
\end{proof}\endgroup

\begin{Rmk}
The simplifications in the computation of~$\lambda(p)$ look rather miraculous\dots\
\emph{A priori} I only expected that $\lambda(p) \leq \Lambda$ on~$(0,1)$
with $\lambda(p) \stackrel{p\longto 0,1}{\longto} \Lambda$).
That I found the exact quasi-eigenvector
associated to the quasi-eigenvalue $\Lambda$ (cf.\ Remark~\ref{r6061})
is purely fortuitous; I have no simple explanation for why things work so well.
\end{Rmk}

\begin{Rmk}\label{r6061}
$\mathcal{L}$ is self-adjoint, hence normal,
so its operator norm is also its spectral radius.
Therefore there is some $(\textit{eigenvalue},\textit{eigenvector})$ pair,
or more precisely (since here the spectral radius of~$\mathcal{L}$ is due to its continuous spectrum)
some `quasi-eigenvalue' and its `quasi-eigenvector'
(cf.~\cite[\S~4]{nonstandard}), which are responsible for the value of the operator norm.

Tracking this quasi-eigenvector throughout the proof of Lemma~\ref{lem7792},
we find that $\Lambda$ is a quasi-eigenvalue of~$\mathcal{L}$
and that the associated quasi-eigenvector is:
\[ f_{\Lambda} \colon p \mapsto \int_{1/2}^p (p'\bar{p}')^{-3/2} \dx{p'} .\]
Obviously $f_{\lambda}$ is not in~$L^2$,
so it is not a true eigenvector; however
one can perturb it slightly to get an element $\tilde{f}_\Lambda\in\ldb(0,1)\setminus\{0\}$
such that $\langle \mathcal{L}\tilde{f}_{\Lambda},\tilde{f}_{\Lambda} \rangle_{\ldb(0,1)}
\mathbin{\big/}
\|\tilde{f}_{\Lambda}\|_{\ldb(0,1)}^2$
is arbitrarily close to~$\Lambda$.
\end{Rmk}

\begin{Rmk}\label{rmk0995}
An interesting feature of~$f_\Lambda$ is that its `$L^2$ mass'
is concentrated about~$0$ and~$1$, so that one needs only look
at what happens near~$0$ and~$1$ to understand
how $f_\Lambda$ contributes to the operator norm of~$\mathcal{L}$.

When one `zooms' more and more to the point $(0,0)$
—the same behaviour would happen about~$(1,1)$ —,
$\mu$ `looks more and more like' the measure~$\mu^*$ on~$(0,\infty)^2$
defined by (see Figure~\ref{fig-mustar}):
\[\label{fo4208}
\forall p,q \in [0,\infty)^2 \quad \mu^*\big[\big\{(x_1,x_2)\in(0,\infty)^2 \colon 
x_1\leq p \text{ and } x_2\leq q\big\}\big] = \epsilon\sqrt{pq} \wedge p \wedge q ,\]
i.e.
\[ \dx\mu^*(p,q) = \1{\{\epsilon^2p<q<\epsilon^{-2}p\}} \frac{\epsilon}{4\sqrt{pq}} \dx{p}\dx{q}
+ \1{\{q=\epsilon^2p\}} \frac{\epsilon^2}{2} \dx{p}
+ \1{\{q=\epsilon^{-2}p\}} \frac{1}{2} \dx{p} .\]
\begin{figure}
\centering
%
%
\begin{tikzpicture}[x=5.25cm, y=5.25cm, line cap=round, >=stealth]


\fill[gray!12.5] (0,0) -- (.25,1) -- (1,1) -- (1,.25) -- cycle ;
\draw[thick] (0,0) -- node[near end,sloped,above]{$q=\epsilon^{-2}p$} (.25,1) ;
\draw[thick] (0,0) -- node[near end,sloped,below]{$q=\epsilon^2  p$} (1,.25) ;


\node[anchor=north east, inner sep=1pt] (0,0) {$0$};
\draw[thin,->] (0,0) -- node[near end,anchor=north]{$p$} (1,0);
\draw[thin,->] (0,0) -- node[near end,anchor=east ]{$q$} (0,1);


\node at (.6,.6) {$C$};

\end{tikzpicture}
\hfill
\begin{tikzpicture}[x=5.25cm, y=5.25cm, line cap=round, >=stealth]



\node[anchor=north east, inner sep=1pt] (0,0) {$0$};
\draw[thin,->] (0,0) -- node[near end,anchor=north]{$p$} (1,0);
\draw[thin,->] (0,0) -- node[near end,anchor=east ]{$q$} (0,1);


\begin{scope}[ultra thin]
\input{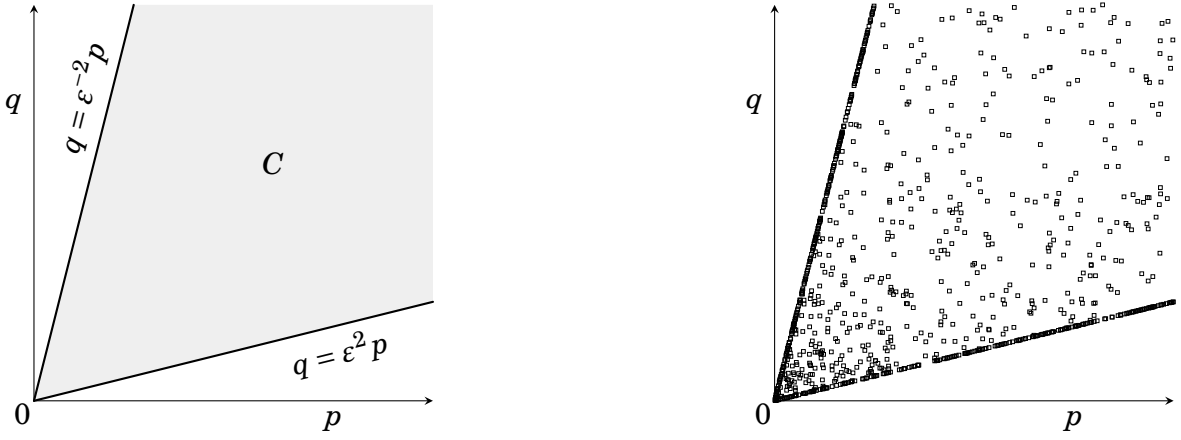}
\end{scope}

\end{tikzpicture}
\caption{The measure~$\mu^*$. On the left, the different zones for the measure;
on the right, a Poisson cloud of points with density $\mu^*$.
The scale and density of the cloud are consistent with Figure~\ref{fig-mu}.}%
\label{fig-mustar}
\end{figure}
So, near $0$, $\mathcal{L}$ behaves like the operator $\mathcal{L}^*$ on~$L^2(0,\infty)$
defined by:
\[\label{fo4259}
\big(\mathcal{L}^*f\big)(p) = \int_{\epsilon^2p}^{\epsilon^{-2}p} \frac{\epsilon}{4\sqrt{pq}} \dx{q}
+ \frac{\epsilon^2}{2} f(\epsilon^2p) + \frac{1}{2} f(\epsilon^{-2}p) .\]

$\mathcal{L}^*$ has scale invariance properties which make it easy to study.
One finds that $\mathcal{L}^*$ is self-adjoint,
that its spectral radius is $\Lambda$,
and that it has $\Lambda$ as a quasi-eigenvalue, associated with the quasi-eigenvector
$(p\mapsto1/\sqrt{p})$.
So, you see that it suffices to study the `local' operator $\mathcal{L}^*$
to compute the spectral radius of the `global' operator $\mathcal{L}$;
in other words, there is a phenomenon of `localization of the spectral radius'
for~$\mathcal{L}$.
\end{Rmk}

\subsection{Optimality of the strong event sufficient condition}%
\label{parOptimalityOfTheStrongEventSufficientCondition}

Now I will prove that Theorem~\ref{thm0367b} is optimal:
\begin{Thm}\label{pro1616}
The factor $\Lambda(\epsilon)$ in~(\ref{for7753}) cannot be improved.
In other words, for all~$\Lambda'<\Lambda(\epsilon)$ it is possible to find
$\sigma$-fields~$\mathcal{F}$ and~$\mathcal{G}$ satisfying
\[\label{for8717}
\forall A \in \mathcal{F}, B \in \mathcal{G} \quad \Pr[A\cap B] - \Pr[A]\Pr[B] \leq \epsilon \sqrt{\Pr[A]\Pr[\C{A}]\Pr[B]\Pr[\C{B}]}, \]
but such that $\{\mathcal{F}:\mathcal{G}\} \geq \Lambda'$.
\end{Thm}
\begin{Rmk}
One can automatically add absolutes values
in the left-hand side of the condition~(\ref{for8717}),
since $-\big(\Pr[A\cap B] - \Pr[A]\Pr[B]\big) = \Pr[A\cap \C{B}] - \Pr[A]\Pr[\C{B}]$.
\end{Rmk}

Actually I will rather prove the following statement,
which is equivalent to the theorem by continuity of the function~$\Lambda(\Bcdot)$:
\begin{Clm}
For all~$\epsilon'>\epsilon$ it is possible to find $\sigma$-fields~$\mathcal{F}$ and~$\mathcal{G}$
satisfying $\{\mathcal{F}:\mathcal{G}\} \geq \Lambda(\epsilon)$, but such that
\[\label{for8717b}
\forall A \in \mathcal{F}, B \in \mathcal{G} \quad \Pr[A\cap B] - \Pr[A]\Pr[B] \leq \epsilon' \sqrt{\Pr[A]\Pr[\C{A}]\Pr[B]\Pr[\C{B}]} .\]
\end{Clm}

\begin{proof}
According to the proof of Theorem~\ref{thm0367b},
the `natural' proof would be to take for space $(\Omega,\mathcal{B},\Pr)$
the set $(0,1)^2$ equipped with its Borel $\sigma$-field
and endowed with the Chogosov law~$\mu$,
and to set $\mathcal{F} = \sigma(p)$ and $\mathcal{G} = \sigma(q)$.
Though it seems to be true that that system satisfies~(\ref{for0256b}),
the complicated structure of~$\mu$ makes existence of a short proof for that property unlikely.
Therefore I will rather adapt the previous idea
to the nicer measure~$\mu^*$ defined by~(\ref{fo4208}),
or more precisely to a `truncation' of it.

My system is the following:
$(\Omega,\mathcal{B},\Pr)$ is the set $(0,1)^2$
equipped with its Borel $\sigma$-field and endowed with a certain measure~$\nu$
(specified just after), and I take $\mathcal{F} = \sigma(p)$, resp.\ $\mathcal{G} = \sigma(q)$.
The measure~$\nu$, which depends on some parameter $x \in (0,1)$ morally close to~$0$,
is a measure on~$(0,1)^2$ having uniform marginals,
which coincides with~$\mu^*$ on~$(0,x]^2$
and which is `as uniform as possible' outside $(0,x]^2$ (see Figure~\ref{fig-nu}).
Technically:
\[ \nu[A\times B] = \begin{cases}
\mu^*(A\times B) & \text{if $A \subset (0,x]$ and $B \subset (0,x]$;} \\
0 & \text{if $A \subset (0,\epsilon^2x]$ and $B \subset (x,1)$;} \\
0 & \text{if $A \subset (x,1)$ and $B \subset (0,\epsilon^2x]$;} \\
\big[ \int_A \big( 1 - \frac{\epsilon}{2}\sqrt{\frac{x}{p}} \big) \dx{p} \big] |B| \big/ (1-x)
& \text{if $A \subset (\epsilon^2x,x]$ and $B \subset (x,1)$;} \\
|A| \big[ \int_B \big( 1 - \frac{\epsilon}{2}\sqrt{\frac{x}{q}} \big) \dx{q} \big] \big/ (1-x)
& \text{if $A \subset (x,1)$ and $B \subset (\epsilon^2x,x]$;} \\
{[}1-(2-\epsilon)x] |A| |B| / (1-x)^2 & \text{if $A \subset (x,1)$ and $B \subset (x,1)$.}
\end{cases} \]
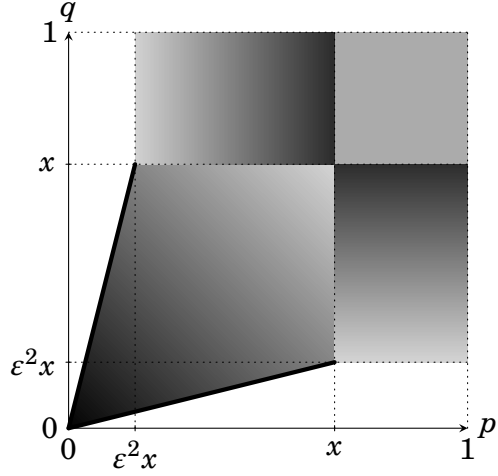
\begin{figure}
\centering
\begin{tikzpicture}[x=5.25cm, y=5.25cm, line cap=round, >=stealth]

\begin{scope}
\clip[draw] (0,0) -- (.1667,.6667) -- (.6667,.6667) -- (.6667,.1667) -- cycle ;
\shade[top color=black!100,bottom color=black!17,shading angle=135] (.3333,.3333) circle (.4714);
\end{scope}

\shade[left color=black!17,right color=black!83] (.1667,.6667) rectangle (.6667,1) ;
\shade[bottom color=black!17,top color=black!83] (.6667,.1667) rectangle (1,.6667) ;
\fill[black!33] (.6667,.6667) rectangle (1,1) ;

\draw[ultra thick] (0,0) -- (.1667,.6667) ;
\draw[ultra thick] (0,0) -- (.6667,.1667) ;
\draw[thin,dotted] (.1667,0) -- +(0,1) ;
\draw[thin,dotted] (.6667,0) -- +(0,1) ;
\draw[thin,dotted] (1,0) -- +(0,1) ;
\draw[thin,dotted] (0,.1667) -- +(1,0) ;
\draw[thin,dotted] (0,.6667) -- +(1,0) ;
\draw[thin,dotted] (0,1) -- +(1,0) ;

\draw[->,thin] (0,0) node[anchor=north]{$0$} --
(1,0) node[anchor=north]{$1$} node[anchor=west]{$p$} ;
\draw[->,thin] (0,0) node[anchor=east]{$0$} --
(0,1) node[anchor=east]{$1$} node[anchor=south]{$q$} ;

\draw (.1667,0) +(0,1pt) -- +(0,-1pt) node[anchor=north]{$\epsilon^2x$} ;
\draw (.6667,0) +(0,1pt) -- +(0,-1pt) node[anchor=north]{$x$} ;
\draw (0,.1667) +(1pt,0) -- +(-1pt,0) node[anchor=east ]{$\epsilon^2x$} ;
\draw (0,.6667) +(1pt,0) -- +(-1pt,0) node[anchor=east ]{$x$} ;

\end{tikzpicture}
\caption{A schematic representation of the measure~$\nu$.}%
\label{fig-nu}
\end{figure}

\noindent\textit{First step: Proof that $\{\mathcal{F}:\mathcal{G}\} \geq \Lambda$.}\quad
Let~$\Lambda' < \Lambda$.
Since $\Lambda$ is in the spectrum of the self-adjoint operator $\mathcal{L}^*$ on~$L^2(0,\infty)$
(see~(\ref{fo4259}) and the lines just below),
there exists $f \in L^2(0,\infty)\setminus\{0\}$
such that $\langle \mathcal{L}^*f,f \rangle / \|f\|_{L^2(0,\infty)} > \Lambda'$.
By a standard truncation argument, we can assume that $f$ has bounded support,
say that $f$ is zero outside $(0,Y]$.
Dividing $f$ by its norm we can also assume that $\|f\|_{L^2} = 1$.

Now, for~$y \in (0,x]$ define the function~$f_y$ by:
\[ f_y(p) \coloneqq \sqrt{\frac{Y}{y}} \, f \! \bigg(\frac{Y}{y}p\bigg) .\]
$f_y$ is zero outside $(0,y] \allowbreak {\subset (0,x]}$; it satisfies $\|f_y\|_{L^2} = 1$ and
\[ \langle \mathcal{L}^*f_y, f_y \rangle = \langle \mathcal{L}^*f,f \rangle > \Lambda' \]
by the scale invariance properties of~$\mathcal{L}^*$.

Denote $m \coloneqq \int f(p) \,\dx{p} / \sqrt{Y}$, which is finite since $f$ is $L^2$ with compact support;
one has $\int f_y(p) \,\dx{p} = \sqrt{y}m$,
so the projection of~$f_y$ on~$\ldb(0,1)$ is the function~$\bar{f}_y = f_y-\sqrt{y}m$.
One has $\|\bar{f}_y\|_{\ldb(0,1)} \leq \|f_y\|_{L^2} = 1$,
and
\[ \EE\big[ \bar{f}_y(p)\bar{f}_y(q) \big] = \langle \mathcal{L}^*f_y,f_y \rangle - m^2y
> \Lambda' - m^2y ,\]
so that $\{\mathcal{F}:\mathcal{G}\} > \Lambda' - m^2y$. Making $y \longto 0$
and then $\Lambda' \longto \Lambda$, one finally gets $\{\mathcal{F}:\mathcal{G}\} \geq \Lambda$.

\noindent\textit{Second step: Proof of~(\ref{for8717b}).}\quad
Let~$\epsilon'>\epsilon$; we want to prove that, provided $x$ is small enough,
(\ref{for8717b}) is satisfied.

Let~$A$ and~$B$ be resp.\ $\mathcal{F}$- and $\mathcal{G}$-measurable events.
One can assume safely that $|A| \leq 1/2$,
since replacing simultaneously~$A$ by~$\C{A}$ and~$B$ by~$\C{B}$ leaves
both sides of~(\ref{for8717b}) unchanged.
One can also assume that $|B| < 1/(1+\epsilon^2)$,
since for $|B|\geq1/(1+\epsilon^2)$, (\ref{for8717b}) comes `for nothing'
by writing
\[ \Pr[A\cap B] - \Pr[A]\Pr[B] \leq \Pr[A](1-\Pr[B]) \leq
\sqrt{\Pr[A]\Pr[\C{A}]} \times \epsilon \sqrt{\Pr[B]\Pr[\C{B}]}
.\]

\begin{NOTA}
In the sequel of this proof we indentify~$A$ and~$B$ with Borel subsets of~$(0,1)$,
rewriting the $p$-measurable event $A$ into the set $A\times(0,1)$, resp.\
the $q$-measurable event $B$ into the set $(0,1)\times B$.
Since both marginals of~$\nu$ are uniform on~$(0,1)$,
one then has $\Pr[A]=|A|$, resp.\ $\Pr[B]=|B|$, so that our goal becomes proving:
\[ \nu[A\times B] - |A||B| \leq \epsilon' \sqrt{|A||B||\C{A}||\C{B}|} .\]
\end{NOTA}

Denote $\check{A} \coloneqq A\cap(0,x]$, resp.\ $\check{B} \coloneqq B\cap(0,x]$.
Provided $x\leq\epsilon/2$, the signed measure~$\dx\nu(p,q) - \dx{p}\dx{q}$ is nonpositive
on~$(0,x]\times(x,1) \cup (x,1)\times(0,x]$,
so that
\[\label{for1899}
\nu[A\times B] - |A||B| \leq
\nu\big[\check{A}\times\check{B}\big] - |\check{A}| |\check{B}|
+ \nu\big[(A\setminus\check{A})\times(B\setminus\check{B})\big]
- \big|A\setminus\check{A}\big| \big|B\setminus\check{B}\big| .\]

Now let us bound above the right-hand side of~(\ref{for1899}):
\begin{itemize}
\item The second term is obviously nonpositive.
\item The third term is $[1-(2-\epsilon)x] \* |A\setminus\check{A}| \* |B\setminus\check{B}| \div (1-x)^2$, so the sum of the two last terms is~%
${(\epsilon x-x^2)} \* {|A\setminus\check{A}|} \* {|B\setminus\check{B}|} \div
(1-x)^2 \ab \leq (\epsilon x - x^2) \* |A| |B|/ \* (1-x)^2$.
Since $|A| \leq 1/2$ and $|B| \leq 1/\allowbreak{(1+\epsilon^2)}$, that quantity
is in turn bounded by~$\frac{\epsilon x-x^2}{\epsilon(1-x)^2} \times
\sqrt{|A||B||\C{A}||\C{B}|}$.
\item For the first term, by Lemma~\ref{clm8768} stated just below, one has
$\nu\big[\check{A}\times \check{B}\big] = \mu^*\big[\check{A}\times \check{B}\big]
\leq \epsilon \sqrt{\mathsmaller{\big|\check{A}\big|\big|\check{B}\big|}}
\leq \epsilon \sqrt{\mathsmaller{\big|\check{A}\big|\big|\C{\check{A}}\big|\big|\check{B}\big|\big|\C{\check{B}}\big|}} \div (1-x)$,
in which, provided $x \leq \epsilon^2 \div (1+\epsilon^2)$, one has
$\sqrt{\mathsmaller{\big|\check{B}\big|\big|\C{\check{B}}\big|}} \leq \sqrt{B|\C{B}|}$
(because then $|\check{B}| \leq |B|\wedge x$ and $|B| \leq 1-x$),
and similarly $\sqrt{\mathsmaller{\big|\check{A}\big|\big|\C{\check{A}}\big|}} \leq \sqrt{A|\C{A}|}$,
so that in the end the first term is bounded by
$\frac{\epsilon}{1-x} \sqrt{|A||B||\C{A}||\C{B}|}$.
\end{itemize}
Summing things up, we get:
\[\label{f4004} \nu[A\times B] - |A||B| \leq
\bigg( \frac{\epsilon}{1-x} + \frac{\epsilon x-x^2}{\epsilon(1-x)^2} \bigg)
\sqrt{|A||B||\C{A}||\C{B}|} .\]
Taking $x$ sufficiently close to~$0$,
the first factor of the right-hand side of~(\ref{f4004}) is $\leq\epsilon'$,
whence the second step of the proof.
\end{proof}

\begin{Lem}\label{clm8768}
For all~$A,B \subset (0,\infty)^2$ with Lebesgue measures $|A|,|B| \ab {<\infty}$,
${\mu^*[A\times B]} \leq \epsilon \sqrt{|A||B|}$.
\end{Lem}
\begingroup\def\proofname{Proof of Lemma~\ref{clm8768}}
\begin{proof}
Recall that $\mu^*$ is the Radon measure on~$(0,\infty)^2$
having density $\epsilon/4\sqrt{pq}$ w.r.t.\ the Lebesgue measure
inside the cone $C = \{(p,q)\colon \epsilon^2p<q<\epsilon^{-2}p\}$,
being zero outside~$C$,
and giving to the borders of~$C$ a lineic mass defined by
$\mu^*\{ (p,\epsilon^2p) \colon p \in A \} = \epsilon^2|A|/2$, resp.\ 
$\mu^*\{ (p,\epsilon^{-2}p) \colon p \in A \} = |A|/2$ (see Figure~\ref{fig-mustar}).
$\mu^*$ is invariant under switching~$p$ and~$q$,
and its marginals both are the Lebesgue measure on~$(0,\infty)$.
Let~$A, B \subset (0,\infty)$ be Borel;
our goal is to show that $\mu^*[A\times B] \leq \epsilon\sqrt{|A||B|}$.

\newcounter{step}
\newcommand{\step}{\addtocounter{step}{1}\textbf{\textit{Step \arabic{step}}.}\quad}
\step If $|A| \leq \epsilon^2|B|$ the result is trivially true,
since then $\mu^*[A\times B] \leq \mu^*[A\times (0,\infty)] = |A| \leq \epsilon\sqrt{|A||B|}$.
Similarly the result is true if $|B| \leq \epsilon^2|A|$.
Therefore in our proof we will always assume that $\epsilon^2|A| < |B| < \epsilon^{-2}|A|$.

\step As for the measure~$\mu$,
decompose the support of~$\mu^*$ into three parts~%
$\mfrak{U}$, \twocirc\ and~$\mfrak{D}$, corresponding resp.\
to the line ``$p=\epsilon^2q$'', the cone $C$ and the line ``$p=\epsilon^{-2}q$''
(see Figure~\ref{fig-mustar}).
Write $\mu^*[A\times B] = m_U + m_2 + m_D$,
where $m_U = \mu^*[(A\times B)\cap\mfrak{U}]$, etc..

Denote by~$\mu^*_q$ the `conditioned version' of~$\mu$ knowing~$q$, i.e. the probability measure such that
\[ \mu^*[X] = \int_0^{\infty} \mu^*_q [\{p\colon (p,q)\in X\}] \,\dx{q} ,\]
which can be computed explicitly to be:
\[\label{for7553} \dx\mu^*_q[p] = \1{\{p=\epsilon^2q\}} \frac{\epsilon^2}{2}
+ \1{\{\epsilon^2q<p<\epsilon^{-2}q\}} \frac{\epsilon}{4\sqrt{pq}} \dx{p} + \1{\{p=\epsilon^{-2}q\}} \frac{1}{2} .\]
The three terms of the right-hand side of~(\ref{for7553}) are respectively due
to~$\mfrak{U}$, $\twocirc$ and~$\mfrak{D}$, so that, integrating the first one, one finds:
\[ m_U = \int_B \frac{\epsilon^2}{2} \1{\{A\ni\epsilon^2q\}} \dx{q}
\leq \frac{\epsilon^2}{2} |B| .\]
Switching the roles of~$p$ and~$q$,
one has similarly $m_D \leq \epsilon^2|A|/2$.
Then it only remains to bound~$m_2$.

\step Let us study further the measures $\mu^*_q$.
If $q\in\epsilon^2A$, then $A\ni\epsilon^{-2}q$
and thus $\mu^*_q[A] \geq \mu^*[\{\epsilon^{-2}q\}] = 1/2$,
and conversely if $q\notin\epsilon^2A$, then $A\not\ni\epsilon^{-2}q$
and thus $\mu^*_q[A] \leq 1-\mu^*[\{\epsilon^{-2}q\}] = 1/2$.
So, $\mu^*_q[A]$ is never smaller if $q\in\epsilon^2A$
than if $q\notin\epsilon^2A$.

As a consequence, let us show that we can always assume
that $\epsilon^2A \subset B$.
Since $|B|>\epsilon^2|A|$,
$|B\setminus \epsilon^2A| > |\epsilon^2A\setminus B|$,
so we can fix some $B^- \subset B\setminus \epsilon^2A$
such that $|B^-| = |\epsilon^2A\setminus B|$.
One has:
\[\label{c7947}
\mu^*[A\times B^-] = \int_{B^-} \mu^*_q(A) \,\dx{q} \leq \frac{|B^-|}{2} \\
= \frac{|\epsilon^2A\setminus B|}{2}
\leq \int_{\epsilon^2A\setminus B} \mu^*_q(A) \,\dx{q}
= \mu^*[A\times(\epsilon^2A\setminus B)] .
\]
Shorthanding ``$(B\setminus B^-)\cup\epsilon^2A$'' into ``$B'$'', (\ref{c7947}) implies that
replacing $B$ by~$B'$ —which does not modify the value of~$|B|$ —%
cannot make $\mu^*[A\times B]$ decrease.
Consequently, if we prove that ${\mu^*[A\times B']} \leq \epsilon\sqrt{|A||B'|}$,
then we will also have proved that $\mu^*[A\times B] \leq \epsilon\sqrt{|A||B|}$.
As $\epsilon^2A\subset B'$, we thus have demonstrated the statement
at the beginning of this paragraph:
one can always assume that $\epsilon^2A\subset B$.

\begin{NOTA}
Switching the roles of~$p$ and~$q$, we will rather impose,
instead of~$\epsilon^2A\subset B$, that $\epsilon^2B\subset A$.
\end{NOTA}

\step Call~$\mu^{\circ}$ the measure~$\mu^*$ restricted to~$C$,
i.e.\ $\dx\mu^{\circ} = \1{C}\dx\mu^*$,
so that $m_2 = \mu^{\circ}[A\times B]$.
$\mu^{\circ}$ is absolutely continuous w.r.t.\ the Lebesgue measure;
denote by~$\mu^{\circ}_q$ its `conditioned version' for fixed~$q$,
i.e.\ the measure such that
\[ \mu^{\circ}[X] = \int_0^\infty \mu^{\circ}_q \big[\big\{p\colon (p,q)\in X\big\}\big] \,\dx{q} ,\]
which has the following explicit density w.r.t.\ the Lebesgue measure:
\[ \label{for5093} \dx\mu^{\circ}_q[p] =
\1{\{\epsilon^2q<p<\epsilon^{-2}q\}} \frac{\epsilon}{4\sqrt{pq}} \,\dx{p} .\]
We perform a change of variables:
for~$y\in(0,|B|)$, define
\[ \beta(y) = \inf\big\{ q\in(0,\infty) \colon |B\cap(0,q)| \geq y \big\} ;\]
so that the push-forward $\beta\pushfwd\dx{q}$
of the Lebesgue measure on~$(0,|B|)$ by the map $\beta$
is equal to~$\1{B}\dx{q}$, the Lebesgue measure restricted to~$B$;
then
\[ m_\text{\twocirc} = \int_B \mu^{\circ}_q[A] \,\dx{q} =
\int_0^{|B|} \mu^{\circ}_{\beta(y)}[A] \,\dx{y} .\]
Our strategy will consist in bounding $\mu^{\circ}_{\beta(y)}[A]$ for all~$y$.

First, we observe that there is some portion of~$A$ which does not contribute to~$\mu^\circ_{\beta(y)}[A]$.
Denote indeed $A_y \coloneqq \{ \epsilon^2q \colon q\in B\cap(0,\beta(y)) \}$;
by the definition of~$\beta$, $|A_y| = \epsilon^2y$,
and one has $A_y \subset \epsilon^2B \subset A$.
But $A_y \subset (0,\epsilon^2\beta(y))$,
so $\mu^{\circ}_{\beta(y)}[A_y] = 0$,
and thus $\mu^{\circ}_{\beta(y)}[A] = \mu^{\circ}_{\beta(y)}[A\setminus A_y]$,
where $|A\setminus A_y| = |A| - |A_y| = |A| - \epsilon^2y$.

Now, for~$q\in(0,\infty)$, the density of~$\mu^{\circ}_q$ is zero for $p\leq\epsilon^2q$
and it is nonincreasing for $p>\epsilon^2q$,
so an immediate coupling argument shows that the maximal value of
$\mu^{\circ}_q[X]$ under the constraint ``$|X|=x$'' is attained for
$X = (\epsilon^2q,\epsilon^2q+x)$.
Applying that result to the conclusion of the previous paragraph, we get that:
\[\label{f1291} \mu^{\circ}_{\beta(y)}[A] \leq
\mu^{\circ}_{\beta(y)} \big[\big(\epsilon^2\beta(y)\,,\,\epsilon^2\beta(y)+|A|-\epsilon^2y\big)\big] .\]
But for~$x\geq 0$, the quantity
$ \mu^{\circ}_q[(\epsilon^2q\,,\,\epsilon^2q+x)] $
can be computed explicitly to be
\[ \mu^{\circ}_q \big[\big(\epsilon^2q\,,\,\epsilon^2q+x\big)\big] =
\left\{\begin{array}{lcl} (1-\epsilon^2)/2 & \quad & \text{if $q\leq x/(\epsilon^{-2}-\epsilon^2)$;} \\
\big( \epsilon \sqrt{\epsilon^2+x/q} - \epsilon^2 \big) /2 & \quad & \text{if $q>x/(\epsilon^{-2}-\epsilon^2)$.} \end{array}\right.\]
In particular, that quantity is a nonincreasing function of~$q$.
Since, by the definition of~$\beta$, one always has $\beta(y) \geq y$,
it follows that (\ref{f1291}) can be improved into:
\[ \mu^\circ_{\beta(y)}[A] \leq \mu^\circ_y [(\epsilon^2y\,,\,|A|)] =
\left\{\begin{array}{lcl} (1-\epsilon^2)/2 &\quad& \text{if $y\leq\epsilon^2|A|$;} \\
\big( \epsilon\sqrt{|A|/y} - \epsilon^2 \big) /2 &\quad& \text{if $y>\epsilon^2|A|$.} \end{array}\right.\]
Integrating, one finds finally:
\begin{multline}
m_2 \leq \int_0^{\epsilon^2|A|} \frac{1-\epsilon^2}{2} \dx{y} + \int_{\epsilon^2|A|}^{|B|}
\bigg( \frac{\epsilon\sqrt{|A|}}{2\sqrt{y}} - \frac{\epsilon^2}{2} \bigg) \dx{y} \\
= \frac{(1-\epsilon^2)\epsilon^2|A|}{2} + \bigg[ \epsilon\sqrt{|A|y} - \frac{\epsilon^2y}{2} \bigg]^{|B|}_{\epsilon^2|A|}
= \epsilon\sqrt{|A||B|} - \frac{\epsilon^2}{2} \big(|A|+|B|\big).
\end{multline}

\step We put our bounds together to get the lemma:
\[ \mu^*[A\times B] \leq m_D + m_U + m_2
= \frac{\epsilon^2}{2} \big(|A|+|B|\big) + \epsilon\sqrt{|A||B|} - \frac{\epsilon^2}{2} \big(|A|+|B|\big)
= \epsilon\sqrt{|A||B|} .\]
\end{proof}\endgroup

\begin{Rmk}
A careful reading of the proof above shows
that the maximal value of~$\mu^*[A\times B]$ is attained for
$A = (0,|A|), B=(0,|B|)$, in which case, provided $\epsilon^2|A|\leq|B|\leq\epsilon^{-2}|A|$,
one has equality in Lemma~\ref{clm8768}.
\end{Rmk}

\chapter{Tensorization}\label{parTensorization}%
\addcontentsline{itc}{chapter}{\protect\numberline{\thechapter}Tensorization}

\section{Subjective correlation}\label{parSubjectiveCorrelation}

In this chapter we will need more advanced definitions for decorrelation.

\begin{Def}
Let~$X$, $Y$ and~$Z$ be random variables. For~$\epsilon\geq 0$,
one says that $X$ and~$Y$ are \emph{subjectively $\epsilon$-decorrelated w.r.t.~$Z$}
(or \emph{$\epsilon$-decorrelated seen from $Z$})
if $X$ and~$Y$ are $\epsilon$-decorrelated under the law $\Law(X,Y|{Z=z})$
for $\Law(Z)$-almost-all~$z$\footnote{The conditional laws $\Law(\Bcdot|{Z=z})$
are only defined up to~$\Law(Z)$-a.e.\ equality, whence the need to specify
``for $\Law(Z)$-almost-all~$z$''.}.

The smallest $\epsilon$ such that $X$ and~$Y$ are $\epsilon$-decorrelated
seen from $Z$ will be called the \emph{subjective correlation level
between~$X$ and~$Y$ w.r.t.~$Z$} (or \emph{correlation level between~$X$ and~$Y$ seen from $Z$});
we denote it ${\{X:Y\}}_{Z}$.
\end{Def}

In \S~\ref{parDefinition}, we had given the definitions
in terms of $\sigma$-algebras rather than random variables.
Of course there is also a $\sigma$-algebra definition for subjective correlation,
though I find it harder to understand:
\begin{Def}\label{def5436}
Let~$\mcal{F}$, $\mcal{G}$ and~$\mcal{H}$ be $\sigma$-algebras.
For~$\epsilon \in [0,1]$, the expression
``$\{\mcal{F}:\mcal{G}\}_{\mcal{H}} \leq \epsilon$'' means that
for all~$f \in \ldb(\mcal{F}\vee\mcal{H})$ and all~$g \in \ldb(\mcal{G}\vee\mcal{H})$
satisfying $\EE[f|\mcal{H}] \equiv 0$, resp.\ $\EE[g|\mcal{H}] \equiv 0$,
one has:
\[ |\EE[fg]| \leq \epsilon \ecty(f) \ecty(g) .\]
\end{Def}

We let the reader check that with that definition,
for~$X$, $Y$ and~$Z$ random variables,
${\{X:Y\}}_{Z} = \{\sigma(X):\sigma(Y)\}_{\sigma(Z)}$.

\begin{Rmk}
The ordinary correlation can be seen as a particular case of subjective correlation,
since $\{\mcal{F}:\mcal{G}\} = \{\mcal{F}:\mcal{G}\}_{\mcal{O}}$ for $\mcal{O} = \{\emptyset,\Omega\}$
the trivial $\sigma$-field.
\end{Rmk}

\begin{Rmk}\label{r4670}
Warning! Writing that ${\{X:Y\}}_Z \leq \epsilon$ does \emph{not} imply
that for all subset $C$ of the range of~$Z$,
$X$ and~$Y$ are $\epsilon$-decorrelated under~$\Law(X,Y|\ab Z\in C)$:
see Examples~\ref{xpl3309} and~\ref{xpl3450} below.
\end{Rmk}

\begin{Rmk}
Warning again! There is no general inequality between~%
${\{X:Y\}}$ and~${\{X:Y\}}_Z$: see Examples~\ref{xpl3308} and~\ref{xpl3309} below.
\end{Rmk}

\begin{Xpl}
Let~$f \colon \RR \to \RR_+$ be a nonnegative continuous function
with $\int_{\RR} f(x)\,\dx{x} = 1$ and let~$(X,Y,Z)$ be a variable
on~$\RR^3$ with density
\[ \dx\Pr\big[(X,Y,Z)=(x,y,z)\big] = \frac{1}{2\pi} f(z)
\exp\Big( \sinh z\cdot xy - {\textstyle\frac{1}{2}} \cosh z\cdot(x^2+y^2) \Big) \,\dx{x}\dx{y}\dx{z} .\]
Then, conditionally to ``$Z=z$'', $(X,Y)$ is a Gaussian vector
with $\Var(X) = \Var(Y) = \cosh z$ and $\Cov(X,Y) = \sinh z$,
so by Theorem~\ref{pro1857}, under the law $\Pr[\Bcdot|Z=z]$ one has ${\{X:Y\}} = |\tanh z|$.
Consequently ${\{X:Y\}}_Z = \sup\{ |\tanh z| \colon f(z)>0 \}$.
\end{Xpl}

The three following examples show that subjective correlation may behave rather wildly,
especially when one changes the $\sigma$-field of reference:
\begin{Xpl}\label{xpl3308}
Let~$X$ and~$Y$ be independent variables with uniform law on~$\RR/\ZZ$
and let~$Z=X+Y$; then ${\{X:Y\}}_{Z} = 1$: under~$\Pr[\Bcdot|\ab Z=z]$ indeed
$Y$ is $X$-measurable (and not constant), since $Y\equiv z-X$.
\end{Xpl}

\begin{Xpl}\label{xpl3309}
Let~$\alpha,\beta,\gamma$ be three independent random variables
uniform on~$\{0,1\}$; define ${X=(\gamma,\alpha)}, \ab Y=(\gamma,\beta)$ and $Z=\gamma$.
Then, conditionally to ``$Z=0$'', $X$ and~$Y$ are independent
with common law uniform on~$\{(0,0),(0,1)\}$,
and similarly $X$ and~$Y$ are independent conditionally to ``$Z=1$'',
so ${\{X:Y\}}_Z=0$.
Yet $X$ and~$Y$ are not independent
since the events ``$X\in\{(0,0),(0,1)\}$'' and ``$Y\in\{(0,0),(0,1)\}$'',
which are non-trivial under~$\Pr$, are equivalent,
so that ${\{X:Y\}}=1$.
\end{Xpl}

\begin{Xpl}\label{xpl3450}
Let $X=(X_1,X_2)$ and $Y=(Y_1,Y_2)$ be independent with uniform laws
on~$\{0,1\}^2$ and define $Z=(X_1,Y_1)$;
then one easily checks that ${\{X:Y\}}_Z=0$.
Now let~$Z'=\1{X_1=Y_1}$, which is $Z$-measurable;
one has ${\{X:Y\}}_{Z'}=1$ since, for instance,
under ``$Z'=1$'' the events ``$X_1=0$'' and ``$Y_1=0$''
are non-trivial and equivalent.
\end{Xpl}

Now we define a more restrictive concept of subjective correlation.
\begin{Def}
A \emph{$\sigma$-metalgebra} $\mcal{M}$ is a set $\{\mcal{H}\colon \mcal{H}\in\mcal{M}\}$
of $\sigma$-algebras which is stable under the ``$\bigvee$'' operator,
i.e.\ such that for any~$\mcal{M}'\subset\mcal{M}$, $\bigvee_{\mcal{H}\in\mcal{M}'}\mcal{H}\in\mcal{M}$.
\end{Def}

One can speak of the `$\sigma$-metalgebra spanned by some set of $\sigma$-algebras',
as states the following immediate proposition:
\begin{Pro}\label{pro0539}
If $(\mcal{H}_k)_{k\in K}$ is a set of $\sigma$-algebras, then there is a smallest
$\sigma$-metalgebra containing all the~$\mcal{H}_k$, which is
\[ \mcal{M} = \Big\{ \bigvee_{k\in K'} \mcal{H}_k \ ;\ K'\subset K \Big\} .\]
\end{Pro}

When one deals with random variables rather than $\sigma$-algebras,
one has the following variant of Proposition~\ref{pro0539}:
\begin{Pro}\label{cor0780}
Let~$(Z_k)_{k\in K}$ be a set of random variables,
then the $\sigma$-metalgebra spanned by~$\{\sigma(Z_k) \colon k\in K\}$
is $\{\sigma(\vec{Z}_{K'})\colon K'\subset K\}$.
\end{Pro}

\begin{Def}
Let~$\mcal{F}$ and~$\mcal{G}$ be $\sigma$-algebras and $\mcal{M}$ be a $\sigma$-metalgebra.
We define the \emph{correlation between~$\mcal{F}$ and~$\mcal{G}$ seen from $\mcal{M}$} by:
\[ \{\mcal{F}:\mcal{G}\}_{\mcal{M}} = \sup_{\mcal{H}\in\mcal{M}} \{\mcal{F}:\mcal{G}\}_{\mcal{H}} .\]
\end{Def}

\begin{Rmk}
Speaking in terms of random variables, if $X$, $Y$ and~$(Z_k)_{k\in K}$ are variables,
denoting by~$\mcal{M}$ the $\sigma$-metalgebra spanned by the~$Z_k$,
then ${\{X:Y\}}_{\mcal{M}}$ is the supremum%
\footnote{More precisely it is a \emph{true} supremum (over $K'$)
of \emph{essential} suprema (over $\vec{z}_{K'}$).}
of the~${\{X:Y\}}$ when taken under all the laws
of kind $\Pr[\Bcdot|\vec{Z}_{K'}=\vec{z}_{K'}]$ for~$K'$ a subset of~$K$
and~$z_k, k\in K'$ elements of the respective ranges of the~$Z_k$.
\end{Rmk}

Finally, the following proposition gathers some easy properties
of relative correlation w.r.t.\ a $\sigma$-metalgebra:
\begin{Pro}\strut
\begin{ienumerate}
\item Call~$\mcal{M}_{\emptyset}$ the trivial $\sigma$-metalgebra, that is,
$\mcal{M}_{\emptyset} = \{\mcal{O}\}$; then for all $\sigma$-algebras~$\mcal{F}$ and~$\mcal{G}$,
$\{\mcal{F}:\mcal{G}\} = \{\mcal{F}:\mcal{G}\}_{\mcal{M}_{\emptyset}}$.
\item If $\mcal{M} \subset \mcal{M}'$, then $\{\mcal{F}:\mcal{G}\}_{\mcal{M}} \leq \{\mcal{F}:\mcal{G}\}_{\mcal{M}'}$.
\item Let~$\mcal{F}$ and~$\mcal{G}$ be $\sigma$-algebras,
let~$\mcal{M}$ be a $\sigma$-metalgebra, and call~$\tilde{\mcal{M}}$ the $\sigma$-metalgebra
spanned by~$\mcal{M}$, $\mcal{F}$ and~$\mcal{G}$; then
$\{\mcal{F}:\mcal{G}\}_{\mcal{M}} = \{\mcal{F}:\mcal{G}\}_{\tilde{\mcal{M}}}$.
\end{ienumerate}
\end{Pro}

\begin{Def}\label{def3128}
In the sequel, the probabilistic systems which we shall consider
will often be made of some `elementary' variables, say $(X_i)_{i\in I}$.
In this case, the so-called \emph{natural $\sigma$-metalgebra} of the system
will mean the $\sigma$-metalgebra spanned by the~$X_i$.
\end{Def}

\section{Simple tensorization}\label{parSimpleTensorization}

Now we turn to tensorization.
First let us deal with `simple' tensorization,
by which I mean that tensorization is performed on only one variable.
The main result of this section will be the `$N$~against~$1$' theorem
(Theorem~\ref{thm5750}).

The problem considered is the following:
Let~$I$ be a set and~$(X_i)_{i\in I}, Y$ be random variables;
call~$\mcal{M}$ the natural $\sigma$-metalgebra of this system, that is,
the $\sigma$-metalgebra spanned by the~$X_i$ and~$Y$ (cf.\ Definition~\ref{def3128}).
Suppose we have bounds ${\{X_i:Y\}_{\mcal{M}}} \leq \epsilon_i$ for all~$i$;
the question is, can we deduce from them a bound on~$\{\vec{X}_I:Y\}$?
We shall prove that the answer is ``yes'',
and moreover the bound~(\ref{for5719}) we will give is optimal in some way
(see \S~\ref{parOptimality}).

For pedagogical purpose,
let us first state and prove a weaker but easier proposition:
\begin{Pro}\label{pro5527}
With the notation above,
\[\label{for5528} \big\{\vec{X}_I:Y\big\} \leq \sqrt{\sum_{i\in I} \epsilon_i^2} .\]
\end{Pro}

\begin{proof}
By Proposition~\ref{pro6559} we may assume $I=\{1,\ldots,N\}$.
Let~$f$ and~$g$ be centered $L^2$ $\vec{X}$-measurable,
resp.\ $Y$-measurable, functions; our goal is to bound $|\EE[fg]|$.

For all~$i\in\{0,\ldots,N\}$, denote
\[ \mcal{F}_i = \sigma(X_1,\ldots,X_i) ,\]
and for all~$i\in\{1,\ldots,N\}$,
\[ f_i = f^{\mcal{F}_i} - \EE[f|\mcal{F}_{i-1}] .\]
Then $f = \sum_i f_i$, where each~$f_i$ is $\mcal{F}_i$-measurable
and centered w.r.t.~$\mcal{F}_{i-1}$ (i.e., $\EE[f_i|\mcal{F}_{i-1}] \equiv 0$).
Consequently, for all~$i_0<i_1$ one has $\EE[f_{i_0}f_{i_1}] = 0$
(since $f_{i_0}$ is $\mcal{F}_{i_0}$-measurable while
$f_{i_1}$ is centered w.r.t.~$\mcal{F}_{i_1-1} \supset \mcal{F}_{i_0}$)
and thus when one expands $\Var(f) = \EE\big[\big(\sum_i f_i\big)^2\big]$
all the non-diagonal terms vanish, yielding:
\[\label{f9207} \Var(f) = \sum_{i=1}^{N} \Var(f_i) .\]

Now, the decomposition ``$f = \sum_i f_i$'' yields
\[ \EE[fg] = \sum_{i=1}^{N} \EE[f_ig] ,\]
so let us bound the~$|\EE[f_ig]|$.
The law of total expectation gives:
\[ \EE[f_ig] =
\int \EE\big[f_ig\big|X_1=x_1,\ldots,X_{i-1}=x_{i-1}\big] \,\dx\Pr[x_1,\ldots,x_{i-1}] .\]
But under~$\dx\Pr[\Bcdot|\ab x_1,\ldots,x_{i-1}]$,
$f_i$ is $X_i$-measurable and centered while $g$ is $Y$-mea\-su\-ra\-ble,
moreover under this law
$\{X_i:Y\} \leq \{X_i:Y\}_{\mcal{F}_{i-1}} \leq \{X_i:Y\}_{\mcal{M}} \leq \epsilon_i$, so:
\[ \big|\EE\big[f_ig\big|Y_1=y_1,\ldots,Y_{i-1}=y_{i-1}\big]\big| \leq \epsilon_i
\ecty\big(f_i\big|x_1,\ldots,x_{i-1}\big) \ecty\big(g\big|x_1,\ldots,x_{i-1}\big) .\]
Using the bound $\ecty(h) \leq \sqrt{\EE[h^2]}$, it follows that:
\begin{multline}\label{cal4816}
\big|\EE[f_ig]\big| \leq \epsilon_i
\int \sqrt{\EE[f_i^2|x_1,\ldots,x_{i-1}]} \sqrt{\EE[g^2|x_1,\ldots,x_{i-1}]}
\,\dx\Pr[x_1,\ldots,x_{i-1}] \\
\footrel{\text{CS}}{\leq} \epsilon_i \sqrt{\int\EE[f_i^2|x_1,\ldots,x_{i-1}]\,\dx\Pr[x_1,\ldots,x_{i-1}]}\sqrt{\textit{the same for~$g$}}
= \epsilon_i \ecty(f) \ecty(g_i) .
\end{multline}
So, summing~(\ref{cal4816}) for all~$i$:
\[\label{f3109} |\EE[fg]| \leq \sum_{i=1}^N \epsilon_i \ecty(f_i) \ecty(g)
\footrel{\text{CS}}{\leq} \sqrt{\sum_i \epsilon_i^2} \sqrt{\sum_i \Var(f_i)} \, \ecty(g)
= \sqrt{\sum_i \epsilon_i^2} \, \ecty(f) \ecty(g) .\]
Since (\ref{f3109}) is true for all~$f,g$, (\ref{for5528}) is proved.
\end{proof}

\label{lbl91}It is striking in Proposition~\ref{pro5527}
that the right-hand side of~(\ref{for5528})
may be greater than $1$, which is never the case for a correlation level.
Actually there is some `loss of optimality' in the proof
of the proposition when we bound above $\Var(f_i|\mcal{F}_{i-1})$
by~$\EE[f_i^2|\mcal{F}_{i-1}]$, since $\EE[f_i^2|\mcal{F}_{i-1}] - \Var(f_i|\mcal{F}_{i-1}) = \EE[f_i|\mcal{F}_{i-1}]^2$
may be different to~$0$. We will use a technique for `recycling' that loss
to get the following result, which \S~\ref{parOptimality} shall prove to be optimal:
\begin{Thm}[`$N$~against~$1$' theorem]\label{thm5750}
Take the same hypotheses as in Proposition~\ref{pro5527}:
$\forall i\in I \quad \{X_i:Y\}_{\mcal{M}} \leq \epsilon_i$, where $\mcal{M}$
is the natural $\sigma$-metalgebra of the system. Then:
\[\label{for5719} \big\{\vec{X}_I:Y\big\} \leq \sqrt{1-\prod_{i\in I}(1-\epsilon_i^2)} .\]
\end{Thm}

\begin{Rmk}
The right-hand side of~(\ref{for5719})
is the~$\bar\epsilon\in[0,1]$ characterized by
$1-\bar\epsilon^2 = \prod_i (1-\epsilon_i)^2$.
\end{Rmk}

\begin{Rmk}
The right-hand side of~(\ref{for5719}) is bounded above by~$\sqrt{\sum_i \epsilon_i^2}$,
so Theorem~\ref{thm5750} gives back Proposition~\ref{pro5527} as a corollary.
\end{Rmk}
\begin{proof}
As in the proof of Proposition~\ref{pro5527},
let~$f$ and~$g$ be centered $L^2$ $\vec{X}$-measurable,
resp.\ $Y$-measurable, functions.
Assume $I=\{1,\ldots,N\}$;
denote $\mcal{F}_i \coloneqq \sigma(X_1,\ldots,X_i)$
and $f_i \coloneqq f^{\mcal{F}_i} - \EE[f|\mcal{F}_{i-1}]$.
Also denote, for~$i\in\{0,\ldots,N\}$,
\[ g^i \coloneqq g - \EE[g|\mcal{F}_i] .\]

As before, one has $\Var(f) = \sum_i\Var(f_i)$ and $\EE[fg] = \sum_{i=1}^N \EE[f_ig]$.
But $f_i$ is centered w.r.t.~$\mcal{F}_{i-1}$
while $(g - g^{i-1})$ is $\mcal{F}_{i-1}$-measurable,
so $\EE[f_ig] = \EE[f_ig^{i-1}]$.
Since, conditionally to~$\mcal{F}_{i-1}$, $f_i$ and~$g^{i-1}$ are both centered and resp.\
$X_i$- and $Y$-measurable, the fact that $\{X_i:Y\}_{\mcal{M}} \leq \epsilon_i$ implies,
by the same argument as in the previous proof,
that
\[ |\EE[f_ig^{i-1}]| \leq \epsilon_i \ecty(f_i) \ecty(g^{i-1}) .\]

Now, for~$i\in\{1,\ldots,N\}$, denote
\[ \bar{g}^i = \EE[g^{i-1}|\mcal{F}_i] .\]
Since $g^{i-1} = \bar{g}^i + g^i$,
where $\bar{g}^i$ is $\mcal{F}_i$-measurable while $g^i$ is centered w.r.t.~$\mcal{F}_i$, one has:
\[ \Var(g^i) = \Var(g^{i-1}) - \Var(\bar{g}^i) .\]
Then, the point consists in making the following observation:
for~$\Var(g^i)$ to be large (that is, close to~$\Var(g^{i-1})$),
$\Var(\bar{g}^i)$ has to be small. But in that case
$|\EE[f_ig]|$ shall be small:
one has indeed, since $f_i$ is $\mcal{F}_i$-measurable,
\[ |\EE[f_ig]| = |\EE[f_ig^{i-1}]|
= \big|\EE\big[f_i(g^{i-1})^{\mcal{F}_i}\big]\big| = \big|\EE\big[f_i\bar{g}^i\big]\big|
\footrel{\text{CS}}{\leq} \ecty(f_i) \ecty(\bar{g}^i) .\]

Let us sum up the relations obtained.
One has, for all~$i\in\{1,\ldots,N\}$:
\begin{eqnarray}
\label{f9041a} |\EE[f_ig]| & \leq & \epsilon_i \ecty(f_i) \ecty(g^{i-1}) ; \\
\label{f9041b} \ecty(g^i) & = & \sqrt{\ecty(g^{i-1})^2 - \ecty(\bar{g}^i)^2} ; \\
\label{f9041c} |\EE[f_ig]| & \leq & \ecty(f_i) \ecty(\bar{g}^i) .
\end{eqnarray}
Now define $\hat\epsilon_i \coloneqq |\EE[f_ig]| \div \ecty(f_i)\ecty(g^{i-1})$,
or $\hat\epsilon_i = 0$ if the right-hand side is $0 \div 0$.
Then (\ref{f9041a}) ensures that $\hat\epsilon_i \leq \epsilon_i$,
and (\ref{f9041c}) means that $\ecty(\bar{g}^i) \geq \hat\epsilon_i\ecty(g^{i-1})$,
so that (\ref{f9041b}) yields $\ecty(g^i) \leq \sqrt{1-\hat\epsilon_i^2} \*
\hspace{.056em} \ecty(g^{i-1})$.
Since $g^0 = g$, one has therefore by induction
$\ecty(g^i) \leq \prod_{i'=1}^{i-1} \ab \sqrt{1-\hat\epsilon_{i'}^2} \,\ecty(g)$,
so that the decomposition ``$\EE[fg]=\sum_i\EE[f_ig]$'' gives:
\[\label{f8485} |\EE[fg]| \leq \sum_{i=1}^N
\Big( \hat\epsilon_i \prod_{i'=1}^{i-1} \sqrt{1-\hat\epsilon_{i'}^2} \Big)
\ecty(f_i) \ecty(g) .\]
By the Cauchy--Schwarz inequality, it follows that:
\[\label{f9280} |\EE[fg]| \leq \sqrt{\sum_{i=1}^N \hat\epsilon_i^2
\prod_{i'=1}^{i-1} \big(1-\hat\epsilon_{i'}^2 \big)} \, \ecty(f)\ecty(g)
= \sqrt{1-\prod_{i=1}^N \big(1-\hat\epsilon_i^2\big)} \, \ecty(f)\ecty(g) .\]
Obviously the maximal value for the right-hand side of~(\ref{f9280})
is when $\hat\epsilon_i = \epsilon_i$ for all~$i$,
then yiel\mbox{di}ng~(\ref{for5719}).
\end{proof}

There is an alternative proof, which is less intuitive
but whose reasoning shall be used again in the proof of Theorem~\ref{thm6413}:
\begingroup\def\proofname{Alternative proof of Theorem~\ref{thm5750}}
\begin{proof}
We use the same notation as in the previous proof.
As $f$ is $\mcal{F}$-measurable, $\EE[fg] = \EE[fg^{\mcal{F}}]$,
so by the Cauchy--Schwarz inequality:
\[\label{f2641} |\EE[fg]| \leq \ecty(f) \ecty(g^{\mcal{F}}) .\]
Now, by associativity of variance $\ecty(g^{\mcal{F}}) = \sqrt{\Var(g)-\Var(g-g^{\mcal{F}})}$,
so by~(\ref{f2641}) it suffices to prove that
\[\label{for5211}
\Var\big(g-g^{\mcal{F}}) \geq \prod_{i=1}^N \big(1-\epsilon_i^2\big) \Var(g) .\]

With our notation, $g-g^{\mcal{F}} = g^N$ and $g = g^0$;
we will prove that for all~$i\in\{1,\ldots,N\}$,
\[\label{for0829} \Var(g^i) \geq (1-\epsilon_i^2) \Var(g^{i-1}) .\]
Since $g^{i-1}$ and~$g^i$ are centered w.r.t.~$\mcal{F}_{i-1}$,
one has
\[ \Var(g^{i-1}) = \int \Var(g^{i-1}|x_1,\ldots,x_{i-1}) \,\dx\Pr[x_1,\ldots,x_{i-1}] ,\]
with a similar decomposition for~$\Var(g^i)$,
so that it suffices to prove~(\ref{for0829}) conditionally to~$\mcal{F}_{i-1}$.

Conditionally to~$\mcal{F}_{i-1}$, $g^{i-1}$ is centered and $Y$-measurable.
Moreover, $g^i = g^{i-1} - (g^{i-1})^{\sigma(X_i)}$,
so by associativity of variance
$\Var(g^i) = \Var(g^{i-1}) - \Var\big( (g^{i-1})^{\sigma(X_i)} \big)$,
and therefore (\ref{for0829}) is equivalent to
\[ \Var\big( (g^{i-1})^{\sigma(X_i)} \big) \geq \epsilon_i^2 \Var(g^{i-1}) ,\]
which follows directly from the assumption ``$\{X_i:Y\}_{\mcal{F}_{i-1}} \leq \epsilon_i$''.
\end{proof}\endgroup

\section{Double tensorization}\label{parDoubleTensorization}

Simple tensorization as itself is already interesting
since it gives an $L^2$-type bound for the correlation between~$X$ and~$\vec{Y}$,
which is better than the $L^1$-type bounds
typically obtained by total variation methods.
Yet it does not exhaust
the full potential of Hilbertian correlations
concerning tensorization, since obviously
it does not contain results like independent tensorization
(cf.\ \S~\ref{parIndependentTensorization}).

The aim of this section is to get sharp tensorization results
where we perform tensorizing on \emph{both} sides,
without having to assume complete independence like in Theorem~\ref{pro1252}.
The price to pay is that the techniques involved,
though similar in their spirit, will be much more tricky,
moreover the bounds obtained will not be completely optimal
(see \S~\ref{parOptimality}).

\subsection{`\texorpdfstring{$N$}{N}~against~\texorpdfstring{$M$}{M}' tensorization}

The following theorem may be considered as the main result of this monograph.
As will be explained in \S~\ref{parAsymptoticOptimality},
it `contains' qualitatively all the other tensorization theorems
(i.e.\ Theorems~\ref{pro1252}, \ref{thm5750} and~\ref{thm0119}).

\begin{Thm}[`$N$~against~$M$' theorem]\label{thm6413}
Let~$I$ and~$J$ be sets, and let~$(X_i)_{i\in I}$ and~$(Y_j)_{j\in J}$ be random variables,
the $\sigma$-metalgebra they generate being denoted by~$\mcal{M}$.
Suppose for any~$i,j$, ${\{X_i:Y_j\}_{\mcal{M}}} \leq \epsilon_{ij}$ for some $\epsilon_{ij} \geq 0$,
and define the operator
\[ \begin{array}{rrcl} \boldepsilon \colon & L^2(J) & \to & L^2(I) \\
& (a_j)_{j\in J} & \mapsto & \big( \sum_{j\in J} \epsilon_{ij}a_j \big)_{i\in I} ,\end{array} \]
then:
\[\label{for3009} \big\{ \vec{X}_I : \vec{Y}_J \big\} \leq \VERT \boldepsilon \VERT \wedge 1 .\]
\end{Thm}

\begin{Rmk}
On~$(\RR_+)^{I\times J}$, $\VERT \boldepsilon \VERT$
is a nondecreasing function of each~$\epsilon_{ij}$.
\end{Rmk}

\begin{NOTA}
As the proof of Theorem~\ref{thm6413} is rather technical,
I found it useful to write down how it goes on a concrete example.
This is performed in Appendix~\ref{parIllustrationOfTheProofs},
which I suggest the reader to look at in parallel with the proof as a complement.
\end{NOTA}

To prove Theorem~\ref{thm6413},
we will need the following
\begin{Lem}\label{lem7265}
Let~$X_1,X_2,\ldots,X_N$ and~$Y$ be random variables,
call~$\mcal{M}$ their natural $\sigma$-metalgebra,
and assume that for all~$i\in\{1,\ldots,N\}$,
\[ \{X_i:Y\}_{\mcal{M}} \leq \epsilon_i .\]
Let~$f$ be an $\ldb(\vec{X})$ function.
For all~$0 \leq i \leq N$, denote $\mcal{F}_i \coloneqq \sigma(X_1,\ldots,X_i)$,
resp.\ $\mcal{F}^*_i \coloneqq \sigma(X_1,\ldots,\ab X_i,Y)$,
and for all~$1 \leq i \leq N$, define
\[ f_i \coloneqq f^{\mcal{F}_i} - \EE[f|\mcal{F}_{i-1}] ,\]
\[ \llap{\text{resp.}\quad} f^*_i \coloneqq f^{\mcal{F}^*_i} - \EE[f|\mcal{F}^*_{i-1}] ,\]
and denote by~$V_i$ and~$V^*_i$ their respective variances. Then,
for all~$1 \leq i \leq N$,
\[\label{for1302}
V^*_i \geq (1-\epsilon_i^2) V_i - 2\epsilon_i \sqrt{V_i} \Big( \sum_{i'>i} \epsilon_{i'} \sqrt{V_{i'}} \Big) .\]
\end{Lem}

\begin{proof}
For~$0\leq i\leq N$, define
\[ \tilde{f}_i \coloneqq f - \EE[f|\mcal{F}_i] ,\]
\[ \llap{\text{resp.}\quad} \tilde{f}^*_i \coloneqq f - \EE[f|\mcal{F}^*_i] ,\]
and call~$\tilde{V}_i$ and~$\tilde{V}^*_i$ their respective variances.
One has $\tilde{f}_i = \sum_{i'>i} f_{i'}$, resp.\ $\tilde{f}^*_i = \sum_{i'>i} f^*_{i'}$.
Moreover, by the same argument as in the proof of Proposition~\ref{pro5527},
all the~$f_i$ are orthogonal (that is, ${i_0\neq i_1} {\enskip\Rightarrow}\enskip \EE[f_{i_0}f_{i_1}] = 0$),
thus
\[ \tilde{V}_i = \sum_{i'>i} V_{i'} ;\]
similarly,
\[ \tilde{V}^*_i = \sum_{i'>i} V^*_{i'} .\]

In a first step, we observe that for all~$i$,
$\tilde{f}_i - \tilde{f}^*_i = (f-f^{\mcal{F}_i}) - (f-f^{\mcal{F}^*_i})
= f^{\mcal{F}^*_i} - f^{\mcal{F}_i} = {f^{\mcal{F}^*_i} - (f^{\mcal{F}_i})^{\mcal{F}^*_i}} \ab = (f - f^{\mcal{F}_i})^{\mcal{F}^*_i} = (\tilde{f}_i)^{\mcal{F}^*_i}$,
which by associativity of variance yields the following
\begin{Clm}\label{clm1908}
\[\label{for4854} \tilde{V}_i - \tilde{V}^*_i = \Var\big((\tilde{f}_i)^{\mcal{F}^*_i}\big).\]
\end{Clm}

Now, the following claim will be the main tool for proving the lemma:
\begin{Clm}\label{clm1959}
For all~$1 \leq i \leq N$,
\[\label{for1960}
\tilde{V}_{i-1} - \tilde{V}^*_{i-1} \leq \Big( \epsilon_i \sqrt{V_i} + \sqrt{\tilde{V}_i-\tilde{V}^*_i} \Big)^{\!2\,}
\footnotemark .\]
\end{Clm}
\footnotetext{Taking the square root of~$(\tilde{V}_i-\tilde{V}^*_i)$
is allowed, since that quantity is nonnegative by Claim~\ref{clm1908}.}

Admit temporarily Claim~\ref{clm1959}. Since $\tilde{V}_N= \tilde{V}^*_N = 0$,
(\ref{for1960}) applied with $i=N$ gives $\tilde{V}_{N-1} - \tilde{V}^*_{N-1} \leq \epsilon_N^2 V_N$,
which in turn we can use in~(\ref{for1960}) with $i=N-1$, and so on,
to finally prove by finite (decreasing) induction that, for all~$i$,
\[\label{for1291}
\tilde{V}_i - \tilde{V}^*_i \leq \Big( \sum_{i'>i}\epsilon_{i'}\sqrt{V_{i'}} \Big)^2 .\]
Now to get~(\ref{for1302}),
we note that $V_i = \tilde{V}_{i-1} - \tilde{V}_i$,
resp.\ $V^*_i = \tilde{V}^*_{i-1} - \tilde{V}^*_i$,
so, using successively the inequalities~(\ref{for1960}) and~(\ref{for1291}),
\begin{multline}
V_i - V^*_i = \big( \tilde{V}_{i-1} - \tilde{V}^*_{i-1} \big)
- \big( \tilde{V}_i - \tilde{V}^*_i \big) \\
\leq \Big( \epsilon_i \sqrt{V_i} + \sqrt{\tilde{V}_i - \tilde{V}^*_i} \Big)^2
- (\tilde{V}_i - \tilde{V}^*_i)
= \epsilon_i^2 V_i + 2 \epsilon_i\sqrt{V_i} \sqrt{\tilde{V}_i - \tilde{V}^*_i} \\
\leq \epsilon_i^2 V_i + 2 \epsilon_i\sqrt{V_i} \Big( \sum_{i'>i}\epsilon_{i'}\sqrt{V_{i'}} \Big),
\end{multline}
which is equivalent to (\ref{for1302}).
\end{proof}

\begingroup\def\proofname{Proof of Claim~\ref{clm1959}}\begin{proof}
Thanks to Claim~\ref{clm1908}, what we have to prove is:
\[\label{for3998}
\Var\big((\tilde{f}_{i-1})^{\mcal{F}^*_{i-1}}\big) \leq \Big( \epsilon_i \sqrt{V_i} + \sqrt{\tilde{V}_i-\tilde{V}^*_i} \Big)^{\!2\,} .\]
By the definition of conditional expectation
and the equality case in the Cauchy--Schwarz inequality,
(\ref{for3998}) is equivalent to saying that for all~$\ldb(\mcal{F}^*_{i-1})$ function~$g$,
\[\label{for3331} \big| \EE \big[ \tilde{f}_{i-1} g \big] \big| \leq
\Big( \epsilon_i \sqrt{V_i} + \sqrt{\tilde{V}_i-\tilde{V}^*_i} \Big) \ecty(g) .\]
So let~$g$ be a centered $L^2$ $\mcal{F}^*_{i-1}$-measurable real function.
Since $\tilde{f}_{i-1} = f_i + \tilde{f}_i$,
$\EE[\tilde{f}_{i-1}g] = \EE[f_ig] + \EE[\tilde{f}_ig]$,
which two terms we shall bound separately.

For the first term, under~$\Pr[\Bcdot|\ab \mcal{F}_{i-1}]$,
$f_i$ is centered and only depends on~$X_i$, and~$g$ only depends on~$Y$.
Since $\{ X_i : Y \}_{\mcal{F}_{i-1}} \leq {\{ X_i : Y \}_{\mcal{M}}} \leq \epsilon_i$,
it follows that
\[ \big| \EE[f_ig|\mcal{F}_{i-1}] \big| \leq \epsilon_i \, \ecty(f_i|\mcal{F}_{i-1}) \, \ecty(g|\mcal{F}_{i-1}) ,\]
which yields upon integrating:
\begin{multline}\label{cal3329}
|\EE[f_ig]| \leq \epsilon_i \int \ecty(f_i|\mcal{F}_{i-1}) \ecty(g|\mcal{F}_{i-1}) \,\dx\Pr \\
\footrel{\text{CS}}{\leq} \epsilon_i \sqrt{\int \Var(f_i|\mcal{F}_{i-1}) \, \dx\Pr} \sqrt{\int \Var(g|\mcal{F}_{i-1}) \, \dx\Pr} \\
= \epsilon_i \sqrt{V_i} \sqrt{\Var(g) - \Var(g^{\mcal{F}_{i-1}})} \leq \epsilon_i \sqrt{V_i} \ecty(g) .
\end{multline}

For the second term, under~$\Pr[\Bcdot|\ab \mcal{F}_i]$,
$g$ only depends on~$Y$, and $\EE[\tilde{f}_i|Y]
\equiv \tilde{f}_i - \tilde{f}^*_i$ as we noticed just before Claim~\ref{clm1908},
so $\EE[\tilde{f}_ig|\mcal{F}_i] = \EE[{(\tilde{f}_i - \tilde{f}^*_i)}g | \mcal{F}_i]$,
which yields upon integrating:
\[\label{for3330} |\EE[\tilde{f}_ig]|
= \big| \EE \big[ (\tilde{f}_i - \tilde{f}^*_i) g \big] \big|
\footrel{\text{CS}}{\leq} \ecty(\tilde{f}_i - \tilde{f}^*_i) \ecty(g) = \sqrt{\tilde{V}_i - \tilde{V}^*_i} \ecty(g) ,\]
the last equality coming from Claim~\ref{clm1908}.
Then it just remains to combine~(\ref{cal3329}) and~(\ref{for3330})
to get~(\ref{for3331}).\linebreak[1]\strut
\end{proof}\endgroup

\begingroup\def\proofname{Proof of Theorem~\ref{thm6413}}\begin{proof}
First, thanks to a by now classical approximation argument
we may assume that $I = \{1,\ldots,N\}$
and $J = \{1,\ldots,M\}$. Denote $\mcal{F} \coloneqq \sigma(\vec{X}_I)$,
resp.\ $\mcal{G} \coloneqq \sigma(\vec{Y}_J)$; our goal is to prove
that for all~$f\in\ldb(\mcal{F})$, all~$g\in\ldb(\mcal{G})$,
one has $|\EE[fg]| \leq ( \VERT \boldepsilon \VERT \wedge 1 ) \ecty(f) \ecty(g)$.
We will use the same trick as in our alternative proof of Theorem~\ref{thm5750}:
by the definition of conditional expectation and the Cauchy--Schwarz inequality,
proving the inequality above is equivalent to showing that for all~$f\in\ldb(\mcal{F})$,
\[ \Var\big( f^{\mcal{G}} \big) \leq \big( \VERT \boldepsilon \VERT^2 \wedge 1 \big) \Var(f) ,\]
which, by associativity of variance, is in turn equivalent to:
\[\label{f3451}
\Var\big( f - f^{\mcal{G}} \big) \geq \big( 1-\VERT\boldepsilon\VERT^2 \big)_+ \Var(f) .\]

For~$0\leq i\leq N$, resp.~$0\leq j\leq M$, define
$\mcal{F}_i = \sigma(X_1,\ldots,X_i)$, resp. $\mcal{G}_j = \sigma(Y_1,\ldots,Y_j)$.
For all~$0 \leq j \leq M$, define
\[ f^j \coloneqq f - \EE[f|\mcal{G}_j] ,\]
and for all~$1 \leq i \leq N$, define moreover
\[ f^j_i \coloneqq f^{\mcal{G}_j\vee\mcal{F}_i} - \EE[f|\mcal{G}_j\vee\mcal{F}_{i-1}] .\]
Denote $V^j \coloneqq \Var(f^j)$, resp.\ $V^j_i \coloneqq \Var(f^j_i)$.
For fixed~$j$, the~$f^j_i$ are pairwise orthogonal
(again by the argument in the proof of Proposition~\ref{pro5527})
and their sum is equal to~$f^j$, so:
\[\label{for3043} V^j = \sum_{i=1}^N V^j_i .\]
Thus, with this notation our goal~(\ref{f3451}) becomes:
\[\label{for7254}
\sum_{i=1}^N V^M_i \geq \big( 1-\VERT\boldepsilon\VERT^2 \big)_{\!+\,} \sum_{i=1}^N V^0_i .\]

The main tool to prove~(\ref{for7254}) will be Lemma~\ref{lem7265}.
Actually the rough formula~(\ref{for1302}) is quite impratical,
so we introduce a \emph{linearized} version of it:
for each~$1\leq i\leq N$ take some $\alpha_i > 0$
(which for the time being is arbitrary), then by the Cauchy--Schwarz inequality,
(\ref{for1302}) implies that:
\[\label{for3261}
V^*_i \geq (1-\epsilon_i^2) V_i - \frac{\epsilon_i V_i}{\alpha_i} \sum_{i'>i} \epsilon_{i'}\alpha_{i'}
- \epsilon_i\alpha_i \sum_{i'>i} \frac{\epsilon_{i'}V_{i'}}{\alpha_{i'}} .\]

\begin{Rmk}
(\ref{for3261}) is devised so that its right-hand side is exactly the same
as in~(\ref{for1302}) if $V_i \propto \alpha_i^2\enskip\forall i$.
\end{Rmk}

Let us reason conditionally to~$\mcal{G}_{j-1}$ for a few lines.
Under this conditioning, call~$\dot{\mcal{F}}_i \coloneqq \sigma(X_1,\ldots,X_i)$,
resp.\ $\dot{\mcal{F}}^*_i \coloneqq \sigma(X_1,\ldots,X_i,Y_j)$, and $\dot{f} \coloneqq f^{j-1}$.
Then $\dot{f}$ is an $\ldb(\vec{X})$ function,
so we are in situation of applying Lemma~\ref{lem7265} to the functions
\begin{eqnarray} \dot{f}_i &\coloneqq& \dot{f}^{\dot{\mcal{F}}_i} - \EE\big[\dot{f}\big|\dot{\mcal{F}}_{i-1}\big] \\
\text{\llap{and \qquad}{}} \dot{f}^*_i &\coloneqq& \dot{f}^{\dot{\mcal{F}}^*_i} - \EE\big[\dot{f}\big|\dot{\mcal{F}}^*_{i-1}\big] .
\end{eqnarray}
But in fact we already know these functions:
namely, $\dot{f}_i = f^{j-1}_i$ and $\dot{f}^*_i = f^j_i$.
Then, applying the linearized version~(\ref{for3261}) of Lemma~\ref{lem7265}
\[ \Var \big( f^j_i|\mcal{G}_{j-1} \big) \geq \\
\bigg( 1-\epsilon_{ij}^2 - \frac{\epsilon_{ij}}{\alpha_i} \sum_{i'>i} \epsilon_{i'j}\alpha_{i'} \bigg)
\Var \big( f^{j-1}_i|\mcal{G}_{j-1} \big)
- \epsilon_{ij}\alpha_i \sum_{i'>i} \frac{\epsilon_{i'j}\Var\big( f^{j-1}_{i'}|\mcal{G}_{j-1} \big)}{\alpha_{i'}} , \]
whence upon integrating:
\[\label{for4382}
V^j_i \geq (1-\epsilon_{ij}^2) V^{j-1}_i
- \Big(\sum_{i'>i} \epsilon_{i'j}\alpha_{i'}\Big) \frac{\epsilon_{ij} V^{j-1}_i}{\alpha_i}
- \epsilon_{ij}\alpha_i \sum_{i'>i} \frac{\epsilon_{i'j}V^{j-1}_{i'}}{\alpha_{i'}} .\]

By Equation~(\ref{for4382}), we have transformed our initial problem
into a purely abstract operator problem, \emph{posed in an $L^1$ setting}.
To handle it, we need a little notation.
Call~$L^1(I)$ the set of real functions on~$I$,
endowed with the $L^1$ norm
\[ \big\| (v_i)_{i\in I} \big\|_1 \coloneqq \sum_{i\in I} |v_i| .\]
The dual space of~$L^1(I)$ is made of the linear forms
$\ell \colon (v_i)_{i\in I} \mapsto \sum \ell_i v_i$,
equipped with the $L^{\infty}$ norm
\[ \| \ell \|_{\infty} \coloneqq \sup_{i\in I} |\ell_i| .\]
We shall write ``$L^1(I) \ni v \geq 0$'' to mean that all the entries of~$v$ are nonnegative,
and ``$(L^1(I))' \ni \ell \geq 0$'' to mean that
$(v \geq 0) \Rightarrow (\ell v \geq 0)$, which is equivalent to say that
all the~$\ell_i$ are nonnegative.
Now I claim the following lemma, whose proof is postponed:
\begin{Lem}\label{lem5453}
Suppose given some nonnegative numbers $V^j_i$ for
$(i,j) \in \{1,\ldots,N\} \times \{0,\ldots,M\}$,
such that Equation~(\ref{for4382}) is satisfied for all~$i,j$.
Call~$\mcal{L}$ the nonnegative linear form on~$L^1(I)$ defined by
\[ \mcal{L}v = \sum_{\substack{j\in J\\i,i'\in I}} \frac{\alpha_{i'}}{\alpha_i} \epsilon_{ij}\epsilon_{i'j} v_i ,\]
and assume $\|\mcal{L}\|_{\infty} \leq 1$,
then:
\[\label{for6376}
\sum_{i=1}^N V^M_i \geq \sum_{i=1}^N V^0_i - \mcal{L}\big( (V^0_i)_{i\in I} \big) .\]
\end{Lem}

Lemma~\ref{lem5453} has the following immediate
\begin{Cor}\label{cor1307}
Suppose given some nonnegative numbers $V^j_i$ for
$(i,j) \in \{1,\ldots,N\} \times \{0,\ldots,M\}$,
such that Equation~(\ref{for4382}) is satisfied for all~$i,j$,
then:
\[\label{for1338} \sum_{i=1}^N V^M_i \geq
\bigg( 1 - \sup_{i\in I} \sum_{\substack{j\in J\\i'\in I}} \frac{\alpha_{i'}}{\alpha_i} \epsilon_{ij}\epsilon_{i'j}
\bigg)_{\!\!+\,\,} \sum_{i=1}^N V^0_i .\]
\end{Cor}

Now we finish the proof of Theorem~\ref{thm6413}:
thanks to Corollary~\ref{cor1307}
we have proved that~(\ref{for1338}) stands true in our situation
for \emph{any} choice of positive $(\alpha_i)_{i\in I}$.
The last step then consists in optimizing that choice.
Denote ``$\alpha>0$'' to mean that all the~$\alpha_i$ are positive. One has:
\begin{multline}
\inf_{\alpha>0} \sup_{i\in I} \sum_{\substack{j\in J\\i'\in I}} \frac{\alpha_{i'}}{\alpha_i} \epsilon_{ij}\epsilon_{i'j}
= \inf \big\{ \lambda \geq 0 \colon
(\exists \alpha > 0) (\forall i) \big( \sum_{\substack{j\in J\\i'\in I}} \epsilon_{ij}\epsilon_{i'j}\alpha_{i'}
\leq \lambda \alpha_i \big) \big\} \\
= \inf \big\{ \lambda \geq 0 \colon (\exists \alpha > 0)
(\boldepsilon\boldepsilon^* \alpha \leq \lambda\alpha) \big\}.
\end{multline}
But $\boldepsilon\boldepsilon^*$ is a nonnegative operator on~$L^2(I)$
(I mean, when seen as a matrix all its entries are nonnegative),
so by Lemma~\ref{l3542} in appendix:
\[\label{for2029}
\inf \big\{ \lambda\geq 0 \colon (\exists \alpha > 0) (\boldepsilon\boldepsilon^* \alpha \leq \lambda\alpha) \big\}
= \rho(\boldepsilon\boldepsilon^*) = \VERT \boldepsilon \VERT^2 .\]
This ends the proof of Theorem~\ref{thm6413}.
\end{proof}\endgroup

\begingroup\def\proofname{Proof of Lemma~\ref{lem5453}}\begin{proof}
We prove Lemma~\ref{lem5453} by induction on~$M$.
The case $M=0$ is trivial.
Suppose $M\geq1$ and assume the result is true for~${(M-1)}$.
We generalize the notation $\mcal{L}$
by defining, for~$\bullet \in \{\text{\textvisiblespace},1,*\}$,
\[ \mcal{L}^{\bullet}v = \sum_{\substack{j\in J^{\bullet}\\i,i'\in I}}
\frac{\alpha_{i'}}{\alpha_i} \epsilon_{ij}\epsilon_{i'j} v_i ,\]
with $J^1 = \{1\}$, resp.\ $J^* = \{2,\ldots,M\}$,
so that $\mcal{L} = \mcal{L}^1 + \mcal{L}^*$.
Notice that $\|\mcal{L}^*\|_{\infty}\leq 1$ since ${\|\mcal{L}\|_{\infty}\leq1}$.
For all~$i \in I$, define
\[ \check{V}^1_i = (1-\epsilon_{i1}^2) V^0_i - \frac{\epsilon_{i1} V^0_i}{\alpha_i} \sum_{i'>i} \epsilon_{i'1}\alpha_{i'}
- \epsilon_{i1}\alpha_i \sum_{i'>i} \frac{\epsilon_{i'1}V^0_{i'}}{\alpha_{i'}} ,\]
which is the value that $V^1_i$ would take if there were equality in~(\ref{for4382})
for $j=1$.
With that notation, (\ref{for4382}) writes
\[\label{for6134} \big( V^1_i - \check{V}^1_i \big)_{i\in I} \geq 0 ,\]
and by induction hypothesis we have:
\[\label{for6135} \sum_{i=1}^N V^M_i \geq \sum_{i=1}^N V^1_i - \mcal{L}^* \big( (V^1_i)_{i\in I} \big) .\]

Introducing the~$\check{V}^1_i$, we have therefore the following chain of inequalities:
\begin{multline}
\sum_{i=1}^N V^M_i \footrel{(\ref{for6135})}{\geq}
\sum_{i=1}^N V^1_i - \mcal{L}^* \big( (V^1_i)_{i\in I} \big) \\
\footrel{(\ref{for6134})}{=} \sum_{i=1}^N \check{V}^1_i + \big\| \big( V^1_i - \check{V}^1_i \big)_{i\in I} \big\|_1
- \mcal{L}^* \big( (\check{V}^1_i)_{i\in I} \big) - \mcal{L}^* \big( (V^1_i - \check{V}^1_i)_{i\in I} \big) \\
\footrel{\|\mcal{L}^*\|_\infty\leq1}{\geq} \sum_{i=1}^N \check{V}^1_i - \mcal{L}^* \big( (\check{V}^1_i)_{i\in I} \big)
= \sum_{i=1}^N V^0_i - \mcal{L}^1 \big( (V^0_i)_{i\in I} \big)
- \mcal{L}^* \big( (\check{V}^1_i)_{i\in I} \big) \\
\footrel{\substack{\check{V}^1\leq V^0\\\mcal{L}^*\geq0}}{\geq}
\sum_{i=1}^N V^0_i - \mcal{L}^1 \big( (V^0_i)_{i\in I} \big)
- \mcal{L}^* \big( (V^0_i)_{i\in I} \big)
= \sum_{i=1}^N V^0_i - \mcal{L} \big( (V^0_i)_{i\in I} \big),
\end{multline}
so~(\ref{for6376}) is true for~$M$, whence the lemma by induction.
\end{proof}\endgroup

\begin{Rmk}
Our proof of Theorem~\ref{thm6413}
handled the~$X_i$ and the~$Y_j$ in a fully nonsymmetric way,
since we began with putting orders on~$I$ and~$J$,
which orders played a crucial role in the decomposition of~$f$.
Yet the bound~(\ref{for3009}) obtained \emph{is} obviously symmetric
by re-labelling the basic variables%
—and this is not due to having proceeded to any `re-symmetrization' step\dots\ 
To date I have no simple explanation for this `coincidence'.
\end{Rmk}

\subsection{`\texorpdfstring{$\ZZ$}{Z}~against~\texorpdfstring{$\ZZ$}{Z}' tensorization}

The proof of the `$N$~against~$M$' theorem was quite more technical than
that of the `$N$~against~$1$' theorem;
because of that, in order to get tractable computations
we had to use suboptimal inequalities at two places:
\begin{itemize}
\item Claim~\ref{clm1959} is suboptimal: it has indeed the same shortcoming
as Proposition~\ref{pro5527} exhibited compared to Theorem~\ref{thm5750},
namely, it does not `recycle the losses' occurring
when one makes $g$ covariate with both~$f_i$ and~$\tilde{f}_i$
(cf.\ the discussion on page~\pageref{lbl91},
just after the proof of Proposition~\ref{pro5527}).
\item Our linearization technique is suboptimal in general,
even after optimizing the~$\alpha_i$.
In fact, as we said before, Inequality~(\ref{for3261}) is optimal
if and only if one has $V_i \propto \alpha_i$;
thus, for~(\ref{for4382}) to be always optimal,
one has to have $V^j_i \propto \alpha_i$
for \emph{all} $j$, with the \emph{same} values for the~$\alpha_i$.
This would imply that all the sequences $(V_i^j)_{0\leq i<n}$
are proportional, which is not true in general.
\end{itemize}

So, Theorem~\ref{thm6413} is certainly not optimal%
\footnote{Though, as we will see in \S~\ref{parAsymptoticOptimality},
it is `asymptotically optimal'.}%
—this is confirmed by the example of \S~\ref{parIllustrationOfTheProofs}.
Nonetheless, there is one particular case in which
an alternative reasoning yields an optimal bound%
\footnote{The bound's being optimal shall be proved by Theorem~\ref{t6657}.}.
This case is when some symmetries in the decorrelation hypotheses
allow us to transform the original two-parameter problem (indexed by~$I\times J$)
into a one-parameter problem (indexed by~$\ZZ$).
Let us state and prove the corresponding result:
\begin{Thm}[`$\ZZ$~against~$\ZZ$' theorem]\label{thm0119}
Let~$I$ and~$J$ be sets isomorphic to~$\ZZ$,
and let~$(X_i)_{i\in I}$ and~$(Y_j)_{j\in J}$ be random variables
such that, $\mcal{M}$ denoting the $\sigma$-metalgebra they generate,
one has for all~$i,j\in\ZZ$
\[ \{X_i : Y_j\}_{\mcal{M}} \leq \epsilon(j-i) \]
for some function~$\epsilon \colon \ZZ \to [0,1]$.

Then
\[ \big\{ \vec{X}_I : \vec{Y}_J \big\} \leq \bar{\epsilon} ,\]
where $\bar{\epsilon} \in [0,1]$ is characterized by:
\[\label{f6444} \arcsin \bar{\epsilon} =
\Big( \sum_{z\in\ZZ} \arcsin \epsilon(z) \Big) \wedge \frac{\pi}{2} .\]
\end{Thm}

\begin{Rmk}
If we apply Theorem~\ref{thm6413} to the situation above, we find
${\{\vec{X}_I:\vec{Y}_J\}} \ab \leq \big( \sum_{z\in\ZZ}\epsilon(z) \big) \wedge 1$
(cf.\ \S~\ref{parPractical}).
The latter expression is always $\geq \bar{\epsilon}$
because of the concavity of the function~$\sin(\Bcdot\wedge\frac{\pi}{2})$ on~$\RR_+$,
and even $>\bar\epsilon$ if $\bar{\epsilon}\neq0,1$;
so, when it is applicable, Theorem~\ref{thm0119} is strictly stronger than Theorem~\ref{thm6413}.
\end{Rmk}

\begin{proof}
Let~$f$ and~$g$ be resp.\ $\vec{X}_I$- and $\vec{Y}_J$-measurable $\ldb$ functions.
Denote $\mcal{F} \coloneqq \sigma(\vec{X})$, resp.\ $\mcal{G} \coloneqq \sigma(\vec{Y})$,
and for~$i\in\ZZ$, resp.~$j\in\ZZ$,
denote $\mcal{F}_i \coloneqq \bigvee_{i'\leq i} \sigma(X_{i'})$,
resp.\ $\mcal{G}_j \coloneqq \bigvee_{j'\leq j} \sigma(Y_{j'})$.
For~$(i,j)\in\ZZ\times\ZZ$, define
\[\label{f9113}
f_i^j \coloneqq f^{\mcal{G}_j\vee\mcal{F}_i} - \EE[f|\mcal{G}_j\vee\mcal{F}_{i-1}] \]
and\footnote{Beware: the definition of~$g_j^i$
is \emph{not} analogous to the definition of~$f_i^j$\,!}
\[ g_j^i \coloneqq g^{\mcal{G}_j} - \EE[g^{\mcal{G}_j}|\mcal{G}_{j-1}\vee\mcal{F}_i] .\]
Denote $V\coloneqq\Var(f)$, $W\coloneqq\Var(g)$, $V_i^j\coloneqq\Var(f_i^j)$, $W_j^i\coloneqq\Var(g_j^i)$;
also denote
\[ S_{ij} \coloneqq \EE[f_i^{j-1}g_j^{i-1}] .\]

Our auxiliary functions were devised so that
\begin{Clm}\label{clm4070}
Provided the sum in the right-hand side is absolutely convergent,
\[\label{for5163} \EE[fg] = \sum_{i,j} S_{ij} .\]
\end{Clm}

\begingroup\def\proofname{Proof of Claim~\ref{clm4070}}\begin{proof}
First define $\bar{f} \coloneqq f^{\mcal{G}}$,
so that $\bar{f}$ is $\mcal{G}$-measurable and
$\EE[fg] = \EE[\bar{f}g]$.
For~$j \ab {\in \ZZ}$, define
$g_j \coloneqq g^{\mcal{G}_j} - \EE[g|\mcal{G}_{j-1}]$,
resp.\ $\bar{f}_j \coloneqq {\bar f}^{\mcal{G}_j} - \EE[\bar{f}|\mcal{G}_{j-1}]$:
we have $g=\sum_jg_j$ and $\bar{f}=\sum_j\bar{f}_j$,
which are the respective decompositions of~$g$ and~$\bar{f}$
on the same basis of orthogonal subspaces of~$\ldb(\mcal{G})$,
so $\EE[fg] = \sum_{j} \EE[\bar{f}_jg_j]$.
The terms of the right-hand side of that formula are unchanged upon replacing
$\bar{f}_j$ by~$f^{j-1} \coloneqq f - \EE[f|\mcal{G}_{j-1}]$,
since $\EE[(f^{j-1} - \bar{f}_j)g_j]$ is zero%
—the function~$(f^{j-1} - \bar{f}_j)$ is indeed equal to~$(f-\EE[f|\mcal{G}_j])$,
which is centered conditionally to~$\mcal{G}_j$,
while $g_j$ is $\mcal{G}_j$-measurable.
In the end we have:
\[\label{for5102} \EE[fg] = \sum_j \EE[f^{j-1}g_j] .\]

So in a first step we have decomposed $\EE[fg]$
into a sum indexed by~$j$. Now we decompose each term of that sum
into a sum indexed by~$i$.
Let us reason conditionally to~$\mcal{G}_{j-1}$.
Then $f^{j-1}$ is an $\ldb(\mcal{F})$ function
and $g_j$ is in~$\ldb(Y_j)$.
We compute $\EE[f^{j-1}g_j]$ as in the first step of this proof:
first we replace $g_j$ by~$\bar{g}_j \coloneqq (g_j)^{\mcal{F}}$;
then we decompose $f^{j-1} = \sum_i f^{j-1}_i$ and $\bar{g}_j = \sum_i \bar{g}_{ji}$,
with $f^{j-1}_i \coloneqq (f^{j-1})^{\mcal{F}_i} - \EE[f^{j-1}|\mcal{F}_{i-1}]$%
\footnote{Notation is consistent: this $f^{j-1}_i$ is indeed the same
as the~$f^{j-1}_i$ defined by~(\ref{f9113}),
since we are reasoning conditionally to~$\mcal{G}_{j-1}$.},
resp.\ $\bar{g}_{ji} \coloneqq \bar{g}_j^{\mcal{F}_i} - \EE[\bar{g}_j|\mcal{F}_{i-1}]$,
and by orthogonal decomposition
we get $\EE[f^{j-1}g_j] = \sum_i \EE[f^{j-1}_i\bar{g}_{ji}]$;
then we conclude by saying that $\EE[f^{j-1}_i\bar{g}_{ji}]$
is actually equal to~$\EE[f^{j-1}_ig^{i-1}_j]$,
since $(g^{i-1}_j - \bar{g}_{ji})$ is centered conditionally to~$\mcal{F}_i$
while $f^{j-1}_i$ is $\mcal{F}_i$-measurable.
In the end we have obtained
\[\label{for5103} \EE[f^{j-1}g_j] = \sum_i \EE[f^{j-1}_ig^{i-1}_j] ,\]
which combined with~(\ref{for5102}) yields~(\ref{for5163}).
\end{proof}\endgroup

So we have expressed $\EE[fg]$ as a function of the~$S_{ij}$.
It is also possible to `read' the values of~$V$ and~$W$ from the~$V_i^j$,
resp.\ from the~$W_j^i$, \emph{via} the formulas:
\begin{eqnarray}
\label{for0923a} V &=& \overset{\nearrow}{\lim_{j\longto-\infty}} \Big( \sum_i V_i^j \Big) ;\\
\label{for0923b} W &=& \sum_j \Big( \overset{\nearrow}{\lim_{i\longto-\infty}} W_j^i \Big) .
\end{eqnarray}

Now we are looking for relations between the~$V_i^j$, the~$W_j^i$ and the~$S_{ij}$.
The first relation comes from the decorrelation hypothesis:
conditionally to~$\mcal{G}_{j-1} \vee \mcal{F}_{i-1}$,
$f_i^{j-1}$ is in~$\ldb(X_i)$, resp.~$g_j^{i-1}$ is in~$\ldb(Y_j)$,
and ${\{X_i:Y_j\}} \leq \epsilon(j-i)$, so:
\[\label{for9472a} |S_{ij}| \leq \epsilon(j-i) \sqrt{V_i^{j-1}W_j^{i-1}} .\]

The second relation means that a large value of~$|S_{ij}|$ forces $W_j^i$ to diminish.
To state it, we observe that, since $f_i^{j-1}$ is $(\mcal{G}_{j-1}\vee\mcal{F}_i)$-measurable,
$S_{ij} = \EE[f_i^{j-1}(g_j^{i-1})^{\mcal{G}_{j-1}\vee\mcal{F}_i}]$,
so by the Cauchy--Schwarz inequality
$|S_{ij}| \leq \ecty(f_i^{j-1}) \*
\ecty\big( (g_j^{i-1})^{\mcal{G}_{j-1}\vee\mcal{F}_i} \big)$.
Moreover, since $g_j^{i-1} - (g_j^{i-1})^{\mcal{G}_{j-1}\vee\mcal{F}_i} = g_j^i$,
one has by orthogonality $\Var\big( (g_j^{i-1})^{\mcal{G}_{j-1}\vee\mcal{F}_i} \big) =
\Var(g_j^{i-1}) - \Var(g_j^i)$,
so our inequality becomes
\[\label{for9472b} |S_{ij}| \leq \sqrt{V_i^{j-1}}\sqrt{W_j^{i-1}-W_j^i} \]
(where it is understood that $W_j^i \leq W_j^{i-1}$),
or more eloquently
\[\label{for9472b'} W^i_j \leq W^{i-1}_j - (S_{ij})^2/V^{j-1}_i \]
provided $V_i^{j-1} > 0$.

The third and last relation means, on the other hand,
that a large value of~$\big|\sum_{i'>i} S_{i'j}\big|$ forces~$\sum_{i'>i} V_i^j$ to diminish.
To state it, we denote
\[ \tilde{f}^j_i \coloneqq f - f^{\mcal{G}_j\vee\mcal{F}_i}
= \sum_{i'>i} f^j_{i'} ,\]
whose variance is $\Var(\tilde{f}^j_i) = \sum_{i'>i} \Var(f_{i'}^j)$
since the~$f_{i'}^j$ are pairwise orthogonal.
One has
\[ \sum_{i'>i} S_{i'j} = \sum_{i'>i} \EE[f_{i'}^{j-1} g_j^{i'-1}]
= \sum_{i'>i} \EE [ f_{i'}^{j-1} g_j^i ] = \EE[\tilde{f}_i^{j-1} g_j^i]
= \EE[(\tilde{f}_i^{j-1})^{\mcal{G}_j\vee\mcal{F}_i} g_j^i] , \]
so by the Cauchy--Schwarz inequality,
\[ \Big| \sum_{i'>i} S_{i'j} \Big| \leq
\ecty\big((\tilde{f}^{j-1}_i)^{\mcal{G}_j\vee\mcal{F}_i}\big) \ecty(g_j^i) .\]
Since $\tilde{f}^{j-1}_i - (\tilde{f}^{j-1}_i)^{\mcal{G}_j\vee\mcal{F}_i} = \tilde{f}^j_i$,
one has by orthogonality
\[ \Var\big((\tilde{f}^{j-1}_i)^{\mcal{G}_j\vee\mcal{F}_i}\big) =
\Var(\tilde{f}^{j-1}_i) - \Var(\tilde{f}^j_i) ,\]
so our inequality becomes
\[\label{for9472c}
\Big| \sum_{i'>i} S_{i'j} \Big| \leq
\sqrt{\sum_{i'>i} V_{i'}^{j-1} - \sum_{i'>i} V_{i'}^j} \sqrt{W_j^i} ,\]
or more eloquently:
\[\label{for9472c'} \sum_{i'>i} V_{i'}^j \leq
\sum_{i'>i} V_{i'}^{j-1} - \Big( \sum_{i'>i} S_{i'j} \Big)^{\!2} \mathbin{\Big/} W_j^i .\]

So, we have transformed our initial probabilistic problem into the following analytic one:
let~$\mbf{A}$ be an array indexed by~$\ZZ\times\ZZ$,
each entry $(i,j)$ of which contains three numbers~$V_i^j\geq 0$, $W_j^i\geq 0$ and~$S_{ij}$,
satisfying~(\ref{for9472a}), (\ref{for9472b}) and~(\ref{for9472c})%
—we will say such an array is \emph{correct}.
We define $V$ by~(\ref{for0923a}) and~$W$ by~(\ref{for0923b}),
and we set $S = \sum_{i,j} S_{ij}$ (provided it makes sense);
our goal is to get a bound of the form
``$|S| \leq \bar{\epsilon}\sqrt{VW}$'',
with~$\bar{\epsilon}$ only depending on~$\epsilon(\Bcdot)$.

Note that \emph{A priori} some problems of summability can arise from $\mbf{A}$'s being infinite,
for instance to check~(\ref{for9472c'}) or to define $S$.
However, in the situations which are of interest to us,
we can restrict to cases in which $\mbf{A}$ is of nice particular form.
To do this, we first approximate $f$ in~$\ldb(\vec{X})$, resp.~$g$ in~$\ldb(\vec{Y})$,
by a function depending only on a finite number of~$X_i$, resp.\ of~$Y_j$
—say, we assume $f$ is $\vec{X}_{\dot{I}}$-measurable
and $g$ is $\vec{Y}_{\dot{J}}$-measurable for finite $\dot{I} \subset I, \dot{J} \subset J$.
Then, we define a new model $(\tilde{X}_i)_{i\in\ZZ}, (\tilde{Y}_j)_{j\in\ZZ}$
by~$\tilde{X}_i = X_i$ for $i\in\dot{I}$, resp.\ $\tilde{Y}_j = Y_j$ for $j\in\dot{J}$,
and $\tilde{X}_i, \tilde{Y}_j = \partial$ for $i\notin\dot{I}, j\notin\dot{J}$,
$\partial$ being some cemetery point.
This new model still gives a correct array,
for which $S/\sqrt{VW}$ is arbitrarily close to the initial value of
$\EE[fg]/\ecty(f)\ecty(g)$;
and the new array is of the following form, which we will call \emph{compact},
for which all the quantities of interest are well defined:
\begin{itemize}
\item $V^j_i$ is zero as soon as $i\notin\dot{I}$, and it does not depend on~$j$
for $j < \min\dot{J}$, nor for $j \geq \max\dot{J}$;
\item Similarly, $W^i_j$ is zero as soon as $j\notin\dot{J}$, and it does not depend on~$i$
for $i < \min\dot{I}$, nor for $i \geq \max\dot{I}$;
\item $S_{ij}$ is zero as soon as $(i,j)\notin\dot{I}\times\dot{J}$.
(This condition automatically follows from the first two if the array is correct).
\end{itemize}

We define the following operations on arrays:
\begin{Def}\strut\begin{itemize}
\item For~$z\in\ZZ$, we define the \emph{translation operator} $\tau^z$ on arrays
such that, if the entries of~$\mbf{A}$ at~$(i,j)$ are $V_i^j,W_j^i,S_{ij}$,
the entries of~$\tau^z\mbf{A}$ at~$(i,j)$ are $V_{i+z}^{j+z},W_{j+z}^{i+z},S_{(i+z)(j+z)}$.
\item For~$\grave{\mbf{A}}$ and~$\acute{\mbf{A}}$ two arrays
with entries $\grave{V}_i^j, \ab \grave{W}_j^i, \ab \grave{S}_{ij}$,
resp.~$\acute{V}_i^j, \ab \text{etc.}$,
for~$\alpha, \beta$ two real numbers,
we define the linear combination $\alpha\grave{\mbf{A}} + \beta\acute{\mbf{A}}$
as the array with entries $\alpha\grave{V}_i^j+\beta\acute{V}_i^j,
\alpha\grave{W}_j^i+\beta\acute{W}_j^i,\text{etc.}$.
\end{itemize}\end{Def}

\begin{Lem}\label{lem0513}
Correct arrays are stable by translations and by nonnegative linear combinations,
i.e., if $\mbf{A}$ and~$\mbf{B}$ are correct arrays,
then for all~$z\in\ZZ$ and~$\alpha,\beta \geq 0$,
$\tau^z\mbf{A}$ and~$\alpha\mbf{A} + \beta\mbf{B}$
are correct too.
\end{Lem}

\begingroup\def\proofname{Proof of Lemma~\ref{lem0513}}\begin{proof}
Recall that being correct means satisfying~(\ref{for9472a}), (\ref{for9472b}) and~(\ref{for9472c}).
These conditions are trivially stable by multiplication by a nonnegative constant
and by translation%
\footnote{Getting stability of Condition~(\ref{for9472a}) by translation
is actually the only place where the symmetries of the problem are used.}.
It remains to see that they are stable by addition.
The technique being the same for all three inequalities,
we just treat the case of~(\ref{for9472b}).
Stability of this condition by addition is a consequence of the following inequality
(which is in fact a particuliar case of the \emph{Brunn--Minkowski inequality},
see~\cite{B-M_inequality}):
\begin{Lem}\label{lem0514} For all~$a_1, b_1, a_2, b_2 \geq 0$,
\[\label{f1709} \sqrt{(a_1+a_2)(b_1+b_2)} \geq \sqrt{a_1b_1} + \sqrt{a_2b_2} .\]
\end{Lem}
\begingroup\def\proofname{Proof of Lemma~\ref{lem0514}}\begin{proof}
Take squares on both sides of~(\ref{f1709})
and notice that ${(a_1+a_2)}\ab {(b_1+b_2)} \ab - {\big(\sqrt{a_1b_1} + \sqrt{a_2b_2}\big)^2}
\ab = a_1b_2 + a_2b_1 - 2\sqrt{a_1b_1a_2b_2} \ab =
{\big(\sqrt{a_1b_2} - \sqrt{a_2b_1}\big)^2} \ab \geq 0$.
\end{proof}\endgroup

For~$\grave{\mbf{A}}$ and~$\acute{\mbf{A}}$ two correct arrays satisfying~(\ref{for9472b}),
applying~(\ref{f1709}) with ${a_1 = \acute{V}^{j-1}_i}, a_2 = \grave{V}^{j-1}_i,
b_1 = \acute{W}^{i-1}_j - \acute{W}^i_j, b_2 = \grave{W}^{i-1}_j - \grave{W}^i_j$,
we get:
\begin{multline} \big|\grave{S}_{ij}+\acute{S}_{ij}\big| \leq
|\grave{S}_{ij}| + |\acute{S}_{ij}| \leq
\sqrt{\grave{V}_i^{j-1}}\sqrt{\grave{W}_j^{i-1}-\grave{W}_j^i}
+ \sqrt{\acute{V}_i^{j-1}}\sqrt{\acute{W}_j^{i-1}-\acute{W}_j^i} \\
\leq \sqrt{\grave{V}_i^{j-1}+\acute{V}_i^{j-1}}
\sqrt{(\grave{W}_j^{i-1}+\acute{W}_j^{i-1}) - (\grave{W}_j^i+\acute{W}_j^i)} ,
\end{multline}
so~(\ref{for9472b}) is still valid for~$(\grave{\mbf{A}}+\acute{\mbf{A}})$.
\end{proof}\endgroup

Now, thanks to Lemma~\ref{lem0513} we will reduce our problem on~$(\ZZ\times\ZZ)$-arrays
into a problem on~$\ZZ$-arrays.
Suppose $\mbf{A}$ is a correct array with certain values of~$V$, $W$ and~$S$.
Then, for~$k \geq 0$, the array
\[\label{for3010} \mbf{A}_k = \frac{1}{2k+1} \sum_{z=-k}^{k} \tau^z\mbf{A} \]
is correct too, with the same values of~$V$, $W$ and~$S$ as~$\mbf{A}$.
Now when~${k \longto \infty}$, $\mbf{A}_k$ `looks more and more like a Toeplitz array',
that is, an array whose entries at~$(i,j)$ only depend on~$(j-i)$.
To state it rigorously, we need some definitions:
\begin{Def}\label{def3111}\strut
\begin{itemize}
\item Here, a \emph{Toeplitz array} will mean a $\ZZ\times\ZZ$ array whose entries at~$(i,j)$
only depend on~$(j-i)$.
For such an array, for~$z\in\ZZ$ we denote by~$V_{(z)}, W_{(z)}, S_{(z)}$ the quantities
characterized by~$V_i^j = V_{(j-i)}$, etc..
\item Actually we can always assume our Toeplitz array is \emph{Toeplitz compact},
which means that there exists some $z^-\leq z^+$ such that:
\begin{itemize}
\item $V_{(z)}$ does not depend on~$z$ for $z<z^-$, nor for $z\geq z^+$;
\item $W_{(z)}$ does not depend on~$z$ for $z\leq z^-$, nor for $z>z^+$;
\item $S_{(z)}$ is zero as soon as $z<z^-$ or $z>z^+$.
\end{itemize}
\item For a compact Toeplitz array, we define $v, w, s$
as `renormalized versions' of~$V, W, S$:
\begin{eqnarray}
\label{for3162a} v &\coloneqq& V_{(z<z^-)} ;\\
\label{for3162b} w &\coloneqq& W_{(z>z^+)} ;\\
\label{for3162c} s &\coloneqq& \sum_{z\in\ZZ} S_{(z)} .
\end{eqnarray}
\item A Toeplitz array is said to be \emph{correct} if it is correct when seen as an ordinary array.
For a Toeplitz array, Equations~(\ref{for9472a}), (\ref{for9472b'}) and~(\ref{for9472c'})
become respectively%
\footnote{Note that the way (\ref{for3534}) follows from~(\ref{for9472c'}) is rather tricky,
because it appears a difference between two infinite quantities,
which has to be `renormalized' in the convenient way.}:
\begin{eqnarray}
\label{for5175}
|S_{(z)}| &\leq& \epsilon(z) \sqrt{V_{(z-1)}W_{(z+1)}} ; \\
\label{for5246}
W_{(z)} &\leq& W_{(z+1)} - S_{(z)}^2 \mathbin{\big/} V_{(z-1)} ; \\
\label{for3534}
V_{(z-1)} &\leq& v - \Big( \sum_{z'<z} S_{(z')} \Big)^{\!2} \mathbin{\Big/} W_{(z)} .
\end{eqnarray}
\end{itemize}
\end{Def}

With that vocabulary, our informal statement can be made precise:
let~$\mbf{A}$ be a compact correct array with entries $V_i^j, W_j^i, S_{ij}$,
and associated quantities $V, W, S$,
and define the arrays $\mbf{A}_k$ by~(\ref{for3010}).
Then when~$k \longto \infty$ one has $(2k+1) \mbf{A}_k \longto \bar{\mbf{A}}$ (in the sense that
each entry of~$(2k+1) \mbf{A}_k$ converges to the corresponding entry of~$\bar{\mbf{A}}$),
where $\bar{\mbf{A}}$ is the Toeplitz array with entries $\bar{V}_i^j,\bar{W}_j^i,\bar{S}_{ij}$
defined by:
\begin{eqnarray}
\bar{V}_{(z)} &=& \sum_{j-i=z} V_i^j ; \\ 
\bar{W}_{(z)} &=& \sum_{j-i=z} W_j^i ; \\
\bar{S}_{(z)} &=& \sum_{j-i=z} S_{ij}.
\end{eqnarray}
This array $\bar{\mbf{A}}$ is Toeplitz compact
with $z^-=\min\dot{J}-\max\dot{I}$, resp.\ $z^+=\max\dot{J}-\min\dot{I}$,
and the quantities~(\ref{for3162a})--(\ref{for3162c}) for~$\bar{\mbf{A}}$
are:
\begin{eqnarray}
\bar{v} &=& V ; \\
\bar{w} &=& W ; \\
\bar{s} &=& S.
\end{eqnarray}
Moreover $\bar{\mbf{A}}$ is correct, because all the~$(2k+1) \mbf{A}_k$ are,
and being correct is clearly conserved by array convergence.

The consequence of this statement is the following claim,
which achieves the reduction to a `$\ZZ$-indexed' problem
I alluded to a few lines above:
\begin{Clm}
The supremum of~$|S|/\sqrt{VW}$ for correct arrays
is not greater than the supremum of~$|s|/\sqrt{vw}$
for correct Toeplitz arrays.
\end{Clm}

So we have to study (compact) correct Toeplitz arrays.
Consider such an array.
Denote $\theta(z) \coloneqq \arcsin \epsilon(z)$;
then~(\ref{for5175}) can be rewritten:
\[\label{f7312} \exists \hat{\theta}(z) \in [\pm\theta(z)] \qquad
S_{(z)} = \sin \hat{\theta}(z) \cdot \sqrt{V_{(z-1)}W_{(z+1)}} .\]
Now, notice that for fixed values of the~$V_{(z)}$, the~$S_{(z)}$ and~$w$,
if we have values $W_{(z)}$ such that (\ref{for5175})--(\ref{for3534}) are satisfied,
we can modify those $W_{(z)}$ so that (\ref{for5246}) becomes an equality for all~$z$,
an operation which keeps (\ref{for5175}) and~(\ref{for3534}) true
since it can only make the~$W_{(z)}$ increase.
So we can suppose that (\ref{for5246}) actually is an equality,
i.e.\ that for all~$z\in\ZZ$,
\[\label{for5357} W_{(z)} = w \prod_{z'\geq z} \cos^2\hat{\theta}(z') .\]
Then it remains to integrate~(\ref{for3534}).
For~$z\in\ZZ$, denote
\[ \Gamma(z) \coloneqq \sum_{z'<z} \Big( \sin \hat{\theta}(z') \cdot
\prod_{z'<z''<z} \cos \hat{\theta}(z'') \cdot \sqrt{V_{(z'-1)}} \Big) ,\]
so that (\ref{for3534}) becomes:
\[\label{for6008} V_{(z-1)} \leq v - \Gamma(z)^2 .\]
$\Gamma(\Bcdot)$ satisfies the recursion equation
\[ \Gamma(z+1) = \sin\hat{\theta}(z) \sqrt{V_{(z-1)}} + \cos\hat{\theta}(z) \Gamma(z) ,\]
so by~(\ref{for6008}):
\[\label{for6399} |\Gamma(z+1)| \leq \sin |\hat{\theta}(z)| \sqrt{v-\Gamma(z)^2}
+ \cos\hat{\theta}(z) |\Gamma(z)| .\]
From~(\ref{for6399}), we will now prove that for all~$z\in\ZZ$:
\[\label{for6450} |\Gamma(z)| \leq
\sin \Big( \frac{\pi}{2} \wedge \sum_{z'<z} \theta(z') \Big) \sqrt{v} .\]
Indeed, (\ref{for6450}) is equivalent to saying that
there exists some $\eta(z) \in [0, \sum_{z'<z} \theta(z')]$
such that $|\Gamma(z)| = \sin\eta(z) \sqrt{v}$,
which we prove by induction.
First, since our Toeplitz array was supposed compact,
$\forall z<z^- \enskip \hat\theta(z) = 0$, so
the formula is true for $z\leq z^-$ with $\eta(z) = 0$.
Next, if the formula is true for~$z$, then (\ref{for6399}) yields
\[|\Gamma(z)| \leq \big( \sin|\hat{\theta}(z)| \cos\eta(z)
+ \cos|\hat{\theta}(z)| \sin \eta(z) \big) \sqrt{v}
= \sin\big( \eta(z) + |\hat{\theta}(z)| \big) \sqrt{v} ,\]
where $\eta(z) + |\hat{\theta}(z)| \leq \sum_{z'<z} \theta(z') + \theta(z)
= \sum_{z'<z+1} \theta(z')$, so the formula is true for~$(z+1)$,
which ends the induction.

To conclude, we write that $s = \sum_z S_{(z)} = \Gamma({z>z^+})\sqrt{w}$.
But by~(\ref{for6450}),
${|\Gamma(z > z^+)|} \ab \leq \sin\bar{\epsilon} \cdot \sqrt{v}$,
so in the end:
\[\label{for4760} |s| \leq \sin\bar{\epsilon} \cdot \sqrt{vw} ,\]
\emph{quod erat demonstrandum.}
\end{proof}

\begin{Cor}[`$\ZZ^n$ against $\ZZ^n$' theorem]\label{corZnZn}
Let~$n\geq 1$; let~$(X_x)_{x\in\ZZ^n}$ and~$(Y_y)_{y\in\ZZ^n}$ be random variables,
and assume there exists a function~$\epsilon \colon \ZZ^n \to [0,1]$
such that for all~$x,y\in\ZZ^n$,
\[ \{ X_x : Y_y \}_{\mcal{M}} \leq \epsilon(y-x) ,\]
$\mcal{M}$ being the natural $\sigma$-metalgebra of the system.
Then $\{\vec{X} : \vec{Y}\} \leq \bar{\epsilon}$,
where $\bar{\epsilon}$ the number in~$[0,1]$ such that
\[\label{for7131} \arcsin (\bar{\epsilon}) =
\Big( \sum_{v\in\ZZ^n} \arcsin \epsilon(v) \Big) \wedge \frac{\pi}{2} .\]
\end{Cor}

\begin{proof}
To alleviate notation, we define the `arcsin-sum'
as the binary operation $\tilde{+} \colon [0,1]^2 \to [0,1]$
defined by:
\[ a \tilde{+} b = \sin \Big( \big( \arcsin a + \arcsin b \big) \wedge \frac{\pi}{2} \Big) .\]
$\tilde{+}$ is associative, commutative and nondecreasing,
so it can be extended into an $\infty$-ary operator $\tilde{\sum}$;
with this notation, (\ref{for7131}) merely writes
$\bar{\epsilon} = \tilde{\sum}_{v\in\ZZ^n} \epsilon(v)$.

Let~$(\mbf{e}_1,\ldots,\mbf{e}_n)$ be a $\ZZ$-basis of~$\ZZ^n$.
For~$1\leq r\leq n$,
we identify $\ZZ^r$ with~$\ZZ\mbf{e}_1 \oplus \ZZ\mbf{e}_2 \oplus \cdots \oplus \ZZ\mbf{e}_r$;
we also denote $\ZZ_r^\perp \coloneqq \ZZ\mbf{e}_{r+1} \oplus \cdots \oplus \ZZ\mbf{e}_n$.
What we will prove is actually the following
\begin{Clm}\label{clm6369}
For all~$r\in\{1,\ldots,n\}$, all~$x,y\in\ZZ_r^\perp$,
\[\label{for7222} \big\{ \vec{X}_{x+\ZZ^r} : \vec{Y}_{y+\ZZ^r} \big\}_{\mcal{M}}
\leq \tilde{\sum}_{v\in\ZZ^r} \epsilon(y-x+v) .\]
\end{Clm}%
\noindent The statement of the lemma then corresponds to the claim for $r=n$.

We prove Claim~\ref{clm6369} by induction on~$r$.
The case $r=1$ is merely Theorem~\ref{thm0119}%
\footnote{More precisely, it is the subjective version of that theorem,
cf.\ \S~\ref{parSubjectiveResults}.}.
Now let us show how to go from the case $r-1$ to the case $r$ for~$r>1$:

Take $x,y \in \ZZ_r^\perp$.
We notice that
\[ \vec{X}_{x+\ZZ^r} = \big( \vec{X}_{x+i\mbf{e}_r+\ZZ^{r-1}} \big)_{i\in\ZZ} ,\]
which we shorthand into $\vec{X}_{x+\ZZ^r} = (\mbf{X}_i)_{i\in\ZZ}$;
similarly we write, with obvious notation, $\vec{Y}_{y+\ZZ^r} \eqqcolon (\mbf{Y}_j)_{j\in\ZZ}$.
By induction hypothesis one has for all~$i,j\in\ZZ$:
\[\label{for7020} \{ \mbf{X}_i : \mbf{Y}_j \}_{\mcal{M}} \leq
\tilde{\sum}_{v\in\ZZ^{r-1}} \epsilon\big(y-x+(j-i)\mbf{e}_r+v\big) .\]
Since the right-hand side of~(\ref{for7020}) only depends on~$(j-i)$,
we can apply Theorem~\ref{thm0119}
to the~$\mbf{X}_i$ and the~$\mbf{Y}_j$, which yields
\[ \big\{ \vec{X}_{x+\ZZ^r} : \vec{Y}_{y+\ZZ^r} \big\}_{\mcal{M}} \leq
\tilde{\sum}_{z\in\ZZ} \Big( \tilde{\sum}_{v\in\ZZ^{r-1}} \epsilon\big(y-x+z\mbf{e}_r+v \big) \Big)
= \tilde{\sum}_{v\in\ZZ^r} \epsilon(y-x+v) ,\]
i.e.~(\ref{for7222}).
\end{proof}

\section{Generalizations of the tensorization results}

\subsection{Minimal Hypotheses}\label{parMinimalHypotheses}

When reading the proofs of the tensorization theorems,
you may have noticed that taking the decorrelation hypotheses
w.r.t.\ the whole $\sigma$-metalgebra of the system was a needlessly strong assumption.
Actually each decorrelation hypothesis can be stated
relatively to only \emph{one} $\sigma$-algebra,
in the following way:
\begin{itemize}
\item For Theorem~\ref{thm5750},
one needs only assume that for all~$i\in I$,
$X_i$ and~$Y$ are $\epsilon_i$-decorrelated
when seen from $\sigma\big((X_{i'})_{i'<i}\big)$;
\item For Theorems~\ref{thm6413} and~\ref{thm0119},
one needs only assume that
$X_i$ and~$Y_j$ are $\epsilon_{ij}$-decorrelated (or $\epsilon(j-i)$-decorrelated)
when seen from $\sigma\big((X_{i'})_{i'<i},(Y_{j'})_{j'<j}\big)$.
\end{itemize}

In practice it is rare that one can bound above
${\{X_i:Y\}_{\vec{X}_{\{i'<i\}}}}$
or~${\{X_i:Y_j\}_{\mathlarger{(}\vec{X}_{\{i'<i\}},\vec{Y}_{\{j'<j\}}\mathlarger{)}}}$
more sharply than~${\{X:Y_i\}}_{\mcal{M}}$, resp.~${\{X_i:Y_j\}_{\mcal{M}}}$;
yet it is worth remembering that the `genuine' decorrelation hypotheses
are weaker than those we wrote,
especially when one gets interested in optimality issues (cf.\ \S~\ref{parOptimality}).

\begin{Rmk}
In our tensorization proofs we took~$I$ and~$J$ finite; yet those proofs,
and therefore everything in this subsection,
remain valid if we take for~$I$ or~$J$ any (countable) well-ordered set,
in particular if $I$ or~$J$ is~$\NN$.
\end{Rmk}

\subsection{Subjective versions of the theorems}\label{parSubjectiveResults}

In the tensorization theorems I stated,
the decorrelation hypotheses were given with regard to
the natural $\sigma$-metalgebra $\mcal{M}$ of the system,
while the results were given in terms of `objective' (I mean, not subjective) decorrelations.
Yet actually it can be shown that our results are still valid w.r.t.~$\mcal{M}$
—or even w.r.t.\ any sharper $\sigma$-metalgebra $\mcal{N}\supset\mcal{M}$,
provided decorrelation hypotheses are stated w.r.t.~$\mcal{N}$.
As an example, let us state and prove the subjective result corresponding to Theorem~\ref{thm5750}:
\begin{Cor}\label{cor4274}
Let~$X$, $(Y_i)_{i\in I}$ and~$(Z_\theta)_{\theta\in\Theta}$ be random variables,
and call~$\mcal{N}$ the $\sigma$-metalgebra they span.
Suppose we have bounds ${\{X:Y_i\}}_{\mcal{N}} \leq \epsilon_i$ for all~$i\in I$;
then:
\[\label{for4193}
\{ X : \vec{Y}_I \}_{\mcal{N}} \leq \sqrt{1-\prod_{i\in I}(1-\epsilon_i^2)} .\]
\end{Cor}

\begin{proof}
Up to making up copies of~$I$ and~$\Theta$,
we can assume that $\{0\}$, $I$ and~$\Theta$ are disjoint,
which allows us to denote $Z_0 \coloneqq X$ and $Z_i \coloneqq Y_i$ for~$i\in I$,
so that $\mcal{N}$ is the $\sigma$-metalgebra spanned by the~$Z_\theta$
for~$\theta\in\bar{\Theta} \coloneqq \{0\} \uplus I \uplus \Theta$.
Then (\ref{for4193}) means that for all~$\Xi \subset \bar{\Theta}$,
for (almost-)all~$\vec{z}_{\Xi}$, one must have:
\[ \{ X : \vec{Y}_I \} \leq \sqrt{1-\prod_{i\in I}(1-\epsilon_i^2)} 
\quad \text{under the law $\Pr\big[\Bcdot\big|\vec{Z}_{\Xi} = \vec{z}_{\Xi}\big]$.} \]
So, Corollary~\ref{cor4274} will ensue from Theorem~\ref{thm5750}
provided we can prove that, denoting by~$\mcal{M}$ the $\sigma$-metalgebra spanned by~$X$ and the~$Y_i$,
one has for all~$i\in I$:
\[\label{for4395} \{ X : Y_i \}_{\mcal{M}} \leq \epsilon_i
\quad \text{under the law $\Pr\big[\Bcdot\big|\vec{Z}_{\Xi} = \vec{z}_{\Xi}\big]$.} \]

But under a law $P$, saying that ${\{X:Y_i\}}_{\mcal{M}} \leq \epsilon_i$ means that for all~$\Upsilon\subset\{0\}\uplus I$,
for (almost-)all~$\vec{z}\,'_{\!\Upsilon}$,
one has ${\{X:Y_i\}}\leq\epsilon_i$ under the law $P[\Bcdot|\vec{Z}_{\Upsilon} = \vec{z}\,'_{\!\Upsilon}]$.
So, for $P = \Pr[\Bcdot\big|\vec{Z}_{\Xi} = \vec{z}_{\Xi}]$,
(\ref{for4395}) means that, for all~$\vec{z}\,'_{\!\Upsilon}$:
\[\label{for4586} \{ X : Y_i \} \leq \epsilon_i
\quad \text{under the law $\Pr\big[\Bcdot\big|\vec{Z}_{\Xi} = \vec{z}_{\Xi}\enskip\text{and}\enskip \vec{Z}_{\Upsilon} = \vec{z}\,'_{\!\Upsilon}\big]$.} \]
In Formula~(\ref{for4586}) we can assume that $z_\theta$ and~$z'_\theta$ coincide
for all~$\theta \in \Xi \cap \Upsilon$,
since otherwise the event ``$\vec{Z}_{\Xi}=\vec{z}_{\Xi}\enskip\text{and}\enskip \vec{Z}_{\Upsilon} = \vec{z}\,'_{\!\Upsilon}$''
would be empty and there would be nothing to say.
Then ``$\vec{Z}_{\Xi}=\vec{z}_{\Xi}\enskip\text{and}\enskip \vec{Z}_{\Upsilon} = \vec{z}\,'_{\!\Upsilon}$''
is of the form ``$\vec{Z}_{\Xi\cup\Upsilon} = \vec{z}_{\Xi\cup\Upsilon}$'',
where $\Xi\cup\Upsilon \subset \bar{\Theta}$,
so that (\ref{for4586}) follows directly from the hypothesis
$\{X:Y_I\}_{\mcal{N}} \leq \epsilon_i$.
\end{proof}

\section{Optimality}\label{parOptimality}

\subsection{Exact Optimality}

With the minimal hypotheses stated in \S~\ref{parMinimalHypotheses},
Theorems~\ref{thm5750} and~\ref{thm0119} are optimal:
\begin{Thm}\label{t6656}
The bound~(\ref{for5719}) in Theorem~\ref{thm5750} is optimal, in the following sense:
for any integer $N$, for all~$(\epsilon_i)_{1\leq i\leq N}$ in~$[0,1]^N$,
one can find random variables $X_1, \ldots, X_N, Y$ such that
for all~$i \in \{1,\ldots,N\}$,
\[\label{f6808} \{X_i:Y\}_{\vec{X}_{\{i'<i\}}} = \epsilon_i \]
and
\[ \{ \vec{X} : Y \} = \sqrt{1-\prod_{i}(1-\epsilon_i^2)} .\]
\end{Thm}

\begin{Thm}\label{thm6657-0}
The bound~(\ref{f6444}) in Theorem~\ref{thm0119} is optimal, in the following sense:
for any integer $N$, for all~$(\epsilon(z))_{-N\leq z\leq N} \in [0,1]^{\{-N,\ldots,N\}}$,
one can find random variables~$(X_i)_{i\in\ZZ}$ and~$(Y_j)_{j\in\ZZ}$ such that for all~$i,j\in\ZZ$,
\[\label{f0794} \{ X_i : Y_j \}_{\mathlarger{(}\vec{X}_{\{i'<i\}},\vec{Y}_{\{j'<j\}}\mathlarger{)}}
= \left\{ \begin{array}{lcl} \epsilon(j-i) & \quad & \text{if $|j-i| \leq N$;} \\
0 & \quad & \text{if $|j-i| > N$} \end{array} \right.\]
and ${\{ \vec{X} : \vec{Y} \}} = \bar{\epsilon}$, with~$\bar{\epsilon}$ defined by:
\[\label{f6595} \arcsin \bar{\epsilon} = \sum_{z=-N}^N \arcsin \epsilon(z)
\wedge \frac{\pi}{2} .\]
\end{Thm}

Actually, as proving Theorem~\ref{thm6657-0} for \emph{all} the~$(\epsilon(z))_{-N\leq z\leq N}$
involves some heavy technicalities~\cite{perso},
I will only prove the slightly weaker following
\begin{Thm}\label{t6657}
For any integer $N$, the exists a neighbourhood $U$ of~$\vec{0}$ in~$[0,1]^{\{-N,\ldots,N\}}$
such that, for all~$(\epsilon(z))_{-N\leq z\leq N} \in U$, one can find random variables
$(X_i)_{i\in\ZZ}$ and~$(Y_j)_{j\in\ZZ}$ satisfying~(\ref{f0794}) and~(\ref{f6595})%
\footnote{Notice that in the neighbourhood of~$\vec{0}$,
one can drop the ``$\wedge\frac{\pi}{2}$'' in the right-hand side of~(\ref{f6595}).}.
\end{Thm}

\begin{Rmk}
On the other hand, Theorem~\ref{thm6413}
is obviously not optimal since, as we pointed out, its bound is strictly weaker
than that of Theorem~\ref{thm0119}.
\end{Rmk}

The proof of Theorem~\ref{t6656} relies on the following important result:
\begin{Lem}\label{l6791}
Let~$(X_1,\ldots,X_N,Y)$ be an $(N+1)$-dimensional Gaussian vector.
For all~$1\leq i\leq N$, define
\[ e_i \coloneqq \{ X_i : Y \}_{\vec{X}_{\{i'<i\}}} ,\]
then one has exactly:
\[\label{f1922} \{ \vec{X} : Y \} = \sqrt{1-\prod_{i}(1-e_i^2)} .\]
\end{Lem}

\begin{Rmk}\label{rmk>3lines}
Maximal correlation, as I told in \S~\ref{parBasicHilbertianDecorrelations},
is fundamentally a Hilbertian concept. When one deals with Gaussian vectors,
the Hilbert spaces involved actually have finite dimensions,
so that Lemma~\ref{l6791} about decorrelations can also be seen as a result about Euclidian spaces.
In Appendix~\ref{par3lines}, I will present an unexpected corollary of this lemma,
stating a geometric property of the $3$-dimensional Euclidian space.
\end{Rmk}

\begingroup\def\proofname{Proof of Lemma~\ref{l6791}}\begin{proof}
To alleviate notation, we denote $\mcal{F}_{i-1} \coloneqq \sigma\big(\vec{X}_{\{i'<i\}}\big)$.
Since $(\vec{X},Y)$ is Gaussian, the law of~$(X_i,Y)$
under~$\Pr[\Bcdot|\ab x_1,\ldots,x_{i-1}]$ is Gaussian and only depends on
$(x_1,\ldots,x_{i-1})$ through an additive constant;
consequently, we can speak of ``\emph{the} Hilbertian correlation between~$X_i$ and~$Y$ conditionally to~$\mcal{F}_{i-1}$'', which is $e_i$,
and also of ``\emph{the} conditional variance of~$X_i$ w.r.t.~$\mcal{F}_{i-1}$'',
resp.\ ``\emph{the} conditional variance of~$Y$'', resp.\ ``\emph{the} conditional covariance of~$(X_i,Y)$'',
which we denote resp.~$\Var(X_i|\mcal{F}_{i-1})$, $\Var(Y|\mcal{F}_{i-1})$,
$\Cov(X_i,Y|\mcal{F}_{i-1})$. By Theorem~\ref{pro1857}, one has:
\[\label{f2664}
\Cov(X_i,Y|\mcal{F}_{i-1}) = \pm e_i \ecty(X_i|\mcal{F}_{i-1}) \ecty(Y|\mcal{F}_{i-1}) .\]

Now take $g(Y) = Y$ and $f(X) = \sum_{i=1}^{N} \beta_i X_i$,
for some $\beta_i \in \RR$ to be chosen later.
Then $g^{i-1}$ is equal to~$Y - \EE[Y|\mcal{F}_{i-1}]$ and $f_i$ is proportional
to~$X_i - \EE[X_i|\mcal{F}_{i-1}]$, thus, by~(\ref{f2664}) and our model's being Gaussian,
all the inequalities until~(\ref{f9041c}) in the proof of Theorem~\ref{thm5750}
actually are equalities for $\epsilon_i = e_i$.
If moreover $\Cov(f_i,g^{i-1}|\mcal{F}_{i-1}) \geq 0$ for all~$i$,
then we can drop the absolute values in their left-hand sides,
and thus (\ref{f8485}) will also be an equality.
Then, to get an equality in~(\ref{f9280}), it just remains to ensure
that the final Cauchy--Schwarz equality is an equality,
i.e.\ to ensure that one has, for all~$i$:
\[\label{f2464} \Var(f_i) \propto e_i^2 \prod_{i'=1}^{i-1} (1-e_{i'}^2) .\]
If all of that is satisfied, then one will have exactly
$\EE[fg] = \sqrt{1-\prod_{i}(1-e_i^2)} \ab \hspace{.056em} \ecty(f)\ecty(g)$,
so that $\{\vec{X}:Y\} \geq \sqrt{1-\prod_{i}(1-e_i^2)}$.
The converse inequality being obviously true
by (the minimal version of) Theorem~\ref{thm5750},
the result will follow.

So, we have to check that the choice of the~$\beta_i$ can be performed so that
(\ref{f2464}) is satisfied, with $\Cov(f_i,g^{i-1}|\mcal{F}_{i-1})$ of the good sign.
To do this, we will choose successively
relevant values for~$\beta_N,\ab \beta_{N-1},\ab \ldots,\ab \beta_1$.

We observe that, if $\beta_N,\ldots,\beta_{i+1}$ have already been fixed, then
$\beta_i \mapsto \Cov(f_i,g^{i-1}|\mcal{F}_{i-1})$
is an affine function with slope
\[ \pm e_i \ecty(Y|\mcal{F}_{i-1}) \frac{\ecty(X_i|\mcal{F}_{i-1})}{\ecty(X_i)} .\]
Moreover, $\Var(f_i) = \Var(f_i|\mcal{F}_{i-1})$ as $f_i$ is centered w.r.t.~$\mcal{F}_{i-1}$;
so, since $f_i \propto X_i - \EE[X_i|\mcal{F}_{i-1}]$, (\ref{f2664}) implies:
\[ \Var(f_i) = \frac{\Cov(f_i,g^{i-1}|\mcal{F}_{i-1})^2}{e_i^2 \Var(Y|\mcal{F}_{i-1})} .\]
So, provided all the three quantities $e_i$, $\ecty(Y|\mcal{F}_{i-1})$ and~$\ecty(X_i|\mcal{F}_{i-1})$ are nonzero,
there exists a (unique) $\beta_i$ satisfying~(\ref{f2464}).

Now if $\ecty(Y|\mcal{F}_{i-1})$ is zero, this means that $Y$ is $\mcal{F}_{i-1}$-measurable;
then one of the~$e_{i'}$ has to be~$1$ and thus the result is trivial.
Next if $\ecty(X_i|\mcal{F}_{i-1})$ is zero,
this means that $X_i$ is $\mcal{F}_{i-1}$-measurable;
then $e_i = 0$ and $f_i \equiv 0$, so that (\ref{f2464}) is automatically satisfied.
Finally if $e_i = 0$ and $\Var(X_i|\mcal{F}_{i-1})>0$, then there exists a (unique) $\beta_i$
such that $f_i \equiv 0$, for which (\ref{f2464}) is satisfied.
So all those particular cases actually work fine too.
\end{proof}\endgroup

\begingroup\def\proofname{Proof of Theorem~\ref{t6656}}\begin{proof}
For technical reasons, we begin with noticing that the theorem
is immediate if some $e_i$ is equal to~$1$,
so that we can assume that all the~$e_i$ are $<1$.
Thanks to Lemma~\ref{l6791}, it suffices to prove that
for any sequence of~$\epsilon_i \in [0,1)$
it is possible to build a Gaussian vector~$(X,\vec{Y})$
for which $e_i = \epsilon_i \ \forall i$.
To do this, let~$\xi, \zeta_1, \ldots, \zeta_N$ be i.i.d.\ $\mcal{N}(1)$ variables,
and take $Y = \xi$ and $X_i = \sqrt{1-\alpha_i}\zeta_i + \sqrt{\alpha_i}\xi$
for some parameters $\alpha_i \in [0,1)$.
We want to choose the~$\alpha_i$ such that $\vec{e}(\vec{\alpha}) = \vec{\epsilon}$;
this is always possible, by the following method:
\begin{itemize}
\item First we compute $\alpha_1$: By Theorem~\ref{pro1857},
one can write down the equation linking~$\alpha_1$ and~$e_1$.
It is clear without knowing the precise form of that equation
(actually, $e_1 = \sqrt{\alpha_1}$)
that $e_1$ is a continuous increasing function of~$\alpha_1$
with $e_1=0$ for $\alpha_1=0$ and $e_1=1$ for $\alpha_1=1$.
Therefore there is a unique $\alpha_1$ such that $e_1=\epsilon_1$.
\item Then we compute $\alpha_2$: As we already know the value of~$\alpha_1$,
we can treat it as a constant and look for the equation linking~%
$\alpha_2$ and~$e_2$, which we compute by Theorem~\ref{pro1857} again.
That equation, though more complicated than in the previous case
(actually, $e_2 = \sqrt{\alpha_2}\sqrt{1-\alpha_1}/\sqrt{1-\alpha_1\alpha_2}$),
exhibits the same behaviour: $e_2$ is a continuous increasing function of~$\alpha_2$
with $e_2(\alpha_2=0)=0$ and $e_2(\alpha_2=1)=1$.
Therefore there is a unique $\alpha_2$ such that $e_2=\epsilon_2$.
\item We carry on this process until having determined all the~$\alpha_i$.
\end{itemize}
\end{proof}\endgroup

\begingroup\def\proofname{Proof of Theorem~\ref{t6657}}\begin{proof}
Again, the principle of the proof will consist in showing
how the optimal bound can be attained for relevant Gaussian vectors
and linear functions of them.

We consider independent $\mcal{N}(1)$ variables~$(\xi_j)_{j\in\ZZ}$
and~$(\omega_{ij})_{(i,j)\in\ZZ\times\ZZ}$.
For all~$i$ we set:
\[ X_i = \sum_{z=-N}^{N} \omega_{i(i+z)} ,\]
resp. for all~$j$:
\[ Y_j = \xi_j + \sum_{z=-N}^{N} \alpha_z \omega_{(j-z)j} \]
for some real parameters $(\alpha_z)_{-N\leq z\leq N}$ to be fixed later.
This model is obviously invariant by translation of the indexes.
For~$z\in\ZZ$, define
\[ \dot{e}_z \coloneqq \frac{\Cov\big(X_i,Y_{i+z}\big|\mcal{F}_{i-1}\vee\mcal{G}_{i+z-1}\big)}
{\ecty\big(X_i\big|\mcal{F}_{i-1}\vee\mcal{G}_{i+z-1}\big)
\, \ecty\big(Y_{i+z}\big|\mcal{F}_{i-1}\vee\mcal{G}_{i+z-1}\big)} ,\]
where the choice of~$i$ does not matter.
Since our model is Gaussian, by Theorem~\ref{pro1857},
\[ \{ X_i : Y_{i+z} \}_{\mathlarger{(}\vec{X}_{\{i'<i\}},\vec{Y}_{\{j'<i+z\}}\mathlarger{)}}
= |\dot{e}_z| .\]

By the properties of Gaussian vectors,
it is possible to write down explicitly the equations linking
the~$\dot{e}_z$ to the~$\alpha_z$.
Though these equations may be quite horrendous, some of their properties can be easily established:
\begin{Clm}\label{c4330}\strut
\begin{ienumerate}
\item\label{i4332} For $|z|>N$, $\dot{e}_z = 0$ (for any choice of the~$\alpha_z$);
\item The map $(\alpha_{-N},\ldots,\alpha_N) \mapsto (\dot{e}_{-N},\ldots,\dot{e}_N)$
is of class $\mcal{C}^1$ on the neighbourhood of~$(0,\ldots,0)$, with:
\[ \left(\frac{\partial \dot{e}_z}{\partial \alpha_y}\right)(\vec0) =
\frac{\1{y=z}}{\sqrt{2N+1}} .\]
\end{ienumerate}
\end{Clm}
\noindent By the inverse function theorem, one can therefore find neighbourhoods~%
$V$ and~$U$ of~$\vec0$ in~$\RR^{\{-N,\ldots,N\}}$ such that
the map $\vec{\alpha} \mapsto \vec{\dot{e}}$ is a $\mcal{C}^1$-diffeomorphism from $V$ onto $U$.
In particular, for~$\vec\epsilon$ in such an $U$ we can always fix the~$\alpha_z$
of our model such that $\forall z\enskip\dot{e}_z = \1{|z|\leq N}\epsilon(z)$,
so that (\ref{f0794}) is satisfied.

Now we have to choose~$f$ and~$g$.
Morally\footnote{I say ``morally'' because nothing ensures
that the supremum~(\ref{for3878}) would actually be a maximum here.}
we have to take the functions~$f$ and~$g$ having maximal Pearson correlation.
Since the model is Gaussian, these functions will be linear,
and since the model is invariant by translation,
they will likely be invariant by translation too.
So we would like to take, formally, $f(\vec{X}) = \sum_{i\in\ZZ} X_i$
and $g(\vec{Y}) = \sum_{j\in\ZZ} Y_j$.
As such functions are not properly defined,
we will rather consider
$f[k](\vec{X}) = \sum_{i=-k}^k X_i$,
resp.\ $g[k](\vec{Y}) = \sum_{j=-k}^k Y_j$,
and then we will let~$k$ tend to infinity.

For these~$f[k]$ and~$g[k]$, define the~$V[k]^j_i$, the~$W[k]_j^i$ and the~$S[k]_{ij}$
as in the proof of Theorem~\ref{thm0119}, which are gathered into the array $\mbf{A}[k]$.
The following properties of the~$\mbf{A}[k]$
follow easily from the structure of our model:
\begin{Clm}\label{c6626}\strut
\begin{ienumerate}
\item All the~$V[k]^j_i, W[k]_j^i, S[k]_{ij}$ are bounded uniformly in~$i,j,k$.
\item\begin{itemize}\item $V[k]^j_i$ is zero as soon as $i\notin\{-k-2N,\ldots,k\}$;
\item $W[k]_j^i$ is zero as soon as $j\notin\{-k,\ldots,k\}$.\end{itemize}
\item $S[k]_{ij}$ is zero as soon as $|j-i| > N$.
\item\label{itmHOT1}\begin{itemize}\item For $-k\leq i\leq k-2N$,
$V[k]^j_i$ only depends on~$(j-i)$, even when $k$ varies. We denote its value by~$V_{(j-i)}$.
\item For $-k\leq j\leq k$,
$W[k]_j^i$ only depends on~$(j-i)$, even when $k$ varies. We denote its value by~$W_{(j-i)}$.
\item For $-k\leq i\leq k-2N$ and $-k\leq j\leq k$,
$S[k]_{ij}$ only depends on~$(j-i)$, even when $k$ varies. We denote its value by~$S_{(j-i)}$.\end{itemize}
\item\label{itmHOT2}\begin{itemize}\item $V_{(z)}$ has some constant value $v$ for $z < -N$;
\item $W_{(z)}$ has some constant value $w$ for $z > N$.
\end{itemize}
\end{ienumerate}
\end{Clm}
\noindent By Claim~\ref{c6626}, $\mbf{A}[k]$ converges pointwise to
some compact Toeplitz array $\mbf{A}$, whose entries are the~$V_{(z)}, W_{(z)}, S_{(z)}$
introduced at Item~(\ref{itmHOT1}) of the claim,
whose values~$v$ and~$w$ are those introduced at Item~(\ref{itmHOT2}),
and whose value $s$ is $\sum_{z=-N}^{N} S_{(z)}$.
All the arrays $\mbf{A}[k]$ are obviously correct since they correspond to true functions,
so by passing to the limit $\mbf{A}$ is correct too.

Since our model is Gaussian,
all the inequalities~(\ref{for9472a}), (\ref{for9472b}) and~(\ref{for9472c})
are actually equalities for the arrays $\mbf{A}[k]$;
moreover, since the~$\dot\epsilon_z$ are nonnegative,
the~$S[k]_{ij}$ are nonnegative.
By letting $k$ tend to infinity, it follows that
all the inequalities (\ref{for5175})--(\ref{for3534}) are actually equalities for the array $\mbf{A}$,
with the~$S_{(z)}$ nonnegative.
Consequently in~(\ref{f7312}) one has $\hat\theta(z) = \theta(z)$,
and all the further inequalities are actually equalities, so that in the end
(\ref{for4760}) becomes:
\[\label{for4760=} \frac{s}{\sqrt{vw}} = \bar\epsilon .\]

Now, defining~$V[k]$, $W[k]$ and~$S[k]$
by resp.~(\ref{for0923a}), (\ref{for0923b}) and~(\ref{for5163}) for the arrays $\mbf{A}[k]$,
Claim~\ref{c6626} shows that, when~$k \longto \infty$,
$V[k] \sim 2kv$, resp.\ $W[k] \sim 2kw$, resp.\ $S[k] \sim 2ks$,
so~(\ref{for4760=}) implies that $S[k]/\sqrt{V[k]W[k]} \longto \bar\epsilon$.
But recall that $V[k]$, $W[k]$ and~$S[k]$ are the respective variances and covariance
of the functions~$f[k] \in \ldb(\vec{X})$ and~$g[k] \in \ldb(\vec{Y})$,
so by the very definition~(\ref{for3878}) of Hilbertian correlations,
\[ \{\vec{X} : \vec{Y}\} \geq \frac{S[k]}{\sqrt{V[k]W[k]}} .\]
Making $k \longto \infty$, it follows that $\{\vec{X}:\vec{Y}\} \geq \bar\epsilon$;
the converse inequality being obviously true
by (the minimal version of) Theorem~\ref{thm0119}, this proves Theorem~\ref{t6657}.
\end{proof}\endgroup

\begin{Xpl}\label{xplOptZZ}
In this example we will carry out explicit computations for a Gaussian model
close to the model presented in the proof above.
We take independent $\mcal{N}(1)$ variables $\ldots,\zeta_{-1},\zeta_0,\zeta_1,\ldots$,
$\ldots,\xi_{-1/2},\xi_{1/2},\xi_{3/2}\ldots$,
$\ldots,\omega_{-1/4},\omega_{1/4},\omega_{3/4},\ldots$,
and we set
\begin{eqnarray}
X_i &=& \zeta_i + \sqrt{\alpha} (\omega_{i-1/4} + \omega_{i+1/4}) , \\
\llap{resp.\quad} Y_j &=& \xi_j + \sqrt{\alpha} (\omega_{j-1/4} + \omega_{j+1/4})
\end{eqnarray}
for all integer $i$, resp.\ all half-integer $j$,
where $\alpha$ is some arbitrary nonnegative parameter.
We are going to show that for this system (\ref{f6444}) is actually an equality,
in accordance with the proof of Theorem~\ref{t6657}.

For half-integer $z$ denote
\[ e_z \coloneqq \{ X_i : Y_{i+z} \}%
_{\mathlarger{(}\vec{X}_{\{i'<i\}},\vec{Y}_{\{j'<i+z\}}\mathlarger{)}} ,\]
where the choice of~$i$ does not matter by translation invariance.
Clearly $e_{-z} = e_z$ for all~$z$ and $e_z = 0$ for $|z| > 1/2$,
so to know all the~$e_z$ the only nontrivial computation is computing $e_{1/2}$.
Let us perform it.

Since everything is Gaussian, by Theorem~\ref{pro1857}, $e_{1/2}$ is the value,
under the law $\Pr[\Bcdot|\ab \vec{X}_{\{i<0\}},\vec{Y}_{\{j<1/2\}} \equiv0]$, of
\[\label{forVespr} |\EE[X_0Y_{1/2}]| \mathbin{\big/} \ecty(X_0)\ecty(Y_{1/2}) .\]
Under the law $\Pr[\Bcdot|\vec{X}_{\{i<0\}},\vec{Y}_{\{j<1/2\}} \equiv0]$,
it is clear that $\zeta_0,\ab \omega_{1/4},\ab \xi_{1/2},\ab \omega_{3/4},\ldots$
have exactly the same (joint) law as under~$\Pr$,
and that $\omega_{-1/4}$ is still independent of these (joint) variables,
though its variance shall have diminished.
So we need only compute
\[ v \coloneqq \Var\big(\omega_{-1/4}\big|\vec{X}_{\{i<0\}},\vec{Y}_{\{j<1/2\}} \equiv0\big) .\]

Denote $\vec{L}_{\mathrm r} \coloneqq (\ldots,X_{-2},Y_{-3/2},X_{-1},Y_{-1/2})$,
resp.\ $\vec{L}_{\mathrm l} \coloneqq (\ldots,X_{-2},Y_{-3/2},X_{-1})$.
We write that (formally)
\[ \dx\Pr\big[ \vec{L}_{\mathrm r} \equiv0 \ \text{and} \ \omega_{-1/4} = x \big]
\propto e^{-x^2/2v} \dx{x}, \]
and also
$\dx\Pr[ \vec{L}_{\mathrm l} \equiv0 \ \text{and} \ \omega_{-3/4} = y]
\propto e^{-y^2/2v} \dx{y}$
by translation invariance.
But under~$\Pr[\Bcdot| \ab {\vec{L}_{\mathrm l} \equiv0} \enskip \text{and} \enskip {\omega_{-3/4} = y}]$,
the law of~$(\xi_{-1/2},\omega_{1/4})$ is the same as under~$\Pr$, so one has:
\begin{multline}
e^{-x^2/2v} \propto
\dx\Pr\big[ \vec{L}_{\mathrm r} \equiv0 \ \text{and} \ \omega_{-1\!/\!4} = x \big] \\
= \int_y \dx{y} \,
\dx\Pr\big[ \vec{L}_{\mathrm l} \equiv0 \ \text{and} \ \omega_{-3\!/\!4} = y \big]
\dx\Pr\big[ Y_{-1/2} = 0 \ \text{and} \ \omega_{-1\!/\!4} = x \big|
\vec{L}_{\mathrm l} \equiv0 \ \text{and} \ \omega_{-3\!/\!4} = y \big] \\
\propto \int_y \dx\Pr\big[ Y_{-1/2} = 0 \ \text{and} \ \omega_{-1\!/\!4} = x \big|
\vec{L}_{\mathrm l} \equiv0 \ \text{and} \ \omega_{-3\!/\!4} = y \big] e^{-y^2/2v} \,\dx{y} \\
= \int_y
\dx\Pr\big[ \xi_{-1\!/\!2} = -\sqrt{\alpha}(x+y) \ \text{and} \ \omega_{-1\!/\!4} = x \big|
\vec{L}_{\mathrm l} \equiv0 \ \text{and} \ \omega_{-3\!/\!4} = y \big]
e^{-y^2/2v} \,\dx{y} \\
\propto \int_y e^{-\alpha(x+y)^2/2} e^{-x^2/2} e^{-y^2/2v} \,\dx{y}
\propto \exp \Big\{ \Big( 1+\alpha-\frac{\alpha^2}{\alpha+1/v} \Big) \frac{x^2}{2} \Big\},
\end{multline}
so that $v$ must satisfy:
\[ 1+\alpha-\frac{\alpha^2}{\alpha+1/v} = \frac{1}{v} ,\]
whose only nonnegative solution is
\[ v = \frac{\sqrt{1+4\alpha}-1}{2\alpha} .\]

So one has $\ecty(X_0|\vec{L}_{\mathrm l} \equiv0) = \sqrt{1+\alpha v+\alpha}
= \big(\sqrt{1+4\alpha}+1\big)/2$,
$\ecty(Y_{1/2}|\vec{L}_{\mathrm l} \equiv0) = \sqrt{1+2\alpha}$
and $\EE[X_0Y_{1/2}|\vec{L}_{\mathrm l} \equiv0] = \alpha$,
so that in the end (\ref{forVespr}) yields:
\[\label{f3260} e_{1/2} = \frac{\sqrt{1+4\alpha}-1}{2\sqrt{1+2\alpha}} .\]
With this value, Theorem~\ref{thm0119} states that one has necessarily
\[\label{forMncz} \{ \vec{X} : \vec{Y} \} \leq \sin (2\arcsin e_{1/2}) \footnotemark
= 2 e_{1/2} \sqrt{1-e_{1/2}^2} = \frac{2\alpha}{1+2\alpha} .\]
\footnotetext{As here one always has $e_{1/2} \leq 1/\sqrt{2}$,
we can drop the ``$\wedge\frac{\pi}{2}$'' of Formula~(\ref{f6444}).}

We show that (\ref{forMncz}) is actually an equality:
take indeed~$f[k](\vec{X}) \coloneqq \sum_{i=1}^{k} X_k$,
resp.\ $g[k](\vec{Y}) \coloneqq \sum_{j=1/2}^{k-1/2} Y_k$,
then $\Var(f[k]) = \Var(g[k]) = k(1+2\alpha)$
and $\EE[fg] = (2k-1)\alpha$,
so that
\[\label{f4151} \{ \vec{X} : \vec{Y} \} \geq \frac{(2k-1)\alpha}{k(1+2\alpha)} 
\stackrel{k\longto\infty}{\longto} \frac{2\alpha}{1+2\alpha} .\]
\end{Xpl}

\begin{Rmk}\label{rmkPhTr}
One can formally set $\alpha = +\infty$ in the previous example,
which actually means that one takes $X_i = \omega_{i-1/4} + \omega_{i+1/4}$,
resp.\ $Y_j = \omega_{j-1/4} + \omega_{j+1/4}$.
In this case, both Formulas~(\ref{f3260}) and~(\ref{f4151}) `pass to the limit',
yielding $e_{1/2} = 1/\sqrt{2}$ and $\{\vec{X}:\vec{Y}\}=1$.
This shows that it is possible indeed that the~$e_z$ have `mild' values
and that yet $\vec{X}$ and~$\vec{Y}$ are fully correlated.
In other words, the ``$\wedge\frac{\pi}{2}$'' in~(\ref{f6444})
is not an `artifact' of the proof of Theorem~\ref{thm0119}%
\footnote{\label{ftn6129}%
On the other hand, it is possible that the ``$\wedge1$'' in~(\ref{for3009})
was such an artifact, since Theorem~\ref{thm6413} is not optimal.},
but the expression of a real `phase transition' phenomenon%
\footnote{There exist indeed situations going
`beyond the phase transition', i.e.\ for which $\sum_{z\in\ZZ} \arcsin(e_z) \ab > \pi/2$,
though this is not the case for Example~\ref{xplOptZZ}.}.
Such a phase transition did not occur for the simple tensorization formula~(\ref{for5719}),
which shows that double tensorization in intrinsically more complicated
than simple tensorization.
\end{Rmk}

\subsection{Asymptotic optimality}\label{parAsymptoticOptimality}

In the previous subsection we saw that
(the minimal versions of) Theorems~\ref{thm5750}
and~\ref{thm0119} were optimal, while Theorem~\ref{thm6413} was not.
However it turns out that that result is nevertheless `asymptotically optimal',
in the sense that the bound it gives is equivalent to the optimal bound
when the correlations between the variables become weak.
Here is a precise statement:
\begin{Thm}
Let $I = \{1,\ldots,N\}$ and $J=\{1,\ldots,M\}$ be finite sets, and define
the function~$\mathit{Opt} \colon [0,1]^{I\times J} \to [0,1]$ by
\[\label{f5516}
\mathit{Opt} \big( \vec\epsilon_{I\times J} \big)
\coloneqq \sup \Big\{ \big\{ \vec{X}_I : \vec{Y}_J \big\} \ ;\ \big(\forall (i,j) \in I\times J\big) \,
\big(\{ X_i:Y_j \}_{\mathlarger{(}\vec{X}_{\{i'<i\}},\vec{Y}_{\{j'<j\}}\mathlarger{)}}
\leq \epsilon_{ij}\big) \Big\} ;\]
then, when~$\vec\epsilon_{I\times J} \longto \vec{0}$,
one has:
\[\label{f7160}
\mathit{Opt} \big(\vec\epsilon\big) \sim \VERT \boldepsilon \VERT .\]
\end{Thm}

\begin{Rmk}
In the same way, the simple bound~(\ref{for5528}) of Proposition~\ref{pro5527}
is asymptotically equivalent to the optimal bound~(\ref{for5719}) of Theorem~\ref{thm5750}.
\end{Rmk}

\begin{proof}
Take $(M+NM)$ i.i.d.\ $\mcal{N}(1)$ variables
$\xi_1,\ldots,\xi_M,\omega_{11},\ldots,\omega_{NM}$.
For~$(\!(\alpha_{ij})\!)_{i,j} \in \RR^{N\times M}$,
set
\[ \left\{ \begin{array}{rcl} X_i &=& \sum_j \omega_{ij} ; \\
Y_j &=& \xi_j + \sum_i \alpha_{ij} \omega_{ij} .\end{array} \right.\]

Denote
\[ e_{ij} \coloneqq
\{ X_i:Y_j \}_{\mathlarger{(}\vec{X}_{\{i'<i\}},\vec{Y}_{\{j'<j\}}\mathlarger{)}} ,\]
and define $\dot{e}_{ij}$ as the Pearson correlation coefficient
of~$X_i$ and~$Y_j$ under the law
$\Pr[\Bcdot|\ab \vec{X}_{\{i'<i\}}, \ab {\vec{Y}_{\{j'<j\}}\equiv 0}]$.
Then, as in the proof of Theorem~\ref{t6657}, one has $e_{ij} = |\dot{e}_{ij}|$, and the function
$\vec\alpha \mapsto \vec{\dot{e}}$ is $\mcal{C}^1$ around~$\vec{0}$, with
\[\label{f8236} \vec{\dot{e}} = \frac{1}{\sqrt{M}}\,\vec\alpha + O(\|\vec\alpha\|^2)
\qquad \text{when $\vec\alpha \longto \vec0$.} \]
By the inverse function theorem,
$\vec\alpha \mapsto \vec{\dot{e}}$ is therefore a diffeomorphism
from some neighbourhood~$V$ of~$\vec{0}$ onto some neighbourhood $U$ of~$\vec{0}$,
whose inverse function is such that
\[ \vec{\alpha} = \sqrt{M}\,\vec{\dot{e}} + O(\|\vec{\dot{e}}\|^2) \qquad
\text{when $\vec{\dot{e}} \longto \vec0$.} \]

Now let~$\vec\epsilon \in (\RR_+)^{N\times M} \cap U$.
Take $\vec\alpha\in V$ such that $\vec{\dot{e}}(\vec\alpha) = \vec\epsilon$,
so that the condition of~(\ref{f5516}) is satisfied.
For~$\phi\in\RR^N,\ab {\psi\in\RR^M}$ with $\|\phi\|,\|\psi\| = 1$, set
\[\left\{\begin{array}{rcl}
f(\vec{X}) &\coloneqq& \sum_i \phi_i X_i ; \\
g(\vec{Y}) &\coloneqq& \sum_j \psi_j Y_j.
\end{array}\right.\]
One has
\[ \Var(f) = M ,\]
\[ \Var(g) = 1 + O(\|\vec\alpha\|^2) = 1 + O(\|\vec\epsilon\|^2) \]
and
\[ \EE[fg] = \sum_{i,j} \alpha_{ij}\phi_i\psi_j
= \langle \phi, \boldsymbol\epsilon \psi \rangle + O(\|\vec\epsilon\|^2) ,\]
where the constants implicit in the ``$O(\|\vec\epsilon\|^2)$'' are uniform
in~$(\phi,\psi)$.
So one has
\[ \mathit{Opt}(\vec{\epsilon}) \geq \{ \vec{X} : \vec{Y} \}
\geq \frac{|\EE[fg]|}{\ecty(f)\ecty(g)}
= |\langle \phi, \boldsymbol\epsilon \psi \rangle| + O(\|\vec\epsilon\|^2) ,\]
whence after taking supremum over $(\phi,\psi)$:
\[ \mathit{Opt}(\vec{\epsilon}) \geq \VERT \boldsymbol\epsilon \VERT + O(\|\vec\epsilon\|^2)
\stackrel{\vec\epsilon\longto\vec0}{\sim} \VERT \boldsymbol\epsilon \VERT .\]
Since on the other hand $\mathit{Opt}(\vec{\epsilon}) \leq \VERT\boldsymbol\epsilon\VERT$
by Theorem~\ref{thm6413}, the proposition follows.
\end{proof}

\begin{Rmk}
If we state decorrelation hypotheses
w.r.t.\ the \emph{whole} $\sigma$-metalgebra of the system (denoted by~$*$),
no quantity analogous to~$\dot{e}_{ij}$ shall exist any more;
then one can only write, denoting $e'_{ij} \coloneqq \{X_i:Y_j\}_{*}$:
\[ e'_{ij} (\vec\alpha) = \frac{|\alpha_{ij}|}{\sqrt{M}} + O(\|\vec\alpha\|^2) .\]
So, to see how the correlations depend on the parameters,
we have to study the map $\vec\alpha\mapsto\vec{e}\,'$,
which is approximated by a homothety \emph{only on the cone $\RR^{N\times M}_+$}
—and which moreover is no better than continuous here.
So we shall replace the inverse function theorem by an alternative technique,
which will yield the slightly weaker theorem stated just below.
\end{Rmk}
\begin{Thm}
Define
\[ \mathit{Opt}' \big( (\epsilon_{ij})_{(i,j)\in I\times J} \big)
\coloneqq \sup \Big\{ \big\{ \vec{X}_I : \vec{Y}_J \big\} \ ;\ \big(\forall (i,j) \in I\times J\big) \,
\big(\{ X_i:Y_j \}_* \leq \epsilon_{ij}\big) \Big\} ;\]
then for any closed cone $C$ of~$\RR^{N\times M}$
contained in~$(\RR_+^*)^{N\times M} \cup \{0\}$,
on~$C$, one has
\[ \mathit{Opt}'(\vec{\epsilon}) \stackrel{\vec\epsilon\longto\vec0}\sim
\VERT \boldepsilon \VERT .\]
\end{Thm}

\section{Machinery for using the tensorization theorems}%
\label{parMachineryForUsingTheTensorizationTheorems}

Up to now we stated the tensorization theorems in a rather `theoretical' form.
To apply these results to `concrete' situations, some additional techniques may be needed.
This section gives such techniques, which we will use later for the applications
of Chapter~\ref{parApplications}.

\begin{NOTA}
In this section, all the probability systems considered
will be endowed with their natural $\sigma$-metalgebras,
cf.\ Definition~\ref{def3128}.
To alleviate notation, I will give no names to these $\sigma$-metalgebras,
but will plainly denote ${\{X:Y\}}_*$ to mean
``the subjective decorrelation between~$X$ and~$Y$
seen from the natural $\sigma$-metalgebra of the underlying system''.
\end{NOTA}

\subsection{The `doubling-up' technique}\label{parTheDoublingUpTechnique}

\begin{Def}
For~$I$ a set and~$\mcal{R}$ a binary relation on~$I$,
$J_1,J_2\subset I$, we will say that ``$J_2$ is $\mcal{R}$-disjoint to~$J_1$''
if $(i,j)\in J_1\times J_2\ \Rightarrow\ i\notR j$.
\end{Def}

\begin{Lem}[`Doubling-up' lemma]\label{lem0306}
Let~$I$ be a (countable) set and let~$(X_i)_{i\in I}$ be random variables
such that for all~$i,j\in I$, $\{X_i:X_j\}_* \leq \epsilon_{ij}$
for a certain family of~$\epsilon_{ij} \in [0,1]$.

Let~$\mcal{R}$ be a binary relation on~$I$;
for~$i,j\in I$, denote $\epsilon^{\mcal{R}}_{ij} \coloneqq \1{i\notR j}\epsilon_{ij}$.

Define $\mbf{I} = I_1 \uplus I_2$ to be a disjoint union of two copies of~$I$;
denote by~$(i_1)_{i\in I}$, resp.~$(j_2)_{j\in I}$, the elements of~$I_1$, resp.~$I_2$.
Assume that the following holds for a certain $\epsilon\in[0,1]$:
``if $(Y_{i_{\kappa}})_{i_{\kappa}\in\mbf{I}}$ are random variables such that
$\forall i, j \in I\enskip\{Y_{i_1}:Y_{j_2}\}_{*}
\leq \epsilon^{\mcal{R}}_{ij}$, then $\{\vec{Y}_{I_1}:\vec{Y}_{I_2}\} \leq \epsilon$''.

Then for all~$J_1,J_2\subset I$ such that $J_2$ is $\mcal{R}$-disjoint to~$J_1$,
$\{\vec{X}_{J_1}:\vec{X}_{J_2}\}_{*} \leq \epsilon$.
\end{Lem}

\begin{Rmk}
The interest of Lemma~\ref{lem0306} is that, by proving
\emph{one} tensorization result on ${\{\vec{Y}_{I_1}:\vec{Y}_{I_2}\}}$,
one gets tensorization results
on \emph{all} the~$\{\vec{X}_{J_1}:\vec{X}_{J_2}\}$
for $J_2$ $\mcal{R}$-disjoint to~$J_1$.
\end{Rmk}

\begin{Xpl}\label{xpl1201}\strut
\begin{enumerate}
\item If you take for~$\mcal{R}$ the equality relation,
then Lemma~\ref{lem0306} gives a decorrelation result
for all disjoint~$J_1$ and~$J_2$.
\item\label{itm1202} If $I$ is equipped with a distance $\mathit{dist}$
and if you take $(i\mcal{R}j) \, \Leftrightarrow \, {\big(\mathit{dist}(i,j)<d_1\big)}$,
then you get a decorrelation result for all~$J_1$ and~$J_2$
such that $\mathit{dist}(J_1,J_2) {\geq d_1}$.
\end{enumerate}
\end{Xpl}

\begin{proof}
Assume that the hypotheses of the lemma hold
and let~$J_1,J_2\subset I$ with~$J_2$ $\mcal{R}$-disjoint to~$J_1$.
For~$i_{\kappa} \in \mbf{I}$, define
\[ Y_{i_{\kappa}} = \begin{cases} X_i &
\text{if\ \ $(\kappa=1\enskip\text{and}\enskip i\in J_1)$\ \ or\ \ $(\kappa=2\enskip\text{and}\enskip i\in J_2)$;} \\
\partial & \text{otherwise,} \end{cases} \]
for~$\partial$ some cemetery point in the range of none of the~$X_i$.
Since a constant variable is always independent of any variable,
the hypothesis ``$\{X_i:X_j\}_* \leq \epsilon_{ij}$'' for all~$i,j \in I$
implies that ${\{Y_{i_1}:Y_{j_2}\}_{*}} \leq \epsilon^{\mcal{R}}_{ij}$,
so, by the assumption of the lemma, ${\{\vec{Y}_{I_1}:\vec{Y}_{I_2}\}} \leq \epsilon$.
But $\vec{X}_{J_1}$ is $\vec{Y}_{I_1}$-measurable,
resp.\ $\vec{X}_{J_2}$ is $\vec{Y}_{I_2}$-measurable,
hence $\{\vec{X}_{J_1}:\vec{X}_{J_2}\} \leq \epsilon$.

Getting the subjective result w.r.t.~$*$
is just a variant of that reasoning, cf.\ \S~\ref{parSubjectiveResults}.
\end{proof}

\subsection{A practical result on~\texorpdfstring{$\ZZ^n$}{Z\textasciicircum n}}\label{parPractical}

In Chapter~\ref{parApplications}, the situations we will handle
shall always be of the following form:
\begin{Hyp}\label{hypZZn}
For some $n\in\NN^*$, the system is made of random variables $X_i$, $i\in\ZZ^n$,
which satisfy the condition
\[ \forall i,j\in\ZZ^n \qquad \{X_i:X_j\}_* \leq \epsilon(j-i) \]
for some symmetric function~$\epsilon \colon \ZZ^n \to [0,1]$.
\end{Hyp}

For systems satisfying Assumption~\ref{hypZZn},
one has the following practical synthetic result:
\begin{Lem}\label{pro8882}
Consider a norm $|\Bcdot|$ on~$\RR^n$,
the associated distance on the affine $\RR^n$ being denoted by~$\dist$.
Then for a system satisfying Assumption~\ref{hypZZn},
for all~$J_1,J_2\subset I$:
\[\label{for7041} \big\{ \vec{X}_{J_1} : \vec{X}_{J_2} \big\} \leq
\Big(\sum_{\mathclap{\substack{z\in\ZZ^n\\|z|\geq\dist(J_1,J_2)}}} \epsilon(z)\Big)
\wedge 1 .\]
\end{Lem}

\begin{proof}
To alleviate notation, denote $d \coloneqq \dist(J_1,J_2)$.
Applying Lemma~\ref{lem0306},
taking for ``$\mcal{R}$'' the relation ``be at distance $<d$''
(cf.\ Example~\ref{xpl1201}-\ref{itm1202}),
our goal becomes the following:
supposing $(Y_{i_\kappa})_{i_\kappa\in\ZZ^n_1\uplus\ZZ^n_2}$ are random variables
such that $\{Y_{i_1}:Y_{j_2}\}_* \leq \1{|j-i|\geq d}\epsilon(j-i)$,
we want to bound above $\{\vec{Y}_{\ZZ^n_1}:\vec{Y}_{\ZZ^n_2}\}$.

To do this we apply Theorem~\ref{thm6413},
and we get that $\{\vec{Y}_{\ZZ^n_1}:\vec{Y}_{\ZZ^n_2}\}$ is bounded by
$\VERT \boldepsilon \VERT \wedge 1$,
where $\boldepsilon$ is the following operator:
\[ \begin{array}{rrcl} \boldepsilon \colon & L^2(\ZZ) & \circlearrowleft & \\
& \big(g(j)\big)_{j\in\ZZ} & \mapsto &
\big( \sum_{j\in\ZZ} \1{|j-i|\geq d} \epsilon(j-i) g(j) \big)_{i\in\ZZ} .
\end{array} \]

To compute $\VERT \boldepsilon \VERT$, we split $\boldepsilon$
as $\sum_{z\in\ZZ} \1{|z|\geq d} \epsilon(z) M_z$, where
$M_z$ is the operator
\[ \begin{array}{rrcl} M_z \colon & L^2(\ZZ) & \circlearrowleft & \\
& \big(g(j)\big)_{j\in\ZZ} & \mapsto &
\big( g(i+z) \big)_{i\in\ZZ} .\end{array} \]
Obviously $\VERT M_z \VERT = 1$,
thus $\VERT \boldepsilon \VERT \leq \sum_{|z|\geq d} \epsilon(z)$
—actually there is even equality\hbox{—,}
which ends the proof of Lemma~\ref{pro8882}.
\end{proof}

\begin{Rmk}
Instead of Theorem~\ref{thm6413}, here we could have used Theorem~\ref{corZnZn},
which would yield a better result;
yet that would be very specific to~$\ZZ^n$ (cf.\ \S~\ref{parNonFlat}),
and the result would actually be almost equivalent to~(\ref{for7041})
(cf.\ \S~\ref{parAsymptoticOptimality}).
\end{Rmk}

\subsection{Avoiding the artificial phase transition}%
\label{parAvoidingThePhaseTransitionForTheToyModel}

Let us look again at Formula~(\ref{for7041}):
the ``$\wedge 1$'' in it is not really relevant
since a correlation level is \emph{always} bounded by~$1$.
In fact the situation is di\-cho\-to\-mic:
denoting $d \coloneqq \dist(I,J)$, either $\sum_{|z|\geq d} \epsilon(z)$ is $<1$ and then
(\ref{for7041}) is a true decorrelation result,
or it is $\geq1$ and then (\ref{for7041}) tells us actually nothing.
In other words, our result has a `phase transition'
depending on the relative values of~$\sum_{|z|\geq d} \epsilon(z)$ and~$1$,
similar to the phenomenon we discussed in Remark~\ref{rmkPhTr}.

However, as I pointed out in Footnote~\ref{ftn6129} on page~\pageref{ftn6129},
it is not clear whether the phase transition we are dealing with is a real phenomenon:
maybe it is rather an artifact due to Theorem~\ref{thm6413}'s bound's being non-optimal,
which could be avoided by a cleverer reasoning.
We are strengthened in that thought by observing that,
if $\sum_{z\in\ZZ^n} \epsilon(z) < \infty$,
then for~$d$ large enough one has
$\sum_{|z|\geq d} \epsilon(d) < 1$,
so that there is no phase transition for long distances;
why would a transition appear all of a sudden for short distances?

This subsection will show that, indeed, phase transitions can be avoided
in the situations we deal with.

\begin{Lem}\label{lem4067}
For a system satisfying Assumption~\ref{hypZZn}
with $\epsilon(z)<1$ as soon as $z\neq0$ and $\sum_{z\neq0} \epsilon (z) < \infty$,
there exists a constant $k<1$ such that, for all disjoint
$J_1,J_2\subset I$, one has ${\{\vec{X}_{J_1}:\vec{X}_{J_2}\}} \ab \leq k$.
\end{Lem}

\begin{proof}
As before, using Lemma~\ref{lem0306}
we have to bound above $\{\vec{Y}_{\ZZ^n_1}:\vec{Y}_{\ZZ^n_2}\}$
in the relevant doubled-up model.
Our plan to avoid the phase transition will consist in reducing to the `long distance' case.

For some $\ell \in \NN^*$,
we split $\ZZ^n_1$, resp.~$\ZZ^n_2$,
into a partition of~$\ell^n \eqqcolon N$ sublattices $Z_1^{(1)}, \ab \ldots, \ab Z_1^{(N)}$,
resp.~$Z_2^{(1)}, \ab \ldots, \ab Z_2^{(N)}$,
each lattice $Z_\kappa^{(u)}$ being of the form~$\ell \ZZ^n + z_u$ for some $z_u\in\ZZ^n \div \ell\ZZ^n$.
I claim two fundamental properties of these sublattices:

\begin{Clm}\label{clm8133}
For all~$u,v\in\{1,\ldots,N\}$,
\[\label{for8590} \big\{ \vec{Y}_{\ZZ_1^{(u)}} : \vec{Y}_{\ZZ_2^{(v)}} \big\}_*
\leq \sum_{z \equiv z_v-z_u} \!\! \1{z\neq0} \, \epsilon(z) .\]
\end{Clm}

\begin{proof} It is analogous to the proof of Lemma~\ref{pro8882}.
\end{proof}

\begin{Clm}\label{clm9816}
Provided $\ell$ is large enough,
the right-hand side of~(\ref{for8590}) is (strictly) less than $1$
for all the possible values of~$u,v$.
\end{Clm}

\begin{proof}
Denote $\zeta \coloneqq \sup_{z\neq0} \epsilon(z)$;
notice that our assumptions imply that $\zeta < 1$.
Since $\sum_{z\in\ZZ^n} \epsilon(z)$ converges,
there exists some $d_1<\infty$ such that $\sum_{|z|>d_1} \epsilon(z) < 1-\zeta$.
Now, denoting $d_0 \coloneqq \min\{ |z| \colon z\in\ZZ^n\setminus\{0\} \}$,
for $\ell > 2d_1 \div d_0$, for all~$u,v$ there is at most one $z$ congruent to
$z_v-z_u$ [mod.\ $\ell$] such that $|z| \leq d_1$,
whence the following uniform bound for the right-hand side of~(\ref{for8590}):
\[ \sum_{z\equiv z_v-z_u} \!\! \1{z\neq0} \, \epsilon(z)
\leq \sum_{|z| > d_1} \epsilon(z)
+ \underbrace{\sum_{\substack{|z|\leq d_1\\z\equiv z_v-z_u\\z\neq 0}} \epsilon(z)}%
_{\substack{\leq \zeta \text{ because}\\\text{the sum has at most}\\\text{one term, being $\leq \zeta$}}}
\leq \underbrace{\sum_{|z| > d_1} \epsilon(z)}_{< 1-\zeta} + \zeta < 1.\]
\end{proof}

Now, suppose $\ell$ large enough so that Cl\-aim~\ref{clm9816} works.
We apply \emph{simple} tensorization (Theorem~\ref{thm5750}) to the~$\vec{Y}_{\ZZ_2^{(v)}}$:
writing that $\vec{Y}_{\ZZ^n_2} = \big( \vec{Y}_{\ZZ_2^{(1)}}, \ldots, \vec{Y}_{\ZZ_2^{(N)}} \big)$,
we get that, for any~$u\in\{1,\ldots,N\}$,
\[\label{for8741} \big\{ \vec{Y}_{\ZZ_1^{(u)}} : \vec{Y}_{\ZZ^n_2} \big\}_*
\leq \sqrt{ 1 - \prod_{v=1}^{N}
\left( 1 - \big\{ \vec{Y}_{\ZZ_1^{(u)}}:\vec{Y}_{\ZZ_2^{(v)}} \big\}_*^2 \right) } < 1 .\]
Now we write $\vec{Y}_{\ZZ^n_1} = \big( \vec{Y}_{\ZZ_1^{(1)}}, \ldots, \vec{Y}_{\ZZ_1^{(N)}} \big)$
and we apply simple tensorization again\ab —this time to the~$\vec{Y}_{\ZZ_1^{(u)}}$ —to get:
\[\label{for1198}
\big\{ \vec{Y}_{\ZZ^n_1} : \vec{Y}_{\ZZ^n_2} \big\}
\leq \sqrt{ 1 - \prod_{u=1}^{N}
\left( 1 - \big\{ \vec{Y}_{\ZZ_1^{(u)}}:\vec{Y}_{\ZZ^n_2} \big\}_*^2 \right) } < 1 .\]
Bound~(\ref{for1198}) achieves our goal.
\end{proof}

\begin{Rmk}
With that proof, the way $k$ depends on~$\epsilon(\Bcdot)$ is rather complicated;
in particular, you cannot express $k$
as a function of only~$\sum_{z\neq0} \epsilon(z)$ and~$\sup_{z\neq0} \epsilon(z)$.
\end{Rmk}

\begin{Rmk}
In the case $n=1$, at first sight Lemma~\ref{lem4067} seems to contradict Theorem~\ref{t6657},
in which we told that Theorem~\ref{thm0119}, which does have a phase transition, was optimal.
The explanation for this paradox stands in the slight difference between the assumptions
of Lemma~\ref{lem4067} and Theorem~\ref{thm0119}:
while in Lemma~\ref{lem4067} we really imposed that $\{X_i:X_j\}_* \leq \epsilon\big(\mfrak{d}(i,j)\big)$,
with ``$*$'' denoting the \emph{full} natural $\sigma$-metalgebra of the system,
in Theorem~\ref{thm0119}—more precisely, in the version of Theorem~\ref{thm0119}
Theorem~\ref{t6657} proved to be optimal, which was the \emph{minimal} version of this theorem
(cf.\ \S~\ref{parMinimalHypotheses})—%
the conditions on subjective decorrelations were a bit looser.
That difference makes all the trick when one performs the steps of simple tensorization in the proof of Lemma~\ref{lem4067},
because these steps require subjective decorrelations w.r.t.\ the~$\vec{Y}_{I_\kappa^{(u)}}$,
which the sole assumptions of Theorem~\ref{t6657} do not provide.
\end{Rmk}

\subsection{Non-flat geometries}\label{parNonFlat}

It is natural to ask what we one can do
when the basic variables $X_i$ are not indexed by~$\ZZ^n$,
but by the vertices of a more general graph,
for instance a tree or a finitely generated group.
This shall occur indeed if the physical space one works in
exhibits some curvature%
—though Chapter~\ref{parApplications} will not handle such situations.

Actually for general graphs there are results analogous to those of
\S\S~\ref{parPractical} and~\ref{parAvoidingThePhaseTransitionForTheToyModel},
with similar (though more technical) proofs~\cite{perso}.
Here I will only give the statements of these results.

In this subsection the situation will be the following:
\begin{Hyp}\label{hypNF}
The system is made of random variables $(X_i)$ indexed by a (countable) set $I$.
There is a group $G$ acting transitively on~$I$,
and $I$ is endowed with a symmetric map $\mfrak{d} \colon {I\times I}\to\mfrak{D}$,
called the `abstract distance', which is preserved by the action of~$G$.
We assume that one has
\[ \forall i,j\in I \quad \{X_i:X_j\}_{*} \leq \epsilon\big(\mfrak{d}(i,j)\big) \]
for some function~$\epsilon \colon \mfrak{D} \to [0,1]$.
\end{Hyp}

\begin{Def}
For~$d\in\mfrak{D}$, we define $\valency(d) \coloneqq \# \{ i\in I \colon \mfrak{d}(o,i) = d \}$,
where the choice of~$o\in I$ does not matter.
\end{Def}

Then the analogous to Lemma~\ref{pro8882} is the
\begin{Lem}\label{pro8882+}
For~$\mfrak{D}'\subset\mfrak{D}$, for all~$J_1,J_2\subset I$ such that
$(i\in J_1,j\in J_2)\,\Rightarrow \mfrak{d}(i,j) \ab \in \nolinebreak[2] \mfrak{D}'$,
\[ \big\{\vec{X}_{J_1} : \vec{X}_{J_2}\big\} \leq
\Big( \sum_{d\in \mfrak{D}'} \valency(d)\,\epsilon(d) \Big) \wedge 1 .\]
\end{Lem}

The analogous of Lemma~\ref{lem4067} is the
\begin{Lem}\label{lem4067+}
Assume that Assumption~\ref{hypNF} is satisfied;
denoting by~$0$ the (common) value of the~$\mfrak{d}(i,i)$,
also assume that $\valency(0)=1$ and that $\epsilon(d) < 1$ as soon as $d \neq 0$.
Assume that $\sum_{d\in\mfrak{D}} \valency(d)\epsilon(d) < \infty$.

Moreover, assume that the action of~$G$ on~$I$ is \emph{profinite}
(cf.~\cite[Definition~1.1]{profinite}), i.e.\ that
there is a subset $\mcal{N} \subset \NN^*$
such that for each~$N\in\mcal{N}$, there is a subgroup $G_N \leq G$ such that:
\begin{ienumerate}
\item\label{itm4759} The action of~$G_N$ splits $I$ into exactly $N$ orbits
$I^{(1)},\ldots,I^{(N)}$;
\item\label{itm4518} $G_N$ is normal,
so that the partition of~$I$ into the~$I^{(u)}$ is stable by the action of~$G$;
\item\label{itm4750} Any two distinct points of~$I$ are ultimately separated
by the partitions induced by the~$G_N$, i.e.:
\[ \limsup_{\substack{N\in\mcal{N}\\N\longto\infty}} (G_{N}\cdot o) = \{o\} .\]
\end{ienumerate}

Then there exists a constant $k<1$ such that, for all disjoint
$J_1,J_2\subset I$, one has $\{\vec{X}_{J_1}:\vec{X}_{J_2}\} \leq k$.
\end{Lem}

\begin{Xpl}
For $I=\ZZ^n$ on which $G=\ZZ^n$ acts by translation,
equipped with the abstract distance $\mfrak{d}(x,y) = \{\pm(y-x)\}$,
the assumptions of Lemmas~\ref{pro8882+} and~\ref{lem4067+} are checked,
and these lemmas re-give resp.\ Lemmas~\ref{pro8882} and~\ref{lem4067}.
\end{Xpl}

\begin{Xpl}\label{x3339}
For $I$ the modular group $\mathit{PSL}_2(\ZZ)$ acting by left multiplication on itself,
equipped with its natural abstract distance (i.e., $\mfrak{d}(i,j) = \{i^{-1}j,j^{-1}i\}$),
the assumptions of Lemmas~\ref{pro8882+} and~\ref{lem4067+} are also checked%
—to see that the action of~$G$ on~$I$ is profinite,
take for the~$G_{N(\ell)}$ the principal congruence subgroups
$\Gamma(\ell)$ of the modular group~\cite{PSL2Z}.
Notice that $\mathit{PSL}_2(\ZZ)$ is an example of graph having negative curvature%
~\cite{Gromov}.
\end{Xpl}

\section{Appendix: Illustration of the proof of Theorem~\ref{thm6413}}%
\label{parIllustrationOfTheProofs}

\begin{NOTA}
This subsection is devised for the readers who would like to
understand better the proof of Theorem~\ref{thm6413}
by seeing how it works on a concrete example.
It only contains pedagogical material,
and thus can be skipped safely.
\end{NOTA}

\subsection{A Gaussian system of variables}\label{parX1X2Y}

In this illustration we take $N=2, M=1$ —since $M=1$, $Y_1$ will merely be denoted by~$Y$ —,
and we take $(X_1,X_2,Y)$ Gaussian (and centered),
whose law is described through a $3\times3$ matrix
\emph{via} writing that, for some standard Gaussian vector $(\xi_1,\xi_2,\xi_3) \in \RR^3$,
\[\label{for5012}
\begin{pmatrix} X_1\\X_2\\Y \end{pmatrix} =
\begin{pmatrix} \alpha_1 & \alpha_2 & \alpha_3 \\
\beta_1 & \beta_2 & \beta_3 \\ \omega_1 & \omega_2 & \omega_3 \end{pmatrix}
\begin{pmatrix} \xi_1\\\xi_2\\\xi_3 \end{pmatrix}.\]
We denote the matrix appearing in~(\ref{for5012}) by~$\mbf{M}$.
The rows of~$\mbf{M}$ will be denoted by~$\alpha,\beta,\gamma \in \RR^3$,
and $(\xi_1,\xi_2,\xi_3)$ will be denoted by~$\xi\in\RR^3$.
On~$\RR^3$ we will use the Euclidian scalar product ``$\cdot$'' and the associated norm ``$\|\Bcdot\|$''.

The advantage of this model is that,
by of the general properties of Gaussian vectors
(in particular Theorem~\ref{pro1857}),
all the quantities of interest are computable exactly.

First we compute the correlation levels:
by Theorem~\ref{pro1857},
\[ \{X_1:Y\} = \frac{|\alpha \cdot \omega|}{\|\alpha\| \, \|\omega\|} ,\]
similarly $\{X_2:Y\} = |\beta \cdot \omega| \div \|\beta\|\|\omega\|$; and
\[ \{\vec{X}:Y\} = \sqrt{1-\frac{|\omega\cdot(\alpha\vec{\times}\beta)|^2}{\|\omega\|^2\,\|\alpha\vec{\times}\beta\|^2}}, \]
where ``$\vec{\times}$'' denotes the cross product on $\RR^3$.
Concerning the conditional quantities,
denote by~$\beta^1$, resp.~$\omega^1$, the (orthogonal) projections of~$\beta$, resp.~$\omega$, on~$\RR\alpha$,
and~$\beta^*$, resp.~$\omega^*$, the projections of the same vectors on~$(\RR\alpha)^\perp$,
i.e.\ (assuming that $\alpha\neq0$):
\[ \beta^1 \coloneqq (\beta\cdot\alpha \div \|\alpha\|^2)\,\alpha , \quad
\omega^1 \coloneqq (\omega\cdot\alpha \div \|\alpha\|^2)\,\alpha ; \]
\[ \beta^* \coloneqq \beta - \beta^1 , \quad
\omega^* \coloneqq \omega - \omega^1 \]
(see Figure~\ref{fig-abo}).
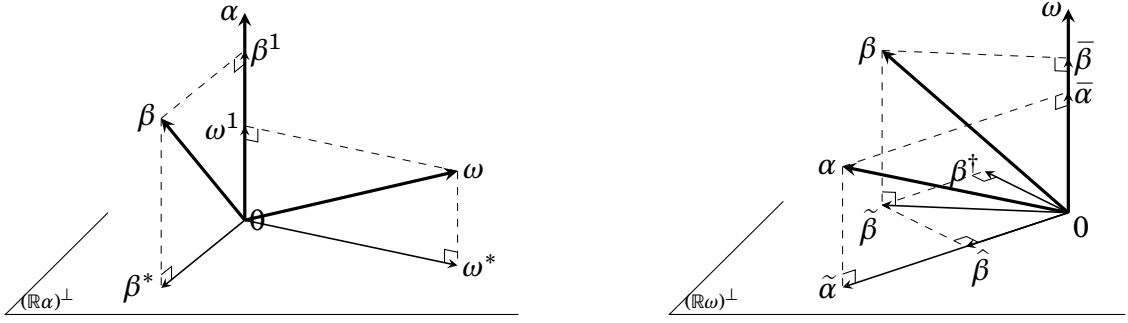
\begin{figure}
\centering
%
%
\noindent\strut\hspace{1em}%
\begin{tikzpicture}[>=stealth,x=.5cm,y=.5cm,inner sep=1.5pt]
\coordinate (O) at (0,0);
\coordinate (A) at (0,5.5);
\coordinate (Bb) at (0,4.5);
\coordinate (Zb) at (0,2.5);
\coordinate (Bt) at (-2.2,-1.8);
\coordinate (Zt) at (5.6,-1.2);
\coordinate (B) at ($(Bb)+(Bt)$);
\coordinate (Z) at ($(Zb)+(Zt)$);
\draw (O) node[anchor=west]{$0$};


\draw[very thick,->] (O)--(A) node[anchor=east]{$\alpha$};
\draw[very thick,->] (O)--(B) node[anchor=east]{$\beta$};
\draw[very thick,->] (O)--(Z) node[anchor=west]{$\omega$};
\draw[semithick,->] (O)--(Bb) node[anchor=west]{$\beta^1$};
\draw[semithick,->] (O)--(Zb) node[anchor=east]{$\omega^1$};
\draw[semithick,->] (O)--(Bt) node[anchor=east]{$\beta^*$};
\draw[semithick,->] (O)--(Zt) node[anchor=west]{$\omega^*$};
\draw[very thin, dashed] (B)--(Bb);
\draw[very thin, dashed] (B)--(Bt);
\draw[very thin, dashed] (Z)--(Zb);
\draw[very thin, dashed] (Z)--(Zt);
\draw[thin] (-3.6,0.2)--(-6.3,-2.5) node[above right]{\scriptsize~~$(\RR\alpha)^\perp$} --(7.2,-2.5);

\rightangle{B}{Bb}{O}
\rightangle{Z}{Zb}{O}
\rightangle{B}{Bt}{O}
\rightangle{Z}{Zt}{O}
\end{tikzpicture}
\hfill
%
%
\begin{tikzpicture}[>=stealth,x=.5cm,y=.5cm,inner sep=1.5pt]
\coordinate (O) at (0,0);
\coordinate (Z) at (0,5.4);
\coordinate (Bd) at (-2.2,1.1);
\coordinate (Bc) at (-2.7,-.9);
\coordinate (At) at ($2.2*(Bc)$);
\coordinate (Ab) at (0,3.2);
\coordinate (Bb) at (0,4.1);
\coordinate (Bt) at ($(Bc)+(Bd)$);
\coordinate (B) at ($(Bt)+(Bb)$);
\coordinate (A) at ($(At)+(Ab)$);
\draw (O) node[anchor=north west]{$0$};


\draw[very thick,->] (O)--(A) node[anchor=east]{$\alpha$};
\draw[very thick,->] (O)--(B) node[anchor=east]{$\beta$};
\draw[very thick,->] (O)--(Z) node[anchor=east]{$\omega$};
\draw[semithick,->] (O)--(Ab) node[anchor=west]{$\bar\alpha$};
\draw[semithick,->] (O)--(At) node[anchor=east]{$\tilde\alpha$};
\draw[semithick,->] (O)--(Bb) node[anchor=west]{$\bar\beta$};
\draw[semithick,->] (O)--(Bt) node[anchor=north east]{$\tilde\beta$};
\draw[semithick,->] (O)--(Bc) node[anchor=north west]{$\hat\beta$};
\draw[semithick,->] (O)--(Bd) node[anchor=east]{$\beta^\dag$};
\draw[very thin, dashed] (A)--(Ab);
\draw[very thin, dashed] (A)--(At);
\draw[very thin, dashed] (B)--(Bb);
\draw[very thin, dashed] (B)--(Bt);
\draw[very thin, dashed] (Bt)--(Bc);
\draw[very thin, dashed] (Bt)--(Bd);
\draw[thin] (-7.5,0.3)--(-10.5,-2.7) node[above right]{\scriptsize~~$(\RR\omega)^\perp$} --(1.5,-2.7);

\rightangle{B}{Bb}{O}
\rightangle{A}{Ab}{O}
\rightangle{B}{Bt}{O}
\rightangle{A}{At}{O}
\rightangle{Bt}{Bc}{O}
\rightangle{Bt}{Bd}{O}
\end{tikzpicture}%
\hspace{1em}\strut
\caption{Visual definitions of the vectors derived from $\alpha$, $\beta$ and~$\omega$:
the left drawing shows how to build $\beta^1,\ab \omega^1,\ab \beta^*,\ab \omega^*$;
the right drawing (with different values for~$\alpha,\ab \beta,\ab \omega$) explains
the construction of~$\protect\bar\alpha,\ab \protect\bar\beta,\ab \protect\tilde\alpha,\ab \protect\tilde\beta,\ab \protect\hat\beta,\ab \beta^\dag$.}%
\label{fig-abo}
\end{figure}%
Then one has
$\EE[X_2|X_1] = \beta^1_1\xi_1 + \beta^1_2\xi_2 + \beta^1_3\xi_3 = \beta^1\cdot\xi$,
resp.\ $\EE[Y|X_1] = \omega^1\cdot\xi$,
thus $X_2-\EE[X_2|X_1] = \beta^*\cdot\xi$,
resp. $Y-\EE[Y|X_1] = \omega^*\cdot\xi$.
As $(X_1,X_2,Y)$ is Gaussian,
the law of~$(X_2-\EE[X_2|X_1],\ab {Y-\EE[Y|X_1]})$
under~$\Pr[\Bcdot|\ab X_1=x]$ does not depend on the value of~$x$;
therefore we know all the conditional laws of~$(X_2,Y)$ under the~$\Pr[\Bcdot|X_1=x]$,
and for all these laws $\{X_2:Y\}$ is equal by Theorem~\ref{pro1857} to
$|\beta^*\cdot\omega^*| \div \|\beta^*\|\|\omega^*\|$, so in the end:
\[ \{ X_2:Y \}_{X_1} = \frac{|\beta^*\cdot\omega^*|}{\|\beta^*\|\,\|\omega^*\|} .\]
$\{ X_1:Y \}_{X_2}$ can be computed by a similar formula.

Now let us `dissect' the proof of Theorem~\ref{thm6413} on our example.
We take $f$ linear, namely
\[\label{forX1+X2} f(X_1,X_2) \coloneqq X_1 + X_2 ,\]
so that all the computations shall again be tractable exactly.

Let us start with computing the quantities linked to~$f^0$:
one has
\begin{eqnarray}
f^0 = f &=& (\alpha+\beta)\cdot\xi ;\\
f_1^0 = f^{\sigma(X_1)} &=& X_1 + (X_2)^{\sigma(X_1)} = (\alpha+\beta^1)\cdot\xi ;\\
f_2^0 = f - f^{\sigma(X_1)} &=& \beta^*\cdot\xi ,
\end{eqnarray}
whence respectively
\begin{eqnarray}
V = V^0 &=& \|\alpha+\beta\|^2 = \|\alpha\|^2 + \|\beta\|^2 + 2\alpha\cdot\beta ; \\
V_1^0 &=& \|\alpha+\beta^1\|^2 = \|\alpha\|^2 + 2\alpha\cdot\beta + \frac{(\alpha\cdot\beta)^2}{\|\alpha\|^2} ; \\
V_2^0 &=& \|\beta^*\|^2 = \|\beta\|^2 - \frac{(\alpha\cdot\beta)^2}{\|\alpha\|^2} .
\end{eqnarray}
By the way we check that, as claimed by Formula~(\ref{for3043}),
$V^0 = V_1^0 + V_2^0$.

Now we turn to the quantities linked to~$f^1$.
First we have to compute the conditional laws of~$(X_1,X_2)$
under the events ``$Y=y$''.
The technique is the same as for computing $\{X_2:Y\}_{X_1}$
a few lines above:
denoting by~$\bar{\alpha}$, resp.\ by~$\bar{\beta}$, the projections of~$\alpha$, resp.~$\beta$, on~$\RR\omega$,
and~$\tilde{\alpha}$, resp.~$\tilde{\beta}$, the projections of the same vectors on~$(\RR\omega^\perp)$, i.e.\ (see Figure~\ref{fig-abo})
\[ \bar{\alpha} \coloneqq (\alpha\cdot\omega \div \|\omega\|^2)\,\omega , \quad
\bar{\beta} \coloneqq (\beta\cdot\omega \div \|\omega\|^2) \,\omega ; \]
\[ \tilde{\alpha} \coloneqq \alpha - \bar{\alpha} , \quad
\tilde{\beta} \coloneqq \beta - \bar{\beta} ,\]
one has $\EE[X_1|Y] = \bar{\alpha}\cdot\xi$,
resp.\ $\EE[X_2|Y] = \bar{\beta}\cdot\xi$, thus
$X_1 - \EE[X_1|Y] = \tilde{\alpha}\cdot\xi$,
resp.\ $X_2 - \EE[X_2|Y] = \tilde{\beta}\cdot\xi$;
and $(X_1 - \EE[X_1|Y], X_2 - \EE[X_2|Y])$
has the same law under all the~$\Pr[\Bcdot|Y=y]$.
So we can compute the quantities linked to~$f^1$ in the same way
as we computed those linked to~$f^0$: denoting
\[ \hat{\beta} \coloneqq \frac{\tilde{\beta} \cdot \tilde{\alpha}}{\|\tilde{\alpha}\|^2} \, \tilde{\alpha} ; \]
\[\beta^\dag \coloneqq \tilde{\beta} - \hat{\beta} \]
(see Figure~\ref{fig-abo}), one finds
\begin{eqnarray}
f^1 &=& (\tilde{\alpha}+\tilde{\beta})\cdot\xi ;\\
f_1^1 &=& (\tilde{\alpha}+\hat{\beta})\cdot\xi ;\\
f_2^1 &=& \beta^\dag \cdot \xi ,
\end{eqnarray}
whence respectively:
\begin{eqnarray}
V^1 &=& \| \tilde{\alpha} + \tilde{\beta} \|^2 ;\\
V_1^1 &=& \|\tilde{\alpha}+\hat{\beta}\|^2 ;\\
V_2^1 &=& \|\beta^\dag\|^2.
\end{eqnarray}
As for~$f^0$, we check that $V^1 = V_1^1 + V_2^1$,
since $\tilde{\alpha} + \tilde{\beta}$ is the orthogonal sum of
$\tilde{\alpha} + \hat{\beta}$ and~$\beta^\dag$.
Moreover one always has $V^1 \leq V^0$, resp.\ $V_2^1 \leq V_2^0$:
the first inequality follows indeed from ${(\tilde{\alpha} + \tilde{\beta})}$'s being
the projection of~${(\alpha + \beta)}$ on~$(\RR\omega)^\perp$,
and the second one from $\beta^\dag$'s being
the projection of~$\beta^*$ on~${(\RR\omega+\RR\alpha)}^\perp$.
These inequalities are consistent with the following corollary of Claim~\ref{clm1908},
obtained by applying the claim conditionally to~$\mcal{G}_{j-1}$
with the role of ``$f$'' played by~$f^{j-1}$ and the role of ``$Y$'' played by~$Y_j$:
\begin{Pro}\label{clm3701}
For all~$1\leq j\leq M$, all~$0\leq i\leq N$,
\[ \sum_{i'>i} V_{i'}^j \leq \sum_{i'>i} V_{i'}^{j-1} .\]
\end{Pro}

\subsection{Numerical computations}\label{par411}

Now let us see a numerical example.
Our parameters will be chosen so that the function~$f$ defined by~(\ref{forX1+X2})
is optimal in the supremum~(\ref{for3878})
defining the Hilbertian correlation coefficient ${\{\vec{X}:Y\}}$;
other than that, the behaviour of our example will be generic:
\[ \mbf{M} = \begin{pmatrix} 4&1&1 \\ 1&4&1 \\ 1&1&4 \end{pmatrix} .\]
For that $\mbf{M}$ the calculations of the previous subsection give:

\begin{center}
\renewcommand\tabcolsep{5pt}
\noindent\begin{tabular}{|c||c|c|c|c|c|c|c|c|c|c|c|c|c|c|c|c|c|c|}
\hline
$\chi$  & $\alpha$ & $\beta$ & $\gamma$ & $\alpha\vec{\times}\beta$ &
$\beta^1$ & $\omega^1$ & $\beta^*$ & $\omega^*$ &
$\bar{\alpha}$ & $\bar{\beta}$ & $\tilde{\alpha}$ & $\tilde{\beta}$ & 
$\hat{\beta}$ & $\beta^\dag$ \\
\hline\hline
$\chi_1$ & $4$ & $1$ & $1$ & $-3$ & $2$  & $2$  & $-1$ & $-1$  & $1/2$ & $1/2$ & $7/2$ & $1/2$ & $7/6$ & $-2/3$ \\
\hline
$\chi_2$ & $1$ & $4$ & $1$ & $-3$ & $1/2$ & $1/2$ & $7/2$ & $1/2$ & $1/2$ & $1/2$ & $1/2$ & $7/2$ & $1/6$ & $10/3$ \\
\hline
$\chi_3$ & $1$ & $1$ & $4$ & $15$ & $1/2$ & $1/2$ & $1/2$ & $7/2$ & $2$  & $2$  & $-1$ & $-1$ & $-1/3$ & $-2/3$ \\
\hline
\end{tabular}
\end{center}

\noindent whence $\{X_1:Y\} = 1/2$ and $\{X_1:Y\}_{X_2} = 1/3$,
thus $\{X_1:Y\}_{\mcal{M}} = 1/2$;
and similarly $\{X_2:Y\}_{\mcal{M}} = 1/2$.

Then Theorem~\ref{thm5750} yields:
\[ \{ \vec{X} : Y \} \leq 1/\sqrt{2} = 0.707\ldots ,\]
and even, according to the refinements of \S~\ref{parMinimalHypotheses}:
\[ \{ \vec{X} : Y \} \leq \sqrt{13}/6 = 0.600\ldots ;\]
on the other hand, the true result is:
\[ \{ \vec{X} : Y \} = 1/\sqrt{3} = 0.577\ldots .\]
So here the bound~(\ref{for3009}) is (fortunately!) correct, and even rather sharp.

Now, as the proof of Theorem~\ref{thm6413} consists in studying the relations
between the~$V^j_i$, let us see what these quantities look like here.
One computes:
\[ \begin{pmatrix} V_1^0 & V_2^0 \\ V_1^1 & V_2^1 \end{pmatrix} =
\begin{pmatrix} 40\frac{1}{2} & 13\frac{1}{2} \\ 24 & 12 \end{pmatrix} .\]
As a first consequence, we can check the conclusions of Proposition~\ref{clm3701}:
$V_2^1 = 12 \leq V_2^0 = 13\frac{1}{2}$,
resp.\ $V_1^1+V_2^1 = 36 \leq V_1^0+V_2^0 = 54$.
Going further, we check the conclusions of Claim~\ref{clm1959},
which forbids the differences~$V_2^0 - V_2^1$ and~$(V_1^0+V_2^0) - (V_1^1+V_2^1)$
to be too large: for the first difference, one has
$V_2^0 - V_2^1 = 1\frac{1}{2} \leq \epsilon_2^2 V_2^0 = 3\frac{3}{8}$%
\footnote{According to \S~\ref{parMinimalHypotheses},
one can replace $\epsilon_2 = 1/2$ by~$\{X_2:Y\}_{X_1} = 1/3$
in this inequality. Then the inequality even becomes an equality:
this is linked to the optimality of certain tensorization results for Gaussian variables,
cf.\ \S~\ref{parOptimality}.}, and for the second one,
$(V_1^0+V_2^0) - (V_1^1+V_2^1) = 18 \leq
\big( \epsilon_1 V_1^0 + \sqrt{V_2^0-V_2^1} \big)^2 = 19.419\dots$.

\subsection{Some traps to avoid}

To finish with this appendix,
I would like to comment on what is true or not about the~$V^j_i$ in general situations.
Proposition~\ref{clm3701} pointed out that for all~$\hat\imath \in \{0,\ldots,N\}$,
$\sum_{i>\hat\imath} V^j_{i}$ is a nonincreasing function of~$j$;
in particular, when one looks at the table of the~$V^j_i$,
the last term ($\hat\imath=N-1$), resp.\ the total ($\hat\imath=0$) of line $j$
can only decrease.
Moreover, if in some line $j$ all the~$V^j_i$ are zero from some position $\hat\imath+1$,
then this property remains true in all the lower lines $j'>j$.
That can be explained very simply, since
saying that all the~$V^j_i$ are zero from position $\hat\imath+1$
means indeed that $f$ is $(\mcal{G}_j\vee\mcal{F}_{\hat\imath})$-measurable,
hence \emph{a fortiori} $(\mcal{G}_{j'}\vee\mcal{F}_{\hat\imath})$-measurable.
The following example, in which $f$ turns out to be $2X_1$, illustrates this phenomenon:
\[ \mbf{M} = \begin{pmatrix} 1&0&0 \\ 1&0&0 \\ 1&0&1 \end{pmatrix}
\qquad \Rightarrow \qquad
\begin{pmatrix} V_1^0 & V_2^0 \\ V_1^1 & V_2^1 \end{pmatrix} =
\begin{pmatrix} 4 & 0 \\ 2 & 0 \end{pmatrix} .\]

However, keep careful: almost anything else
you would like to say about the table of the~$V_i^j$ would be false!
In particular, for $i<N$, $V_i^j$ is not a nonincreasing function of~$j$ in general;
it is not even true that $V_i^j = 0 \ \Rightarrow V_i^{j'>j} = 0$,
as shown by the following example:
\[\label{for9186a} \mbf{M} = \left(\begin{array}{ccc} 1&0&0 \\ \!\!-1&1&0 \\ 1&1&1 \end{array}\right)
\qquad \Rightarrow \qquad
\begin{pmatrix} V_1^0 & V_2^0 \\ V_1^1 & V_2^1 \end{pmatrix} =
\begin{pmatrix} 0 & 1 \\ 1/6 & 1/2 \end{pmatrix} .\]

It is not true either that,
if $V^j$ remains unchanged from one line to another
(that is, the total of the~$V^j_i$ remains unchanged),
then all the~$V^j_i$ are unchanged.
In fact, that $V^{j+1}$ is equal to~$V^j$ means that, conditionally to~$\mcal{G}_j$,
$f^j$ is centered w.r.t.~$Y_{j+1}$, and then $f^{j+1} = f^j$.
However, the way $f^{j+1}$ decomposes into a sum of~$f^{j+1}_i$
may be different to the way $f^j$ decomposed into a sum of~$f^j_i$,
because conditioning w.r.t.~$Y_{j+1}$ may make the law of the~$X_i$ change!
That is what happens in the following example:
\[\label{for9186b} \mbf{M} = \begin{pmatrix} 1&0&1 \\ 0&1&\!\!-1 \\ 0&0&1 \end{pmatrix}
\qquad \Rightarrow \qquad
\begin{pmatrix} V_1^0 & V_2^0 \\ V_1^1 & V_2^1 \end{pmatrix} =
\begin{pmatrix} 1/2 & 3/2 \\ 1 & 1 \end{pmatrix} .\]

\section{Appendix: A corollary of the Perron--Fro\-be\-nius theorem}%
\label{parACorollaryOfThePerronFrobeniusTheorem}

In this appendix I handle a lemma used
in the proof of Theorem~\ref{thm6413}.
We are working on the vector space $\RR^N$ for some $N>0$;
a vector or a matrix is said to be ${>0}$ if all its entries are positive,
resp.~${\geq 0}$ if all its entries are nonnegative.
Then the \emph{Perron--Frobenius theorem}~\cite[Theorem~8.3.1]{HornJohnson} states that
if a square matrix $A$ is ${\geq 0}$, then $A$ has some ${\geq 0}$ eigenvector
for the eigenvalue $\rho(A)$. Our goal here is prove the following corollary:
\begin{Lem}\label{l3542}
Let~$A \geq 0$ be a square matrix,
then:
\[\label{for6491} \inf \big\{ \lambda\geq 0 \colon (\exists u > 0)(Au \leq \lambda u) \big\} = \rho(A) .\]
\end{Lem}

\begin{proof}
We prove separately each way of the equality.
Let us begin with way~``$\leq$''.
Let~$v \geq 0$ be some eigenvector of~$A$ for the eigenvalue $\rho(A)$.
If $v > 0$, then the value $\lambda = \rho(A)$ checks the condition in the infimum and we are done.
Otherwise if $v \ngtr 0$, up to a permutation of indices it has the form
$(0,\ldots,0,v'_{n+1},\ldots,v'_N)$ with $0<n<N$ and all the~$v'_i$ positive.
Reasoning by induction,
assume that we have proved the way~``$\leq$'' of the lemma
for all~$n<N$.
Then the form of the eigenvector~$v$ forces~$A$ to write blockwise
\[ A = \left( \begin{array}{cc} \tilde{A} & 0 \\ * & * \end{array} \right) \]
with $\RR^{n\times n} \ni \tilde{A} \geq 0$.
I claim that $\rho(\tilde{A}) \leq \rho(A)$, since
if $\tilde{v}$ is an eigenvector of~$\tilde{A}$ for the eigenvalue~$\rho(\tilde{A})$,
then for~$t \geq 0$
\[ A^{t} \big(\tilde{v},0,\ldots,0\big) = \big(\rho(\tilde{A})^{t}\,\tilde{v},*,\ldots,*\big) ,\]
so
\[\label{for5711}
\limsup_{t\longto\infty} \rho(\tilde{A})^{-t} \big|A^{t}(\tilde{v},0,\ldots,0)\big| > 0 \footnotemark \]
\footnotetext{Equation~(\ref{for5711}) is meaningless if $\rho(\tilde{A})=0$,
but in this case there is nothing to prove.}%
and consequently $\rho(A) \geq \rho(\tilde{A})$.
Now let~$\epsilon > 0$.
By induction hypothesis
there exists some $\RR^n \ni w \ab > 0$ such
that $\tilde{A}w \leq {(\rho(\tilde{A})+\epsilon) w}$.
Thus for~$\eta>0$, $\RR^N \ni (\eta w,v') > 0$ and
\[
A(\eta w,v') = \big( \eta \tilde{A} w,\rho(A)v'+O(\eta) \big)
\leq \big( \eta (\rho(\tilde{A})+\epsilon) w, \rho(A)v' + O(\eta) \big)
\stackrel{\eta \searrow 0}{\leq} \big(\rho(A)+\epsilon\big) (\eta w, v') .
\]
So $(\rho(A)+\epsilon)$ checks the condition in the right-hand side of the infimum,
which ends the proof of the way~``$\leq$'' of~(\ref{for6491}).

For the way~``$\geq$'', consider any~$\RR^N \ni u > 0$
and let again $v \geq 0$ be some eigenvector of~$A$
for the eigenvalue $\rho(A)$.
Then there exists a (unique) $\beta \geq 0$
such that $u-\beta v \geq 0$ but $u - \beta v \ngtr 0$.
For this $\beta$, one of the entries of~$\beta v$ and~$u$
is the same, say $\beta v_{i_0} = u_{i_0}$.
So if $\lambda < \rho(A)$,
\[ \lambda u_{i_0} < \rho(A) u_{i_0} = \rho(A) \beta v_{i_0} = \big(A(\beta v)\big)_{i_0}
\leq \big(A(\beta v)\big)_{i_0} + \big(A(u-\beta v)\big)_{i_0}
= (Au)_{i_0} ,\]
thus $Au \nleqslant \lambda u$.
That relation being true for any~$u>0$, $\lambda$ does not check the condition in the infimum,
which proves the way~``$\geq$'' of~(\ref{for6491}).
\end{proof}

\section{Appendix: A geometric consequence of results on correlations}\label{par3lines}

As I pointed out in Remark~\ref{rmk>3lines}, for Gaussian vectors
Hilbertian correlations can be interpreted in terms of Euclidian spaces.
In this appendix I will present a funny corollary of Lemma~\ref{l6791}
following from this interpretation.
Actually that result itself is more or less a pretext:
the real goal of this appendix is in fact to show in an eloquent way
the geometric meaning of maximal correlations
and the Hilbertian frame that underlies them.

First we need some vocabulary about Euclidian spaces:
\begin{Def}\strut
\begin{enumerate}
\item For~$L_1,L_2$ two vector lines in the Euclidian space $\RR^2$,
or more generally in any Hilbert space,
we call \emph{geometric angle} between~$L_1$ and~$L_2$,
denoted by~$\widehat{L_1L_2}$, their ``angle'' in the elementary sense:
for arbitrary $\vec{a}\in L_1\setminus\{0\}, \vec{b}\in L_2\setminus\{0\}$,
\[\widehat{L_1L_2}= \arccos \frac{|\vec{a}\cdot\vec{b}|}{\|\vec{a}\|\,\|\vec{b}\|} \in [0,\pi/2].\]
\item For~$L_1,L_2$ and~$L_3\neq L_1,L_2$ three vector lines in the Euclidian space $\RR^3$
(or any Hilbert space),
we call \emph{apparent angle between~$L_1$ and~$L_2$ seen from $L_3$}
the geometric angle that an observer located somewhere on~$L_3\setminus\{0\}$
would have the impression, due to perspective, that $L_1$ and~$L_2$ make
(see Figure~\ref{fig-3lines}): technically,
it is the geometric angle $\widehat{L'_1L'_2}$,
where $L_1'$ and~$L_2'$ are the respective orthogonal projections of~$L_1$ and~$L_2$
onto the plane $(L_3)^\perp$.
\end{enumerate}
\end{Def}

Then one has the following corollary of Lemma~\ref{l6791}:
\begin{Thm}\label{c8818}
Let~$L_1,L_2,L_3$ be three distinct vector lines of~$\RR^3$.
Denote $\widehat{A} \coloneqq \widehat{L_2L_3}, \widehat{B} \coloneqq \widehat{L_3L_1},
\widehat{\Omega} \coloneqq \widehat{L_1L_2}$,
and denote by~$\widehat{A}'$ the apparent angle between~$L_2$ and~$L_3$ seen from $L_1$,
resp.~$\widehat{B}'$ the apparent angle between~$L_3$ and~$L_1$ seen from $L_2$,
etc..
Then the relative order of~$\widehat{A}$ and~$\widehat{A}'$ is the same
as the relative order of~$\widehat{B}$ and~$\widehat{B}'$
and as the relative order of~$\widehat{\Omega}$ and~$\widehat{\Omega}'$, i.e.,
``$\widehat{A}'<\widehat{A}$'' (resp.\ ``$\widehat{A}'=\widehat{A}$'',
resp.\ ``$\widehat{A}'>\widehat{A}$'')
is equivalent to ``$\widehat{B}'<\widehat{B}$'' (resp.\ ``$\widehat{B}'=\widehat{B}$'',
resp.\ ``$\widehat{B}'>\widehat{B}$''),
etc..
\end{Thm}

\begin{Rmk}
I found Theorem~\ref{c8818} by chance, one day that I was looking for a situation
where one would have $\widehat{B}' > \widehat{B}$ but $\widehat{A}' < \widehat{A}$,
in order to build a `nice' example for \S~\ref{par411}.
I thought that such a situation would be generic,
but after having looked for it without success,
I realized that it was actually impossible, and that the explanation had a simple
interpretation in terms of correlations.
\end{Rmk}

\begin{figure}
\centering
\includegraphics[width=\textwidth]{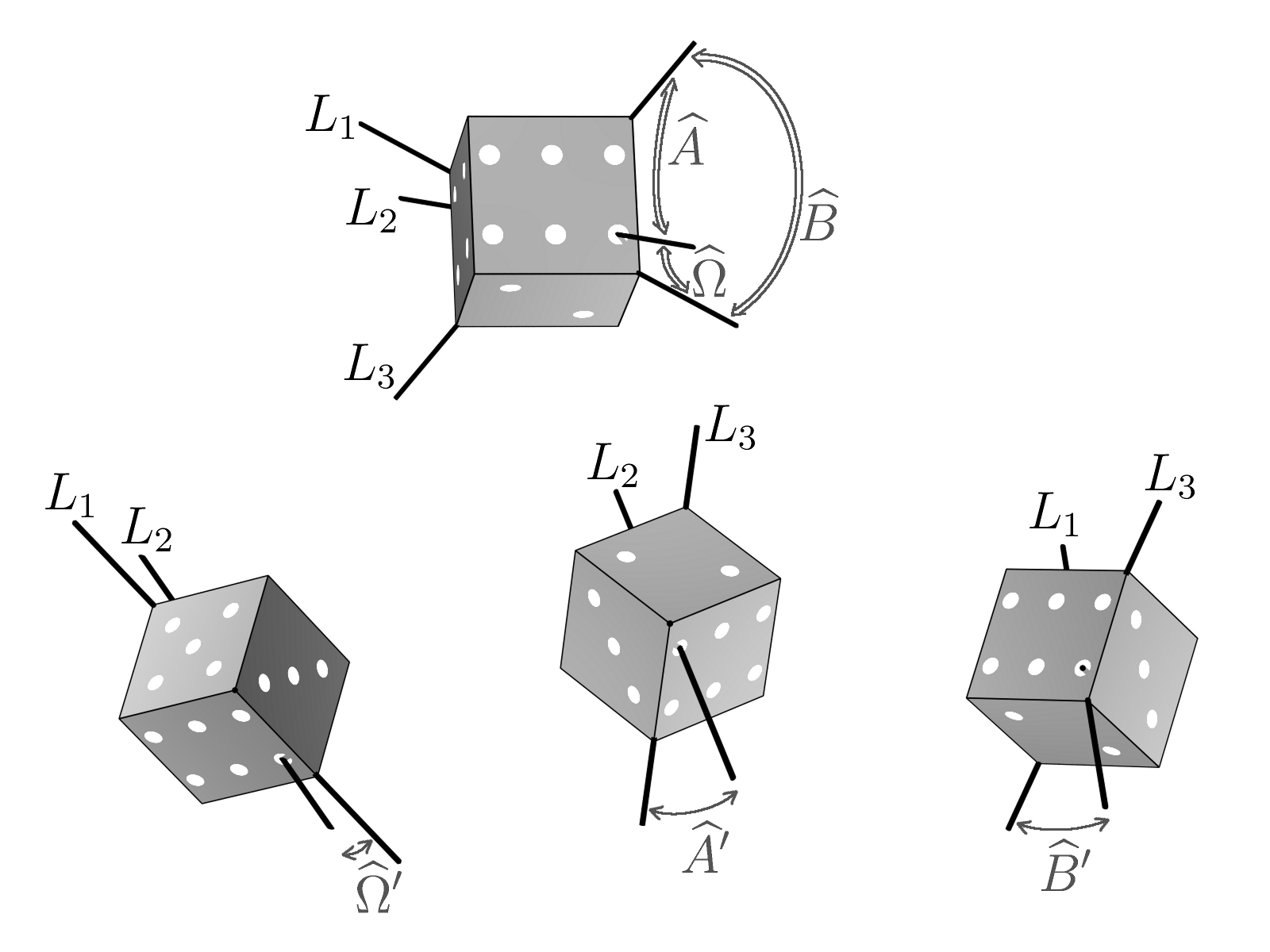}
\caption{This figure shows four different views of the same $3$-dimensional object.
What interests us actually is only the three concurrent lines $L_1,L_2,L_3$,
but we added a die
centered at their point of concurrency to see depth better on the pictures
{\footnotesize[Recall that on a die, the total number of points on two opposite faces is always $7$.]}.
On the top picture, the die is shown in generic position. We represent the angles~%
$\widehat{A}$, $\widehat{B}$ and~$\widehat\Omega$; these angles are $3$-dimensional angles,
which we underline by drawing them with double strokes.
On each of the bottom pictures, the die is viewed from the direction of one of the lines
(from the left to the right, $L_3$, $L_1$ and~$L_2$),
so that this line appears completely foreshortened.
We represent the angles $\widehat\Omega'$, $\widehat{A}'$ and~$\widehat{B}'$
made by the two other lines {\it as they appear on the drawing};
we underline that these angles are $2$-dimensional by drawing them with simple strokes
{\footnotesize[Note that in the case of~$\widehat{B}'$,
the angular sector representing $\widehat{B}'$
is not the projection of the angular sector representing $\widehat{B}$, but its supplementary%
—otherwise $\widehat{B}'$ would be greater than $\pi/2$, which would contradict
our `geometric' definition of angles.]}.
On this example one has
$\widehat{A}\simeq58\,^{\circ}, \widehat{B}\simeq71\,^{\circ}, \widehat\Omega\simeq15\,^{\circ}$ and $\widehat{A}'\simeq30\,^{\circ}, \ab \widehat{B}'\simeq34\,^{\circ}, \ab 
\widehat\Omega'\simeq9\,^{\circ}$; so, perspective makes angles appear
smaller than they are really for all three pairs of lines,
which is in accordance with Theorem~\ref{c8818}.}
\label{fig-3lines}
\end{figure}

\begin{proof}
Fix three arbitrary nonzero vectors $\alpha, \ab \beta, \ab \omega$ of resp.~$L_1, \ab L_2, \ab L_3$;
and consider the Gaussian system~(\ref{for5012}) of \S~\ref{parIllustrationOfTheProofs}
for these vectors.
Then the correlation coefficients between~$X_1$, $X_2$ and~$Y$
can be interpreted as angles between~$L_1$, $L_2$ and~$L_3$;
more precisely, one has the following correspondance:
\begin{Pro}\label{procscs}\strut
\begin{ienumerate}
\item\label{itmCos2}
$\{X_1:X_2\}$ is the cosine of the geometric angle between~$L_1$ and~$L_2$;
\item\label{itmCos3}
$\{X_1:X_2\}_Y$ is the cosine of the apparent angle between~$L_1$ and~$L_2$ seen from $L_3$.
\end{ienumerate}
\end{Pro}
\begin{proof}
(\ref{itmCos2}) is nothing but the Euclidian interpretation of Theorem~\ref{pro1857}.
(\ref{itmCos3}) follows from the fact that,
in the vector space spanned by jointly Gaussian real random variables,
conditional expectation corresponds to orthogonal projection
and independence corresponds to orthogonality.
\end{proof}

By Proprosition~\ref{procscs}, in our situation Lemma~\ref{l6791} gives:
\[\label{for0803a} \big\{(X_1,X_2):Y\big\} = \sqrt{1 - \sin^2\widehat{B} \, \sin^2\widehat{A}'} .\]
Obviously the roles of~$X_1$ and~$X_2$ can be interchanged in the above argument,
yielding:
\[\label{for0803b} \big\{(X_2,X_1):Y\big\} = \sqrt{1 - \sin^2\widehat{A} \, \sin^2\widehat{B}'} .\]
But $(X_1,X_2)$ and~$(X_2,X_1)$ generate the same $\sigma$-algebra,
so ${\{(X_1,X_2):Y\}} \ab = {\{(X_2,X_1):Y\}}$,
and thus, comparing~(\ref{for0803a}) and~(\ref{for0803b}):
\[ \frac{\sin\widehat{A}'}{\sin\widehat{A}} = \frac{\sin\widehat{B'}}{\sin\widehat{B}} .\]
This implies in particular that
$\sin\widehat{A},\sin\widehat{A}'$ and~$\sin\widehat{B},\sin\widehat{B}'$
have the same relative order, so also do $\widehat{A},\widehat{A}'$ and~$\widehat{B},\widehat{B}'$.
A cyclic permutation of~$L_1$, $L_2$ and~$L_3$
shows that the result is still valid for~$\widehat{\Omega},\ab \widehat{\Omega}'$.\linebreak[1]\strut
\end{proof}

\begin{Xpl}
See Figure~\ref{fig-3lines} on page~\pageref{fig-3lines}.
\end{Xpl}

\chapter{Other applications of tensorization techniques}%
\label{parOtherApplicationsOfTensorizationTechniques}%
\addcontentsline{itc}{chapter}{\protect\numberline{\thechapter}Other applications of tensorization techniques}

In the previous chapter we have been seeing
how Hilbertian decorrelation hypotheses between pairs of variables
could yield `global' results on an arbitrary number of variables,
by splitting functions of several variables into relevant telescopic sums.
I used the word ``tensorization'' to qualify these results,
as the conclusions were of the same nature as the hypotheses.

But the techniques of \S~\ref{parTensorization} can also be applied to get other types of results.
In this chapter I am going to show how, from Hilbertian decorrelation hypotheses,
one can get results on some classical features of particle systems
which are not linked with Hilbertian correlations \emph{a priori}.

I will deal with two such features.
First, I will look at the implications of $\rho$-mixing
on the existence of a central limit theorem%
—more precisely, of a \emph{spatial} central limit theorem,
since I am more interested in random fields than in sequences
(variables indexed by~$\ZZ^n$ rather than by~$\ZZ$).
Very sharp results concerning this issue are already known;
however, I find interesting to show how it goes with my `tensorization-like' approach:
this approach takes indeed a quite different way to do the job,
which may be neater by certain sides.
Moreover, the results are stated with a slighlty different vocabulary%
—namely, \emph{subjective} Hilbertian correlations.

Next, I will look at the question of spectral gap for Glauber dynamics.
Though this point has already been thouroughly studied in a $\beta$-mixing paradigm,
this work, to the best of my knowledge,
is the first to show how $\rho$-mixing can be used to tackle this issue.

My main goal here is just to show
\emph{how} the techniques of this work may be applied to the problems
of spatial central limit theorem and convergence of the Glauber dynamics.
Accordingly, I favoured the simplicity on proofs against the refinement of the results.

\section{Spatial central limit theorem}

\subsection{Introduction}

A fundamental result in probability theory is the central limit theorem (CLT),
which, in its standard statement, requires an assumption of complete independence.
It is natural to wonder whether that assumption can be relaxed
into an hypothesis of `near independence'.
Hilbertian decorrelations are a natural frame for such a generalization,
since the CLT already takes place in an $L^2$ setting.

Our point of view is motivated by statistical physics.
Let~$\ZZ^n$ be a lattice, on each vertex $i$ of which there is a random `spin' $X_i$
ranged in some space $\mcal{X}$ not depending on~$i$.
We assume that the law of the system is translation invariant,
i.e.\ that for all~$z\in\ZZ^n$, $(X_{i+z})_{i\in\ZZ^n}$ has the same law
as~$\vec{X}_{\ZZ^n}$.
Then, for all~$z \in \ZZ^n$, we denote
\[\label{f5933} \epsilon_z = \{ X_i : X_{i+z} \}_* .\]
We are interested in situations where the~$\epsilon_z$ are sufficiently `rapidly decreasing'
as $|z| \longto \infty$ so that $\sum_{z\in\ZZ^n} \epsilon_z < \infty$.

Let~$f \colon \mcal{X} \to \RR$ be a function such that $f(X_0)$ is square-integrable and centered.
The question is, does one get a CLT when summing $f(X_i)$ for~$i$ in a large subset of~$\ZZ^n$,
i.e., does the sum grow as the square root of the number of its terms
and have asymptotically normal distribution?
For instance, we would like the law of the variable
\[ \frac{1}{\sqrt{\ell^n}} \sum_{\substack{i\in\ZZ^n\\0\leq i_1,\ldots,i_n<\ell}} f(X_i) \]
to weakly converge, when~$\ell \longto \infty$, to some Gaussian distribution.

\begin{Rmk}
Note that the limit distribution, if it exists, will have to be centered,
but its variance will not be equal to~$\Var\big(f(X_0)\big)$ in general.
\end{Rmk}

In the case $n=1$, extremely sharp results for this topic have been known from long;
let us cite, among many others, \cite{Rosenblatt56, Ibragimov75, Peligrad87, Bradley92}.
For $n\geq 2$, similar results also exist; see e.g.\ \cite[Theorem~5]{Bradley92} for such a result,
and~\cite[\S~29]{Bradley-book} for a survey of the topic.
All these proofs relie on some `coupling' between (bunches of) the spins
and other convenient variables which are close to them,
but which are \emph{actually} independent,
so as to deduce the CLT for the former from the CLT for the latter.
On the other hand, my proof will mimick
Lévy's proof of the CLT, hence needing no coupling argument.

\emph{A priori} the results presented here do not improve the state of the art;
however, when turning to quantitative versions of these results,
it is likely that the difference between the usual method and mine
would yield a difference in the corresponding non-asymptotic bounds obtained.

\subsection{Product of weakly coupled variables}

My results relie on the following
\begin{Lem}\label{l1840}
Let~$N \geq 1$ and let~$\dot{\mcal{F}}_1, \ldots, \dot{\mcal{F}}_N$ be $\sigma$-algebras with
$\{ \dot{\mcal{F}}_i : \dot{\mcal{F}}_j \}_* \leq \epsilon_{ij}$%
\footnote{As in \S~\ref{parMachineryForUsingTheTensorizationTheorems},
``$*$'' stands for ``the natural $\sigma$-metalgebra of the system''.
Moreover, in the same way as in \S~\ref{parMinimalHypotheses},
it is actually possible in the statement of the lemma
to replace that~$\sigma$-metalgebra by smaller ones.},
and denote
\[ \bar\epsilon = \sup_i \sum_{j\neq i} \epsilon_{ij} .\]
Let~$\Phi_1, \ldots, \Phi_N$ be complex-valued random variables with $|\Phi_i|\leq1$ a.s.,
such that $\Phi_i$ is $\dot{\mcal{F}}_i$-measurable for all~$i$,
with all the~$\Phi_i$ having the same distribution.
Then, denoting by~$\phi$ the common value of the~$\EE[\Phi_i]$,
\[\label{fo9356} \Big| \EE\Big[ \prod_{i} \Phi_i \Big] - \phi^N \Big| \leq
N \bar\epsilon(1+\bar\epsilon) (1-|\phi|^2) .\]
\end{Lem}

\begin{proof}
Denote $\delta \coloneqq \ecty(\Phi)$. Since $\EE[|\Phi^2|] \leq 1$,
the definition of (complex) variance ensures that $\delta \leq \sqrt{1-|\phi|^2}$.

For all~$i \in \{0,\ldots,N\}$, denote $\mcal{F}_i \coloneqq \bigvee_{i'\leq i}\dot{\mcal{F}}_i$ ;
denote $\Psi^{(i)} \coloneqq \prod_{i'\leq i} \Phi_{i'}$; define
\[ \Psi^{(i)}_j \coloneqq (\Psi^{(i)})^{\mcal{F}_j} - \EE[\Psi^{(i)}|\mcal{F}_{j-1}] \]
and denote $\Delta^{(i)}_j \coloneqq \ecty(\Psi^{(i)}_j)$.
Also denote
\[ \Phi_{i,j} \coloneqq (\Phi_i)^{\mcal{F}_j} - \EE[\Phi_i|\mcal{F}_{j-1}] .\]

Usual manipulation on conditioning shows that, for~$i\geq1$,
\[ \Psi^{(i)}_j = \Psi^{(i-1)}_j (\Phi_i)^{\mcal{F}_j}
+ \Psi^{(i-1)}_{j-1} \Phi_{i,j} .\]
Since $\|\Phi_i\|_{L^\infty} \leq 1$,
one has also $\| (\Phi_i)^{\mcal{F}_j} \|_{L^\infty} \leq 1$, hence
\[ \ecty\big( \Psi^{(i-1)}_j (\Phi_i)^{\mcal{F}_j} \big)
\leq \ecty(\Psi^{(i-1)}_j) = \Delta^{(i-1)}_j .\]
Similarly, it is obvious that $\|\Psi^{(i-1)}\|_{L^\infty} \leq 1$, whence
\[ \ecty\big( \Psi^{(i-1)}_{j-1} \Phi_{i,j} \big) \leq \ecty(\Phi_{i,j}) .\]

Now, I claim that
\begin{Clm}\label{clm-dld}
\[ \ecty(\Phi_{i,j}) \leq \epsilon_{ij} \delta .\]
\end{Clm}
\begingroup\def\proofname{Proof of Claim~\ref{clm-dld}}
\begin{proof}
Conditionally to~$\mcal{F}_{j-1}$, $\Phi_{i,j}$ is indeed the projection on~$\dot{\mcal{F}}_{j}$ of the centered $\dot{\mcal{F}}_i$-measurable function $(\Phi_i - \EE[\Phi_i|\mcal{F}_{j-1}])$, whose standard deviation is less than $\ecty(\Phi_i) = \delta$ by associativity of the variance, so that $\ecty(\Phi_{i,j}) \leq {\{\dot{\mcal{F}}_i:\dot{\mcal{F}}_j\}_{\mcal{F}_{j-1}}} \delta \ab \leq \epsilon_{ij} \delta$.
\end{proof}\endgroup

In the end, we got that
\[ \Delta^i_j \leq \Delta^{(i-1)}_j + \epsilon_{ij} \delta .\]
Since $\Delta^0_j = 0$, one has therefore:
\[ \forall i,j \qquad \Delta^i_j \leq (1+\bar\epsilon) \delta .\]
Now, denoting $\psi^{(i)} \coloneqq \EE[\Psi^{(i)}]$, one has
\[ | \psi^{(i)} - \phi\psi^{(i-1)} |
= \Big| \sum_{j<i} \EE\big[ \Psi^{(i-1)}_j \Phi_{i,j} \big] \Big|
\leq \sum_{j<i} \Delta^{(i-1)}_j \ecty(\Phi_{i,j})
\leq \bar\epsilon (1+\bar{\epsilon}) \delta^2 ,\]
and finally
\[ | \psi^{(N)} - \phi^N |
= \sum_{i=1}^N |\phi|^{N-i} \big|\psi^{(i)} - \phi\psi^{(i-1)}\big|
\leq \sum_{i=1}^N \big|\psi^{(i)} - \phi\psi^{(i-1)}\big| \\
\leq N \bar\epsilon(1+\bar\epsilon) \delta^2 ,\]
which is (\ref{fo9356}) if you recall that $\delta^2 \leq 1-|\phi|^2$.
\end{proof}

\subsection{A spatial CLT}

First I state and prove a CLT on cubes:
\begin{Thm}\label{thm0318}
Consider a translation-invariant spin model on a lattice $\ZZ^n$
and define $\epsilon_z$ by~(\ref{f5933}).
Assume that $\sum_{z\in\ZZ^n} \epsilon_z < \infty$.
Then for any centered square-summable function~$f \colon \mcal{X} \to \RR$,
there exists a constant $\sigma < \infty$ such that
\[\label{f7877}
\frac{1}{\sqrt{\ell^n}} \sum_{\substack{i\in\ZZ^n\\0\leq i_1,\ldots,i_n<\ell}} f(X_i)
\stackrel{\ell\longto\infty}{\rightharpoonup} \mcal{N}(\sigma^2) ,\]
where ``$\rightharpoonup$'' denotes convergence in law.
\end{Thm}

\begin{proof}
Denote by~$F(\ell)$ —or merely $F$ —the left-hand side of~(\ref{f7877}).

What will be the value of~$\sigma$?
Clearly we must have
\[\label{fo0116} \sigma^2 = \lim_{\ell\longto\infty} \Var \big( F(\ell) \big) ,\]
which yields
\[ \sigma = \sqrt{ \sum_{z\in\ZZ} \EE\big[ f(X_0) f(X_z) \big] } ,\]
where the expression under the root sign,
which is necessarily nonnegative,
is finite because $|\EE[f(X_0)f(X_z)]| \leq
\epsilon_z \ecty(f(X_0))\ecty(f(X_z)) = \epsilon_z \Var(f)$.
By the way, we will denote
\[ \sigma_*^2 \coloneqq \Big( \sum_z \epsilon_z \Big) \|f\|_{L^2}^2 .\]

Fix some arbitrary $\eta > 0$.
The assumption that $\sum \epsilon_z < \infty$
implies the existence of an $\ell_0 < \infty$
such that
\[ \sum_{|z|_\infty > \ell_0} \epsilon_z \leq \eta ,\]
where $|z|_\infty$ denotes $\max(|z_1|,\ldots,|z_n|)$.
By~(\ref{fo0116}), we can also fix an $\ell_1 < \infty$
such that
\[\label{f1939} \big| \Var \big( F(\ell) \big) - \sigma^2 \big| \leq \eta .\]

Now we will `tile' the cube of size $\ell$ into a `patchwork'
made of cubes of size $\ell_1$ which I call ``tiles'',
each tile being at distance at least $\ell_0$ from the others, plus some ``scrap''.
I denote by~$\tilde{F}$ the part of~$F$ due to the tiles
and by~$F^*$ the part of~$F$ due to the scrap.

Index the tiles by~$\{ 1,\ldots,N \}$, with
$N \coloneqq \big\lfloor (\ell+\ell_0) \div (\ell_1+\ell_0) \big\rfloor^n$.
We write, with obvious notation, $\tilde{F} \eqqcolon F_1 + \cdots + F_N$.
For~$\lambda\in\RR$, denote
\begin{eqnarray}
\Psi (\lambda,\ell) \coloneqq \exp(i\lambda \tilde{F}), &\qquad&
\psi(\lambda,\ell) \coloneqq \EE[\Psi (\lambda,\ell)] ; \\
\Phi_j(\lambda,\ell) \coloneqq \exp(i\lambda F_j), &\qquad&
\phi(\lambda,\ell) \coloneqq \EE[\Phi_j(\lambda,\ell)] .
\end{eqnarray}
Then we are exactly in situation of applying Lemma~\ref{l1840},
which yields:
\[\label{f1890} \big| \psi(\lambda,\ell) - \phi(\lambda,\ell)^N \big| \leq
N \eta(1+\eta) \big( 1-|\phi(\lambda,\ell)|^2 \big) .\]

Let us look at the asymptotics of Formula~(\ref{f1890}) when~$\ell \longto \infty$.
We observe that, denoting
\[ F_{\rm t} \coloneqq \frac{1}{\sqrt{\ell_1^n}} \sum_{i\in\text{\it\ fixed tile}} f(X_i) ,\]
one has
\[ F_j \stackrel{\text{law}}{=} \frac{\sqrt{\ell_1^n}}{\sqrt{\ell^n}} F_{\rm t} .\]
Since $F_{\rm t}$ is centered, its Fourier transform satisfies $\hat{F}_{\rm t}(0)=1$,
$\hat{F}_{\rm t}'(0)=0$ and $\hat{F}_{\rm t}'' = \Var(F_{\rm t})$, so that
\[ \lim_{\ell\longto\infty} \ell^n \big(1-\phi(\lambda,\ell)\big)
= \frac{\lambda^2}{2} \ell_1^n \Var(F_{\rm t}) ,\]
where, denoting $\sigma^2_{\ell_1} \coloneqq \Var(F_{\rm t})$,
we recall that $\ell_1$ has been taken sufficiently large so that
$|\sigma^2_{\ell_1} - \sigma^2| \leq \eta$.
Then, since $N \stackrel{\ell\longto\infty}{\sim} \ell^n \div (\ell_1+\ell_0)^n$,
one has the following asymptotics for~(\ref{f1890}):
\begin{eqnarray}
\phi(\lambda,\ell)^N & \stackrel{\ell\longto\infty}{\longto} &
\exp \left[ -\sigma^2_{\ell_1} \bigg( \frac{\ell_1}{\ell_1+\ell_0} \bigg)^{\!n}
\frac{\lambda^2}{2} \right] ; \\
N \big( 1-|\phi(\lambda,\ell)|^2 \big) & \stackrel{\ell\longto\infty}{\longto} &
\sigma^2_{\ell_1} \bigg( \frac{\ell_1}{\ell_1+\ell_0} \bigg)^{\!n}\lambda^2 .
\end{eqnarray}

It remains to control the contribution of~$F^*$.
\begin{Clm}\label{c7901}
There are at most
\[ \bigg[1-\bigg(\frac{\ell_1}{\ell_1+\ell_0}\bigg)^{\!n}\bigg]\ell^n + n\ell_1\ell^{n-1} \]
scrap spins.
\end{Clm}
By Claim~\ref{c7901},
\[ \|F^*\|_{L^1} \leq \Var(F^*) \leq
\bigg[1-\bigg(\frac{\ell_1}{\ell_1+\ell_0}\bigg)^{\!n} + n\frac{\ell_1}{\ell} \bigg] \sigma_*^2 ;\]
then the contribution of~$F^*$ is controlled using the following immediate
\begin{Lem}
Let~$X$ and~$H$ be real random variables with $\|H\|_{L^1} < \infty$.
Then, for~$\lambda\in\RR$,
\[ \big| \EE\big[e^{i\lambda(X+H)}\big] - \EE\big[e^{i\lambda X}\big] \big| \leq |\lambda| \|H\|_{L^1} .\]
\end{Lem}

In the end, putting everything together we get:
\begin{multline}\label{f6829}
\limsup_{\ell\longto\infty} \big| \psi(\lambda,\ell) - e^{-\sigma^2\lambda^2/2} \big| \leq \\
\Big|e^{\big{-(\sigma^2-\eta) [\ell_1 \div (\ell_1+\ell_0)]^n \lambda^2/2}} - e^{-\sigma^2\lambda^2/2}\Big|
+ \eta(1+\eta) \bigg( \frac{\ell_1}{\ell_1+\ell_0} \bigg)^{\!n} (\sigma^2+\eta) \lambda^2
+ \sqrt{ 1-\bigg(\frac{\ell_1}{\ell_1+\ell_0}\bigg)^{\!n} } \sigma_* .
\end{multline}
Since there were no upper restriction on the value of~$\ell_1$, we can assume
that we have taken it such that $[\ell_1 \div (\ell_1+\ell_0)]^n \geq 1-\eta$.
Then (\ref{f6829}) becomes:
\[\label{f6890} \limsup_{\ell\longto\infty}
\big| \psi(\lambda,\ell) - e^{-\sigma^2\lambda^2/2} \big| \leq
\Big| e^{\big{-(\sigma^2-\eta)(1-\eta)\lambda^2/2}} - e^{-\sigma^2\lambda^2/2} \Big|
+ \eta(1+\eta)(1-\eta) (\sigma^2+\eta) \lambda^2
+ \sqrt{\eta} \sigma_* .\]
The right-hand side of~(\ref{f6890}) can be made arbitrarily close to~$0$
by taking $\eta$ small enough, so we have proved that
\[ \forall \lambda \in \RR \quad \EE\big[e^{i\lambda F(\ell)}\big]
\stackrel{\ell\longto\infty}{\longto} e^{-\sigma^2\lambda^2/2} .\]
By L\'evy's theorem on characteristic functions,
this is tantamount to saying that $F(\ell)$ converges in law to~$\mcal{N}(\sigma^2)$.
\end{proof}

The CLT should remain valid for other shapes than a cube,
since morally the random field $f(X_i)$
should look like a Gaussian white noise at large scales.
Indeed, the same proof as above yields a CLT for general shapes,
where moreover convergence is uniform in the shape considered in some way.
Let us give a precise statement:
\begin{Def}
An open subset $U \subset \RR^n$ (not necessarily connected)
is said to be \emph{$\mcal{C}^2$} if its boundary $M$ is a $\mcal{C}^2$ submanifold
of~$\RR^n$ (of codimension $1$).
We define the \emph{roughness} of~$U$, denoted by~$\kappa(U)$, as
\[ \kappa(U) \coloneqq \sup_{x\in M} \VERT \II(x) \VERT, \]
where $\II(\Bcdot)$ denotes the shape tensor of~$M$~\cite[Chapter~10]{Differential},
which measures the local deviation of~$M$ from being flat.
Also, the Lebesgue measure of~$U$ will be denoted by~$\volume(U)$.
\end{Def}

\begin{Thm}\label{thm3015}
Consider a translation-invariant spin model on a lattice $\ZZ^n$
and define $\epsilon_z$ by~(\ref{f5933}).
Assume that $\sum_{z\in\ZZ^n} \epsilon_z < \infty$.
Then for any centered square-summable function~$f \colon \mcal{X} \to \RR$,
if $(U_k)_{k\in\NN}$ is a sequence of $\mcal{C}^2$ bounded subsets of~$\RR^n$
with $\sup_k \kappa(U_k) < \infty$ and
$(\ell_k)_{k\in\NN}$ is a sequence of positive numbers tending to infinity,
\[ \frac{1}{\sqrt{\ell_k^n\volume(U_k)}} \sum_{i\in\ell_kU_k\cap\ZZ^n} f(X_i)
\stackrel{k\longto\infty}{\rightharpoonup} \mcal{N}(\sigma^2) ,\]
where $\sigma^2$ is the same as in Theorem~\ref{thm0318}.
\end{Thm}

\begin{proof}
Just copy the proof of Theorem~\ref{thm0318}.
The only difference lies in proving the analoguous of Claim~\ref{c7901},
which is where one needs the~$\kappa(U_k)$ to be bounded.
Observe that we use the non-asymptotic form of our intermediate bounds
to get a result independent of the precise shape of the~$U_k$.
\end{proof}

\begin{Rmk}
Another generalization of the CLT, still based on the idea that
the field $f(X_i)$ looks like a Gaussian white noise at large scales,
is the statement that for~$\phi$ a continuous function with compact support,
\[ \frac{1}{\sqrt{\ell^n}} \sum_{i\in\ZZ^n} \phi\big(X_i\div\ell\big) f(X_i)
\stackrel{\ell\longto\infty}{\rightharpoonup}
\mcal{N}\big(\sigma^2 \mathsmaller{\int_{\RR^n} \phi(x)^2\,\dx{x}} \big) .\]
This can be proved with the same methods as before.
\end{Rmk}

\section{Spectral gap for the Glauber dynamics}

\subsection{Introduction}

In this section we are looking at a probabilistic system
made of a large number of `elementary' random variables $(X_i)_{i\in I}$
— $I$ may be seen as lattice and $X_i$ as the state of the particle being at site~$i$.
As is customary by now, theorems will only be stated in the case where $I$ is finite,
the infinite case being got by passing to the limit.

\begin{Def}
Denoting by~$\Omega$ the states space of~$\vec{X}_I$,
let~$\Pr$ be a probability measure on~$\Omega$.
The \emph{Glauber dynamics}~\cite{Glauber, Dobrushin71}
associated to~$\Pr$ is the Markov process on~$\Omega$ having the following law:
on each~$i\in I$ there is an alarm clock,
all the clocks being independent and ringing with law $\mathit{Poisson}(1)$.
When a clock rings, the state of spin $X_i$ —and only it—is flipped
so that the state of~$X_i$ immediately after the flip follows the law
$\Pr\big(X_i|\vec{X}_{I\setminus\{i\}}\big)$.

In formal terms, the Glauber dynamics is the Markov process
whose generator $\mcal{L}$ on~$L^\infty(\Omega)$ is defined by:
\[ \big(\mcal{L}f\big)(\vec{x}_I)
= \sum_{i\in I} \EE \big[ f(\vec{X}_I) - f(\vec{x}_I) \big|
\vec{X}_{I\setminus\{i\}} = \vec{x}_{I\setminus\{i\}} \big] .\]
\end{Def}

Let us recall some basic facts on the Glauber dynamics
(see~\cite[Chapter~IV]{Liggett} for more details).
By construction $\Pr$ is a reversible equilibrium measure for the dynamics,
so $\mcal{L}$ is self-adjoint on~$L^2(\Pr)$.
Since obviously $\mcal{L}1\equiv0$, one can also define $\mcal{L}$ on~$\ldb(\Pr)$,
on which it is self-adjoint too. This leads to the following definition:
\begin{Def}
The \emph{energy} of~$f\in\ldb(\Pr)$ is
\[ \mcal{E}(f,f) = \langle Lf, f \rangle \]
\end{Def}
The following immediate identity shows that $\mcal{E}$ is always a nonnegative bilinear form:
\begin{Pro}\label{p1345}
\[ \mcal{E}(f,f) = \int_\Omega \dx\Pr[\vec{x}_I]
\sum_i \Var\big( f \big| \vec{X}_{I\setminus\{i\}} = \vec{x}_{I\setminus\{i\}} \big) .\]
\end{Pro}

\begin{Def}
For~$\lambda > 0$, the Glauber dynamics is said to have \emph{spectral gap $\geq\lambda$}
if, for all~$f\in\ldb(\Pr)$,
\[\label{fo8103} \mcal{E}(f,f) \geq \lambda \Var(f) .\]
\end{Def}
What makes spectral gap interesting is that its positiveness
is equivalent to exponential convergence to~$0$
of the semigroup $(e^{-t\mcal{L}})_{t\geq0}$ on $\ldb(\Pr)$,
the rate of convergence being equal to the width of the spectral gap.
As the Glauber dynamics is one of the easiest ways to simulate the law $\Pr$ for complicated models,
the stake of having exponential convergence for it is evident.

Many works have been done on the spectral gap of the Glauber dynamics,
see for instance Martinelli's St-Flour course~\cite{Martinelli}.
Several results state that,
the less spins are correlated, the larger the spectral gap is.
Yet the researchers who work on this topic generally express
the decorrelation between the spins in terms of $\beta$-mixing (cf.\ Definition~\ref{def3167}),
while it seems be more natural to look at them in terms of Hilbertian decorrelations,
since the formula~(\ref{fo8103}) stating the spectral gap problem
takes place in a Hilbertian frame itself.
Thus my goal here will be to find a control on the spectral gap
expressed in terms of $\rho$-mixing conditions.
Since Hilbertian correlations look to be the minimal frame
to study the spectral gap for the Glauber dynamics,
hopefully the bounds yielded by this method will be sharp.

Another noticeable feature of my approach
is that it remains at a quite abstract level:
no symmetry property of~$I$ or $\Pr$ need be assumed,
all the work essentially consisting in manipulating relevant quadratic forms.

\subsection{A lower bound for the spectral gap}

The central theorem of this section is the following:
\begin{Thm}\label{t7230}
Take $I = \{1,\ldots,N\}$.
Suppose that for all distinct $i,j\in I$ one has ${\{X_i:X_j\}_\ast} \ab \leq \epsilon_{ij} < 1$
—we will make the costless assumption that $\epsilon_{ji} = \epsilon_{ij}$.
For~$i \in I$, denote
\[ \tilde1_i \coloneqq \frac{1}{ \prod_{i<j\leq N} (1-\epsilon_{ij}^2) }
= \frac{\teps_{iN}}{\epsilon_{iN}} ,\]
and for~$i < j$, denote
\[ \teps_{ij} \coloneqq
\frac{\epsilon_{ij}}{\prod_{i<j'\leq j}(1-\epsilon_{ij'}^2)} .\]
Then the Glauber dynamics has spectral gap at least $\VERT M \VERT^{-2}$,
where $M$ is the $(N \times N)$ matrix defined by
\[\label{for4393} M = \begin{pmatrix} 1 & -\teps_{12} & \cdots & -\teps_{1N} \\
0 & \ddots & \ddots & \vdots \\
\vdots & \ddots & \ddots & -\teps_{(N-1)N} \\
0 & \cdots & 0 & 1 \end{pmatrix}^{-1}
\begin{pmatrix} \tilde1_1 & 0 & \cdots & 0 \\
0 & \ddots & \ddots & \vdots \\
\vdots & \ddots & \ddots & 0 \\
0 & \cdots & 0 & \tilde1_N \end{pmatrix} .\]
\end{Thm}
\begin{Rmk}\label{rmk4123}
The form of the first matrix in the right-hand side of~(\ref{for4393})
ensures that it is invertible.
Since moreover all the~$\epsilon_{ij}$ were supposed $<1$,
all the~$\tilde\epsilon_{ij}$ and the~$\tilde{1}_i$ are finite;
thus, the lower bound~$\VERT M \VERT^{-2}$ is strictly positive.
\end{Rmk}

\begin{proof}
Let~$f$ be a centered square-integrable function on~$(\Omega,\Pr)$.
For~$I' \subset I$, denote $\mcal{F}_{I'} \coloneqq \sigma\big(\vec{X}_{I'}\big)$.
For~$i\in I$, $I' \subset I\setminus\{i\}$,
denote
\[ f_i^{I'} \coloneqq f^{\mcal{F}_{I'\uplus\{i\}}} - \EE[f|\mcal{F}_{I'}] ;\]
define moreover
\begin{eqnarray}
f_i^{\neq} & \coloneqq & f_i^{I\setminus\{i\}} ; \\
f_i^{<}  & \coloneqq & f_i^{\{1,\ldots,i-1\}}.
\end{eqnarray}
Then by Proposition~\ref{p1345}, one has
\[ \mcal{E}(f,f) = \sum_i \Var(f_i^{\neq}) ,\]
while the usual telescopic argument shows that
\[ \Var(f) = \sum_i \Var(f_i^<) .\]

So to prove the theorem, we have to establish links between
the different values $\Var(f_i^{I'})$.
It will be convenient to introduce the shorthands $\Delta_i^{I'} = \ecty(f_i^{I'})$.
One has the following
\begin{Clm}\label{l8326}
For~$I' \subset I$ and~$i,j\in I \setminus I'$ with $j\neq i$,
\[\label{f4464} \Delta_i^{I'} \leq \Delta_i^{I'\uplus\{j\}} + \epsilon_{ij} \Delta_j^{I'} .\]
\end{Clm}

\begin{proof}
Assume in a first time that $I'=\emptyset$,
and denote $f_i \coloneqq f_i^{\emptyset}$, $f_j \coloneqq f_j^{\emptyset}$, $f_i^j \coloneqq f_i^{\{j\}}$
and $\mcal{F}_i \coloneqq \mcal{F}_{\{i\}}$.
Projecting the decomposition ``$f_i = f_i^j + (f_i-f_i^j)$''
on~$L^2(\mcal{F}_i)$, one has $f_i = (f_i^j)^{\mcal{F}_i} + (f_j)^{\mcal{F}_i}$,
whence by the Cauchy--Shwarz inequality:
\[\label{f4448}
\ecty(f_i) \leq \ecty\big((f^j_i)^{\mcal{F}_i}\big) + \ecty\big((f_j)^{\mcal{F}_i}\big) .\]
One has trivially
$\ecty\big((f^j_i)^{\mcal{F}_i}\big) \leq \ecty(f^j_i)$;
on the other hand, $f_j$ is $X_j$-measurable,
so $\ecty\big((f_j)^{\mcal{F}_i}\big) \leq \epsilon_{ij} \ecty(f_j)$.
In the end, (\ref{f4448}) becomes
\[ \ecty(f_i) \leq \ecty(f^j_i) + \epsilon_{ij} \ecty(f_j) ,\]
which is~(\ref{f4464}) for $I' = \emptyset$.

In the case $I' \neq \emptyset$, the same reasoning can be performed, except that
one have to work conditionally to~$\mcal{F}_{I'}$.
Then, taking
$f_i = f^{I'}_i$,
$f_j = f^{I'}_j$,
$f^j_i = f^{I'\uplus\{j\}}_i$,
$\mcal{F}_i = \mcal{F}_{I'\uplus\{i\}}$,
one gets
\[\label{fo5917}
\ecty(f_i|\mcal{F}_{I'}) \leq \ecty(f^j_i|\mcal{F}_{I'}) + \epsilon_{ij} \ecty(f_j|\mcal{F}_{I'}) .\]
Now
\[\ecty(f_i) = \sqrt{\int \ecty\big(f_i|\vec{X}_{I'}=\vec{x}_{I'}\big)^2 \dx\Pr[\vec{x}_{I'}]} ,\]
with similar formulas for~$f_j$ and~$f^j_i$, since all these functions
are centered w.r.t.~$\mcal{F}_{I'}$.
Therefore, integrating~(\ref{fo5917}) and applying Minkowski's inequality yields:
\[ \ecty(f_i) \leq \ecty(f^j_i) + \epsilon_{ij} \ecty(f_j) ,\]
i.e.~(\ref{f4464}).
\end{proof}

For~$i \leq j$, let us denote
\[ \Delta^{[j]}_i = \Delta^{\{1,\ldots,j\} \setminus \{i\}}_i .\]
Claim~\ref{l8326} will be used through the following corollary:
\begin{Clm}\label{l3526}
For all~$i<j$,
\[\label{f7631} \Delta^{[j-1]}_i \leq \frac{1}{1-\epsilon_{ij}^2}
\big( \Delta^{[j]}_i + \epsilon_{ij} \Delta^{<}_j \big) .\]
\end{Clm}

\begin{proof}
We have to bound $\Delta^{[j-1]}_i$, which here we rather denote
$\Delta^{[b-1]}_a$ to avoid confusion with the notation of Claim~\ref{l8326}.
Applying Claim~\ref{l8326} with $I' = \{1,\ldots,b-1\} \setminus \{a\}$,
$i = a$ and $j = b$, one has
\[\label{f7625a} \Delta^{[b-1]}_a = \Delta^{\{1,\ldots,b-1\}\setminus\{a\}}_{a}
\leq \Delta^{\{1,\ldots,b\}\setminus\{a\}}_{a}
+ \epsilon_{ab} \Delta^{\{1,\ldots,b-1\}\setminus\{a\}}_b
= \Delta^{[b]}_{a} + \epsilon_{ab} \Delta^{\{1,\ldots,b-1\}\setminus\{a\}}_b .\]
But applying again Claim~\ref{l8326}, this time with $I' = \{1,\ldots,b-1\} \setminus \{a\}$,
$i = b$ and $j = a$, one has
\[\label{f7625b} \Delta^{\{1,\ldots,b-1\}\setminus\{a\}}_{b}
\leq \Delta^{\{1,\ldots,b-1\}}_{b}
+ \epsilon_{ab} \Delta^{\{1,\ldots,b-1\}\setminus\{a\}}_a
= \Delta^{<}_b + \epsilon_{ab} \Delta^{[b-1]}_a .\]
Combining~(\ref{f7625a}) and~(\ref{f7625b}) then yields~(\ref{f7631}).
\end{proof}

Now let us show how Claim~\ref{l3526} implies the theorem.
To avoid heavy formalism, I will detail the computations for $I = \{1,2,3,4\}$
(rather denoted by~$I = \{a,b,c,d\}$ here to avoid confusions
with ``$1$'' and~``$2$'' taken as numbers),
hoping that generalizing is obvious then.

First, note that
\[\label{f4236d} \Delta^<_d = \Delta^{\neq}_d .\]
Now, by a direct use of Claim~\ref{l3526},
\[\label{f4236c} \Delta^<_c = \Delta^{[c]}_c
\leq \frac{1}{1-\epsilon_{cd}^2} \big( \Delta^{[d]}_c + \epsilon_{cd} \Delta^{<}_d \big)
= \tilde1_c \Delta^{\neq}_c + \tilde{\epsilon}_{cd} \Delta^{\neq}_d .\]
To bound $\Delta^<_b$, we have to iterate Claim~\ref{l3526} twice:
\begin{multline}\label{f4236b}
\Delta^<_b = \Delta^{[b]}_b
\leq \frac{1}{1-\epsilon_{bc}^2} \Delta^{[c]}_b + \teps_{bc} \Delta^{<}_c \\
\leq \frac{1}{(1-\epsilon_{bc}^2)(1-\epsilon_{bd}^2)}
\big( \Delta^{[d]}_b + \epsilon_{cd} \Delta^{<}_d \big) + \teps_{bc} \Delta^<_c
= \tilde1_b \Delta^{\neq}_b + \teps_{bd} \Delta^{\neq}_d + \teps_{bc} \Delta^<_c \\
\footrel{(\ref{f4236c})}\leq
\tilde1_b \Delta^{\neq}_b + \tilde1_c \teps_{bc} \Delta^{\neq}_c
+ (\teps_{bd} + \teps_{bc}\teps_{cd}) \Delta^{\neq}_d .
\end{multline}
Last, bounding $\Delta^{<}_a$ requires iterating Claim~\ref{l3526} three times:
\begin{multline}\label{f4236a}
\Delta^<_a = \Delta^{[a]}_a
\leq \frac{1}{1-\epsilon_{ab}^2} \Delta^{[b]}_a + \teps_{ab} \Delta^<_b \\
\leq \frac{1}{(1-\epsilon_{ab}^2)(1-\epsilon_{ac}^2)} \Delta^{[c]}_a
+ \teps_{ac} \Delta^<_c + \teps_{ab} \Delta^<_b 
\leq \tilde1_a \Delta^{\neq a} + \teps_{ad} \Delta^<_d
+ \teps_{ac} \Delta^<_c + \teps_{ab} \Delta^<_b \\
\footrel{(\ref{f4236c},\ref{f4236b})}
\leq \tilde1_a \Delta^{\neq a} + \tilde1_b \teps_{ab} \Delta^{\neq}_b
+ \tilde1_c (\teps_{ac} + \teps_{ab}\teps_{bc}) \Delta^{\neq}_c
+ (\teps_{ad} + \teps_{ac}\teps_{cd}
+ \teps_{ab}\teps_{bd}
+ \teps_{ab}\teps_{bc}\teps_{cd}) \Delta^{\neq}_d.
\end{multline}

One can sum up Equations~(\ref{f4236d})--(\ref{f4236a}) into the matricial expression
\[\label{f5482}
\begin{pmatrix} \Delta_a^< \\ \Delta_b^< \\ \Delta_c^< \\ \Delta_d^< \end{pmatrix} \leq
\begin{pmatrix} 1 & \teps_{ab} & \teps_{ac} + \teps_{ab}\teps_{bc} &
\teps_{ad} + \teps_{ac}\teps_{cd} + \teps_{ab}\teps_{bd} + \teps_{ab}\teps_{bc}\teps_{cd} \\
0 & 1 & \teps_{bc} & \teps_{bd} + \teps_{bc}\teps_{cd} \\
0 & 0 & 1 & \teps_{cd} \\
0 & 0 & 0 & 1 \end{pmatrix}
\begin{pmatrix} \tilde1_a \Delta^{\neq}_a \\ \tilde1_b \Delta^{\neq}_b \\
\tilde1_c \Delta^{\neq}_c \\ \Delta^{\neq}_d \end{pmatrix} .\]
If we look back at how the square matrix in~(\ref{f5482}) has been constructed,
we find that
\[ \begin{pmatrix}
\text{\it square} \\
\text{\it matrix} \\
\text{\it in} \\
\text{\it (\ref{f5482})} \\
\end{pmatrix} =
\sum_{k=0}^{\infty} \begin{pmatrix}
0 & \teps_{ab} & \teps_{ac} & \teps_{ad} \\
0 & 0 & \teps_{bc} & \teps_{bd} \\
0 & 0 & 0 & \teps_{cd} \\
0 & 0 & 0 & 0 \end{pmatrix}^{k}
= \begin{pmatrix}
1 & -\teps_{ab} & -\teps_{ac} & -\teps_{ad} \\
0 & 1 & -\teps_{bc} & -\teps_{bd} \\
0 & 0 & 1 & -\teps_{cd} \\
0 & 0 & 0 & 1 \end{pmatrix}^{-1} ,\]
so in the end we obtain that
\[ \begin{pmatrix} \Delta_a^< \\ \Delta_b^< \\ \Delta_c^< \\ \Delta_d^< \end{pmatrix} \leq
M \begin{pmatrix} \Delta^{\neq}_a \\ \Delta^{\neq}_b \\ \Delta^{\neq}_c \\ \Delta^{\neq}_d
\end{pmatrix} ,\]
where $M$ is given by~(\ref{for4393}).
Then it is immediate that $\Var(f) = \sum_i \big(\Delta_i^<\big)^2 \leq
\VERT M \VERT^2 \sum_i \big(\Delta_i^{\neq}\big)^2 = \VERT M \VERT^{2} \,\mcal{E}(f,f)$, \textsc{qed}.
\end{proof}

The bound we have obtained for the spectral gap
is not symmetric by permutation of the indexes in~$I$.
It can however can be bounded by a simpler expression,
which is nearly as good as the original one in concrete situations:
\begin{Cor}
In Theorem~\ref{t7230}, $M$ can be replaced by the matrix
\[ M' = \begin{pmatrix} 1 & -\epsilon_{12} & \cdots & -\epsilon_{1N} \\
-\epsilon_{12} & \ddots & \ddots & \vdots \\
\vdots & \ddots & \ddots & -\epsilon_{(N-1)N} \\
-\epsilon_{1N} & \cdots & -\epsilon_{(N-1)N} & 1 \end{pmatrix}^{-1} ,\]
provided $\rho(\mbf{I}_N-M') < 1$.
\end{Cor}

\begin{proof}
Each entry of~$M$ is actually bounded by the corresponding entry of~$M'$.
To see it, we `expand' the entries of~$M$, resp.~$M'$.
First, notice that $1/(1-\epsilon_{ij}^2)$ can be expanded
into $1+\epsilon_{ij}\epsilon_{ji}
+ \epsilon_{ij}\epsilon_{ji}\epsilon_{ij}\epsilon_{ji} + \cdots$,
so that one has the expansions
\[ \tilde1_i = \sum_{i<j_1\leq\cdots\leq j_k} \prod_{\ell=1}^k
\epsilon_{ij_\ell} \epsilon_{j_\ell i} \]
and
\[ \teps_{ij} = \sum_{i<j_1\leq\cdots\leq j_k\leq j} \Big( \prod_{\ell=1}^k
\epsilon_{ij_\ell} \epsilon_{j_\ell i} \Big) \epsilon_{ij} .\]
Then, using the inversion formula $(\mrm{I}-A)^{-1} = \sum_{k=0}^\infty A^k$
for triangular arrays, one obtains that
\[ M_{ij} = \sum_{\substack{(i_0,i_1,\ldots,i_k)\\\text{\it first condition}}}
\prod_{\ell=0}^{k-1} \epsilon_{i_\ell i_{\ell+1}} ,\]
where the meaning of ``\textit{first condition}'' is given by the following
\begin{Def}\label{d8225}
A sequence $(i_0,\ldots,i_k)$ is said to satisfy the \textit{first condition} if:
\begin{ienumerate}
\item\label{i8226a} $i_0 = i$ and $i_k = j$;
\item\label{i8226b} $i_\ell \neq i_{\ell+1}$ for all~$\ell$;
\item $i_{\ell+1} < i_\ell$ only if $\ell \geq 1$ and $i_{\ell+1} = i_{\ell-1}$;
\item If $i_{\ell+1} < i_\ell$ and $\ell \leq k-2$,
then $i_{\ell+2} \geq i_\ell$.
\end{ienumerate}
\end{Def}

One has a similar formula for~$M'$:
\[ M'_{ij} = \sum_{\substack{(i_0,i_1,\ldots,i_k)\\\text{\it second condition}}}
\prod_{\ell=0}^{k-1} \epsilon_{i_\ell i_{\ell+1}} ,\]
where
\begin{Def}
A sequence $(i_0,\ldots,i_k)$ is said to satisfy the \textit{second condition} if
it satisfies Conditions~(\ref{i8226a}) and~(\ref{i8226b}) of Definition~\ref{d8225}.
\end{Def}

Since the \textit{second condition} is obviously weaker than the \textit{first condition},
one has $M_{ij} \leq M'_{ij}$.\linebreak[1]\strut
\end{proof}

There is a still weaker but even simpler formula:
\begin{Cor}\label{cor3632}
Defining
\[ \begin{array}{rrcl} \boldepsilon \colon & L^2(I) & \to & L^2(I) \\
& (a_i)_{i\in I} & \mapsto & \big( \sum_{j\neq i} \epsilon_{ij}a_j \big)_{i\in I} ,\end{array} \]
the spectral gap of the Glauber dynamics is at least
\[\label{fo6385} \big( 1 - \VERT \boldepsilon \VERT \big)_+^2 .\]
\end{Cor}
\begin{proof}
One has $M' = (\mrm{I}-\boldepsilon)^{-1}$, so, provided $\VERT A\VERT < 1$,
\[ \VERT M' \VERT = \big\VERT (\mrm{I}-A)^{-1} \big\VERT = \Big\VERT \sum_{k=0}^\infty A^k \Big\VERT
\leq \sum_{k=0}^\infty \VERT A \VERT^k = (1-\VERT A\VERT)^{-1} .\]
In the case $\VERT A\VERT \geq 1$, (\ref{fo6385}) is trivial.
\end{proof}

\subsection{Avoiding the articial phase transition}

A common situation in which we would like to apply the previous results is when $I = \ZZ^n$
and $\epsilon_{ij}$ is of the form $\epsilon(j-i)$
for some symmetric function~$\epsilon \colon \ZZ^n \to [0,1]$.
Then Corollary~\ref{cor3632} tells
that the Glauber dynamics has a (strictly) positive spectral gap as soon as
$\sum_{z\neq 0} \epsilon(z) < 1$.
But like in \S~\ref{parAvoidingThePhaseTransitionForTheToyModel},
we are going to prove that that bound is somehow `artificial'
and that it can be relaxed into the neater condition
``$\sum_{z\neq 0} \epsilon(z) < \infty$'':
\begin{Thm}\label{thm7431}
Suppose that $I = \ZZ^n$ and that for all~$i,j\in \ZZ^n$
one has ${\{X_i:X_j\}_\ast} \ab \leq \epsilon(j-i)$
for some symmetric function~$\epsilon \colon \ZZ^n \to [0,1]$
such that $\epsilon(z) < 1$ as soon as $z \neq 0$.
Then if $\sum_{z\in\ZZ^n} \epsilon(z) < \infty$,
the spectral gap of the Glauber dynamics is positive.
\end{Thm}

\begin{proof}
The assumption on~$\sum_z\epsilon(z)$ allows us to take $\ell < \infty$ large enough so that
\[ \sum_{z\in\ell\ZZ^n\setminus\{0\}} \epsilon(z) < 1 .\]
We split $\ZZ^n$ into a partition of~$\ell^n \eqqcolon N$ sublattices $Z_1, \ldots, Z_N$,
each lattice $Z_u$ being of the form $\ell \ZZ^n + z_u$ for some $z_u\in\ZZ^n/\ell\ZZ^n$.
Then we define an auxiliary dynamics:
\begin{Def}\label{def4108}
The \emph{sublattice Glauber dynamics} is the Glauber dynamics for~$\vec{X}_{\ZZ^n}$
\emph{considered as the finite-dimensional vector
$(\vec{X}_{Z_1},\ldots,\vec{X}_{Z_N})$}.
In other words, on each~$u\in\{1,\ldots,N\}$
there is an independent $\mathit{Poisson}(1)$ alarm clock,
and when clock $u$ rings, the state of the whole $\vec{X}_{Z_u}$ is flipped in one shot
according to~$\Pr\big(X_{Z_u}\big|\vec{X}_{\ZZ^n\setminus Z_u}\big)$.
\end{Def}

Now let~$f \ab {\in \ldb(\Omega)}$.
In addition to the notation of the proof of Theorem~\ref{t7230},
we introduce the following definition:
\begin{Def}
For~$u \in \{1,\ldots,N\}$, we define
\[ f_{(u)}^{\neq} \coloneqq f - \EE\big[ f \big| \vec{X}_{\ZZ^n \setminus Z_u} \big] .\]
\end{Def}
\begin{Rmk}\label{rmk8840}
The~$f_{(u)}^{\neq}$ are the equivalent of the~$f_i^{\neq}$ for the sublattice Glauber dynamics.
\end{Rmk}

Fixing some `boundary condition' $\vec{x}_{\ZZ^n\setminus Z_u}$ on~$\ZZ^n\setminus Z_u$,
we can apply Corollary~\ref{cor3632} to the Glauber dynamics
for~$\vec{X}_{Z_u}$ under the law
$\Pr\big[\Bcdot\big|\vec{X}_{\ZZ^n\setminus Z_u}=\vec{x}_{\ZZ^n\setminus Z_u}\big]$.
After integrating, one gets that
\[\label{fo4048} \Var\big( f_{(u)}^{\neq} \big) \leq
\big(1-\VERT\boldsymbol\zeta\VERT\big)^{-2} \sum_{i\in Z_u} \Var\big( f_i^{\neq} \big) ,\]
where $\boldsymbol\zeta$ is the operator on~$L^2(\ell\ZZ^n)$ defined by
\[ \big(\boldsymbol\zeta g\big)(i)
= \sum_{z\in\ell\ZZ^n\setminus\{0\}} \epsilon(z) g(i+z) ,\]
whose norm is obviously bounded by~$\sum_{z\in\ell\ZZ^n\setminus\{0\}} \eqqcolon \zeta < 1$.
Then, summing~(\ref{fo4048}) for all~$u$:
\[\label{fo8710a} \sum_{u=1}^N \Var\big( f_{(u)}^{\neq} \big) \leq
\big(1-\zeta\big)^{-2} \mcal{E}(f,f) .\]

Now, let us apply Theorem~\ref{t7230} to the sublattice Glauber dynamics [Definition~\ref{def4108}].
It yields that
\[\label{fo8710b} \Var(f) \leq \VERT M\VERT^2 \sum_{u=1}^N \Var\big( f_{(u)}^{\neq} \big) ,\]
where $M$ is some $(N\times N)$ matrix depending on the
${\{ \vec{X}_{Z_u} : \vec{X}_{Z_v} \}_*}$.
But by Theorem~\ref{lem4067},
${\{ \vec{X}_{Z_u} : \vec{X}_{Z_v} \}_*} < 1$ for all~$u\neq v$,
thus $\VERT M\VERT < \infty$ by Remark~\ref{rmk4123}.
Combining~(\ref{fo8710a}) and~(\ref{fo8710b}), we finally get that
the spectral gap of the Glauber dynamics
for~$\vec{X}_{\ZZ^n}$ is bounded below by~$\VERT M\VERT^{-2} \times (1-\zeta)^2 > 0$.\linebreak[1]\strut
\end{proof}

\begin{Rmk}
Like Theorem~\ref{lem4067}, Theorem~\ref{thm7431} could actually be stated
in the general case of `abstract' metric spaces on which some group acts profinitely.
\end{Rmk}

\chapter{Concrete examples}\label{parApplications}%
\addcontentsline{itc}{chapter}{\protect\numberline{\thechapter}Concrete examples}

It is now time to see what the results of Chapters~\ref{parTensorization}
and~\ref{parOtherApplicationsOfTensorizationTechniques}
yield for concrete models of statistical physics.
I will try to give rather different types of examples,
so as to illustrate the advantages of working with Hilbertian correlations:
this frame is indeed quite general, as it requires little structure on the models considered.

In \S~\ref{parBackToIsingsModel} we will look back at Ising's model,
seeing how tensorization of Hilbertian decorrelations
improves the results of \S~\ref{parSomeResultsOnIsingsModel},
and what other results are given by the theorems of
\S~\ref{parOtherApplicationsOfTensorizationTechniques}.
We will also consider two kinds of generalizations, namely
when the range of interactions becomes infinite
and when the strength of the interactions is random (\emph{spin glasses}).
In the two next sections we will look at models with continuous states spaces:
first a quite general class of linear models [\S~\ref{parQuadraticModels}],
then a family of nonlinear models [\S~\ref{parNonlinearLatticeOfParticles}].
Finally in \S~\ref{parAHypocoerciveSystemOfInteractingParticles} we will see
how one can consider time as a supplementary dimension of the system
to get contractivity results for non-reversible Markov chains
(\emph{hypocoercivity}) on an infinite system of particles.

\begin{NOTA}
In this chapter, all the probability systems considered
will be endowed with their natural $\sigma$-metalgebras,
cf.\ Definition~\ref{def3128}.
To alleviate notation, I will give no names to these $\sigma$-metalgebras,
but will plainly write ``${\{X:Y\}_*}$'' to mean
``the subjective decorrelation between~$X$ and~$Y$
seen from the natural $\sigma$-metalgebra of the underlying system''.
\end{NOTA}

\section{Back to Ising's model}\label{parBackToIsingsModel}

\subsection{Standard Ising's model}

In all this section, we work on the lattice $\ZZ^n$
equipped with its natural distance $\dist$;
accordingly $|\Bcdot|$ will denote the $\ell^1$ norm on~$\RR^n$.
Recall the definition of Ising's model and the related notation that we introduced
in \S~\ref{parSomeResultsOnIsingsModel}, and Theorem~\ref{t4505}
on the existence of a completely analytical regime.

The following theroem states that Ising's model in completely analytical regime
is $\rho$-mixing, i.e.\ that two distant bunches of spins
are little correlated in the sense of maximal correlation:
\begin{Thm}\label{t3730}
For Ising's model on~$\ZZ^n$ in the completely analytical regime,
\begin{ienumerate}
\item\label{i3743a} There exists some $\psi'>0$ (the same as in Theorem~\ref{t4505})
such that for all disjoint $I,J\subset\ZZ^n$, one has when~$\dist(I,J)\longto\infty$ that
\[\label{f4615}
\{ \vec\omega_I : \vec\omega_J \} \leq \exp\big[-\big(\psi'+o(1)\big) \, \dist(I,J)\big] ,\]
where the ``$o(1)$'' can be
easily computed as an explicit function of~$\dist(I,J)$, $n$, $T$, $\psi'$ and the
$C'$ appearing in Theorem~\ref{t4505}.
\item\label{i3743b} There exists some $k<1$ such that for all disjoint $I,J\subset\ZZ^n$,
\[ \{ \vec\omega_I : \vec\omega_J \} \leq k .\]
\item\label{i5202}
Points~(\ref{i3743a}) and~(\ref{i3743b}) remain valid uniformly under any law of the form
$\Pr[\Bcdot|{\vec\omega_K=\vechat{\omega}_K}]$,
for~$K \subset \ZZ^n$ and~$\vechat{\omega}_K \in \{\pm1\}^K$ a `boundary condition' on~$K$.
\end{ienumerate}
\end{Thm}

\begin{Rmk}
Let us compare Theorem~\ref{t3730} with Theorem~\ref{t4505}.
Both theorems state decorrelation between distant bunches of spins
above temperature $T_{\mrm{c}}'$; the difference relies in using $\rho$-mixing
rather than $\beta$-mixing to quantify dependence between the bunches
in Theorem~\ref{t3730}.

Both results give an exponential decay of correlations,
with the same exponential constant $\psi'$,
but Theorem~\ref{t3730} is more powerful
in the sense that the bound~(\ref{f4615}) is uniform in the size of~$I$ and~$J$
while (\ref{f7869}) was not.
Moreover, thanks to Point~(\ref{i3743b}) we get a non-trivial result
for any choice of disjoint~$I$ and~$J$, which was not the case beforehand.
Recall that the drawbacks of Theorem~\ref{t4505} were inherent to $\beta$-mixing,
as Theorem~\ref{c3793} shew.

Both result remain valid under conditioning.
However, if one takes a \emph{random} boundary condition%
—that is, if one works under the law $\Pr[\Bcdot|{\vec\omega_K\in C}]$
for some non-singleton $C \subset \{\pm1\}^K$\hbox{ —,}
then Point~(\ref{i5202}) of Theorem~\ref{t3730} fails (cf.\ Remark~\ref{r4670}),
while (\ref{f7869}) is still valid by convexity of the total variation norm.
\end{Rmk}

\begin{Rmk}
Let us compare Theorem~\ref{t3730} with Theorem~\ref{c6177}.
The result of Theorem~\ref{c6177} can be rewritten:
\[ \big\{ \vec\omega_{\{0\}\times\ZZ^{n-1}} : \vec\omega_{\{x\}\times\ZZ^{n-1}} \big\}
\leq e^{-\psi x} . \]
Theorem~\ref{t3730} can be seen as a generalization of that result
to the case where $I$ and~$J$ have arbitrary shapes.%
\footnote{Note that in the case $I$ and~$J$ are hyperplanes,
we shew on page~\pageref{parHyperplanesInIsingsModel}
that (\ref{f4615}) could be improved into
\[ \big\{ \vec\omega_{\{0\}\times\ZZ^{n-1}} : \vec\omega_{\{x\}\times\ZZ^{n-1}} \big\}
\leq e^{-\psi' x} .\]}
Moreover, Point~(\ref{i5202}) also gives the existence of a conditional version,
which we did not have before.

There is however a price to pay for this greater generality,
since we had to require complete analyticity rather than just weak mixing,
which can be really more restrictive in some cases
(cf.\ Footnote~\ref{Tc'=Tc?} on page~\pageref{Tc'=Tc?}).
\end{Rmk}

\begin{Rmk}
Continuing the previous remark,
a natural open question is whether one can tensorize maximal decorrelation
under assumptions of weak mixing type.
In the case of Ising's model at least, I expect $\rho$-mixing to remain true%
—even for arbitrary shapes—as soon as $T>T_{\mrm{c}}$,
because on the one hand Theorem~\ref{c6177}
proves $\rho$-mixing between parallel hyperplanes,
while on the other hand $\rho$-mixing
seems to hold also in the `opposite extreme case'
when $I$ and~$J$ make a check pattern.

By the way, it is likely that the natural condition
should not be weak mixing itself but rather something like \emph{strong mixing for cubes}
(often called merely \emph{strong mixing}%
\footnote{Strong mixing \emph{stricto sensu} is actually
the same as complete analyticity, so that mathematicians have got used to undermeaning
``for cubes''—but strong mixing for cubes is strictly weaker than complete analyticity!%
~\cite[\S~2]{MO}.},
which means that when a boundary condition is fixed outside a cube of arbitrary edge,
changing one spin on the boundary has an effect in total variation
which decreases exponentially with the distance to the spin changed.
In fact it has been proved~\cite{MOS} that in dimension~2,
weak mixing is equivalent to strong mixing.
\end{Rmk}

\begingroup\def\proofname{Proof of Theorem~\ref{t3730}}\begin{proof}
Theorem~\ref{t3730} will be a direct consequence of the work of Chapter~\ref{parTensorization}
as soon as we show that, denoting by~$*$ the natural $\sigma$-metalgebra of the system
(i.e.\ the $\sigma$-metalgebra generated by the~$\omega_i$), for all distinct $i,j\in\ZZ^n$,
one has
\[\label{f7295} \{ \omega_i : \omega_j \}_* \leq c_0C' e^{-\psi' \dist(i,j)} \wedge k_0 \]
for some explicit $c_0 < \infty$ and~$k_0 < 1$ only depending on~$n$ and~$T$.
Then indeed, Proposition~\ref{pro8882} yields
\begin{multline}
\{ \vec\omega_I : \vec\omega_J \} \leq
\sum_{\substack{\delta\in\ZZ^n\\|\delta|\geq\dist(I,J)}}
c_0C' e^{-\psi'|\delta|}
= c_0C' \sum_{d=\dist(I,J)}^\infty
\#\{\delta\in\ZZ^n\colon |\delta|=d\} e^{-\psi'd} \\
\stackrel{\mathsmaller{\dist(I,J)\longto\infty}}{\sim}
c_0C' \sum_{d=\dist(I,J)}^\infty \frac{2^nd^{n-1}}{(n-1)!} e^{-\psi'd}
\sim \frac{c_0C'2^n}{(n-1)!} \dist(I,J)^{n-1} e^{-\psi'\dist(I,J)} \\
= e^{-(\psi'+o(1))\dist(I,J)},
\end{multline}
whence Point~(\ref{i3743a}). Moreover, since
\[ \sum_{\substack{\delta\in\ZZ^n\setminus\{0\}
\\|\delta|\leq\dist(I,J)}} c_0C' e^{-\psi' \dist(i,j)}
< \infty ,\]
Point~(\ref{i3743b}) follows from Lemma~\ref{lem4067},
and finally (\ref{i5202}) is a consequence of \S~\ref{parSubjectiveResults}
about subjective results.

So, we have to prove~(\ref{f7295}).
Let~$\vechat{\omega}_K \in \{\pm1\}^K$, $K\subset\ZZ^n$, be some arbitrary boundary condition,
and denote by~$\Pr_{\mathrm{con}}$ the associated law,
that is, $\Pr_{\mathrm{con}} = \Pr[\Bcdot|{\vec\omega_K = \vechat{\omega}_K}]$;
our goal is to show that under~$\Pr_{\mathrm{con}}$, for all distinct $i,j\in\ZZ^n$, one has
$\{ \omega_i : \omega_j \} \leq c_0C' e^{-\psi' \dist(i,j)} \wedge k_0$.

The result is immediate if $i\in K$, resp.\ $j\in K$
(since then $\omega_i$, resp.~$\omega_j$, is constant and thus independent of everything),
so we assume $i,j\notin K$.
We begin with observing that if $K$ is the set~$N(i)$ of all the neighbours of~$i$,
equilibrium at~$i$ implies that, whatever the boundary condition may be:
\[\label{f7537} \Pr_{\mathrm{con}}[\omega_i = -1], \Pr_{\mathrm{con}}[\omega_i = +1]
\geq \big( e^{4n/T} + 1 \big)^{-1} \]
—the extremal cases being when $\vechat{\omega}_{N(i)} \equiv +1$,
resp.\ $\vechat{\omega}_{N(i)} \equiv -1$.
Now in the general case $K \subset \ZZ^n\setminus\{i\}$,
$\Law_{\mathrm{con}}[\omega_i]$ is an average of laws of the form
$\Law(\omega_i|{\vec\omega_{N(i)}=\vechat{\omega}_{N(i)}})$,
so that (\ref{f7537}) remains valid.
Similarly, equilibrium on~$\{i,j\}$ gives that
for all~$a,b\in\{\pm1\}$,
\[ \label{f7538} \Pr_{\mathrm{con}}[\omega_i = a \text{ and } \omega_j = b]
\geq \big( e^{8n/T} + 2e^{(4n+2)/T} + 1 \big)^{-1} .\]

Now, recall that the correlation level between two two-ranged variables
can be computed by Formula~(\ref{f0699}), where
$|p_a^b-p_ap^b|$ is also $\beta(X,Y)/2$.
Thus the bound ``$\{\omega_i:\omega_j\} \leq C_0 e^{-\psi' \dist(i,j)}$''
is a direct consequence of Theorem~\ref{t4505}, with
\[ c_0 = \frac{1/2}{(e^{4n/T}+1)^{-1}\big(1-(e^{4n/T} + 1)^{-1}\big)}
= \tanh(4n/T) + 1 .\]

It remains to prove the bound ``$\{\omega_i:\omega_j\} \leq k_0$''.
We will use the following corollary of~(\ref{f0699}):
\begin{Lem}\label{l1373}
With the notation of Remark~\ref{rmk9549},
there exists $a,b$ in the respective ranges of~$X,Y$ such that
\[ \{X:Y\} \leq 1 - 4p_a^b .\]
\end{Lem}

\begingroup\def\proofname{Proof of Lemma~\ref{l1373}}%
\begin{proof}
The difference~$p_a^b-p_ap^b$ gets its sign changed whenever $a$, resp.~$b$, changes,
so there are some~$a$ and~$b$ for which this value is nonpositive;
moreover, denoting by~$\{a,a'\}$ and~$\{b,b'\}$ the respective ranges of~$X$ and~$Y$,
$p_{a'}^{b'}-p_{a'}p^{b'}$ is also nonpositive.
Now one has
\[ \frac{p_ap^b}{\sqrt{p_ap_{a'}p^bp^{b'}}} \times \frac{p_{a'}p^{b'}}{\sqrt{p_ap_{a'}p^bp^{b'}}}
= 1 ,\]
so that either $p_ap^b$ or $p_{a'}p^{b'}$ is $\leq \sqrt{p_ap_{a'}p^bp^{b'}}$.
Up to changing notation we can assume that it is~$p_ap^b$,
and then
\[ \{X:Y\} = \frac{|p_a^b-p_ap^b|}{\sqrt{p_ap_{a'}p^bp^{b'}}}
= \frac{p_ap^b - p_a^b}{\sqrt{p_ap_{a'}p^bp^{b'}}}
\leq 1 - \frac{p_a^b}{\sqrt{p_ap_{a'}p^bp^{b'}}} \leq 1 - 4p_a^b .\]
\end{proof}\endgroup

Combining Lemma~\ref{l1373} with~(\ref{f7538}),
we then get the desired bound, with
\[ k_0 = 1 - 4\big( e^{8n/T} + 2e^{(4n+2)/T} + 1 \big)^{-1} < 1 .\]
\end{proof}\endgroup

Formula~(\ref{f7295}) is also what we need to apply the results of Chapter~%
\ref{parOtherApplicationsOfTensorizationTechniques}. Indeed, denoting
$\epsilon(z) \coloneqq \{ X_i : X_{i+z} \}_*$, it gives that $\sum_{z\in\ZZ^n} \epsilon(z) < \infty$
with $\epsilon(z) < 1$ as soon as $z \neq 0$,
so that Theorems~\ref{thm3015} and~\ref{thm7431} yield respectively:
\begin{Thm}
In completely analytical regime, the spins Ising's model satisfies
the central limit theorem, in the sense that the conclusions of Theorem~\ref{thm3015}
hold for them.
\end{Thm}
\begin{Thm}
In completely analytical regime, the Glauber dynamics for Ising's model
has a (strictly) positive spectral gap,
and this remains valid uniformly if one fixes
a `boundary condition' on the spins of some $K \subset \ZZ^n$.
\end{Thm}

\begin{Rmk}
As I told in Chapter~\ref{parOtherApplicationsOfTensorizationTechniques},
results of these kinds have already been studied by other methods
(see e.g.~\cite{Bradley92, Dedecker} for the CLT and~\cite{Martinelli} for the spectral gap).
For the standard Ising model in completely analytical regime,
which is ``very nice'', these previous works apply well,
so the two theorems above are not new.
They are interesting however because of the new method used to prove them,
which is quite direct and likely to apply to a broader class of models.
Such models will be presented in the sequel of this chapter.
\end{Rmk}

\subsection{Generalizations of Ising's model}

The previous results can be adapted to several kinds of generalizations of Ising's model.
Let us expose some of them.

\subsubsection*{Long-range Ising models}

A physically important case is the \emph{long-range Ising models} on~$\ZZ^n$.
In these models, the states space is unchanged, but the Hamiltonian $H$ becomes
\[ H(\vec\omega) = - \frac{1}{2} \sum_{i\neq j} J(j-i) \omega_i\omega_j ,\]
where $J \colon \ZZ^n\setminus\{0\} \to \RR$ is some symmetric function with non-compact support
such that $J(z) \stackrel{|z|\longto\infty}{=} O(|z|^{-(n+\alpha)})$ for some $\alpha>0$.

Let us state a decorrelation result for this class of models.
The frame of the proof of the following proposition will work as well
for the other generalizations of Ising's model.
\begin{Pro}\label{pro5088}
There exists an temperature $T_1 < \infty$ such that, provided $T\geq T_1$:
\begin{ienumerate}
\item\label{i0797} Equilibrium for the long-range Ising model is unique;
\item\label{i96a} Uniformly in~$i,j$,
$\{\omega_i:\omega_j\}_* \stackrel{|j-i|\longto\infty}{=} O(|j-i|^{-(n-\alpha)})$;
\item\label{i96b} There exists some $k_0 < 1$ such that for all~$i\neq j$,
$\{\omega_i:\omega_j\}_* \leq k_0$.
\end{ienumerate}
\end{Pro}

\begin{proof}
The principle of the proof consists in coupling two Glauber dynamics
with different initial conditions.
Recall that the Glauber dynamics is defined as follows:
each spin has an independent clock ringing with rate $1$,
and when the clock of a spin rings, this spin is flipped so that its final state is drawn
according to its equilibrium measure conditionnally to the state of all other spins.
Namely, if the clock of spin $i$ rings at time~$t$,
denoting as usual $\beta = T^{-1}$,
\[ \Pr[\omega_i(t+)=+1] = \frac{\exp\big(\beta\sum_{j\neq i}J(j-i)\omega_j(t)\big)}
{2\cosh\big(\beta\sum_{j\neq i}J(j-i)\omega_j(t)\big)} \]
and $\Pr[\omega_i(t+)=-1] = 1-\Pr[\omega_i(t+)=+1]$.

To couple the Glauber dynamics, we will assume that, rather than just ``ringing'' at time~$t$,
the clock of~$i$ is a Poisson process on~$\RR_+ \times (0,1)$,
points of which are denoted by~$(t,y)$.
Then, if at time~$t$ the clock of spin $i$ has a point $(t,y)$,
spin $i$ flips to~$+1$ if $y<\Pr[\omega_i(t+)=+1]$,
resp.\ to~$-1$ if~$y\geq\Pr[\omega_i(t+)=+1]$.

Now, consider two Glauber dynamics~$\vec\omega^-$ and~$\vec\omega^+$
having the same Poisson process, but starting with different initial conditions.
It will be convenient%
\footnote{In the cases where interactions can be antiferromagnetic ($J<0$),
monotonicity does not stand any more;
the proof however remains valid with a heavier formalism,
replacing ``$>$'' by ``$\neq$'' and putting absolute values at the right places.}
to assume that $\vec\omega^-(t=0) \leq \vec\omega^+(t=0)$ almost-surely:
then, as we will see, for the coupled dynamics
one has (a.s.) $\vec\omega^-(t) \leq \vec\omega^+(t) \ \forall t$.
At time $t$, denote by~$\Theta(t)$ the set of points where~$\vec\omega^-$ and~$\vec\omega^+$ differ:
\[ \Theta(t) = \big\{ i\in\ZZ^n \colon \big(\omega^-_i(t), \omega^+_i(t)\big) = (-1,+1) \} .\]
When the clock at spin $i$ rings at time $t$, three cases have to be distinguished:
\begin{enumerate}
\item If $y < \exp\big(\beta\sum_{j\neq i}J(j-i)\omega^-_j(t)\big) \mathbin{\big/}
2\cosh\big(\beta\sum_{j\neq i}J(j-i)\omega^-_j(t)\big)$,
then both $\omega^+_i$ and~$\omega^-_i$ flip into state~$+1$;
\item If $y > \exp\big(\beta\sum_{j\neq i}J(j-i)\omega^+_j(t)\big) \mathbin{\big/}
2\cosh\big(\beta\sum_{j\neq i}J(j-i)\omega^+_j(t)\big)$,
then both $\omega^+_i$ and~$\omega^-_i$ flip into state~$-1$;
\item If $\frac{\exp\big(\beta\sum_{j\neq i}J(j-i)\omega^-_j(t)\big)}
{2\cosh\big(\beta\sum_{j\neq i}J(j-i)\omega^-_j(t)\big)} < y
< \frac{\exp\big(\beta\sum_{j\neq i}J(j-i)\omega^+_j(t)\big)}
{2\cosh\big(\beta\sum_{j\neq i}J(j-i)\omega^+_j(t)\big)}$,
then $\omega^+_i$ flips into state $+1$ while $\omega^-_i$ flips into state $-1$.
\end{enumerate}
Denoting
\[ \mcal{J} \coloneqq \sum_{z\in\ZZ^n\setminus\{0\}} J(z) ,\]
which is always finite by the assumption on~$J$,
the probability of each of the two first cases is bounded below by
$e^{-\beta\mcal{J}} / 2\cosh(\beta\mcal{J})$.
The probability of the third case is
\[ \frac{\sinh\big(2\beta\sum_{j\in\Theta(t)\setminus\{i\}} J(j-i)\big)}
{2\cosh\big(\beta\sum_{j\neq i}J(j-i)\omega^-_j(t)\big)
\cosh\big(\beta\sum_{j\neq i}J(j-i)\omega^+_j(t)\big)} ,\]
which is bounded above by~$\beta \sum_{j\in\Theta(t)\setminus\{i\}} {J(j-i)}$
thanks to the following computational
\begin{Lem}
For~$a \leq b$ two real numbers,
\[ \sinh(b-a) \leq (b-a) \cosh a \cosh b .\]
\end{Lem}
\begin{proof}
Making the change of variables $x = (a+b)/2, t = (b-a)/2$,
we have to prove that for~${x \in \RR}, \ab t \geq 0$, one has:
\[\label{fo4994} \sinh(2t) \leq 2 t \cosh(x-t)\cosh(x+t) .\]
If we consider the right-hand side of~(\ref{fo4994})
as a function of~$x$, it is symmetric (since $\cosh$ is symmetric)
and its logarithm is convex (since $\ln\circ\cosh$ is convex,
its derivative being the increasing function~$\tanh$),
so its minimum is attained for $x=0$;
thus it suffices to prove~(\ref{fo4994}) in that case,
i.e.\ to prove that $\sinh(2t) \leq 2t\cosh^2t$ for all~$t\geq 0$.
But $\sinh(2t) = 2\sinh t\cosh t$, so we can simplify both sides by~$2\cosh t$,
and then it suffices to prove that
$\sinh t \leq t \cosh t$, which is true since
$\tanh t \leq t$ for all~$t\geq 0$.
\end{proof}

Thanks to these estimates, we can define a process Markovian $\Theta^*(t)$ on~$\mathfrak{P}(\ZZ^n)$
such that almost-surely, $\Theta^*(t) \supset \Theta(t)\ \forall t$.
This process has the following law:
\begin{Def}
The law of~$\Theta^*$ is defined thanks to independent Poissonian clocks
indexed by~$(\ZZ^n)^2$.
For $i\neq j$ the clock $(i,j)$ has rate $\beta J(j-i)$,
while the clock $(i,i)$ has rate $e^{-\beta\mcal{J}} / \cosh(\beta\mcal{J})$.
At $t=0$ one has $\Theta^*(0) = \Theta(0)$.
If at time $t$ the clock $(i,j)$ rings, with $j\neq i$, then:
\begin{itemize}
\item Either $i\in\Theta^*(t-)$ and then $\Theta^*$ changes so that
$\Theta^*(t+) = \Theta^*(t-) \cup \{j\}$%
\footnote{Of course, if $j\in\Theta^*(t-)$ then $\Theta^*$ does actually not change.};
\item Or $i\notin\Theta^*(t-)$ and then $\Theta^*$ does not change.
\end{itemize}
On the other hand, if at time $t$ the clock $(i,i)$ rings,
then $\Theta^*$ changes so that $\Theta^*(t+) = \Theta^*(t-) \setminus \{i\}$.
\end{Def}

Let~$\lambda \coloneqq \beta\mcal{J} - \big(e^{-\beta\mcal{J}}/\cosh(\beta\mcal{J})\big)$.
If we take $\EE\big[\#\Theta(t=0)\big] < \infty$%
\footnote{The general case where $\Theta^*$ can be infinite can be got
from the finite case by passing to the limit,
despite some technicalities of little interest.},
it is immediate that $\#\Theta^*(t)/e^{\lambda t}$ is a supermartingale.
So, provided $T$ is large enough so that $\lambda < 0$, i.e.
\[\label{for5229} \beta\mcal{J} < \frac{e^{-\beta\mcal{J}}}{\cosh(\beta\mcal{J})} ,\]
the two processes~$\vec{\omega}^-(t)$ and~$\vec{\omega}^+(t)$
tend to be equal when~$t\longto\infty$; in particular they have the same equilibrium.
That proves Point~(\ref{i0797}) of the Lemma,
since any initial condition stands between the `extreme' conditions
$\vec\omega^-(t=0) \equiv -1$ and $\vec\omega^+(t=0) \equiv +1$.

Observe that the previous reasoning remains entirely valid if one reasons conditionally
to some boundary condition of the form ``$\vec\omega_K = \vechat{\omega}_K$'',
with the same condition on~$T$.

Now we are turning to the correlation between two distant spins.
Let~$i\in\ZZ^n$ and let~$\vechat{\omega}_K$ be some boundary condition on
some $K \subset \ZZ^n\setminus\{i\}$.
Suppose $T$ satisfies~(\ref{for5229});
I want to compare the Glauber dynamics corresponding to the boundary condition
``$\vec\omega_{K\uplus\{i\}} = \big(\vechat{\omega}_K, (+1)^{\{i\}}\big)$''%
—where $(\vechat{\omega}_K, (+1)^{\{i\}})$ stands for the function on
$K\uplus\{i\}$ which is equal to~$\hat\omega$ on~$K$ and to~$+1$ at~$i$ —%
with the Glauber dynamics corresponding to the boundary condition
``$\vec\omega_{K\uplus\{i\}} = \big(\vechat{\omega}_K, (-1)^{\{i\}}\big)$''.
In this frame, one defines the process $\Theta^*$ as previously,
except that one imposes that $\Theta^*(t) \cap K = 0$ and $i \in \Theta^*(t)$ for all~$t$.
This time, it is the equilibrium behaviour of~$\Theta^*$ which interests us.
Denote by~$\Pr_{\text{eq}}$ the equilibrium law of~$\Theta^*$;
for~$j'\in\ZZ^n\setminus K$, denote $\theta(j) \coloneqq \Pr_{\text{eq}}[j\in \Theta^*]$.
Then $\theta$ satisfies the following discrete subelliptic equation
with Dirichlet boundary conditions:
\[\label{fo7960} \left\{ \begin{array}{ll}
\forall j\notin K\uplus\{i\}
\quad \frac{e^{-\beta\mcal{J}}}{\cosh(\beta\mcal{J})} \theta(j) \leq
\beta \sum_{\substack{i'\in\ZZ^n\setminus K \\ i'\neq j}} J(i'-j) \theta(i') ;\\
\forall k\in K \quad \theta(k) = 0; \qquad \theta(i) = 1.
\end{array} \right. \]

Define the convolution kernel $a$ on~$\ZZ^n$ by
\[ \left\{ \begin{array}{ll} a(0) = 1 ; \\
\forall z\neq 0 \quad a(z) = - \frac{\cosh(\beta\mcal{J})}{e^{-\beta\mcal{J}}}\beta J(z) ,
\end{array}\right. \]
so that (\ref{fo7960}) writes in the bulk:
\[ a * \theta \leq 0 .\]
Writing $a \eqqcolon \delta_0 - \tilde{a}$,
Condition~(\ref{for5229}) ensures that $\|\tilde{a}\|_{\ell^1} < 1$.
Since $\ell^1(\ZZ^n)$ is a Banach algebra for the convolution operator $*$,
with neutral element $\delta_{0}$, it follows that $a$ is invertible with inverse
\[\label{fo7178} a^{-*} =
\delta_{0} + \tilde{a} + \tilde{a} * \tilde{a} + \tilde{a} * \tilde{a} * \tilde{a} + \cdots .\]
Since $\tilde{a} \geq 0$, $a^{-*}$ is nonnegative everywhere with $a^{-*}(0)>0$.
Therefore the function~$F \coloneqq \big(a^{-*}(0)\big)^{-1} \delta_i * a^{-*}$
satisfies:
\[\label{fo7960*} \left\{ \begin{array}{ll}
\forall j\notin K\uplus\{i\} \quad \big(a \ast F\big)(j) = 0 ;\\
\forall k\in K \quad F(k) \geq 0; \qquad F(i) = 1.
\end{array}\right.\]
Comparing~(\ref{fo7960}) with~(\ref{fo7960*}),
since~(\ref{fo7960}) is subelliptic,
we can apply a maximum principle to it%
\footnote{The maximum principle is generally stated in a PDE context,
see for instance~\cite[\S~3.1]{elliptique}, but it works exactly the same for discrete equations.},
which yields that $\theta \leq F$ everywhere.
But $J(z) = O(|z|^{-(n+\alpha)})$,
so by Lemma~\ref{l1939b} in appendix, $F(j) = O(|j-i|^{-(n+\alpha)})$,
and therefore
\[\label{fo9696} \Pr\big[ \omega_j=+1 \big| \vec\omega_K = \vechat{\omega}_K , \omega_i = 1 \big]
- \Pr\big[ \omega_j=-1 \big| \vec\omega_K = \vechat{\omega}_K , \omega_i = -1 \big]
= O(|j-i|^{-(n+\alpha)}) ,\]
uniformly in~$i, j, K, \vechat{\omega}_K$.

The end of the proof, namely deducing Point~(\ref{i96a}) from~(\ref{fo9696})
and proving Point~(\ref{i96b}), is then performed in the same way as to establish~(\ref{f7295})
in the proof of Theorem~\ref{t3730}.
\end{proof}

Thanks to Proposition~\ref{pro5088}, we can apply the results of
Chapters~\ref{parTensorization} and~\ref{parOtherApplicationsOfTensorizationTechniques}.
One gets the following
\begin{Thm}\label{thm8777}
For the long-range Ising model on~$\ZZ^n$ at~$T\geq T_1$,
\begin{ienumerate}
\item\label{i204a} For all disjoint $I,J\subset\ZZ^n$, uniformly in~$I,J$, one has an estimate
\[\label{fo2580}
\{ \vec\omega_I : \vec\omega_J \} \leq O \big(\dist(I,J)^{-\alpha}\big) ,\]
where the~$O(\Bcdot)$ can be turned into an explicit constant only depending
on~$\mcal{J}$ and~$T$.
Moreover, there exists some $k<1$
(still explicit and only depending on~$\mcal{J}$ and~$T$)
such that for all disjoint $I,J\subset\ZZ^n$,
\[ \{ \vec\omega_I : \vec\omega_J \} \leq k .\]
\item\label{i204b} The spins satisfies
the central limit theorem, in the sense that the conclusions of Theorem~\ref{thm3015}
hold for them.
\item\label{i204c} The Glauber dynamics has a positive spectral gap.
\item\label{i204d}
Points~(\ref{i204a}) and~(\ref{i204c}) remain valid uniformly under any law of the form
$\Pr[\Bcdot|{\vec\omega_K=\vechat{\omega}_K}]$,
for~$K \subset \ZZ^n$ and~$\vechat{\omega}_K \in \{\pm1\}^K$ a `boundary condition' on~$K$.
\end{ienumerate}
\end{Thm}
\begin{proof}
The proof is the same as the work done in the previous subsection.
The only difference is to prove~(\ref{fo2580}),
which follows from the following computation:
denoting $D\coloneqq\dist(I,J)$, one has that, when~$D\longto\infty$,
\[ \sum_{\substack{z\in\ZZ^n\\|z|\geq D}} \frac{1}{|z|^{n+\alpha}}
\leq \sum_{d=D}^\infty
\frac{\#\{z\in\ZZ^n\colon |z|=d\}}{d^{n+\alpha}}
= \sum_{d=D}^\infty \frac{O(d^{n-1})}{|d^{n+\alpha}|}
= O\Big( \sum_{d=D}^\infty \frac{1}{d^{1+\alpha}} \Big) = O(D^{-\alpha}). \]
\end{proof}

\subsubsection*{Spin glasses}

\emph{Spin glasses} are another generalization of Ising's model.
In these models, the interaction constants are not invariant by translation any longer.
The Hamiltonian writes
\[ H(\vec\omega) = - \frac{1}{2} \sum_{i\neq j} J(i,j) \omega_i\omega_j \]
(with $J(j,i) = J(i,j)$), where the~$J(i,j)$ themselves are random.
We make the following assumptions on the interaction constants:
\begin{Hyp}\label{hyp79}
For distinct unordered pairs $\{i,j\}$, all the~$J(i,j)$ are independent.
Moreover, $J(i,j)$ is distributed according to some law $P_J^{(j-i)}$ only depending on~$(j-i)$%
\footnote{Observe that one has necessarily $P_J^{(-z)} = P_J^{(z)}$ for all~$z$;
in particular the function~${J_\infty \colon} \ab {\ZZ^n \setminus \{0\}} \to \RR_+$
shall always be symmetric.}.
We will assume that all the~$P_J^{(z)}$ have bounded support,
and we denote by~$J_\infty(z)$ the smallest number such that
$P_J^{(z)} \big[ |J| \leq J_\infty(z) \big] = 1$.
\end{Hyp}

\begin{Rmk}
Here the~$J(i,j)$ can be negative, which corresponds to antiferromagnetic interactions.
\end{Rmk}

\begin{NOTA}
In spin glass models, there are two levels of randomness: first to fix the~$J(i,j)$,
next to take $\vec\omega$ according to the Gibbs measure associated to~$H$.
When both levels of randomness are taken into consideration, one speaks of \emph{annealed} law.
Here I am only interested in the \emph{quenched} laws,
which deal with the second level of randomness for fixed~$J(i,j)$.
I will write sentences beginning with ``for almost-all quenched systems'',
which mean that what follows is valid for almost-all Gibbs measures
when the~$J(i,j)$ are taken randomly according to Assumption~\ref{hyp79}.
\end{NOTA}

The machinery exposed above still works for spin glass models.
We obtain the
\begin{Thm}
Suppose that when~$|z| \longto \infty$,
$J_\infty(z)$ decreases at least as fast as $O(|z|^{-(n+\alpha)})$ for some $\alpha>0$.
Then there is a $T_1 < \infty$ such that, for the spin glass model on~$\ZZ^n$ at~$T\geq T_1$,
for almost-all quenched systems,
\begin{ienumerate}
\item If $J_\infty(z) = O(|z|^{-(n+\alpha)})$,
then for all disjoint $I,J\subset\ZZ^n$, uniformly in~$I,J$, one has an estimate
\[\label{fo8727} \{ \vec\omega_I : \vec\omega_J \} \leq O \big(\dist(I,J)^{-\alpha}\big) .\]
If moreover $J_\infty(z)$ has exponential decay
(see Definition~\ref{def726} in the appendix),
then the right-hand side of~(\ref{fo8727}) can even be replaced by
``$\theta\big(\dist(I,J)\big)$'' for some function~$\theta(\Bcdot)$ with exponential decay.

In both cases, there exists some $k<1$ such that for all disjoint $I,J\subset\ZZ^n$,
\[ \{ \vec\omega_I : \vec\omega_J \} \leq k .\]
\item Points~(\ref{i204b})--(\ref{i204d}) of Theorem~\ref{thm8777} hold.
\end{ienumerate}
\end{Thm}

\subsection*{Synthetic vocabulary}

For all the models considered in this section,
the techniques used and the results stated walked along the same lines.
First, one establishes a bound $\{X_i:X_j\}_* \leq \epsilon(j-i) \wedge k_0$ for all~$i\neq j$,
for some sufficiently rapidly decreasing function
$\epsilon \colon \ZZ^n \to [0,1]$ and some $k_0 < 1$.
Then, one applies the results of
Chapters~\ref{parTensorization} and~\ref{parOtherApplicationsOfTensorizationTechniques},
which yield maximal decorrelation for distant bunches of spins
(which is sometimes called \emph{(interlaced) $\rho^*$-mixing})
with uniformly non-full correlation between any two disjoint bunches of spins
(which is sometimes denoted ``$\rho^*(1)<1$''),
central limit theorem, and spectral gap for the Glauber dynamics.

Since this method will be used again in the following sections,
it will be convenient to introduce some synthetic vocabulary:
\begin{Def}\label{def-WRM}
If a spin model (spins can have arbitrary range) $\vec{X}$ on~$\ZZ^n$ satisfies some bound
``${\{X_i:X_j\}_*} \leq \epsilon(j-i) \wedge k_0$'' for all distinct $i,j \in \ZZ^n$,
with $\sum_{z\in\ZZ^n\setminus\{0\}} \epsilon(z) < \infty$ and $k_0 < 1$,
we say that this model is \emph{well-$\rho$-mixing}.
According to our results, for such a model one has $\rho^*$-mixing with $\rho^*(1)<1$,
CLT and spectral gap.

Moreover,
\begin{ienumerate}
\item If $\epsilon(z) = O(|z|^{-(n+\alpha)})$ when~$|z|\longto\infty$,
then we say that the model is \emph{$\alpha$-polynomially $\rho$-mixing}.
According to our results, in this case $\rho^*$-mixing is polynomial with rate $\alpha$,
i.e.\ Formula~(\ref{fo2580}) holds.
\item If $\epsilon(z)$ has exponential decay (cf.\ Definition~\ref{def726}),
then we say that the model is \emph{exponentially $\rho$-mixing}.
According to our results, in this case $\rho^*$-mixing has an exponential speed of decay
(but not with the same rate as $\epsilon(\Bcdot)$, cf.\ Remark~\ref{rmk9890}),
i.e.\ a formula similar to~(\ref{f4615}) holds.
\end{ienumerate}
\end{Def}

\section{Quadratic models}\label{parQuadraticModels}

\begin{NOTA}
In this subsubsection, an arbitrary norm $|\Bcdot|$ on~$\ZZ^n$ is fixed.
\end{NOTA}

\begin{Def}
In our \emph{quadratic model}, the states space is $\Omega = \RR^{\ZZ^n}$
for some $n\in\NN^*$. For~$\vec\omega_{\ZZ^n} \in \Omega$, $i\in\ZZ^n$,
the real number $\omega_i$ will be called the \emph{polarization of particle~$i$}.
Each particle $i$ is submitted to two types of forces:
\begin{itemize}
\item A \emph{pinning force}, preventing the particle from having a too large polarization,
which derives from the quadratic potential $\omega_i^2/2$;
\item \emph{Interaction forces}: each particle $j\neq i$ exerts a force on~$i$
which tends to make the polarizations of particles~$i$ and~$j$ equal;
this force derives from a quadratic potential $\gamma_{j-i} (\omega_j-\omega_i)^2/2$.
\end{itemize}
In other words, the Hamiltonian of the system is formally defined by
\[ H(\vec\omega) = \frac{1}{2}\sum_{i\in\ZZ^n} \omega_i^2
+ \frac{1}{4}\sum_{i\neq j} \gamma_{j-i} (\omega_j-\omega_i)^2 ,\]
where the~$\gamma_z$, for~$z\in\ZZ^n\setminus\{0\}$,
are nonnegative numbers which we impose to satisfy the symmetry condition
$\gamma_z = \gamma_{-z}$ for all~$z$.
Moreover we impose the that the sum of the~$\gamma_z$ is convergent, and we denote
\[\label{f2380} \Gamma \coloneqq \sum_{z\in\ZZ^n\setminus\{0\}} \gamma_z < \infty .\]
\end{Def}

The Hamiltonian $H$ is a quadratic function of~$\vec\omega$,
so at fixed parameter $\beta$ the (infinite-dimensional) random vector~$\vec\omega$
will be Gaussian (and centered).
Let us compute its covariance:
the probability density of~$\vec\omega$ w.r.t.\ the `Lebesgue measure' on~$\Omega$
is formally defined by
\[ \frac{\dx\Pr_\beta [\vec\omega]}{\prod_{i\in\ZZ^n}\dx{\omega_i}} \propto
\exp \Big({\textstyle\frac{1}{2}} \T{\omega} (\beta Q) \omega \Big), \]
where $Q$ is the (infinite-dimensional) symmetric matrix defined by
\[\label{f1243} \left\{ \begin{array}{ccl} Q_{ij} \coloneqq -\gamma_{j-i} & \quad & \text{for $i\neq j$;} \\
Q_{ii} \coloneqq 1 + \Gamma & \quad & \text{on the diagonal,} \end{array} \right.\]
thus the covariance matrix of~$\vec\omega$ is $(\beta Q)^{-1}$.
So we have to compute $Q^{-1}$, the inverse matrix of~$Q$.
Since $Q$ is a Toeplitz matrix (with $n$-dimensional indexes)%
\footnote{Recall that saying that matrix $Q$ is \emph{Toeplitz}
means that its entries $Q_{ij}$ only depend on~$(j-i)$.},
$Q^{-1}$ —if it exists—will be of the same form.
Now, knowing that it is a Toeplitz matrix, $Q$ is described by the function
$a_Q \colon \ZZ^n \to \RR$ such that for all~$i,j$, $Q_{ij} = a_Q(j-i)$.
With this notation, (\ref{f1243}) rewrites:
\[ \forall z\in\ZZ^n \qquad a_Q(z) = \1{z=0} (1+\Gamma) - \1{z\neq0} \gamma_z .\]
When coded by functions like $a_Q$,
the multiplication of Toeplitz matrices becomes the convolution product:
\[ \forall M, N \ \text{Toeplitz} \qquad a_{MN} = a_M * a_N .\]
So, $Q^{-1}$ will be the Toeplitz matrix
whose $a_{Q^{-1}}$ is the inverse of~$a_Q$ for the convolution product.
Thanks to Condition~(\ref{f2380}), such an inverse always exists:
indeed we can write $a_Q = (1+\Gamma) (\delta_{0} - \tilde{a}_Q)$,
where $\tilde{a}_Q$ is a nonnegative function with
$\|\tilde{a}_Q\|_{\ell^1} = \Gamma \div (1+\Gamma) < 1$,
so that $a_Q$ is invertible with
\[\label{f0024} a_Q^{-*} = (1+\Gamma)^{-1} \big(\delta_{0} + \tilde{a}_Q +
\tilde{a}_Q * \tilde{a}_Q + \tilde{a}_Q * \tilde{a}_Q * \tilde{a}_Q + \cdots \big) .\]
In the end, at parameter $\beta>0$ the covariance matrix of~$\vec\omega$ has entries:
\[\label{f3366} \Cov(\omega_i,\omega_j) = \frac{a_{Q^{-1}}(j-i)}{\beta} .\]

\begin{Rmk}
All the entries of~$\Cov(\vec\omega)$ are nonnegative,
which reflects the fact that all the interaction forces are attractive.
\end{Rmk}

\begin{Rmk}
Since $\Cov(\vec\omega)$ depends on~$\beta$ only through a constant factor,
the behaviour of the system is exactly the same, up to a multiplicative constant,
for all~$\beta > 0$. Hence the study of correlations will not depend on~$\beta$.
\end{Rmk}

\begin{NOTA}
In the sequel, we fix arbitrarily $\beta = 1$ and we denote~$\Pr$ for~$\Pr_{\beta=1}$.
\end{NOTA}

Since the model is Gaussian, by~(\ref{f3366}) and Theorem~\ref{pro1857}
one has for all~$i \neq j$:
\[ \{ \omega_i : \omega_j \} = \frac{a_{Q^{-1}}(j-i)}{a_{Q^{-1}}(0)} .\]
Now we have the following claim, with an immediate key corollary:
\begin{Clm}\label{l3475}
For all~$i\neq j$, for all~$K \subset \ZZ\setminus\{i,j\}$,
\[ \{ \omega_i : \omega_j \}_{\vec\omega_K} \leq \{ \omega_i : \omega_j \} .\]
\end{Clm}

\begin{Cor}\label{c6834}
Denoting by~$*$ the natural $\sigma$-metalgebra of the system, for all~$i\neq j$,
\[\label{fc6834} \{ \omega_i : \omega_j \}_* = \{ \omega_i : \omega_j \}
= \frac{a_{Q^{-1}}(j-i)}{a_{Q^{-1}}(0)} .\]
\end{Cor}

\begin{proof}
The proof of Claim~\ref{l3475} relies on the following claims:
\begin{Clm}\label{c4577a}
Up to an additive constant, $\Law(\vec\omega_{\ZZ^n}|\vec\omega_K=\vechat{\omega}_K)$
is the same for all~$\vechat{\omega}_K \in \RR^K$, i.e.\ there exists a vector-valued function
$\vechat{\omega}_K \mapsto \offset(\vechat{\omega}_K) \in \RR^{\ZZ^n}$
such that the law of~$\vec\omega_{\ZZ^n}$ under~$\Pr[\Bcdot|\ab {\vec\omega_K=\vechat{\omega}_K}]$
is the same as the law of~$\vec\omega_{\ZZ^n} + \offset(\vechat{\omega}_K)$
under~$\Pr[\Bcdot|\ab {\vec\omega_K \equiv 0}]$.
\end{Clm}

\begin{Lem}\label{c4577b}
For~$(X,Y)$ a two-dimensional centered Gaussian vector with~$X$ and~$Y$ non-degenerate,
\[\label{f5223} \{X:Y\} = \sqrt{\frac{\EE[X^2]}{\EE[Y^2]}} \, \big|\EE[Y|X=1]\big| .\]
\end{Lem}

\begin{Clm}\label{c4577c}
For~$K\subset\ZZ^n$, the function~$\offset$
defined in Claim~\ref{c4577a} is nondecreasing,
in the sense that each of the entries of~$\offset(\vechat{\omega}_K)$
is a nondecreasing function of each~$\hat\omega_k$ for~$k\in K$.
\end{Clm}

\begin{Clm}\label{c4577d}
For~$i\in\ZZ^n$, $K\subset\ZZ^n\setminus\{i\}$:
\[ \offset\big(1^{\{i\}},0^K\big) \leq \offset\big(1^{\{i\}}\big) ,\]
where $(1^{\{i\}},0^K)$ stands for the function on
$K\uplus\{i\}$ which is equal to~$1$ at~$i$ and to~$0$ on~$K$,
resp.~$1^{\{i\}}$ stands for the function on~$\{i\}$ mapping $i$ to~$1$.
\end{Clm}

Admit temporarily the claims.
Let~$i,j$ be distinct points of~$\ZZ^n$, let~$K\subset\ZZ^n\setminus\{i,j\}$
and let~$\vechat{\omega}_K \in \RR^K$; our goal is to compute
$\{ \omega_i : \omega_j \}$ under~$\Pr[\Bcdot|\ab \vec\omega_K=\vechat{\omega}_K]$.
First, by Claim~\ref{c4577a} we can suppose that $\vechat{\omega}_K \equiv 0$.
Now under~$\Pr[\Bcdot|\ab \vec\omega_K\equiv0]$, $(\omega_i,\omega_j)$ is still Gaussian
by the properties of Gaussian vectors, and it is centered by symmetry,
therefore by Lemma~\ref{c4577b}, $\{\omega_i:\omega_j\}$ is equal to
\[\label{f5276}
\sqrt{\frac{\EE[\omega_i^2|\vec\omega_K\equiv0]}{\EE[\omega_j^2|\vec\omega_K\equiv0]}} \, \big| \EE\big[\omega_j\big|\vec\omega_K\equiv0\enskip\text{and}\enskip \omega_i=1\big] \big|
= \sqrt{\frac{\EE[\omega_i^2|\vec\omega_K\equiv0]}{\EE[\omega_j^2|\vec\omega_K\equiv0]}} \, \big(\offset(1^{\{i\}},0^K)\cdot j\big) \]
—one has indeed $\offset(1^{\{i\}},0^K)\cdot j \geq 0$,
since by Claim~\ref{c4577c}, $\offset(1^{\{i\}},0^K) \geq \offset(0^{\{i\}\uplus K}) \equiv 0$.

Now, taking $K=\emptyset$ in~(\ref{f5276}), we find that under the law $\Pr$:
\[ \{\omega_i:\omega_j\} = \sqrt{\frac{\EE[\omega_i^2]}{\EE[\omega_j^2]}} \,
\big(\offset(1^{\{i\}})\cdot j\big) ,\]
which is $\geq \sqrt{\EE[\omega_i^2] \div \EE[\omega_j^2]} \, \big(\offset(1^{\{i\}},0^K)\cdot j\big)$
by Claim~\ref{c4577d}.
But up to switching the roles of~$i$ and~$j$, we can assume that
$\EE[\omega_i^2] \div \EE[\omega_j^2] \geq
\EE[\omega_i^2|\vec\omega_K\equiv0] \div \EE[\omega_j^2|\vec\omega_K\equiv0]$,
thus getting the desired result:
\[ \{\omega_i:\omega_j\}
\geq \sqrt{\frac{\EE[\omega_i^2|\vec\omega_K\equiv0]}{\EE[\omega_j^2|\vec\omega_K\equiv0]}} \, \big(\offset(1^{\{i\}},0^K)\cdot j\big) = \{\omega_i:\omega_j\}_{\vec\omega_K} .\]
\end{proof}

\begingroup\def\proofname{Proof of the claims}
\begin{proof}\strut\\
\indent\textit{Claim~\ref{c4577a}\ --}
It is a well-known property of Gaussian vectors,
which here is stated in an infinite-dimensional setting.

\indent\textit{Claim~\ref{c4577b}\ --}
Since $(X,Y)$ is centered Gaussian,
$Y^{\sigma(X)}$ is the orthogonal projection of the $L^2$ variable $Y$ on~$\RR X$,
so $\EE[Y|X=x] \propto x$.
Thus one has:
\[\label{c5212} \EE[XY] = \int x\, \EE[Y|X=x]\, \dx\Pr[X=x] \\
= \int x^2\, \EE[Y|X=1]\, \dx\Pr[X=x] = \EE[Y|X=1] \, \EE[X^2] .\]
But for such a Gaussian vector, Theorem~\ref{pro1857} gives that
\[ \{X:Y\} = \frac{|\Cov(X,Y)|}{\ecty(x)\ecty(y)} = \frac{|\EE[XY]|}{\sqrt{\EE[X^2]\EE[Y^2]}} ,\]
which combined with~(\ref{c5212}) gives~(\ref{f5223}).

\indent\textit{Claim~\ref{c4577c}\ --} First, notice that
$\EE[\vec\omega_{\ZZ^n}|\vec\omega_K\equiv0] = \vec0$, so that
\[\offset(\vechat{\omega}_K) =
\EE\big[\vec\omega \big| \vec\omega_K=\vechat{\omega}_K \big] .\]
Now, allowing temporarily $\beta$ to vary again, by the properties of Gaussian vectors,
the vector-valued variable $\EE_\beta[\vec\omega_{\ZZ^n}|\vec\omega_K=\vechat{\omega}_K]$
is Gaussian with constant expectation and covariance matrix proportional to~$\beta$.
Therefore, the common expectation of all these laws
is equal to the constant value of~$\vec\omega_{\ZZ^n}$ for $\beta=0$,
which is the~$\vec\omega$ minimising $H$ under the constraint ``$\vec\omega_K=\vechat{\omega}_K$'':
\[ \offset(\vechat{\omega}_K) =
\argmin_{\vec\omega_K=\vechat{\omega}_K} H(\vec\omega) .\]
Since it minimizes energy, the state $\offset(\vec\omega_K)$ is at equilibrium outside $K$.
In other words, it is the solution of the following subelliptic system:
\[\label{f4463} \left\{ \begin{array}{ll}
\forall i\in\ZZ^n\setminus K
\quad -\omega_i + \sum_{j\neq i} \gamma_{j-i}(\omega_j-\omega_i) = 0 ;\\
\forall i\in K \quad \omega_i = \hat\omega_i .\end{array} \right.\]
(That system is clearly subelliptic
because the pinning and interaction forces all are attractive).
By the maximum principle,
the solution of~(\ref{f4463}) is an increasing function of the boundary condition,
which was our claim.

\indent\textit{Claim~\ref{c4577d}\ --} Denote
$\vec\omega^{1}_{\ZZ^n} = \offset(1^{\{i\}})$.
Since $\offset(0^{\{i\}}) = 0^{\ZZ^n}$, by Claim~\ref{c4577c}
one has ${\vec\omega^1_{\ZZ^n} \geq 0^{\ZZ^n}}$.
Since obviously $\vec\omega^1_i = 1$, one has even
$\vec\omega^1_{\ZZ^n} \geq (1^{\{i\}},0^{\ZZ^n\setminus\{i\}})$.
In particular, $\vec\omega^1_{\{i\}\uplus K} \geq (1^{\{i\}},0^K)$;
therefore, using again Claim~\ref{c4577c},
\[\label{f3643} \offset\big(\vec\omega^1_{\{i\}\uplus K}\big) \geq \offset\big(1^{\{i\}},0^K\big).\]
Now, we defined $\vec\omega^1_{\ZZ^n}$ as $\offset(1^{\{i\}})$,
so by Formula~(\ref{f4463}) it satisfies
\[ -\omega^1_{i'} + \sum_{j\neq i'} \gamma_{j-i'}(\omega_j-\omega_i') = 0 \]
for all~$i' \in \ZZ^n\setminus\{i\}$, hence \emph{a fortiori}
for all~$i'\in \ZZ^n\setminus(\{i\}\uplus K)$.
Since moreover $\vec\omega^1$ obviously coincides with~$\vec\omega^1_{\{i\}\uplus K}$
on~$\{i\} \uplus K$, this implies, by Formula~(\ref{f4463}) again, that
\[ \offset\big( \vec\omega^1_{\{i\}\uplus K} \big) =
\vec\omega^1_{\ZZ^n} = \offset\big( 1^{\{i\}} \big) .\]
So, (\ref{f3643}) becomes ``$\offset(1^{\{i\}}) \geq \offset(1^{\{i\}},0^K)$'',
what we wanted.
\end{proof}\endgroup

Thanks to Corollary~\ref{c6834} our tensorization theorems give decorrelation results
for the quadratic model:
\begin{Thm}\label{thmTchou}
Provided Condition~(\ref{f2380}) holds:
\begin{ienumerate}
\item\label{i4805} The quadratic model is well-$\rho$-mixing, cf.\ Definition~\ref{def-WRM}.
If $\Gamma < 1$, one can be more specific about the property ``$\rho^*(1)<1$'':
for all disjoint $I,J\subset\ZZ^n$, ${\{\vec\omega_I:\vec\omega_J\}} \leq \Gamma$.
\item\label{i4806} Moreover, if there is polynomial decay of interactions
$\gamma_{z} = O(1/|z|^{n+\alpha})$, then the model is $\alpha$-polynomially $\rho$-mixing,
and if $\gamma_{z}$ has exponential decay,
then the model is exponentially $\rho$-mixing (but not with the same rate as $\gamma_z$ in general).
\end{ienumerate}
\end{Thm}

\begin{proof}
To prove Point~(\ref{i4805}), we have to show that
\[ \sum_{z\in\ZZ^n\setminus\{0\}} \frac{a_{Q^{-1}}(z)}{a_{Q^{-1}}(0)} \leq \Gamma \]
—recall that we assumed $\Gamma < \infty$.
We write:
\[\label{f1610} \sum_{z\in\ZZ^n\setminus\{0\}} \frac{a_{Q^{-1}}(z)}{a_{Q^{-1}}(0)}
= \frac{\sum_{z\in\ZZ^n}a_{Q^{-1}}(z)}{a_{Q^{-1}}(0)} - 1 .\]
There, $\sum_{z\in\ZZ^n}a_{Q^{-1}}(z)$ is equal to~$1$:
indeed, $a_{Q^{-1}}$ is the convolution inverse of~$a_Q$,
so by Fubini's theorem:
\[ \sum_{z\in\ZZ^n} a_{Q^{-1}}(z) = \Big( \sum_{z\in\ZZ^n} a_Q(z) \Big)^{-1} ,\]
where
\[ \sum_{z\in\ZZ^n} a_Q(z) = \sum_{z\in\ZZ^n\setminus\{0\}}\big(-\gamma(z)\big) + (1+\Gamma)
= -\Gamma + 1 + \Gamma = 1 .\]
Now, by~(\ref{f0024}),
$a_{Q^{-1}}(0)$ is obviously bounded below by~$(1+\Gamma)^{-1}$, so in the end:
\[ \sum_{z\in\ZZ^n\setminus\{0\}} \frac{a_{Q^{-1}}(z)}{a_{Q^{-1}}(0)}
\leq \frac{1}{(1+\Gamma)^{-1}} - 1 = \Gamma .\]

To prove Point~(\ref{i4806}), we have to show that
polynomial decay of~$\gamma_z$ implies polynomial decay of~$a_{Q^{-1}}$
with the same exponent, resp.\ that exponential decay of~$\gamma_z$
implies exponential decay of~$a_{Q^{-1}}$.
This is achieved resp.\ by Lemmas~\ref{l1939a} and~\ref{l1939b} in the appendix.
\end{proof}

\section{Nonlinear lattice of particles}\label{parNonlinearLatticeOfParticles}

In this section we will consider a model with continuous spins,
but where interactions are nonlinear, so that we cannot use the properties of Gaussian variables.
One has a lattice of particles indexed by~$\ZZ^n$ (equipped with its $\ell^1$ graph structure),
each particle $i$ being described by its ``polarization'' $\omega_i\in\RR$.
Each particle is submitted to a pinning force deriving from a potential $V$,
and to interaction forces with its neighbours, the interactions deriving from a potential $W$.
In other words, the Hamiltonian is formally
\[\label{forHNL} H(\vec\omega) = \sum_{i\in\ZZ} V(\omega_i) +
\frac{1}{2} \sum_{i\sim j} W(\omega_j-\omega_i) .\]

We make the following assumptions:
\begin{Hyp}\label{hypVW}
Both $V$ and~$W$ are convex; moreover
$V$ is uniformly strictly convex and the Hessian of~$W$ is bounded,
i.e.\ there exist constants~$v_*>0$ and~$w_*<\infty$
such that for all~$x \in \RR$, $v_* \leq V''(x)$ and $W''(x) \leq w^*$.
\end{Hyp}
We are interested in the equilibrium state of the system at some inverse temperature
$0 < \beta < \infty$. (In the sequel we suppose that $\beta$ is fixed).

Let~$i\neq j\in\ZZ$, $K\subset\ZZ\setminus\{i,j\}$ and~$\vechat{\omega}_K \in \RR^K$;
we want to study the law of~$(\omega_i,\omega_j)$
under the law $\Pr[\Bcdot|\vec\omega_K=\vechat{\omega}_K]$.
Then, the probability distribution of the system is formally described by
\[ \dx\Pr \big( \omega_i,\omega_j,\vec\omega_{\C{K}\setminus\{i,j\}} \big) \propto
\exp \Big( -\beta
H \big(\omega_i,\omega_j,\vec\omega_{\C{K}\setminus\{i,j\}},\vechat{\omega}_K\big) \Big) .\]
Our assumptions ensure that the function
$H(\Bcdot,\Bcdot,\Bcdot,\vechat{\omega}_K)$ is uniformly convex,
so that the equilibrium exists and is unique.

For the sequel, we need to recall the definition of the $W_{\infty}$ Wasserstein distance:
\begin{Def}[see also~\cite{Winfty}]
For~$\mu_1,\mu_2$ two measures on some metric space $(X,d)$,
``$W_{\infty}(\mu_1,\mu_2) \leq \epsilon$'' means that there exists a probability measure~$\gamma$
on~$E^2$ such that the two respective marginals of~$\gamma$ are~$\mu_1$ and~$\mu_2$
and such that $d(x_1,x_2) \leq \epsilon \quad \gamma\text{-a.s.}$.
This defines a (possibly infinite) distance on the probability measures on~$E$.
\end{Def}

The fundamental lemma of this subsection is the following
\begin{Clm}\label{lem7014}
For~$\hat\omega_j \in \RR$, denote by~$\mu(\hat\omega_j)$ the law of~$\omega_i$
under~$\Pr[\Bcdot|\ab \vec{\omega}_{K\uplus\{j\}}=(\vechat{\omega}_K,\hat\omega_j)]$.
There exists a function~$\epsilon \colon \ZZ \to [0,1]$ with
$\epsilon(d) < 1$ as soon as $d > 0$ and $\epsilon(d) \stackrel{d\longto\infty}{\leq} Ce^{-\psi d}$
for some~${\psi > 0}$ and~${C < \infty}$, such that
\[ \forall \hat\omega_j^1, \hat\omega_j^2 \in \RR \qquad
W_\infty \big( \mu(\hat\omega_j^1), \mu(\hat\omega_j^2) \big)
\leq \epsilon(|j-i|) \big| \hat\omega_j^2 - \hat\omega_j^1 \big| .\]
\end{Clm}

\begin{proof}
The proof relies on the `explicit' construction of a coupling measure~$\gamma$
between~$\mu(\hat\omega_j^1)$ and~$\mu(\hat\omega_j^2)$.
To do that, we will construct $\Pr[\Bcdot|\vec{\omega}_{K\uplus\{j\}}=(\vechat{\omega}_K,\hat\omega_j)]$
thanks to a reversible Fokker--Planck dynamics, and then couple the dynamics
for~$\hat\omega_j^1$ and~$\hat\omega_j^2$.

We define the Fokker--Planck dynamics thanks to independent white noises
$(\dx B^{i'}_t)_{t\in\RR}$ for~$i' \in {\ZZ^n \setminus (K\uplus\{j\})}$.
The motion of point $i'$ is defined by:
\[ \dx\omega_{i'} = -\beta \big( V'(\omega_{i'}) +
\sum_{i''\sim i'} W'(\omega_{i'}-\omega_{i''}) \big) + \sqrt{2} \,\dx{B^{i'}_t} ,\]
with the boundary condition $\vec\omega_{K\uplus\{j\}} = \vechat{\omega}_{K\uplus\{j\}}$
for all times.
Coupling then consists in taking the same noise for the two processes.
The initial condition is not very important since it is asymptotically forgotten,
so we will suppose that the two systems have been coupled for an infinite time,
so that at any time both systems follow their equilibrium law.
We denote by~$\vec\omega^1(t)$ the system correponding to the boundary condition
``$\vec\omega_{K\uplus\{j\}} = (\vechat{\omega}_K,\hat\omega_j^1)$'',
resp.\ by~$\vec\omega^2(t)$ the system correponding to the other boundary condition.
We denote $\Delta_{i'}(t) \coloneqq \omega_{i'}^2(t) - \omega_{i'}^1(t)$.
Then when the dynamics are coupled,
$\vec{\Delta}$ evolves according to the following equation:
\[ \dx(\Delta_{i'}) =
-\beta \Big[ \big( V'(\omega^2_{i'}) - V'(\omega^1_{i'}) \big) +
\sum_{i''\sim i'} \big( W'(\omega^2_{i'}-\omega^2_{i''}) - W'(\omega^1_{i'}-\omega^1_{i''}) \big) \Big] .\]
Obviously the right-hand side is not a deterministic function of
$\vec{\Delta}(t)$, but it can nonetheless be written as
\[ -\beta \big[ v(i',t) \Delta_{i'}(t)
+ w(i',i'',t) \big( \Delta_{i'}(t) - \Delta_{i''}(t) \big) \big] ,\]
for some~$v(i',t)$ and~$w(i',i",t)$ satisfying
\begin{eqnarray}
\label{for0346v} v(i',t) & \geq & v_* \quad \text{and} \\
\label{for0346w} 0 \leq w(i',i",t) & \leq & w^* \end{eqnarray}
by Assumption~\ref{hypVW}.
Moreover, one has the boundary conditions:
\[ \forall t \qquad \left\{
\begin{array}{rcl} \vec{\Delta}_K &\equiv& 0 ;\\
\Delta_j &=& \hat\omega_j^2 - \hat\omega_j^1 . \end{array} \right.\]

So, $\vec{\Delta}$ is the solution of some discrete `damped heat equation',
whose coefficients can vary along time though having to satisfy bounds~(\ref{for0346v}) and~(\ref{for0346w}).
Such an equation has no stationary solution \emph{stricto sensu};
however there exists some $\vec\Delta^+$ such that
\[ \vec{\Delta}(t) \leq \vec\Delta^+ \qquad \Rightarrow \qquad
\forall t'\geq t \quad \vec{\Delta}(t') \leq \vec\Delta^+ ;\]
namely, this $\vec\Delta^+$ is defined as the solution
of the following system of equations:
$\vec\Delta^+_K \equiv 0$, $\Delta^+_j = \hat\omega_j^2 - \hat\omega_j^1$,
and for all~$i'\notin K\uplus\{j\}$,
\[\label{for7756a} 0 = - v_* \Delta^+_{i'} + \sum_{i''\sim i'}
\1{\Delta^+_{i''}\geq\Delta^+_{i'}}w^* (\Delta^+_{i''} - \Delta^+_{i'}) .\]
One has similarly that
\[ \vec{\Delta}(t) \geq \vec0 \qquad \Rightarrow \qquad
\forall t'\geq t \quad \vec{\Delta}(t') \geq \vec0 .\]
Consequently, I claim that for all~$t$ one has
\[\label{for4700} \vec0 \leq \vec\Delta(t) \leq \vec\Delta^+ :\]
indeed if the initial condition of the system satisfies~(\ref{for4700}),
then that property remains valid for all subsequent times;
now, as I told, initial conditions are asymptomatically forgotten,
so in fact (\ref{for4700}) is always satisfied.

One has the following control on~$\vec\Delta^+$:
\begin{Clm}\label{lem3945}
There exists a function~$\epsilon \colon \ZZ \to [0,1]$ with
$\epsilon(d) < 1$ as soon as $d > 0$ and $\epsilon(d) \stackrel{d\longto\infty}{\leq} Ce^{-\psi d}$
for some~$\psi > 0$ and~$C < \infty$, such that
\[ \forall i \in \ZZ^n \qquad \Delta^+_i \leq \epsilon(|j-i|) .\]
Moreover, the function~$\epsilon$ does not depend on~$K$ nor on~$j$.
\end{Clm}
Combining~(\ref{for4700}) with Claim~\ref{lem3945} ends the proof of Claim~\ref{lem7014}.
\end{proof}

\begingroup\def\proofname{Proof of Claim~\ref{lem3945}}
\begin{proof}
First, notice that Equation~(\ref{for7756a}) satisfies a maximum principle,
so we know in advance that $\Delta^+$ is uniquely defined with $0 \leq \Delta^+ \leq 1$ everywhere.

For~$i'\sim i''$, denote $w_{i'}(i'') \coloneqq \1{\Delta^+_{i''} \geq \Delta^+_{i'}} w^*$.
Then (\ref{for7756a}) can be rewritten into:
\[\label{for7756a+} \Delta_{i'} =
\sum_{i''\sim i'} \frac{w_{i'}(i'')}{v_*+\sum_{i''\sim i'}w_{i'}(i'')} \times \Delta_{i''}
+ \frac{v_*}{v_*+\sum_{i''\sim i'}w_{i'}(i'')} \times 0 .\]

Now I define the following Markov chain on~$\ZZ^n \uplus \{\partial\}$,
$\partial$ denoting a cemetery point:
\begin{Def}\strut
\begin{itemize}
\item If at some time the particle is on some point $i'$ of~$\ZZ^n \setminus (K\uplus\{j\})$,
at next time it jumps onto the neighbour $i''$ of~$i'$ with probability
$w_{i'}(i'') \mathbin{\big/} \big(v_*+\sum_{i''\sim i'}w_{i'}(i'')\big)$,
and it jumps onto~$\partial$ with probability
$v_* \mathbin{\big/} {\big(v_*+\sum_{i''\sim i'}w_{i'}(i'')\big)}$;
\item If the particle is somewhere in~$K \uplus \{\partial,j\}$ at some time,
then it does not move any more.
\end{itemize}
\end{Def}
Call~$(X_t)_{t\in\NN}$ such a Markov chain and denote by~$\mcal{L}$ its generator.
It is clear that with probability one,
$X_t$ eventually remains at some point of~$K\uplus\{\partial,i\}$.
Extend $\Delta^+$ to~$\ZZ^n\uplus\{\partial\}$ by setting $\Delta^+_\partial = 0$;
then, (\ref{for7756a+}) merely means that $\Delta^+$ is $\mcal{L}$-harmonic,
and it follows that
\[ \Delta^+_i = \EE \big[ f(X_\infty) \big| X_0=i \big] .\]
Thus, to bound above $\Delta^+_i$ I write that
\begin{multline} \EE\big[f(X_\infty)\big|X_0=i\big]
= \sum_{\substack{i=i_0\sim\cdots\sim i_t=j\\i_1,\ldots,i_{t-1}\notin K\uplus\{j\}}}
\,\,\prod_{u=0}^{t-1} \frac{w_{i_u}(i_{u+1})}{v_*+\sum_{i''\sim i_u}w_{i'}(i'')} \\
\leq \sum_{\substack{i=i_0\sim\cdots\sim i_t=j\\i_1,\ldots,i_{t-1}\neq j}}
\,\,\prod_{u=0}^{t-1} \frac{w_{i_u}(i_{u+1})}{\sum_{i''\sim i_u}w_{i'}(i'')}
\bigg( \frac{2dw^*}{2dw^*+v_*} \bigg)^{\!t} \\
\leq \bigg( \frac{2dw^*}{2dw^*+v_*} \bigg)^{|j-i|}\!\!
\underbrace{\sum_{\substack{i=i_0\sim\cdots\sim i_t=j\\i_1,\ldots,i_{t-1}\neq j}}
\,\,\prod_{u=0}^{t-1} \frac{w_{i_u}(i_{u+1})}{\sum_{i''\sim i_u}w_{i'}(i'')}}_{\leq 1}
\leq \bigg( \frac{2dw^*}{2dw^*+v_*} \bigg)^{\!|j-i|}.
\end{multline}
\end{proof}\endgroup

From Claim~\ref{lem7014}, we take the following
\begin{Cor}
For a Lipschitzian function~$f \colon \RR\to\RR$,
denote by~$\|f\|_{\mathit{Lip}}$ the optimal Lipschitz constant for~$f$.
On~$\ldb(\omega_i)$, define the (possibly infinite) norm $\|\Bcdot\|_{\mathit{Lip}}$
such that $\|f(\omega_i)\|_{\mathit{Lip}} = \|f\|_{\mathit{Lip}}$%
\footnote{This definition can be ambiguous if the support of~$\omega_i$ is not the whole $\RR$; in this case, just add an infimum in the definition.};
denote by~$\barLip(\omega_i)$ the corresponding Banach space.

Then under the law $\Pr[\Bcdot|\vec\omega_K=\vechat{\omega}_K]$,
the map $\pi_{\omega_j\omega_i}$ defined by~(\ref{for0595}) is $\epsilon(|j-i|)$-contracting
when seen as an application from $\barLip(\omega_i)$ into $\barLip(\omega_j)$.
\end{Cor}

Consequently, the map
$\pi_{\omega_i\omega_j\omega_i} \colon \barLip(\omega_i) \to \barLip(\omega_i)$
is $\epsilon(|j-i|)^2$-contracting.
But the canonical embedding $\barLip(\omega_i) \mapsto \ldb(\omega_i)$ is continuous
as our hypotheses ensure that $\Law(\omega_i)$ is uniformly log-concave,
therefore for all~$f \in \barLip(\omega_i)$ one has
\[ \limsup_{k\longto\infty}
\big|\langle \pi_{\omega_i\omega_j\omega_i}^k f, f \rangle_{\ldb(\omega_i)}\big|^{1/k}
\leq \epsilon(|j-i|)^2 .\]
Since $\pi_{\omega_i\omega_j\omega_i}$ is self-adjoint in~$\ldb(\omega_i)$
and $\barLip(\omega_i)$ is a dense subset of~$\ldb(\omega_i)$,
it follows by Lemma~\ref{l3253} that $\pi_{\omega_i\omega_j\omega_i}$ is
$\epsilon(|j-i|)^2$-contracting also in~$\ldb(\omega_i)$.
This, by Remark~\ref{rmk1265}, is equivalent to saying that
\[\label{forECRAN} \{\omega_i:\omega_j\}_{\vec\omega_K} \leq \epsilon(|j-i|) . \]

(\ref{forECRAN}) is what we need to apply Lemma~\ref{lem4067};
in the end, we get the
\begin{Thm}\label{thmLama}
The model~(\ref{forHNL}) is exponentially $\rho$-mixing.
\end{Thm}

\section{A hypocoercive system of interacting particles}%
\label{parAHypocoerciveSystemOfInteractingParticles}

For the time being we have only been dealing with \emph{spatial} decorrelations.
Yet I have had the idea that
the ability of Hilbertian decorrelations to get tensorized for infinite sets
could be well adapted to the study of \emph{temporal} relaxation of an infinite stochastic system:
one can consider indeed time as an extra dimension for the particle system,
which leads to a situation analogous to the parallel hyperplanes of \S~\ref{parc6177}.
In the reversible case, we saw that spectral techniques
make it possible to get $L^2$ results from $L^1$ results, cf.\ Theorem~\ref{c6177}.
Here I will show how Hilbertian decorrelations can be used
for a \emph{non-reversible} particle stochastic system.

The system which we will study here as an example
is governed by a \emph{kinetic Fokker-Planck equation}.
This equation, which arises naturally in phy\-sics, corresponds to a Hamiltonian evolution
perturbed by some noise \emph{acting on speeds}.
The study of such systems is made complicated by the fact that diffusion is only performed
along certain directions of the states space,
so that the non-reversibility of the evolution is essential to ensure convergence to equilibrium.
In~\cite{hypocoercivity}, Villani proves $L^2$ convergence for such systems
in situations where the state of the system lives in a finite-dimensional manifold.
Here we will use tensorization of Hilbertian decorrelations in a fundamental way
to get a result valid in an infinite-dimensional setting.
Moreover, we will get non-trivial bounds for arbitrary small times,
which is a new feature compared to~\cite{hypocoercivity}.

\begin{Def}\label{d7463}
For real parameters $m, \omega, c, T, \lambda > 0$%
\footnote{$m$ is the mass of each particle,
$\omega$ is the frequency corresponding to the pinning potential,
$c$ is more or less the speed of sound, expressed in inter-atomic distances by unit of time,
$T$ is the temperature and $\lambda$ is the relaxation constant of the friction.
Physical homogeneity of these constants are resp.~$[\mathsf{M}], \ab 
[\mathsf{T}^{-1}], \ab [\mathsf{T}^{-1}], \ab [\mathsf{M}\mathsf{L}^2\mathsf{T}^{-2}], \ab 
[\mathsf{T}^{-1}]$.},
we consider a system of particles $i$ indexed by~$\ZZ$,
each particle being described by its momentum $p_i \in \RR$ and its position $q_i \in \RR$.
We consider the Hamiltonian
\[ H(\vec{p},\vec{q}) = m^{-1} \sum_{i\in\ZZ} \frac{p_i^2}{2}
+ m\omega^2 \sum_{i\in\ZZ} \frac{q_i^2}{2}
+ mc^2 \sum_{i\in\ZZ} \frac{(q_{i+1}-q_i)^2}{2} .\]
Then the system $(\vec{p}(u),\vec{q}(u))$ evolves according to the Hamiltonian $H$,
plus a white noise independent on each~$p_i$, plus a friction force
$F_i = -\lambda p_i$ on each~$i$ which dissipates the energy brought by the white noise,
friction being adjusted to the noise so that their association constitutes
a (volumic) thermal bath at temperature $T$.
One computes that this means that the quadratic variation on~$p_i$ is given by
$\dx[p_i] = 2T\lambda m\,\dx{u}$.

In other words, if $\big(W_i(u)\big)_{i\in\ZZ}$ denotes a family of independent brownian motions,
the evolution of the system is given by
\[ \left\{ \begin{array}{rcl}
\dx{p_i} &=& \big( - m\omega^2 q_i + mc^2 (q_{i-1}+q_{i+1}-2q_i) - \lambda p_i \big) \,\dx{u}
+ \sqrt{2T\lambda m} \,\dx{W_i} \\
\dx{q_i} &=& m^{-1} p_i \,\dx{u} .
\end{array} \right. \]
\end{Def}

\begin{Rmk}
The system of Definition~\ref{d7463} is to be thought
as a toy model for a large class of similar systems
obtained by generalizing it in several ways.
A first example, which would change almost nothing but complicating the formalism,
is to replace the states space
$\RR\times\RR$ of each particle by~$\RR^n\times\RR^n$,
or to replace the lattice $\ZZ$ by~$\ZZ^n$.
A trickier generalization is to consider the case of non-harmonic interactions:
then I expect the results stated below to remain qualitatively true,
but proving them might be far more difficult
since one cannot use the properties of Gaussian vectors any more.
Also, if one allows for infinite-ranged interactions, which speed of decay is required
to get temporal decorrelations?

All these questions look quite worthwhile to me, though answering them
is out of the scope of this work.
Here I will only show how Hilbertian correlations make everything work fine for the toy model,
hoping that it shall be useful for the general situation.
\end{Rmk}

Let us consider the equilibrium dynamics of our system.
We fix an arbitrary time $0 < t < \infty$.
Denote by~$(p_i,q_i)$ the state of particle $i$ at time~${u=0}$,
resp.\ by~$(p'_i,q'_i)$ the state of particle $i$ at time~${u=t}$.
We have to prove the
\begin{Clm}\label{lem1278}
Provided $t$ is small enough, for all~$i,j\in\ZZ$ (possibly identical),
one has ${\{p_i:p'_j\}_*}, \ab {\{p_i,q'_j\}_*}, \ab {\{q_i:p'_j\}_*}, \ab 
{\{q_i,q'_j\}_*} < 1$, uniformly in~$i,j$.
Moreover, still uniformly in~$i,j$, these quantities are bounded by
$O(e^{-\gamma|j-i|})$ for some~$\gamma > 0$.
\end{Clm}

\begin{proof}
We denote by~$\eta \text{\ (resp.\ $\eta'$)} \in \RR^{\ZZ\times\{p,q\}}$
the global state $(p_i,q_i)_{i\in\ZZ}$ (resp.~$(p'_i,q'_i)_{i\in\ZZ}$) at time~$0$ (resp.~$t$).
We also denote by~$(\phi^u)_{u\geq0}$ the semigroup of operators on~$\RR^{\ZZ\times\{p,q\}}$
corresponding to the evolution of the system in absence of noise, but with the friction remaining.
Since the system is linear, the~$\phi^u$ are linear operators.

By the work of \S~\ref{parQuadraticModels},
we know that $\eta$ is distributed according to the centered Gaussian law
with covariance matrix $T^{-1}\check{C}$, where $\check{C}$ is defined as $\check{Q}^{-1}$,
the matrix $\check{Q}$ being in turn defined by:
\begin{eqnarray}
\label{forQpp} \check{Q}_{p_ip_i} &\coloneqq& m^{-1} ;\\
\label{forQpq} \check{Q}_{q_iq_i} &\coloneqq& m (\omega^2 + 2c^2) ;\\
\label{forQqq} \check{Q}_{q_iq_{i\pm1}} &\coloneqq& -mc^2 ,
\end{eqnarray}
the other entries of~$\check{Q}$ being zero.
Observe that, as the matrix of a quadratic form,
$\check{Q}$ is bounded (this is obvious from~(\ref{forQpp})--(\ref{forQqq}));
moreover, $\check{Q}^{-1}$ (actually exists and) is also bounded:
that follows from $\check{Q}$'s being bounded below by
the matrix having the same expression with~$c$ replaced by~$0$,
which we denote by~$\check{Q}^\circ$, which is a strictly positive `scalar' matrix
(modulo some homogeneity constant).

Because of the linear nature of the system, we have moreover that,
conditionally to~$\eta$, the law of~$\eta'$ is some Gaussian vector
of the form $\phi^{t}\eta + \theta$, where $\theta$ is a centered Gaussian vector
whose law does not depend on~$\eta$.
Let us denote by~$\hat{C}$ the covariance matrix of~$\theta$,
and $\hat{Q} = \hat{C}^{-1}$ —though for the time being it is not clear that $\hat{Q}$ exists.

Then, we can formally write the covariance matrix $\bar{C}$ of~$(\eta,\eta')$ as
$\bar{C} = \bar{Q}^{-1}$, with:
\[\label{for7743}
\bar{Q}(\eta,\eta') = \check{Q}(\eta) + \hat{Q}(\eta'-\phi^{t}\eta) .\]
(Note that $\bar{Q}$ is a quadratic form on~$\RR^{\ZZ\times\{p,q,p',q'\}}$,
while $\check{Q}$ and~$\hat{Q}$ were defined on~$\RR^{\ZZ\times\{p,q\}}$).

\begin{Not}
In the sequel, we shorthand ``$\ZZ\times\{p,q\}$'' into~``$\ZZ^{\uplus2}$'',
resp.\ ``$\ZZ\times\{p,q,p',q'\}$'' into~``$\ZZ^{\uplus4}$''.
\end{Not}

Now I claim that there exists
constants $0<r\leq R<\infty$ such that
$r\mathbf{I} \leq \bar{Q} \leq R\mathbf{I}$.
Well, this is meaningless \emph{stricto sensu}, because all the entries of~$\bar{Q}$
do not have the same physical homogeneity,
so we have to `convert' momenta into positions by dividing them
by some homogeneity parameter~$\chi$, say $\chi = m\omega$
—but other choices may be more relevant.

First, I claim that $\bar{Q} \geq \frac{1}{2}(\chi^2 m^{-1}\wedge m\omega^2) \mbf{I}$.
Let indeed $(\eta,\eta') = (\vec{p}_{\ZZ},\vec{q}_{\ZZ},\vec{p'}_{\ZZ},\vec{q'}_{\ZZ})
\in \RR^{\ZZ^{\uplus4}}$ with finite support.
We observe that
\[ \|(\eta,\eta')\|^2 = \sum_{i\in\ZZ} \big(\chi^{-2}p_i^2+q_i^2+\chi^{-2}{p'}_i^2+{q'}_i^2\big)
= \| \eta \|^2 + \| \eta' \|^2 ,\]
so that either $\|\eta\|^2 \geq \frac{1}{2} \|(\eta,\eta')\|^2$
or $\|\eta'\|^2 \geq \frac{1}{2} \|(\eta,\eta')\|^2$.
Now, recalling the definition of~$\check{Q}^\circ$ a few lines above,
$\check{Q}(\eta) \geq \check{Q}^\circ(\eta) \geq (\chi^2 m^{-1}\wedge m\omega^2) \|\eta\|^2$,
so by~(\ref{for7743}), $\bar{Q}(\eta,\eta') \geq (\chi^2 m^{-1}\wedge m\omega^2) \|\eta\|^2$.
Since reversing the direction of time yields the same system with the sign of speeds reversed, which does not change the norms of~$\eta$ and~$\eta'$,
one has similarly $\bar{Q}(\eta,\eta') \geq (\chi^2 m^{-1}\wedge m\omega^2) \|\eta'\|^2$.
The claim follows.

The second point consists in proving that $\bar{Q}$ is bounded above.
On the one hand, by~(\ref{forQpp})--(\ref{forQqq}),
\[ \check{Q}(\eta) \leq \big(m^{-1}\chi^2\vee m(\omega^2+4c^2)\big) \|\eta\|^2
\leq \big(m^{-1}\chi^2\vee m(\omega^2+4c^2)\big) \|(\eta,\eta')\|^2 .\]

Next, the difficult point is to prove that $\hat{Q}(\eta'-\phi^{t}\eta)$
(exists and) can be bounded above by a multiple of~$\|(\eta,\eta')\|^2$.
We begin with transforming the original problem
of bounding a quadratic form on~$\RR^{\ZZ^{\uplus4}}$
into a problem on~$\RR^{\ZZ^{\uplus2}}$.
Indeed, $\|\phi^t\eta\|$ is bounded by a multiple of~$\|\eta\|$,
since the operator $\phi^t$ dissipates the energy $H(\eta)$,
energy which the previous work on~$\check Q$
proved to be controlled below and above by~$\|\eta\|^2$;
therefore, it suffices to prove that the quadratic form $\hat{Q}(\eta)$ on~$\RR^{\ZZ^{\uplus2}}$
is bounded by a multiple of~$\|\eta\|^2$ to achieve our goal.

The natural quantity to be computed for~$\theta$
(recall that $\theta$ denotes the total effect of noise between times~$0$ and~$t$)
is its covariance matrix $\hat{C}$.
Its expression is the following (the notation is explained just below):
\[\label{for5483} \hat{C} = 2T\lambda m \int_0^{t} \phi^{t-u} \mbf{I}_p \T{(\phi^{t-u})} \,\dx{u} ,\]
where $\mbf{I}_p$ is the diagonal matrix being~$1$ on diagonal entries indexed by some~$p_i$
and~$0$ on diagonal entries indexed by some~$q_i$,
and $\T{(\phi^{t-u})}$ is the transpose of the linear operator $\phi^{t-u}$
seen as a square matrix indexed by~$\ZZ^{\uplus2}$.
This decomposition means that we are summing the contributions
of all the elementary noises occuring at times $u \in [0,t]$,
using that these elementary noises are independent.

Now we need an approximate expression for~$\phi^u$, $u\in[0,t]$.
Here for the sake of legibility I will remain at a formal level, giving only limited expansions;
it is essential nevertheless to keep in mind that all the ``$O(*)$''
can be made explicit by using Gronwall's lemma,
and that these explicit values ensure that the~$O(*)$ behave well provided $t$ is small enough.
One finds that
\begin{eqnarray}
\label{for5476a}
\phi^u \delta_{p_i} \cdot p_j &=& c^{2|j-i|} \frac{u^{2|j-i|}}{(2|j-i|)!} + O(u^{2|j-i|+2}) ;\\
\phi^u \delta_{p_i} \cdot q_j &=& m^{-1}c^{2|j-i|} \frac{u^{2|j-i|+1}}{(2|j-i|+1)!}
+ O(u^{2|j-i|+3}) ;\\
\phi^u \delta_{q_i} \cdot p_i &=& m\omega^2 u + O(u^3) ;\\
\phi^u \delta_{q_i} \cdot p_{j\neq i} &=& mc^{2|j-i|} \frac{u^{2|j-i|-1}}{(2|j-i|-1)!}
+ O(u^{2|j-i|+1}) ;\\
\label{for5476z}
\phi^u \delta_{q_i} \cdot q_j &=& c^{2|j-i|} \frac{u^{2|j-i|}}{(2|j-i|)!} + O(u^{2|j-i|+2}) .
\end{eqnarray}
Injecting Equations~(\ref{for5476a})--(\ref{for5476z}) into~(\ref{for5483}), one finds that:%
\footnote{Recall that $\hat{C}$, as a covariance matrix, is symmetric.}
\begin{eqnarray}
\label{for5579a}
\hat{C}_{p_ip_i} &=& 2T\lambda m t + O(t^3) ;\\
\hat{C}_{p_iq_i} &=& T\lambda t^2 + O(t^4) ;\\
\hat{C}_{q_iq_i} &=& \mathsmaller{\frac{2}{3}}T\lambda m^{-1} t^3 + O(t^5) ;\\
\hat{C}_{p_ip_{j\neq i}} &=& O(t^{2|j-i|+1}) ;\\
\hat{C}_{p_iq_{j\neq i}} &=& O(t^{2|j-i|+2}) ;\\
\label{for5579z}
\hat{C}_{q_iq_{j\neq i}} &=& O(t^{2|j-i|+3}).
\end{eqnarray}

Consequently, the covariance matrix $\hat{C}$ can be seen
as a perturbation of the matrix $\hat{C}^\circ$
which is defined by Equations~(\ref{for5579a})--(\ref{for5579z}),
but with the ``$O(*)$'' terms replaced by~$0$.
Since $\hat{C}^\circ$ is invertible, with an explicitly computable inverse,
one finds that $\hat{C}$ is invertible too with:
\begin{eqnarray}
\hat{Q}_{p_ip_i} &=& 2T^{-1}\lambda^{-1}m^{-1}t^{-1} + O(t) ;\\
\hat{Q}_{p_iq_i} &=& -3T^{-1}\lambda^{-1}t^{-2} + O(1) ;\\
\hat{Q}_{q_iq_i} &=& 6T^{-1}\lambda^{-1}mt^{-3} + O(t^{-1}) ;\\
\hat{Q}_{p_ip_{j\neq i}} &=& O(t^{2|j-i|-1}) ;\\
\hat{Q}_{p_iq_{j\neq i}} &=& O(t^{2|j-i|-2}) ;\\
\hat{Q}_{q_iq_{j\neq i}} &=& O(t^{2|j-i|-3}).
\end{eqnarray}
In the end, provided that $t$ is small enough, we have proved that
$\hat{Q}(\eta)/\|\eta\|^2 \leq 6T^{-1}\lambda mt^{-3} + O(t^{-1}) \ab {< \infty}$.

Actually we have proved more than that:
not only we have a bound on the operator norm of~$\bar{Q}$,
but we have bounded it entry-wise.
More precisely, expanding the~$O(*)$, we find that provided $t$ is small enough,
there exists constants~$A < \infty$ and~$\gamma > 0$ such that for all~$i,j\in\ZZ$,
\[ \underbrace{\hat{Q}_{p_ip_j}, \hat{Q}_{p_iq_j}, \ldots, \hat{Q}_{q'_iq'_j}}%
_{\text{all 16 possibilities}} \leq A e^{-\gamma|j-i|} .\]

\begin{Not}
From now on we denote the basic variables $p_i, q_i, p'_i, q'_i$ of our system by~$X_i$,
$i \in \ZZ^{\uplus4}$.
\end{Not}

Now the question is: for~$i\neq j \in \ZZ^{\uplus4}$,
$K \subset \ZZ^{\uplus4} \setminus \{i,j\}$,
what is the value of~${\{X_i:X_j\}}_{\vec{X}_K}$?
By the properties of Gaussian variables [Theorem~\ref{pro1857}], the answer is the following.
Let~$\bar{Q}_{|_{\ZZ^{\uplus4}\setminus K}}$ be the restriction
of~$\bar{Q}$ to indexes in~$(\ZZ^{\uplus4}\setminus K)$.
Since $r\mathbf{I} \leq \bar{Q} \leq R\mathbf{I}$, the same holds
for~$\bar{Q}_{|_{\ZZ^{\uplus4}\setminus K}}$, so this matrix is invertible;
denote by~$\bar{C}^{|^{\ZZ^{\uplus4}\setminus K}}$ its inverse.
This matrix is the covariance matrix of (the centered version of)
$\vec{X}_{\ZZ^{\uplus4}\setminus K}$ under some fixed value for~$\vec{X}_K$;
thus:
\[\label{for4972} \{X_i:X_j\}_{\vec{X}_K} =
\frac{\Big| \bar{C}^{|^{\ZZ^{\uplus4}\setminus K}}_{ij} \Big|}
{\sqrt{\bar{C}^{|^{\ZZ^{\uplus4}\setminus K}}_{ii}
\bar{C}^{|^{\ZZ^{\uplus4}\setminus K}}_{jj}}} .\]

It remains to control the entries of~$\bar{C}^{|^{\ZZ^{\uplus4}\setminus K}}$, uniformly in~$K$.
We need two types of control: first an exponential control when $i$ is far away from $j$,
then a non-trivial control for the values of~$i$ and~$j$ corresponding to close
(or even identical) atoms.

Let us start with the first one.
$\bar{C}^{|^{\ZZ^{\uplus4}\setminus K}}_{ii}$ and
$\bar{C}^{|^{\ZZ^{\uplus4}\setminus K}}_{jj}$ are bounded below by~$R^{-1}$,
so we just have to bound above $\bar{C}^{|^{\ZZ^{\uplus4}\setminus K}}_{ij}$.
This is achieved by a direct use of Lemma~\ref{lem4060} in appendix.

Concerning the uniform non-trivial control, since $r\mbf{I} \leq \bar{Q} \leq R\mbf{I}$
one has $r\mbf{I} \leq \bar{Q}_{|_{\ZZ^{\uplus4}\setminus K}} \leq R\mbf{I}$,
hence $R^{-1}\mbf{I} \leq \bar{C}^{|^{\ZZ^{\uplus4}\setminus K}} \leq r^{-1}\mbf{I}$,
hence $R^{-1}\mbf{I} \leq \big(\bar{C}^{|^{\ZZ^{\uplus4}\setminus K}}\big)_{|_{\{i,j\}^2}} \leq r^{-1}\mbf{I}$;
from this and~(\ref{for4972}),
\[ \forall i \neq j \in I \qquad \{X_i:X_j\}_* \leq \frac{R-r}{R+r} < 1 .\]
\end{proof}

From Claim~\ref{lem1278}, we get the main result of this subsection:
\begin{Thm}\label{thmParis}
For the model of Definition~\ref{d7463}, for all~$t>0$, $\{\eta,\eta'\} < 1$.
\end{Thm}

\begin{proof}
First, if $t$ is small enough so that Claim~\ref{lem1278} holds,
direct application of Lemma~\ref{lem4067} proves the result,
as the~$\{X_i,X_j\}_*$ are summable (since they decrease exponentially)
and they all are $<1$.

Now for larger $t$, fix some $0<t_1<t$ so that Claim~\ref{lem1278} holds for~$t_1$.
Then we notice that $\eta \to \eta(t_1) \to \eta'$ is a Markov chain
(with ``$\eta(t_1)$'' standing for ``$\big(\vec{p}(t_1),\vec{q}(t_1)\big)$''),
so by Proposition~\ref{cor0706}, $\{\eta,\eta'\} \leq \{\eta,\eta(t_1)\} < 1$.
\end{proof}

\section{Appendix: Inverses of `nearly diagonal' matrices}%
\label{parInversesOfDiagonalConcentratedMatrices}

The goal of this appendix is to state and prove a few lemmas sharing the same spirit:
``\emph{if a matrix is `nearly diagonal', then it shall be invertible
and its inverse shall also be `nearly diagonal' with the same type of decay}''.

\subsection{Matrices with exponential decay}

The goal of this subsection is to prove the following
\begin{Lem}\label{lem4060}
Let~$I \subset \ZZ$ and let~$(\!(M_{ij})\!)_{(i,j)\in I^2}$ be a matrix.
Assume that, when seen as a quadratic form on~$L^2(I)$, one has
$r\mbf{I} \leq M \leq R\mbf{I}$ for $0<r\leq R<\infty$ —in particular, $M$ is invertible.
Assume moreover that there exists constants~$A < \infty$ and~$\gamma > 0$ such that
for all~$i,j\in I$, $|M_{ij}| \leq A e^{-\gamma|j-i|}$.

Then there exist constants~$A' < \infty$ and~$\gamma' > 0$
which are explicit functions of~$r,R,\gamma,A$ (so \emph{they do not depend on~$I$}),
such that one has the following control on the entries of~$M^{-1}$:
\[ \forall i,j \in I \qquad (M^{-1})_{ij} \leq A' e^{-\gamma'|j-i|} .\]
\end{Lem}

\begin{proof}
Up to multiplying by a scalar, one can assume that $R=1$.
Then $M$ writes $M = \mbf{I} - H$, where $0 \leq H \leq (1-r) \mbf{I}$;
since $H$ is symmetric, that inequality means that $\VERT H \VERT \leq 1-r < 1$.
Therefore, for all~$k \in \NN$ one has $\VERT H^k \VERT \leq (1-r)^k$,
which allows us to write $M^{-1}$ as a series expansion:
\[ M^{-1} = \sum_{k=0}^\infty H^k .\]

Up to replacing $A$ by~$A+1$,
we have the same entry-wise control on~$H$ as on~$M$.
Then one sees by induction that for all~$k\in\NN$,
\[\label{for5051}
\forall i,j\in I \qquad \big| (H^k)_{ij} \big| \leq A_1^k e^{-\gamma_1|j-i|} ,\]
where $\gamma_1$ is an arbitrary parameter in~$(0,\gamma)$ and
\[ A_1 \coloneqq \sum_{z\in\ZZ} Ae^{-\gamma|z| + \gamma_1 z}
= \frac{(1-e^{-2\gamma})A}{(1-e^{-(\gamma-\gamma_1)})(1-e^{-(\gamma+\gamma_1)})} \]
—observe that $I$ does not appear in the expression of~$A_1$.
Since $A_1$ is greater than $1$,
(\ref{for5051}) is not enough to get an entry-wise control on~$M^{-1}$.
But now observe that the bound $\VERT H^k \VERT \leq (1-r)^k$ implies
that all the~$(H^k)_{ij}$ are bounded by~$(1-r)^k$ in absolute value;
thus:
\[ \big| (M^{-1})_{ij} \big| \leq
\sum_{k=0}^\infty \big( e^{-\gamma_1|j-i|} A_1^k \wedge (1-r)^k \big) \\
\leq \bigg( \frac{A_1}{A_1-1} + \frac{1}{r} \bigg)
\exp \left( -\frac{|\ln(1-r)| \gamma_1}{|\ln(1-r)| + \ln A_1} |j-i| \right) ,
\]
from which you read suitable values for~$A'$ and~$\gamma'$.
\end{proof}

\subsection{Convolution inverses of rapidly decreasing functions}%
\label{parDeconvolutionOfRapidlyDecreasingFunctions}

\begin{NOTA}
In all this subsection, we work on~$\ZZ^n$ for some $n\in\NN^*$;
$\RR^n$ is endowed with some fixed norm~$|\Bcdot|$.
\end{NOTA}

\begin{Rmk}
Here I will deal with fonctions on~$\ZZ^n$, but the results of this subsection
could also be tranposed for functions on~$\RR^n$.
\end{Rmk}

\begin{Def}
If $a:\ZZ^n\to\RR$ is some integrable function with
$\|a\|_{\ell^1} < 1$, we define
\[\label{f8928} B[a] = a + a* a + a* a* a + \cdots ,\]
which is the sum of a convergent series in~$\ell^1(\ZZ^n)$.
$B[a]$ is the function~$b\in\ell^1(\ZZ^n)$ characterized by:
\[ (\delta_0-a) * (\delta_0+b) = \delta_0 .\]
\end{Def}

\begin{Def}\label{def726}
A function~$a:\ZZ^n\to\RR$ is said to have \emph{exponential decay} if there exists some $\beta > 0$
such that, for all~$\beta'<\beta$, $a(z) = O(e^{-\beta'|z|})$ when~$|z|\longto\infty$.
The minimal $\beta$ satisfying that property is called the
\emph{(exponential) rate of decay} of~$a$.
\end{Def}

\begin{Lem}\label{l1939a}
Let~$a\in\ell^1(\ZZ^n)$ with $\|a\|_{\ell^1} < 1$.
If $a$ has exponential decay, then so does $B[a]$.
\end{Lem}

\begin{proof}
Denoting by~$|a|$ the function defined by~$|a|(z) = |a(z)|$,
it is clear by~(\ref{f8928}) that
\[ \forall z\in\ZZ^n \qquad \big|B[a](z)\big| \leq B\big[|a|\big](z) ,\]
therefore it suffices to prove the case where $a$ is nonnegative.
In that case, $B[a]$ will also be nonnegative.

Let~$(\RR^n)^*$ denote the dual space of~$\RR^n$, endowed with the dual norm
\[ \forall \lambda \in (\RR^n)^* \qquad
|\lambda|_* = \sup_{\substack{z\in\RR^n\\|z|=1}} |\langle \lambda, z \rangle| .\]
For a nonnegative function~$a$, we define its Laplace transform
$\mcal{L}\{a\} \colon (\RR^n)^* \to \RR_+ \cup \{+\infty\}$ by
\[\label{f9677} \mcal{L}\{a\}(\lambda) = \sum_{z\in\ZZ^n} e^{\langle\lambda,z\rangle} a(z) .\]
Then, saying that $a$ has exponential decay with rate $\gamma$
is equivalent to saying that, for all~$\lambda \in (\RR^n)^*$ with $|\lambda|_* < \gamma$,
$\mcal{L}\{a\}(\lambda)$ is finite.

Since Laplace transform is linear and turns convolution into ordinary product,
(\ref{f8928}) yields, for all~$\lambda \in (\RR^n)^*$:
\[\label{f9574} \mcal{L}\{B[a]\}(\lambda) = \mcal{L}\{a\}(\lambda) + \mcal{L}\{a\}(\lambda)^2
+ \mcal{L}\{a\}(\lambda)^3 + \cdots, \]
which converges if and only if $\mcal{L}\{a\}(\lambda) < 1$.

Now, since $a$ is nonnegative, by~(\ref{f9677}) the function~$\mcal{L}\{a\}$ is convex,
so it is continuous on the interior of the domain where it is finite.
By the exponential decay hypothesis, that domain contains a neighbourhood of~$0$,
so $\mcal{L}\{a\}$ is continuous at~$0$.
And since $\mcal{L}\{a\}(0) = \sum_{z\in\ZZ^n} a(z) = \|a\|_{\ell^1} < 1$,
there is a neighbourhood of~$0$ on which $\mcal{L}\{a\}<1$ and thus $\mcal{L}\{B[a]\}<\infty$.
This implies that $B[a]$ has exponential decay.
\end{proof}

\begin{Rmk}\label{rmk9890}
This proof also shows that (for nonnegative $a$)
the rate of decay of~$B[a]$ will never be greater than the rate of decay of~$a$.
In general, it is even strictly smaller,
since all the values of~$\lambda$ for which $1\leq\mcal{L}\{a\}(\lambda)<\infty$ yield
a finite Laplace tranform for~$a$ but an infinite one for~$B[a]$.
For example, take $n=1$ and $a = e^{-1}\delta_{1}$,
which has exponential decay with infinite rate since it is compactly supported;
then the $k$-th convolution power of~$a$ is $a^{*k} = e^{-k}\delta_k$,
so that $B[a]$ is the function
\[ B[a](z) = \1{z>0} e^{-z} ,\]
which also has exponential decay, but with rate $1$ only.
\end{Rmk}

\begin{Lem}\label{l1939b}
If $\|a\|_{\ell^1(\ZZ^n)} < 1$ and $a(z) = O(1/|z|^\alpha)$ when~$|z|\longto\infty$
for some $\alpha > n$, then $B[a](z) \ab {= O(1/|z|^\alpha)}$ when~$|z|\longto\infty$.
\end{Lem}

\begin{proof}
Let~$a$ satisfy the assumptions of the lemma for some $\alpha$.
Like in the proof of Lemma~\ref{l1939b}, we can assume that $a$ is nonnegative.
For~$d > 0$, we define the function~$\phi_d \colon \ZZ^n\to\RR$ by:
\[ \phi_d(z) \coloneqq 1/(|z|\wedge d)^\alpha ,\]
which is in~$\ell^1(\ZZ^n)$ since $\alpha > n$.
Then the key claim is the following sub-lemma, whose proof is postponed:
\begin{Lem}\label{c9125}
Under the assumptions of Lemma~\ref{l1939b},
there exists some $\rho < 1$ and some $d \in (0,\infty)$ such that, pointwise,
\[\label{f9148} \phi_d * a \leq \rho\phi_d .\]
\end{Lem}

Admitting Lemma~\ref{c9125}, take~$\rho$ and~$d$ such that (\ref{f9148}) is satisfied.
The assumption on~$a$ implies that there exists some $C < \infty$ such that $a \leq C\phi_d$;
therefore by~(\ref{f9148}) one also has $a * a \leq C \phi_d * a \leq \rho C \phi_d$,
whence by~(\ref{f9148}) again $a * a * a \leq \rho C \phi_d * a \leq \rho^2 C \phi_d$,
etc.. In the end,
\[ B[a] \leq C \phi_d + \rho C \phi_d + \rho^2 C \phi_d + \cdots \leq \frac{C}{1-\rho} \phi_d ,\]
which implies that $B[a](z) = O(1/|z|^\alpha)$.
\end{proof}

\begingroup\def\proofname{Proof of Lemma~\ref{c9125}}
\begin{proof}
Denote $S \coloneqq \|a\|_{\ell^1}$, which by hypothesis is $<1$,
and fix $\epsilon \in (0,1/2)$ such that $(1-\epsilon)^\alpha > S$.
Let~$d \in (0,\infty)$, devised to be quite large; our goal is to bound above
$\big(\phi_d * a\big)(z)$ for all~$z \in \ZZ^n$.
Since $\phi_d$ is bounded above by~$d^{-\alpha}$, one has obviously for all~$z \in \ZZ^n$:
\[\label{f2348} \big(\phi_d * a\big)(z) \leq d^{-\alpha} \sum_{z\in\ZZ^n} a(z) = S d^{-\alpha} ,\]
whence $\big(\phi_d * a\big)(z) \leq S \phi_d(z)$ for all~$z$ with $|z|\leq d$.
Since $S<1$, the claim is therefore okay for $|z|\leq d$. 

Now, let~$z\in\ZZ^n$ with $|z|>d$. We have to bound above
\[ \big(\phi_d*a\big)(z) = \sum_{\substack{x,y\in\ZZ^n\\x+y=z}} \phi_d(x)\,a(y) .\]
We decompose this sum into three pieces:
\[ \big(\phi_d*a\big)(z) = \sum_{|y|\leq\epsilon|z|} \phi_d(z-y)\,a(y) +
\sum_{\substack{|x|>\epsilon|z|\\|z-x|>\epsilon|z|}} \phi_d(x)\,a(z-x) +
\sum_{|x|\leq\epsilon|z|} \phi_d(x)\,a(z-x) ,\]
which we shorthand into ``$\onecirc + \twocirc + \threecirc$''.

We bound these three terms separately.
For~\onecirc, we observe that for $|y| \leq \epsilon|z|$, $|z-y| \geq (1-\epsilon)|z|$
by the triangle inequality, thus $\phi_d(z-y) \leq {\big((1-\epsilon)|z|\big)^{-\alpha}} \ab 
= (1-\epsilon)^{-\alpha} \phi_d(z)$, whence by summing:
\[\label{f1034a} \onecirc \leq (1-\epsilon)^{-\alpha} \phi_d(z) \sum_{|y|\leq\epsilon|z|} a(y)
\leq (1-\epsilon)^{-\alpha} S \phi_d(z) .\]
Similarly, for $|x| \leq \epsilon|z|$, $C$ denoting a constant such that $a\leq C\phi_d$,
one has ${a(z-x)} \ab \leq C\big((1-\epsilon)|z|\big)^{-\alpha}$, thus:
\[\label{f1034b} \threecirc \leq (1-\epsilon)^{-\alpha} C \|\phi_d\|_{\ell^1} \, \phi_d(z) .\]
Of course, $\|\phi_d\|_{\ell^1}$ depends on~$d$; the important point is that,
by dominated convergence, $\|\phi_d\|_{\ell^1} \longto 0$ when~$d \longto \infty$.

Finally, provided $d$ is large enough, Term~\twocirc\ will be well approximated by an integral:
\[\label{f1820} \twocirc \leq
\sum_{\substack{x\in\ZZ^n\\|x|,|z-x|>\epsilon|z|}} \frac{1}{|x|^\alpha}\times\frac{C}{|z-x|^\alpha}
\simeq \int_{\begin{subarray}{l}x\in\RR^n\\|x|,|z-x|>\epsilon|z|\end{subarray}}
\frac{C}{(|x|\,|z-x|)^\alpha}\, \dx{x} ,\]
where ``$\simeq$'' means that the ratio between the quantites at each side of that symbol
can be made arbitrarily close to~$1$ when~$d \longto \infty$, uniformly in~$z$.
Indeed, the difference between the sum and the integral
is due to two causes: first, approximating the integral on a unit square of~$\RR^n$
by the value of the integrand at the center of this square, second, summing (or not summing)
terms of the discrete sum corresponding to squares
that are not entirely in the domain of the integral.
For the first cause,
on the domain of the integral, $C/(|x||z-x|)^\alpha$
varies of at most $O(1/|z|)$ in relative value on all the unit squares.
For the second cause, the border of the domain of the integral
is made of two $(n-1)$-dimensional spheres of radius $\epsilon |z|$,
so it crosses $O(|z|^{n-1})$ unit squares. Since $C/(|x||z-x|)^\alpha$
is bounded by~$C(\epsilon(1-\epsilon))^{-\alpha} |z|^{-2\alpha}$
on the domain of the integral, the (absolute) error due to boundary squares
is at most $O(|z|^{n-1-2\alpha})$. As the integral itself is proportional to
$|z|^{n-2\alpha}$ (cf.\ the change of variables below),
the relative error due to boundary squares is at most $O(1/|z|)$ too,
and $O(1/|z|) = o(1)$ since $|z| > \epsilon d$.

Making the change of variables $x = |z|x'$, (\ref{f1820}) becomes:
\[ \twocirc \lesssim C|z|^{n-2\alpha}
\int_{|x'|,|1-x'|>\epsilon} \frac{1}{|x'|^\alpha|1-x'|^\alpha} \, \dx{x'} ,\]
which I shorthand into ``$\twocirc \lesssim \mcal{I}C|z|^{n-2\alpha}$''.
Since $|z|>d$ and $\alpha > n$, this bound implies:
\[\label{f1034c} \twocirc \lesssim \frac{\mcal{I}C}{d^{\alpha-n}} \, \phi_d(z) .\]

Combining~(\ref{f1034a}), (\ref{f1034b}) and~(\ref{f1034c}),
one finally gets that when~$d\longto\infty$, for all~$|z|\geq d$,
\[ \big(\phi_d*a\big)(z) \leq \rho(d) \, \phi_d(z) ,\]
with
\[ \rho(d) = (1-\epsilon)^{-\alpha} (S+C\|a\|_{\ell^1})
+ (1+o(1)) \mcal{I}C/d^{\alpha-n} .\]
$\rho(d)$ tends to~$(1-\epsilon)^{-\alpha} S < 1$ when~$d\longto\infty$,
so it is $<1$ provided $d$ is large enough, which is what we wanted.\linebreak[1]\strut
\end{proof}\endgroup

\bibliographystyle{abbrv}
\bibliography{biblio}

\end{document}